\documentclass[10pt]{article}
\usepackage[utf8]{inputenc}
\usepackage{amsmath,amsthm,amsfonts}
\usepackage{graphicx}
\usepackage{color}
\usepackage{multirow}
\newtheorem{theorem}{Theorem}
\newtheorem{lemma}{Lemma}

\usepackage{subcaption}
\graphicspath{ {./figs/}}

\usepackage{tikz}
\usetikzlibrary{shapes.geometric, arrows, positioning}
\tikzstyle{process} = [rectangle, minimum width=3cm, minimum height=1cm, text centered, draw=black]
\tikzstyle{decision} = [diamond, minimum width=3cm, minimum height=1cm, text centered, draw=black]
\tikzstyle{arrow} = [thick,->,>=stealth]
\tikzstyle{block} = [rectangle, draw, text width=7cm, align=center, text centered, rounded corners, minimum height=4cm]

\newcommand{\rev}[1]{{\color{black}#1}}

\begin{document}

\title{Generalized Multiscale Finite Element Method \\
for discrete network (graph) models}

\author{
Maria Vasilyeva\thanks{Department of Mathematics \& Statistics, Texas A\& M University - Corpus Christi, Corpus Christi, TX, USA.  
Email: {\tt maria.vasilyeva@tamucc.edu}.} 
}

\maketitle

\begin{abstract}
In this paper, we consider a time-dependent discrete network model with highly varying connectivity.  The approximation by time is performed using an implicit scheme.  We propose the coarse scale approximation construction of network models based on the Generalized Multiscale Finite Element Method.  An accurate coarse-scale approximation is generated by solving local spectral problems in sub-networks.  
Convergence analysis of the proposed method is presented for semi-discrete and discrete network models.  We establish the stability of the multiscale discrete network.  
Numerical results are presented for structured and random heterogeneous networks.   
\end{abstract}

\section{Introduction}

Multiphysics models on large networks are used in many applications, for example, pore network models in reservoir simulation, epidemiological models of disease spread, ecological models on multispecies interaction, medical applications such as multiscale multidimensional simulations of blood flow, fibrous materials, electric power systems, and many others \cite{gortz2020numerical}. 
In porous media flow simulation,  instead of direct numerical simulation of the flow in the exact pore geometry-based Navier-Stokes flow models, one can be approximated by a simplified network model of throats and pores \cite{bluntpnm}. This technique reduces the computational complexity and allows for simulation using larger computational domains.
In simulations of the class of insulation materials that are composed of a large number of fibers, the network models are used to represent the high-conductive fibrous \cite{iliev2010fast, gortz2022iterative}
The application of the discrete network model to river network simulations is presented in \cite{betrie2018scalable}. The implementation of the model is based on the PETSc library for high-performance computing systems (HPC).  
The application of the network models to transient hydraulic simulations is considered in \cite{abhyankar2020petsc, maldonado2017scalable} and performed for problems such as water distribution in urban distribution systems, oil distribution, and hydraulic generation. 
In \cite{jalving2017graph},  the authors present a graph-based computational framework that facilitates the construction and analysis of large-scale optimization and simulation applications of coupled infrastructure networks.
The dynamic optimal electric power flow is simulated using network models in \cite{sundermann2023parallel, sundermann2022efficient}.  
In  \cite{della2010distributed}, traffic flows are considered in complex networks, where the model is based on a graph or network of streets in which vehicles can move.
The application of the spatial networks for disease transmission in epidemiological models is considered in \cite{keeling2005networks}. 
In \cite{reichold2009vascular}, a cerebral blood flow is modeled as fluid flow driven through a network of resistors by pressure gradients. The authors introduce a vascular graph modeling framework based on these principles to compute blood pressure, flow, and scalar transport in realistic vascular networks. 
The simulation of blood flow in microvascular networks and the surrounding tissue is considered in \cite{vidotto2019hybrid}. To reduce the computational complexity of this issue, the network structures are modeled by a one-dimensional graph whose location in space is determined by the centerlines of the three-dimensional vessels. 
Despite eliminating a significant portion of complexity through this approach, efficiently solving the resulting model remains a common challenge.
Homogenization and multiscale methods are commonly used techniques for implementing upscaling to manage the extensive computational complexity associated with large network models.

Multiscale problems arise in many areas of science and engineering and typically involve multiple spacial length scales.   Traditional numerical methods, such as finite element or finite volume methods, can become prohibitively expensive when the number of degrees of freedom required to capture all relevant scales becomes large.
To address this issue, multiscale and homogenization methods have been developed that seek to efficiently capture the essential features of the problem at each length scale.  
In homogenization methods, the fine-scale problem is replaced by an equivalent coarse-scale problem that captures the essential features of the original problem at a lower computational cost \cite{sanchez1980non, jikov2012homogenization, bakhvalov2012homogenisation, allaire1992homogenization}.  The effective properties for the coarse grid approximation are computed by analyzing the behavior of unit cells. 
The most common homogenization method is based on periodicity, which assumes that the fine-scale problem can be represented as a periodic array of repeating unit cells. 
Compared to the homogenization techniques, a multiscale methods form a broader class of numerical techniques, for example,  the Multiscale Finite Element Method  \cite{hou1997multiscale, efendiev2009multiscale},  the Heterogeneous Multiscale Method \cite{abdulle2012heterogeneous}, the Local Orthogonal Decomposition method \cite{henning2014localized},  the Variational Multiscale Method \cite{hughes1998variational}, the Generalized Multiscale Finite Element Method \cite{efendiev2011multiscale, efendiev2013generalized},  the Multiscale Finite Volume Method  \cite{lunati2006multiscale, hajibeygi2008iterative} and many others. 
Most multiscale methods are based on constructing the multiscale basis functions in the local domains to capture fine-scale behavior.

There are various applications of multiscale and upscaling methods for network problems that have been studied extensively.
Upscaling traffic flows in complex networks is considered in \cite{della2010distributed}. The problem is considered in two-dimensional regions whose size (macroscale) is much greater than the characteristic size of the network arcs (microscale).  
The numerical upscaling method is presented in \cite{kettil2020numerical}, where a finite element approximation with a  localized orthogonal decomposition method represents the macroscale model. Moreover, the application to a two-dimensional network model of paper-based materials in the form of fiber networks is considered \cite{gortz2022iterative, kulachenko2012direct}.   
In \cite{gortz2020numerical}, the nodal displacement in a fiber network model is analyzed using a multiscale method based on the LOD. 
In \cite{chu2012multiscale, chu2013multiscale}, the heterogeneous multiscale method (HMM) is proposed to couple a network model on the microscale with a continuum scale over the same physical domain.
The coarsening procedures for graph Laplacian problems written in a mixed saddle point form were presented in \cite{barker2017spectral, barker2021multilevel}. 
The numerical methods for computing the effective heat conductivity of fibrous insulation materials are presented in  \cite{iliev2010fast}. The fast algorithm is constructed based on the upscaling procedure. It contains the solution of the auxiliary boundary value problems of the steady-state heat equation in a representative elementary volume occupied by fibers and air. The presented approach ignores air and is further simplified by taking advantage of the slender shape of the fibers and assuming that they form a network. 
A multiscale method for networks representing flows in a porous medium is presented in \cite{chu2012multiscale, chu2013multiscale}. 
In \cite{balhoff2008mortar},  the mortar coupling is presented to couple pore-scale network models to additional pore-scale or continuum-scale models using mortars. Mortars are finite-element spaces in two dimensions connecting distinct subdomains by ensuring pressure and flux continuity at their shared interfaces.
In \cite{reichold2009vascular}, a cerebral blood flow is modeled as fluid flow in a complex network.   The authors construct an upscaling algorithm that significantly reduces the computational cost. Furthermore,  the upscaled model no longer requires extensive information regarding the topology of the capillary bed.
The reduction of the computational complexity of the simulation of blood flow in microvascular networks and the surrounding tissue is considered in  \cite{vidotto2019hybrid}. The authors employ homogenization to the microvascular network's fine-scale structures in the study, leading to a new hybrid approach. This approach models the fine-scale structures as a heterogeneous porous medium, while the larger vessels' flow is modeled using one-dimensional flow equations.

This paper introduces the novel approach for upscaling the complex network model based on the Generalized Multiscale Finite Element Method \cite{efendiev2011multiscale, efendiev2013generalized}. The GMsFEM has a significant advantage in incorporating small-scale features from heterogeneities into coarse-grid basis functions. The multiscale basis functions are constructed by solving local eigenvalue problems. The online solutions can be calculated for any suitable boundary condition or forcing by these greatly reduced-dimension multiscale basis functions.
In this work, we present the construction of the reduced-order model for complex, highly heterogeneous networks. The network model represents the fine-scale model. In contrast, the coarse-scale approximation is represented by a much coarser finite element mesh than the fine-scale network. We design multiscale basis functions to account for the networks' microscale features by solving the local spectral problems in the primary local network cluster.   The constructed multiscale approximation can handle highly varying connectivity and random network structure with a large system size reduction. Convergence analysis of the proposed method is presented for semi-discrete and discrete network models. We establish the stability of the multiscale discrete network for implicit approximation.   
Numerical implementation of the network model is performed based on the DMNetwork framework. DMNetwork is a part of PETSc library  \cite{balay2019petsc} for high-resolution multiphysics simulations on the large-scale complex network  \cite{abhyankar2020petsc, maldonado2017scalable}. DMNetwork provides data and topology management, parallelization for multiphysics systems over a network, and hierarchical solvers. 
We present the main components of the multiscale method for constructing the coarse-scale system.   
To test the presented upscaling approach, we consider regular cubic lattice networks with and without random elimination processes and random networks \cite{gostick2016openpnm, raoof2010new}. We show that the multiscale method can provide an accurate solution with a large system size reduction and fine-scale solution reconstruction. 
\rev{This paper aims to fill the gap in the existing literature, where the most effort is given to the homogenization and upscaling techniques for discrete network models. Regular flux averaging defines one macroscopic variable to capture all underlying physics and may not work well for complex high-construct networks. 
Several recent multiscale methods developed for network models differ from homogenization approaches \cite{balhoff2008mortar, chu2012multiscale, chu2013multiscale, kettil2020numerical, gortz2022iterative}. However, these approaches are still based on a single macroscopic variable for each coarse region and are inaccurate for problems with high contrast and complex structure. In the proposed approach, we use local eigenvectors to accurately represent fine-scale behavior on a coarse level, offering a fresh perspective on the network simulations and generalizing an existing GMsFEM-based approach for a larger class of heterogeneous structures and applications. }

The paper is organized as follows. In Section 2, we consider the fine-scale model of the complex network in semi-discrete and discrete forms with stability analysis. In Section 3, we present the construction of the coarse-scale model using the multiscale method and provide a priory estimate.   The numerical results are presented in Section 4 for structured and random heterogeneous three-dimensional network models. The conclusion and future works discussion are given in Section 5.

\section{Problem formulation}

We consider network represented as a undirected graph $G = (P, E)$, where 
$P = \{v_1, v_2, ..., v_{N_v} \}$ is a set of vertices (nodes $v_i \in \mathcal{R}^d$, $d=2,3$) and 
$E$ is a set of two-elements subsets of $P$ (connections or edges $e_{ij} = \{v_i, v_j \}$ that connect vertices $v_i$ (head) and $v_j$(tail) and $i \neq j$).  
Here $N_v$ is the total number of vertices/nodes, and $N_e$ is the total number of connections \cite{gallier2016spectral}.  We suppose that the graph is connected and weighted. 
Furthermore, we assume that the network is embedded into the rectangular cuboid
\[
\Omega  = [0, L_1] \times ... \times [0, L_d]. 
\]

\begin{figure}[h!]
\centering
\includegraphics[width=0.25\linewidth]{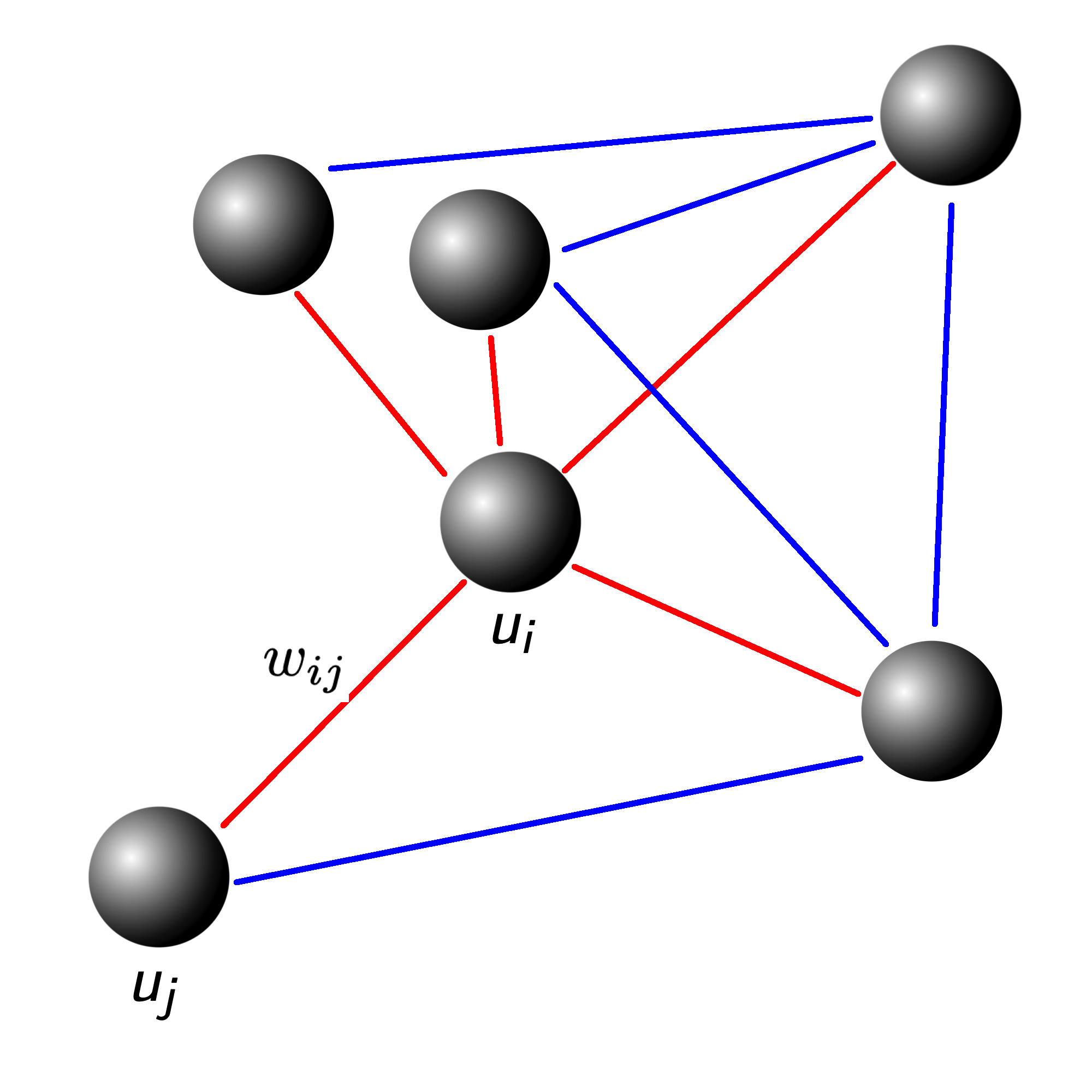}
\caption{Nodes and connections (edges)}
\label{fig:conn}
\end{figure}

We assign a heterogeneous property $c_i$ for each node $x_i$  that can be associated with the volume in the pore-network model \cite{bluntpnm}.  Then, we associate a weight $w_{ij}$  to each connection proportional to the area of the edge and inversely proportional to the distance between nodes \cite{gostick2016openpnm} (Figure \ref{fig:conn}).  
Additionally, we label nodes associated with the top and bottom boundaries ($\Gamma \subset \partial \Omega$)  to set Dirichlet boundary conditions.


\begin{figure}[h!]
\centering
\begin{subfigure}{1\textwidth}
\centering
\includegraphics[width=0.32\linewidth]{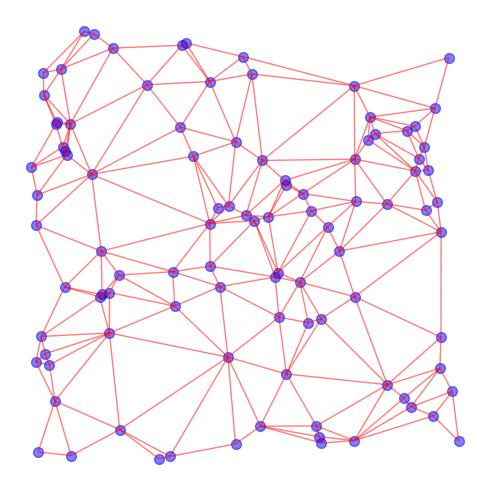}
\includegraphics[width=0.32\linewidth]{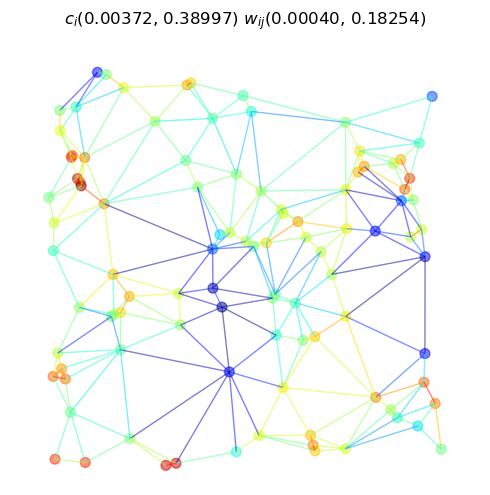}
\includegraphics[width=0.26\linewidth]{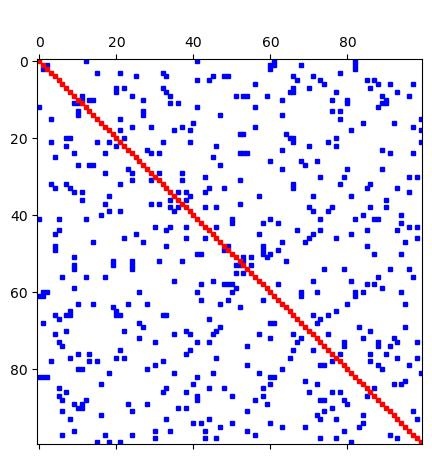}
\caption{\rev{2D network $G$ with 100 nodes and 260 connections; random weights $w_{ij}$ of connections and random node coefficient $c_i$; and structure of the graph Laplacian ($L(G)$), where degree matrix ($D(G)$) elements  are depicted in red color and weight matrix ($W(G)$) elements are depicted in blue color.}}
\end{subfigure}

\vspace{5pt}

\begin{subfigure}{1\textwidth}
\centering
\includegraphics[width=0.35\linewidth]{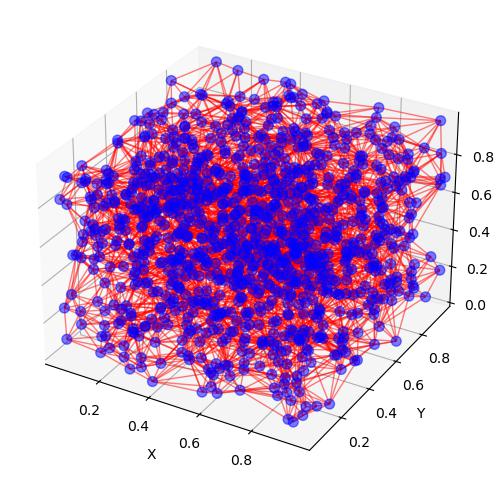}
\includegraphics[width=0.35\linewidth]{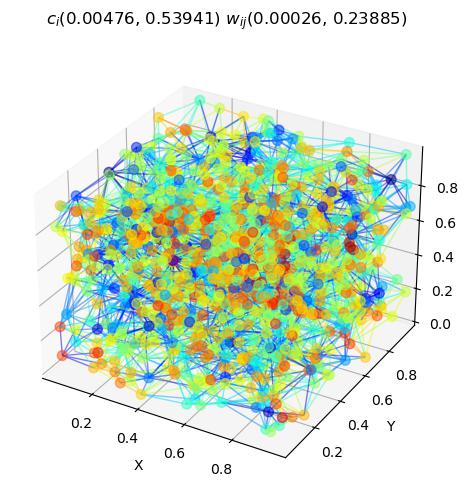}
\includegraphics[width=0.26\linewidth]{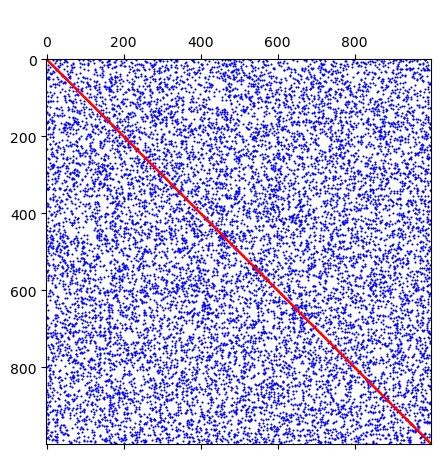}
\caption{\rev{3D network $G$ with 1000 nodes and 6674 connections; random weights $w_{ij}$ of connections and random node coefficient $c_i$; and structure of the graph Laplacian ($L(G)$), where degree matrix ($D(G)$) elements  are depicted in red color and weight matrix ($W(G)$) elements are depicted in blue color.}}
\end{subfigure}
\caption{\rev{Illustration of the 2D and 3D networks and corresponding graph Laplacian.}}
\label{fig:net-mat}
\end{figure}

Let $D(G)$ be a degree matrix, $W = \{w_{ij}\}_{i,j=1}^{N_v}$ be a symmetric weight matrix ($w_{ij} \geq 0$,  $w_{ii} = 0$ and $w_{ij} > 0$ if $e_{ij}$ is an edge) and $L(G) = D(G) - W$ be a graph Laplacian of $G$, where  $D(G) = \text{diag}(d_1,  \ldots, d_{N_v})$ with $d_{i} = \sum_{j=1}^{N_v} w_{ij}$ \cite{gallier2016spectral}.  Note that we will write $L$ instead of $L(G)$ for simplicity.  
\rev{In Figure \ref{fig:net-mat}, we illustrate the two and three-dimensional networks with the corresponding graph Laplacian. We show a typical graph structure with irregular connections in the first column, where we depict nodes in blue and connections in red. In the second column, we plot random coefficients $w_{ij}$ and $c_i$ that are represented by different colors of connections and nodes, respectively. The plot of coefficients is given in the logarithmic scale, and the range of coefficient values is depicted in the title. In the third column, we present a structure of the corresponding graph Laplacian. From the plot, we can observe how network (graph) connectivity affects the matrix sparsity. The size of the matrix is $N_v \times N_v$, where $N_v$ is the number of nodes. The elements corresponding to the degree matrix ($D$) are depicted in red; elements corresponding to the weight matrix $W$ are depicted in blue, and the resulting graph Laplacian is $L = D-W$. }  
For network generation, we use OpenPNM library \cite{gostick2016openpnm}.  We note that the method can be constructed for a general case without embedding it into a hyper-rectangle.   However, it will affect the coarse-grid construction, and a more general way should be considered based on the graph partitioning.

For all $u \in \mathcal{R}^{N_v}$, we have 
\begin{equation}
(Lu)_i = \sum_{j \sim  i} w_{ij} (u_i - u_j),
\end{equation}
where we write $j \sim  i$ if  nodes $x_i$ and $x_j$ are connected. 

We suppose that $w_{ij}$ are bounded weights ($0 < \underline{w} \leq w_{ij} \leq \overline{w} < \infty$) then   
\begin{equation}
\begin{split}
u^T L u
&= u^T D u - u^T W u
= \sum_{i=1}^{N_v} d_i u_i^2 - \sum_{i, j=1}^{N_v} w_{ij} u_i u_j\\
&= \frac{1}{2} \left( 
 \sum_{i=1}^{N_v} d_i u_i^2  
 - 2 \sum_{i, j=1}^{N_v}  w_{ij} u_i u_j 
 + \sum_{i=1}^{N_v} d_i u_i^2
\right)
 = \frac{1}{2} \sum_{i, j=1}^{N_v} w_{ij} (u_i - u_j)^2 \geq 0,
\end{split}
\end{equation}
and operator $L$ is positive semidefinite. 

Furthermore, we can write the following representation of the graph Laplacian $L$ 
\[
L = B M B^T,
\]
where matrix $M=\{m_{ij}\}_{i,j=1}^{N_e}$ is the diagonal matrix filled with the edge weights and $B = \{b_{ij}\}$ is the vertex-edge incident matrix with size $N_v \times N_e$ and element $b_{ij} = 1$ if $v_i$ is a head of edge $e_j$,  $b_{ij} = -1$ if $v_i$ is a tail of edge $e_j$ and zero otherwise.  
Them, we have  $(Lu, u) = u^T L u = (B^T u)^T M (B^T u) = \frac{1}{2} \sum_{i, j=1}^{N_v} w_{ij} (u_i - u_j)^2$.

\subsection{Semi-discrete network model}

\rev{We consider a sufficiently fine weighted undirected graph $G$ embedded into the cube (see Figure \ref{fig:netf}).} 
On $G$, we consider the following time-dependent  problem 
\begin{equation}
c_i \frac{\partial u_i}{\partial t}  + \sum_j w_{ij} (u_i - u_j) = f_i,  
\quad \forall i  \in \mathcal{N},   
\quad 0 < t \leq T,
\label{mm00}
\end{equation}
where $\mathcal{N}=\{1,\ldots,N_v\}$ is the set of indexes,  $u_i$ is defined on on the node $x_i$, $f_i$ is the given nodal source term and $c_i$ is the bounded coefficient, $0 < \underline{c} \leq c_i \leq \overline{c} < \infty$.

We consider equations \eqref{mm00} with initial conditions
\begin{equation}
u_i = u_{i,0}, 
\quad \forall i  \in \mathcal{N}, 
\quad t = 0, 
\end{equation}
and boundary conditions 
\begin{equation}
u_i = g_i, 
\quad i \in \mathcal{N}_D, 
\quad  0 < t \leq T,
\end{equation}
where $\mathcal{N}_D$ is the set of node corresponded to the boundary \cite{kettil2020numerical}.

\begin{figure}[h!]
\centering
\begin{subfigure}{0.45\textwidth}
\includegraphics[width=1\linewidth]{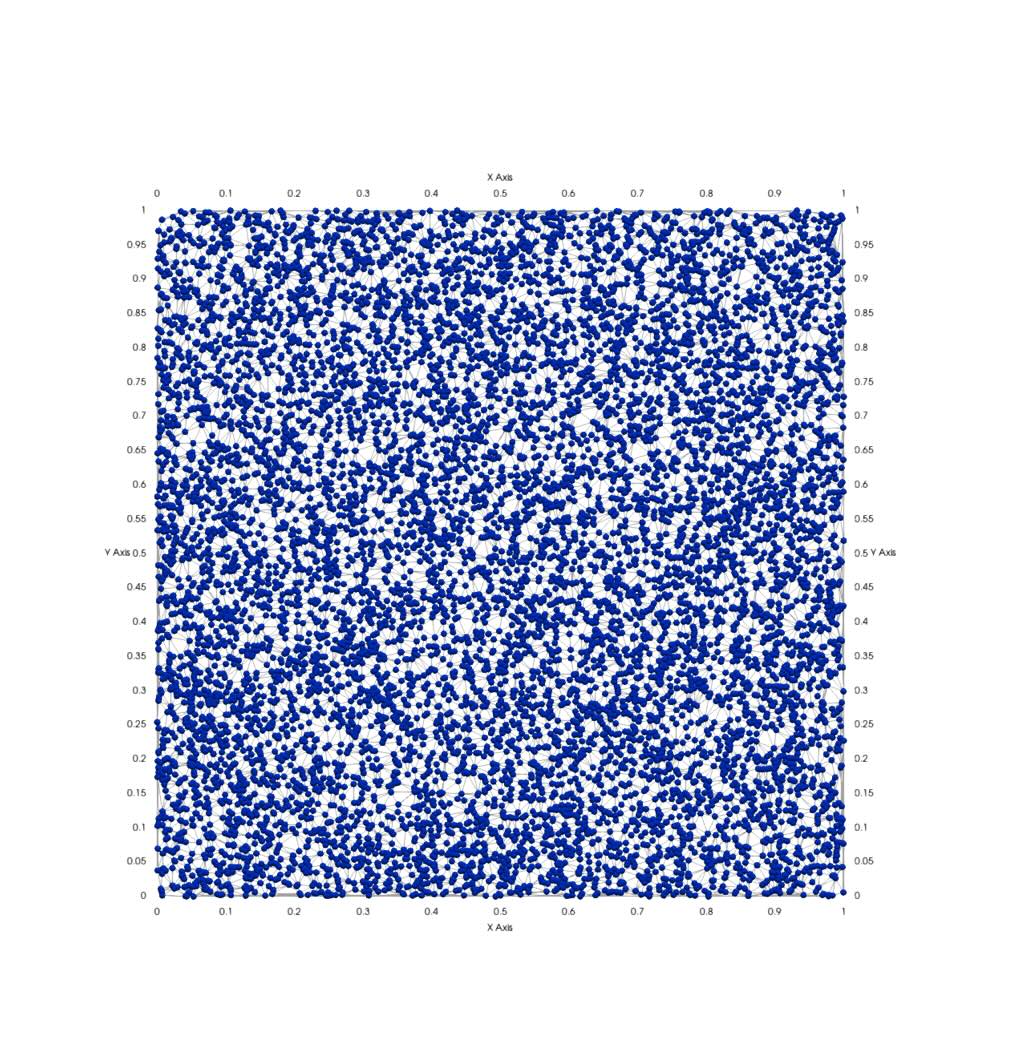}
\caption{2D}
\end{subfigure}
\ \ \ \ 
\begin{subfigure}{0.45\textwidth}
\includegraphics[width=1\linewidth]{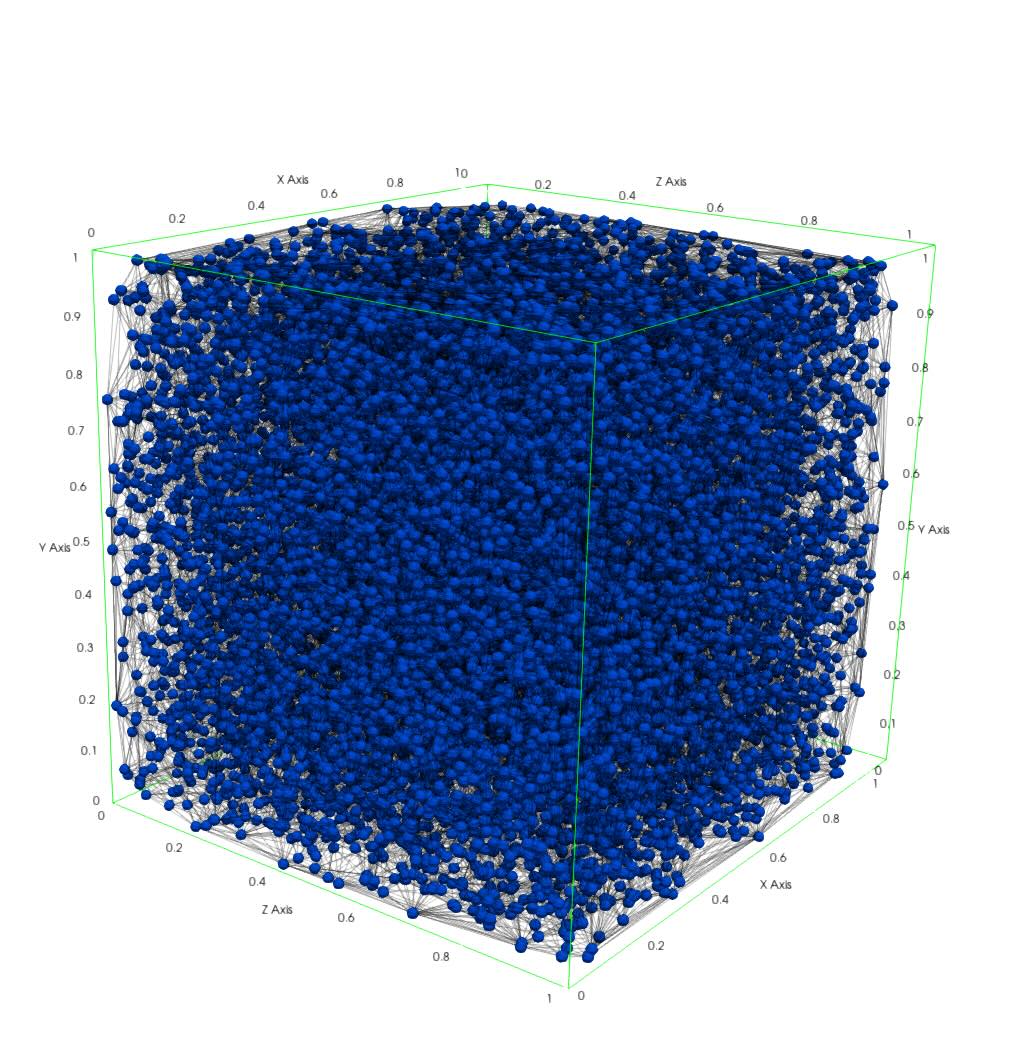}
\caption{3D}
\end{subfigure}
\caption{\rev{Illustration of two and three-dimensional fine-scale networks embedded into the cube }}
\label{fig:netf}
\end{figure}

Let $u = (u_1, \ldots, u_{N_v})^T \in \mathcal{R}^{N_v}$ be the vector defined on the set of nodes,  $f = (f_1, \ldots, f_{N_v})^T \in \mathcal{R}^{N_v}$ be the given source term vector, and $C$ be the diagonal matrix, $C=\text{diag}(c_1,  \ldots, c_{N_v})$. Then we can write problem \eqref{mm0} in the following matrix form
\begin{equation}
C \frac{\partial u}{\partial t}  + L u = f, \quad 0 < t \leq T.
\label{mm0}
\end{equation}
with initial conditions $u = u_0$ for $t=0$. 
Note that, in this formulation, the matrices and right-hand side vector are modified to incorporate boundary conditions. For instance, to have a symmetric matrix, we do not include boundary nodes and impose boundary conditions by setting flux on the nodes connected to the boundary.

Let  $V \subset \mathcal{R}^{N_v}$ be the subspace of real-valued functions defined on the set of nodes satisfying Dirichlet boundary conditions  \cite{hauck2022super}.  Then we can write problem formulation in the following equivalent form: find $u \in  V$ such that 
\begin{equation}
\left( C \frac{\partial u}{\partial t}, v \right)  + (L u, v) = (f, v), 
\quad \forall v \in V, 
\quad 0 < t \leq T,
\label{mm1}
\end{equation}
where 
$(u, v) = u^T v$ and $(Lu, v) = u^T L v$ is scalar products.

Next, we introduce the following two norms $||u|| = \sqrt{(u,u)}$ and $||u||_L = \sqrt{ (L u, u)}$  and  show a stability estimate for a semi-discrete fine-scale scheme \eqref{mm1} \cite{samarskii2001theory, vabishchevich2013additive, vasilyeva2023efficient, kolesov2014splitting}. 

\begin{lemma}
\label{t:t1}
The  solution of the problem \eqref{mm1} satisfies the following a priory estimate
\begin{equation}
\label{t1}
||u(t)||_{L}^2
 \preceq 
 ||u_0||_{L}^2 
+\int_0^t ||f||^2_{C^{-1}} ds.
\end{equation}
\end{lemma}
\begin{proof}
Let $v = \frac{du}{dt}$ in  \eqref{mm1} then we have
\begin{equation}
\left(C \frac{du}{dt}, \frac{du}{dt}\right) 
+ \left( L u,   \frac{du}{dt} \right) = \left(f, \frac{du}{dt} \right). 
\end{equation}
Using Young's inequality 
and 
$\frac{1}{2} \frac{d}{dt} \left(L u,  u\right) =  \left( Lu,   \frac{du}{dt} \right)$, we obtain the following estimate
\begin{equation}
\left\| \frac{du}{dt} \right\|^2_C  + \frac{1}{2} \frac{d}{dt} \left( Lu,  u \right) 
=
 \left(f, \frac{du}{dt} \right)
\leq
\left\| \frac{du}{dt} \right\|^2_C + \frac{1}{4} \left\| f \right\|^2_{C^{-1}},
\end{equation}
or
\begin{equation}
\frac{d}{dt}  \left\| u \right\|_{L}^2
\leq 
\frac{1}{2} \left\| f \right\|^2_{C^{-1}}.
\end{equation}
Finally,  after integration by time, we obtain inequality \eqref{t1} . 
\end{proof}

\rev{The Lemma \ref{t:t1} gives the simplest stability estimates of the solution on the initial data and the right-hand side \cite{vabishchevich2013additive}.}

\subsection{Time-approximation and discrete network model}

Let $u^n = u(t_n)$ and $u^{n-1} = u(t_{n-1})$, where $t_n = n \tau$,  $n=1,2, ...$ and $\tau > 0$ be the uniform time step size.  
For approximation by time, we use backward Euler's  to obtain an implicit fully discrete scheme: find $u^n \in V$ such that 
\begin{equation}
\label{mm2}
\left( C \frac{ u^n - u^{n-1} }{\tau}, v\right) + (L u^n, v) = (f^n, v), 
\quad \forall v \in V, 
\quad n = 1,2,...
\end{equation}
with initial conditions $u^0 = u_0$. 
Similarly to the semi-discrete problem \eqref{mm1}, we have the following stability estimate \cite{samarskii2001theory, vabishchevich2013additive}. 

\begin{lemma}
\label{t:t2}
The solution of the  problem \eqref{mm2} is unconditionally stable and satisfies the following estimate
\begin{equation}
\label{t2}
||u^n||_L^2 
\preceq 
||u_0||_L^2 
+ \tau \sum_{k=1}^n ||f^k||^2_{\left(C + \frac{\tau}{2} L \right)^{-1}}.
\end{equation}
\end{lemma}
\begin{proof}
The equation \eqref{mm2} can be written as follows
\begin{equation}
\left( C \frac{ u^n - u^{n-1} }{\tau}, v\right) + (L u^n, v) 
= 
\left( \left(C + \frac{\tau}{2} L \right) \frac{ u^n - u^{n-1} }{\tau}, v  \right) 
+ \frac{1}{2} (L (u^n  + u^{n-1}), v) = (f^n, v).
\end{equation}
Let $v = \frac{ u^n - u^{n-1} }{\tau}$ then we have
\begin{equation}
\left( \left(C + \frac{\tau}{2} L \right) \frac{ u^n - u^{n-1} }{\tau} , \frac{ u^n - u^{n-1} }{\tau} \right) 
+ \frac{1}{2}  \left( L (u^n + u^{n-1}),   \frac{ u^n - u^{n-1} }{\tau} \right) = \left(f^n, \frac{ u^n - u^{n-1} }{\tau} \right). 
\end{equation}
Using  the following inequality for the right-hand side
\begin{equation}
\left(f^n, \frac{ u^n - u^{n-1} }{\tau} \right) \leq 
\left\| \frac{ u^n - u^{n-1} }{\tau} \right\|_{ \left(C + \frac{\tau}{2} L \right) }^2 
+ \frac{1}{4} 
\left\| f^n \right\|_{ \left(C + \frac{\tau}{2} L \right)^{-1}}^2,
\end{equation}
we obtain the following estimate for $L = L^T$
\begin{equation}
\frac{1}{2 \tau} \left( L (u^n + u^{n-1}),   (u^n - u^{n-1})  \right)
=
\frac{1}{2 \tau} (L u^n,u^n) + 
\frac{1}{2 \tau} (L u^{n-1}, u^{n-1}) 
\leq 
\frac{1}{4} 
\left\| f^n \right\|_{ \left(C + \frac{\tau}{2} L \right)^{-1}}^2,
\end{equation}
or
\begin{equation}
||u^n||^2_L \leq ||u^{n-1}||^2_L  + \frac{\tau}{2} 
\left\| f^n \right\|_{ \left(C + \frac{\tau}{2} L \right)^{-1}}^2.
\end{equation}
\end{proof}
\rev{The Lemma \ref{t:t2}  refer to general stability results for operator-difference schemes and show that the two-level time approximation \eqref{mm2} using the backward Euler method for a discrete network model is unconditionally stable \cite{vabishchevich2013additive,vabishchevich2024computational}.}
We note that the presented discrete network model with implicit approximation by time leads to solving the large system of equations on each time step. To reduce the size of the system, we use a homogenization approach and construct a coarse-scale approximation by introducing local spectral multiscale basis functions.

\section{Multiscale model order reduction}

Multiscale methods form a broad class of numerical techniques where most multiscale methods are based on constructing the multiscale basis functions in the local domains to capture fine-scale behavior.  In this work, we construct a local spectral multiscale basis function for the network model described above and follow the procedure defined in the Generalized Multiscale Finite Element Method (GMsFEM) \cite{akkutlu2016multiscale, efendiev2013generalized, chung2016generalized}.

\begin{figure}[h!]
\centering
\begin{subfigure}{0.4\textwidth}
\includegraphics[width=1\linewidth]{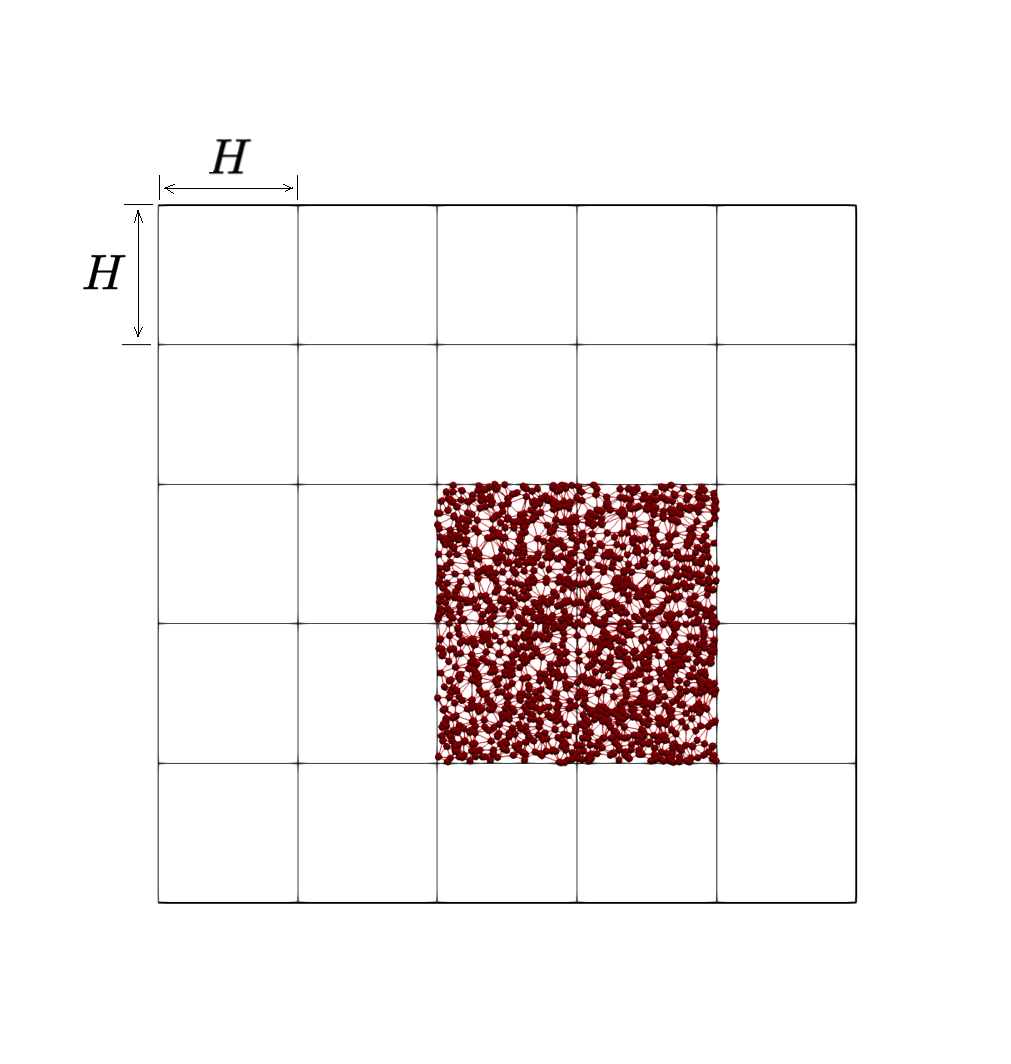}
\caption{2D}
\end{subfigure}
\ \ \ \ 
\begin{subfigure}{0.4\textwidth}
\includegraphics[width=1\linewidth]{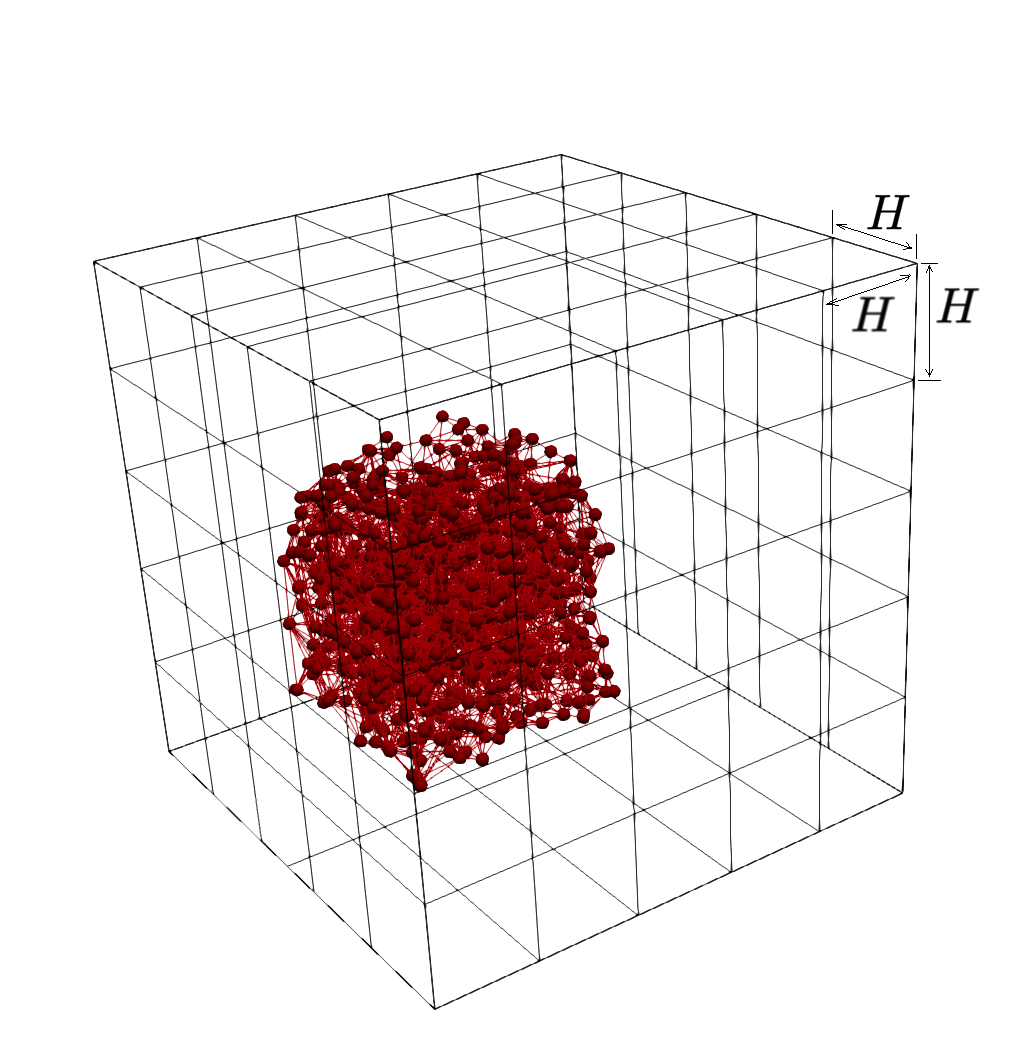}
\caption{3D}
\end{subfigure}
\caption{Illustration of coarse grid $\mathcal{T}_H$ and local domain $\omega_i$ with sub-network $G^{\omega_i}$}
\label{fig:netc}
\end{figure}

Let $\mathcal{T}_H$ be a coarse mesh in $\Omega$
\[
\mathcal{T}_H = \bigcup_{i=1}^{N_c} K_i,
\]
where $K_i$ is the coarse cell, and $N_c$ is the number of coarse cells. 
We denote $\{y_i\}_{i=1}^N$ the nodes of the coarse mesh $\mathcal{T}_H$, where $N$ is the number of coarse grid nodes \cite{abreu2019convergence}.  We define $\omega_i$ as a neighborhood of the node $y_i$
\[
\omega_i = \bigcup \{K_j \in \mathcal{T}_H: \ y_i \in \bar{K_j} \},
\]
and $\omega^{K_j}$ as a the neighborhood of the element ${K_j}$
\[
\omega^{K_j} = \bigcup\{\omega_i : K_j \in \omega_i  \}.
\]
In this work we suppose that $\Omega  = [0, L_1] \times ... \times [0, L_d]$, where $d = 2,3$ is the dimension.  Furthermore,  we consider a uniform mesh with square and cubic cells for simplicity
\[
K_i = [y_i, y_i+H e_1] \times ... \times [y_i, y_i+H e_d],
\]
where $H$ is the coarse grid size.

Next, we associate a sub-network $G^{\omega_i}$ to each $\omega_i$. 
In Figure \ref{fig:netc}, we depict a coarse $5 \times 5$ mesh for 2D network and $5 \times 5 \times 5$ coarse grid for 3D network with local domain $\omega_i$ and corresponded sub-network $G^{\omega_i}$.  
Note that in implementing the sub-network extraction, we ensure that any node from the global graph $G$ is contained on exactly one coarse element. 

The objective of this paper is to develop and analyze a multiscale finite element method for discrete network models defined above.  The main idea of the multiscale method is in construction of the coarse-scale approximation of the network by construction of the accurate multiscale basis functions $\{\rev{{\psi}_r^{\omega_i}}\}_{r=1}^{M_i}$ in each local domain $G^{\omega_i}$ (sub-network), where $M_i$ is the number of local basis functions.  
In the multiscale method, we define offline and online stages:
\begin{itemize}
\item 
\textit{Offline stage:}
\begin{itemize}
\item Coarse grid ($\mathcal{T}_H$) and local domains construction (sub-network, $G^{\omega_i}$).
\item Multiscale space construction by the solution of the local eigenvalue in each $G^{\omega_i}$
\[
V_{ms} = \text{span} \{\rev{{\psi}_r^{\omega_i}}: 1 \leq i \leq N, \ 1 \leq r \leq M_i  \}.
\]
\end{itemize} 
\item 
\textit{Online stage: }  
\begin{itemize}
\item Solution of the coarse-scale problem (Galerkin approximation): find $u \in V_{ms}$ such that 
\begin{equation}
\left( C \frac{\partial u_{ms}}{\partial t}, v_{ms} \right)  + (L u_{ms}, v_{ms}) = (f, v_{ms}), 
\quad \forall v_{ms} \in V_{ms}.
\label{ms1}
\end{equation}
\end{itemize}
\end{itemize}

We use a same-time approximation as for fine-scale problem with $u_{ms}^n = u_{ms}(t_n)$, where $t_n = n \tau$,  $n=1,2, ...$ and $\tau > 0$.  
By backward Euler's scheme, we obtain a fully discrete scheme: find $u_{ms}^n \in V_{ms}$ such that 
\begin{equation}
\label{ms2}
\left( C \frac{ u_{ms}^n - u_{ms}^{n-1} }{\tau}, v_{ms}\right) + (L u_{ms}^n, v_{ms}) = (f^n, v_{ms}), 
\quad \forall v_{ms} \in V_{ms}, 
\quad n = 1,2,...
\end{equation}

Next, we discuss details of the multiscale space construction and coarse-scale model generation with some implementation aspects.

\subsection{Multiscale basis functions and algorithm}

In this work, we use the OpenPNM \cite{gostick2016openpnm} to construct a network and assign properties. The constructed network is saved in the text files with the corresponding properties. To implement the numerical solution of the fine-scale discrete network model, we use a PETSc library with DMNetwork framework \cite{abhyankar2020petsc, maldonado2017scalable}. DMNetwork provides data and topology management, parallelization for multiphysics systems over a large network, and hierarchical solvers. 

\rev{We start with extracting the local subnetwork. For the general case, we applied a preprocessing step for the network extraction to avoid possible disconnected clusters. Let $\tilde{G}^{\omega_i}$ be the subnetwork of $G$ that contains all nodes and connections within the local domain $\omega_i$. Then, we extract a main cluster $G^{\omega_i} \subset \tilde{G}^{\omega_i}$ on which we solve a local eigenvalue problem. We suppose that the size of the primary cluster is much larger than the others. In order to incorporate disconnected clusters into the local multiscale space, we use indicator vectors $\eta_i=[\eta_{i,1}, \ldots, \eta_{i,N_v^{\omega_i}}]$ with $\eta_{i,l} = 0$ if $x_l \in G^{\omega_{i}}$ and 1 otherwise. Note that we solve an eigenvalue problem in $G^{\omega_i}$ to incorporate local fine-scale behavior and construct an accurate approximation.} For network preprocessing, we use existing tools from the OpenPNM library \cite{gostick2016openpnm}.  In the numerical implementation of the eigenvalue problem solution, we use the SLEPc library \cite{hernandez2005slepc}. 

For the local network $G^{\omega_i}$, we have
\rev{
\begin{equation}
L^{\omega_i} = D^{\omega_i} - W^{\omega_i},  \quad 
D^{\omega_i} = \text{diag}(d_1, \ldots,d_{N_v^{\omega_i}}), \quad 
d_i = \sum_{j=1}^{N_v^{\omega_i}} w_{ij}, 
\end{equation}
where $L^{\omega_i}$, $W^{\omega_i}$ and  $D^{\omega_i}$ are the matrices defined on the sub-network $G^{\omega_i}$ and  $N_v^{\omega_i}$ is the number of nodes of local network $G^{\omega_i}$ associated with the local domain $\omega_i$.
Based on the preprocessing and considering the main cluster, we suppose that $G^{\omega_i}$ has no isolated vertices, and therefore the local degree matrix $D^{\omega_i}$ contains positive entities and is invertible. }


To construct a multiscale basis functions in each sub-network $G^{\omega_i}$, we solve a local generalized eigenvalue problem 
\begin{equation}
\label{sp1}
L^{\omega_i} {\phi}_r^{\omega_i}
= \lambda_r^{\omega_i} D^{\omega_i} {\phi}_r^{\omega_i}.
\end{equation}
Then, we choose first $M_i$ eigenvectors corresponded to the smallest eigenvalues,  $\phi_r^{\omega_i}$, $r = 1,\ldots,M_i$ in each $G^{\omega_i}$ with $i = 1,\ldots,N$.  
For the  sub-network $G^{\omega_i}$,  the local matrix $L^{\omega_i}$ is positive semidefinite for $w_{ij} > 0$.  Therefore, the eigenvalues $0 = \lambda^{\omega_i}_1 \leq \lambda^{\omega_i}_2 \leq \ldots \leq \lambda^{\omega_i}_{N_v^{\omega_i}}$  of $L^{\omega_i}$ are real and nonnegative,  and there is an orthonormal basis of eigenvectors of  $L^{\omega_i}$;  and the smallest eigenvalue $\lambda^{\omega_i}_1 = 0$ and corresponded eigenvector is one  \cite{gallier2016spectral}. 

Then similarly to the  GMsFEM for standard finite element-based definition, we use a generalized eigenvalue problem that gives a good approximation space \cite{efendiev2013generalized, abreu2019convergence}. 
We choose first $M_i$ eigenvectors corresponded to the smallest eigenvalues,  $\phi_r^{\omega_i}$, $r = 1,\ldots,M_i$ in each $G^{\omega_i}$ with $i = 1,\ldots,N$.  
In order to form a continuous space, we multiply eigenvalues to the linear partition of unity functions and form a multiscale space
\begin{equation}
V_{ms} = \text{span} \{\rev{{\psi}_r^{\omega_i}}: 1 \leq i \leq N, \ 1 \leq r \leq M_i  \}.
\end{equation}
where
\rev{${\psi}_{1}^{\omega_i}= \chi_i \eta_i$, ${\psi}_{{r+1}}^{\omega_i}= \chi_i \phi_r^{\omega_i}$ and $\chi_i$ is the projection of the linear partition of unity functions in $\omega_i$ to the nodes of corresponded sub-network $G^{\omega_i}$.}

We implement the presented multiscale algorithm based on the algebraic way by forming the projection matrix 
\begin{equation}
R =[  
\rev{{\psi}_1^{\omega_1}}, \ldots, \rev{{\psi}_{M_1}^{\omega_1}}, \ldots ,  
\rev{{\psi}_{1}^{\omega_N}}, \ldots,  \rev{{\psi}_{M_N}^{\omega_N}} ]^T.
\end{equation}
\rev{We form a projection matrix in an offline stage and use it to construct and solve a coarse grid system. }


\begin{figure}[h!]
\centering
\begin{tikzpicture}[node distance=2cm]
\node (local_block) [block] at (0, 0) {
    \textbf{Offline stage}\\
    (local calculations)\\[0.3cm]
    \begin{tikzpicture}[node distance=1.5cm, scale=0.8, transform shape]
        \node (start_local) [process] {Identify nodes and connections within $\omega_i$};
        \node (local_calc1) [process, below of=start_local,  yshift=-0.5cm] {Extract subnetwork $\tilde{G}^{\omega_{i}}$ from $G$; \\ identify main cluster $G^{\omega_{i}}$;\\
        form $\eta_i=[\eta_{i,1}, \ldots, \eta_{i,N_v^{\omega_i}}]$,\\
         $\eta_{i,l} = 0$ if $x_l \in G^{\omega_{i}}$ and 1 otherwise};
        \node (local_calc2) [process, below of=local_calc1,  yshift=-0.5cm] {Extract submatrices $L^{\omega_{i}}$ from $L$ corresponded to $G^{\omega_{i}}$;\\ find $D^{\omega_{i}} = \text{diag}(L^{\omega_{i}})$};
        \node (local_calc3) [process, below of=local_calc2] {Solve generalized eigenvalue problem \eqref{sp1} in $G^{w_{i}}$};
        \node (local_calc4) [process, below of=local_calc3,  yshift=-0.1cm] {Choose $M_i$ eigenvectors and form matrix\\ $R^{\omega_i} = [{\psi}_{1}^{\omega_i}, \ldots,  {\psi}_{M_i}^{\omega_i} ]^T$ \\
        with ${\psi}_{1}^{\omega_i} = \chi_i, \eta_i$ and ${\psi}_{r+1}^{\omega_i} = \chi_i {\psi}_{r}^{\omega_i}$};
        \node (local_end) [process, below of=local_calc4, yshift=-0.5cm] {For global matrix by mapping local blocks $R^{\omega_i}$ \\
        $R = \begin{bmatrix}
       		R^{\omega_1}\\
       		\ldots\\
       		R^{\omega_N}
			\end{bmatrix}$};
        \draw [arrow] (start_local) -- (local_calc1);
        \draw [arrow] (local_calc1) -- (local_calc2);
        \draw [arrow] (local_calc2) -- (local_calc3);
        \draw [arrow] (local_calc3) -- (local_calc4);
        \draw [arrow] (local_calc4) -- (local_end);
    \end{tikzpicture}
};
\node (macro_block) [block, right of=local_block, xshift=6.5cm] {
    \textbf{Online stage}\\
    (global coarse calculations)\\[0.3cm]
    \begin{tikzpicture}[node distance=1.5cm, scale=0.8, transform shape]
        \node (start_macro) [process] {For a given system $C u_t + L u = F$,\\
        project matrices and vectors to coarse grid\\
        $C_H = R C R^T, \quad  L_H = R L R^T, \quad F_H= R F$};
        \node (macro_calc1) [process, below of=start_macro,  yshift=-0.5cm] {Solve time-dependet problem\\
        $(C_H + \tau L_H) u_H^n = \tau F_H + C_H u_H^{n-1}$\\
        with $u^0_H = R u_0$};
        \node (macro_calc2) [process, below of=macro_calc1,  yshift=-0.3cm] {Reconstruction  fine-scale solution\\
         $ u^n_{ms} = R^T u^n_H$};
        \draw [arrow] (start_macro) -- (macro_calc1);
        \draw [arrow] (macro_calc1) -- (macro_calc2);
    \end{tikzpicture}
};
\draw [arrow] (local_block.east) -- ++(0.6,0) node[anchor=south] {$R$} -- (macro_block.west);
\end{tikzpicture}
\caption{\rev{Schematic flow chart of the multiscale scheme.}}
\label{fig:ms}
\end{figure}

To define a coarse grid approximation, we use a projection matrix and form a coarse-scale matrix and right-hand side vector
\begin{equation}
C_H = R C R^T, \quad  L_H = R L R^T, \quad F_H= R F.
\end{equation}
Finally,  we solve a coarse-scale system each time $t_n$
\begin{equation}
C_H \frac{u_H^n - u_H^{n-1}}{\tau} + L_H u_H^n = F_H, 
\end{equation}
and reconstruction of the fine-scale solution $ u^n_{ms} = R^T u^n_H$. 
\rev{In Figure \ref{fig:ms} we present a schematic illustration of the multiscale solver.}
We note that all offline calculations are independent for each local domain $\omega_i$ and can be done entirely parallelly. All coupling is done on the coarse level and generally does not require parallelization. However, for large systems, the size of the coarse-scale approximation can still lead to the large system, and further parallelization of the linear system solution at time $t_n$ can be done based on the PETSc implementation of the parallel linear or nonlinear solvers (KSP and SNES classes). Furthermore, the more advanced models can be upscaled similarly with advanced time-stepping techniques available in the TS framework of PETSc.   We will consider such advanced coupling in future works.

\subsection{Convergence analysis}

In this part of the work, we present a priory error estimate of the multiscale method for network models.  We note that the analysis presented below is closely related to the analysis in \cite{efendiev2011multiscale, abreu2019convergence, fu2022generalized, guan2023coupling, poveda2023convergence}.  
We start with the definition of global norms used in the presented analysis.

Let $d(u, u) = (Du,u) =  u^T Du$ and $A = D^{-1} L$  then 
\begin{equation}
\begin{split}
& d(u,u) = ||u||^2_D, \\
& d(Au, u) = (Au)^T D u = u^T A^T D u  
= u^T (D^{-1} L)^T D u 
= u^T L u = a(u,u) = ||u||^2_L,\\
& d(Au, Au) = (Au)^T D (Au) = u^T (D^{-1} L)^T D (D^{-1} L) u 
=  u^T (L^T  D^{-1} L) u  = ||Au||^2_D,
\end{split}
\end{equation}
for $D= D^T$ and $L = L^T$ \cite{abreu2019convergence}.

For each $\omega_i$ ($i=1,\ldots,N$) we define the local projection
\begin{equation}
\label{pr}
P_{M_i}^{\omega_i} v 
= \sum_{r=1}^{M_i} c_{ri} {\phi}_r^{\omega_i}, \quad c_{ri} = ({\phi}_r^{\omega_i})^T D^{\omega_i} v, 
\end{equation}
for $v \in V$  and given $M_i$ in the local eigenvalue problem \eqref{sp1}.

For $P_{M_i}^{\omega_i}$ the following inequalities holds \cite{abreu2019convergence, guan2023coupling}
\begin{equation}
\label{pest}
\begin{split}
&||v - P_{M_i}^{\omega_i} v ||^2_{D^{\omega_i} } \leq 
\frac{1}{ \lambda_{M_i+1}^{\omega_i} } ||v - P_{M_i}^{\omega_i} v ||^2_{L^{\omega_i}} \leq
\frac{1}{ \lambda_{M_i+1}^{\omega_i} } ||v||^2_{L^{\omega_i}}, 
\\
&||v - P_{M_i}^{\omega_i} v ||^2_{D^{\omega_i}} \leq 
\frac{1}{ (\lambda_{M_i+1}^{\omega_i})^2 } ||A (v - P_{M_i}^{\omega_i} v) ||^2_{D^{\omega_i}} \leq\frac{1}{ (\lambda_{M_i+1}^{\omega_i})^2 } ||Av||^2_{D^{\omega_i}},
\\
&||v - P_{M_i}^{\omega_i} v ||^2_{L^{\omega_i}} \leq 
\frac{1}{ \lambda_{M_i+1}^{\omega_i} } ||A (v - P_{M_i}^{\omega_i} v) ||^2_{D^{\omega_i}} \leq\frac{1}{ \lambda_{M_i+1}^{\omega_i} } ||Av||^2_{D^{\omega_i}}.
\end{split}
\end{equation}

Next, we define the coarse projection $\Pi: V \rightarrow V_{ms}$ by
\begin{equation}
\label{pr}
\Pi v = \sum_{l=1}^N  \chi_l (P_{M_l}^{\omega_l} v).
\end{equation}
and $v - \Pi v  = \sum_{l=1}^N \chi_l (v - P_{M_l}^{\omega_l}  v)$.

\begin{lemma}
\label{t:tms1}
Assume that  $u \in V$ is the fine scale solution of \eqref{mm1} then the  following estimate  holds
\begin{equation}
\begin{split}
||u - \Pi u ||_D^2 \leq 
\frac{1}{\lambda_{M+1}^2}  ||A u||^2_D
\end{split}
\end{equation}
where 
$\lambda_{M+1} = \min_K  \lambda_{K, M+1}$ and  
$ \lambda_{K, M+1} = \min_{y_l \in K}  \lambda_{M_l+1}^{\omega_l}$.
\end{lemma}
\begin{proof}
Using that $\chi_l \leq 1$ we have
\begin{equation}
||u - \Pi u ||_{D^K}^2 
\leq 
\sum_{y_l \in K}  ||\chi_l (v - P_{M_l}^{\omega_l}  u)||^2_{D^{K}}
\leq
\sum_{y_l \in K} ||u - P_{M_l}^{\omega_l}  u||_{D^{\omega_l}}^2.
\end{equation}
By combing with estimate \eqref{pest}, we obtain the result.
\end{proof}

\begin{lemma}
\label{t:tms2}
Assume that  $u \in V$ is the fine scale solution of \eqref{mm1} then the  following estimate  holds
\begin{equation}
||u - \Pi u ||_L^2 \leq 
\left( \frac{1}{H^2 \lambda_{M+1}^2} + \frac{1}{\lambda_{M+1}} \right) ||A u||^2_D
\end{equation}
where 
$\lambda_{M+1} = \min_K  \lambda_{K, M+1}$ and  
$ \lambda_{K, M+1} = \min_{y_l \in K}  \lambda_{M_l+1}^{\omega_l}$.
\end{lemma}
\begin{proof}
For $z = \{z_j\}_{j=1}^{N_v}$ with $z_j = \chi_{l,j} v_j$,  we have
\begin{equation}
\begin{split}
||z||^2_{L^{\omega_l}}
&  = z^T L^{\omega_l} z
= \frac{1}{2} \sum_{i, j=1}^{N^{\omega_l}_v} 
w_{ij} (z_i - z_j)^2 
= \frac{1}{2} \sum_{i, j=1}^{N^{\omega_l}_v} 
w_{ij} (\chi_{l,i} v_i - \chi_{l,j} v_j)^2 \\
& = \frac{1}{2} \sum_{i, j=1}^{N^{\omega_l}_v} 
w_{ij} (\chi_{l,i} (v_i - v_j) - v_j( \chi_{l,j} - \chi_{l,i}))^2
\leq 
\frac{1}{2} \sum_{i, j=1}^{N^{\omega_l}_v} w_{ij} \chi_{l,i}^2 (v_i - v_j)^2
+ 
\frac{1}{2} \sum_{i, j=1}^{N^{\omega_l}_v} w_{ij} v_j^2 (\chi_{l,i} - \chi_{l,j})^2.
\end{split}
\end{equation}
Then for $z = \chi_l (v - P_{M_l}^{\omega_l}  v)$  using properties of partition of unity function \cite{melenk1996partition},   
we have
\begin{equation}
||u - \Pi u ||_{L^K}^2 
\leq 
\sum_{y_l \in K}  ||\chi_l (v - P_{M_l}^{\omega_l}  u)||^2_{L^{\omega_l}}
\leq
\sum_{y_l \in K} \frac{1}{H^2} ||u - P_{M_l}^{\omega_l}  u||_{D^{\omega_l}}^2  
+  
\sum_{y_l \in K} ||u - P_{M_l}^{\omega_l}  u||_{L^{\omega_l}}^2.
\end{equation}
By combing with estimate \eqref{pest}, we obtain  
\begin{equation}
||u - \Pi u ||_{L^K}^2 
\leq 
\left( \frac{1}{H^2 \lambda_{K, M+1}^2} + \frac{1}{ \lambda_{K, M+1}} \right) \sum_{y_l \in K} 
||Au||^2_{D^{\omega_l}}
\end{equation}
with $ \lambda_{K, M+1} = \min_{y_l \in K}  \lambda_{M_l+1}^{\omega_l}$. 

Then, we have
\begin{equation}
\begin{split}
||u - \Pi u ||_{L}^2   
 = \sum_{K \in \mathcal{T}_H}  ||u - \Pi u||_{L^K}^2 
& \leq \sum_{K \in \mathcal{T}_H} 
\left( \frac{1}{H^2 \lambda_{K, M+1}^2} + \frac{1}{ \lambda_{K, M+1}} \right) \sum_{y_l \in K}  ||A (u - P_{M_l}^{\omega_l}  u) ||^2_{D^{\omega_l}}\\
& \leq 
\left( \frac{1}{H^2 \lambda_{M+1}^2} + \frac{1}{\lambda_{M+1}} \right) \sum_{y_l \in K}||Au||_{D^{\omega_l}}^2,
\end{split}
\end{equation}
with $\lambda_{M+1} = \min_K  \lambda_{K, M+1}$. 
\end{proof}

\subsubsection{Multiscale semi-discrete network}

In this subsection, we consider semi-discrete network \eqref{ms1} and give a priory estimates for multiscale approximation. 

Stability estimate derivation is similar to \ref{t1}. We let $v_{ms} = \frac{\partial u_{ms}}{\partial t} = (u_{ms})_t$ in  \eqref{ms1} then we have
\begin{equation}
\left(C (u_{ms})_t,  (u_{ms})_t \right) 
+ \left( L u_{ms},   (u_{ms})_t \right) = \left(f, (u_{ms})_t \right). 
\end{equation}
Then by Young's inequality, we obtain the following estimate 
\begin{equation}
\left\| (u_{ms})_t \right\|^2_C  + \frac{1}{2} \frac{d}{dt} \left( Lu_{ms},  u_{ms} \right) 
\leq
\left\|(u_{ms})_t \right\|^2_C + \frac{1}{4} \left\| f \right\|^2_{C^{-1}}.
\end{equation}
Then after integration by time, we obtain the following lemma. 

\begin{lemma}
\label{t:tms3}
The  solution of the problem \eqref{ms1} satisfies the following a priory estimate
\begin{equation}
||u_{ms}(t)||_{L}^2 \leq ||u_{ms}(0)||_{L}^2 
+ \frac{1}{2} \int_0^t ||f||^2_{C^{-1}} ds.
\end{equation}
\end{lemma}

Next, we can obtain the error estimate of the GMsFEM for the semi-discrete network model. 
\begin{theorem}
\label{t:tms4}
Assume that  $u \in V$ is the fine scale solution of \eqref{mm1} and 
$u_{ms} \in V_{ms}$ is the multiscale solution of \eqref{ms1}  then the  following error estimate  holds
\begin{equation}
\begin{split}
||(u - u_{ms})(t)||_C^2 + \int_0^t || u - u_{ms}||_L^2 ds
&\preceq
||(w - u_{ms})(0)||_C^2 
+ \frac{1}{\lambda_{M+1}^2} ||A u||^2_D \\
&+ \int_0^t  \left(  
\left( \frac{1}{H^2 \lambda_{M+1}^2} + \frac{1}{\lambda_{M+1}} \right)  || A u||^2_D 
+  \frac{1}{\lambda_{M+1}^2} || (A u)_t ||^2_D 
\right) ds
\end{split}
\end{equation}
\end{theorem}
\begin{proof}
Subtracting \eqref{mm1}  from  \eqref{ms1} , we have 
\begin{equation}
\left( C (u -u_{ms})_t, v \right)  + (L (u - u_{ms}), v)  = 0
\quad \forall v \in V_{ms}.
\end{equation}

We write $u - u_{ms} = (u  - w) + (w - u_{ms})$ with $w = \Pi u$ and since $w \in V_{ms}$ we take $v = w - u_{ms}$ and obtain
\begin{equation}
\label{teqe}
\left( C (w -u_{ms})_t, w - u_{ms} \right)  + (L (w - u_{ms}), w - u_{ms})  
=
\left( C (w - u)_t, w - u_{ms} \right)  + (L (w - u), w - u_{ms}).
\end{equation}

For the left part of \eqref{teqe}, we have
\begin{equation}
\left( C(w -u_{ms})_t, w - u_{ms} \right) = \frac{1}{2} \frac{d}{dt}||w - u_{ms}||_C^2, \quad
(L (w - u_{ms}), w - u_{ms})   = ||w - u_{ms}||^2_L.
\end{equation}

For the right part of \eqref{teqe} by Young's inequality, we obtain
\begin{equation}
\left( C(w - u)_t,  w - u_{ms} \right) \leq 
\frac{1}{4 \delta_1}|| (w - u)_t ||^2_{C} + \delta_1 ||w - u_{ms} ||^2_{C},
\end{equation}
and
\begin{equation}
(L (w - u), w - u_{ms}) \leq 
\frac{1}{4 \delta_2} ||w - u||^2_L + \delta_2 ||w - u_{ms}||^2_L.
\end{equation}

Combing estimates for right and left parts of \eqref{teqe} with $\delta_1 = \delta_2 = 1/2$, we can obtain 
\begin{equation}
\frac{d}{dt}||w - u_{ms}||_C^2 + ||w - u_{ms}||^2_L 
\leq
|| (w - u)_t ||^2_{C} + ||w - u_{ms} ||^2_{C} + ||w - u||^2_L.
\end{equation}
By the Gronwall Lemma 
we obtain
\begin{equation}
|| (w - u_{ms})(t)||_C^2 
+ \int_0^t ||w - u_{ms}||^2_L  ds
\leq
|| (w - u_{ms})(0)||_C^2
+ \int_0^t \left( || (w - u)_t ||^2_{C} + ||w - u||^2_L  \right) ds.
\end{equation}
Then using triangle inequality we have
\begin{equation}
\begin{split}
|| (u - u_{ms})(t)||_C^2 
+ \int_0^t ||u - u_{ms}||^2_L  ds
& \leq
|| (w - u_{ms})(0)||_C^2 + || (w-u)(t) ||_C^2 \\
&+ \int_0^t \left( ||(w-u)_t ||^2_{C} + ||w-u||^2_L \right) ds
\end{split}
\end{equation}
Finally with $||u-v||^2_C \preceq ||u-v||^2_D$ and using Lemmas \ref{t:tms1} and \ref{t:tms2}, we obtain the result.
\end{proof}

Note that, we have $Au = D^{-1} Lu = D^{-1} (f - C u_t)$ then
\begin{equation}
||Au||_D^2 = ||f- C u_t||^2_{D^{-1}}.
\end{equation} 
Therefore the error estimate in Theorem \ref{t:tms4} can be written as follows
\begin{equation}
\begin{split}
||(u - u_{ms})(t)||_C^2 &+ \int_0^t || u - u_{ms}||_L^2 ds 
\preceq
||(w - u_{ms})(0)||_C^2 
+ \frac{1}{\lambda_{M+1}^2} ||f- C u_t||^2_{D^{-1}} \\
&+ \int_0^t  \left(  
\left( \frac{1}{H^2 \lambda_{M+1}^2} + \frac{1}{\lambda_{M+1}} \right) ||f- C u_t||^2_{D^{-1}}
+  \frac{1}{\lambda_{M+1}^2}||(f- C u_t)_t||^2_{D^{-1}} 
\right) ds
\end{split}
\end{equation}
If we map the local eigenvalue problem to size one domain \cite{abreu2019convergence}, then the eigenvalues scale with $H^{-2}$ \rev{($\lambda_{M+1} \approx \Lambda^* H^{-2}$) and we obtain 
\begin{equation}
||(u - u_{ms})(t)||_C^2 + \int_0^t || u - u_{ms}||_L^2 ds 
\preceq_T  H^2 (\Lambda^*)^{-1}
\end{equation}
with accurate choose of $u_{ms}(0)$  and $a \preceq_T b$ means $a \leq C_T b$ where $C_T$ is the constant depending on $T$ \cite{chetverushkin2021computational}.}

\subsubsection{Multiscale discrete network}

Finally, we present a priory estimate for a fully discrete problem \eqref{ms2}. 

Stability estimate derivation is similar to Lemma \ref{t2}. We let $v_{ms} = ( u_{ms}^n - u_{ms}^{n-1})/ \tau$ then 
\begin{equation}
\begin{split}
\left( C \frac{ u_{ms}^n - u_{ms}^{n-1} }{\tau},  \frac{ u_{ms}^n - u_{ms}^{n-1} }{\tau}\right) & + (L u_{ms}^n, \frac{ u_{ms}^n - u_{ms}^{n-1} }{\tau}) \\
&= 
\left( \left(C + \frac{\tau}{2} L \right) \frac{ u_{ms}^n - u_{ms}^{n-1} }{\tau} , \frac{ u_{ms}^n - u_{ms}^{n-1} }{\tau} \right) 
+ \frac{1}{2}  \left( L (u_{ms}^n + u_{ms}^{n-1}),   \frac{ u_{ms}^n - u_{ms}^{n-1} }{\tau} \right) \\
&= 
\left(f^n, \frac{ u_{ms}^n - u_{ms}^{n-1} }{\tau} \right)
\leq 
\left\| \frac{ u_{ms}^n - u_{ms}^{n-1} }{\tau} \right\|_{ \left(C + \frac{\tau}{2} L \right) }^2 
+ \frac{1}{4} 
\left\| f^n \right\|_{ \left(C + \frac{\tau}{2} L \right)^{-1}}^2,
\end{split}
\end{equation}
Then we have
\begin{equation}
\frac{1}{2 \tau} \left( L (u_{ms}^n + u_{ms}^{n-1}),   (u_{ms}^n - u_{ms}^{n-1})  \right)
=
\frac{1}{2 \tau} (L u_{ms}^n,u_{ms}^n) + 
\frac{1}{2 \tau} (L u_{ms}^{n-1}, u_{ms}^{n-1}) 
\leq 
\frac{1}{4} 
\left\| f^n \right\|_{ \left(C + \frac{\tau}{2} L \right)^{-1}}^2,
\end{equation}
or
\begin{equation}
||u_{ms}^n||^2_L \leq ||u_{ms}^{n-1}||^2_L  + \frac{\tau}{2} 
\left\| f^n \right\|_{ \left(C + \frac{\tau}{2} L \right)^{-1}}^2.
\end{equation}

\begin{lemma}
\label{t:tms5}
The solution of the  problem \eqref{ms2} is unconditionally stable and satisfies the following estimate
\begin{equation}
\label{t2}
||u_{ms}^n||_L^2 
\preceq 
||u_{ms}^0||_L^2 
+ \tau \sum_{k=1}^n ||f^k||^2_{\left(C + \frac{\tau}{2} L \right)^{-1}}.
\end{equation}
\end{lemma}

Next, we present the error estimate of the multiscale method  for the discrete  network model. 

\begin{theorem}
\label{t:tms6}
Assume that  $u^n \in V$ is the fine scale solution of \eqref{mm2} and 
$u^n_{ms} \in V_{ms}$ is the multiscale solution of \eqref{ms2}  then the  following error estimate  holds
\begin{equation}
\begin{split}
||u^n - u^n_{ms}||_C^2 + \tau \sum_{k=1}^n || u^k - u^k_{ms}||_L^2
&\preceq
||u^0 - u_{ms}^0||_C^2 \\
+ 
& \sum_{k=1}^n \left( 
\tau \left( \frac{1}{H^2 \lambda_{M+1}^2} + \frac{1}{\lambda_{M+1}} \right)  || A u^k||^2_D
+
 \frac{1}{\lambda_{M+1}^2} || A u^k||^2_D
 \right). 
\end{split}
\end{equation}
\end{theorem}
\begin{proof}
Subtracting \eqref{mm2}  from  \eqref{ms2} , we have 
\begin{equation}
\left( C \frac{ (u^n -u^n_{ms}) }{\tau}, v \right) - \left( C \frac{ (u^{n-1} -u^{n-1}_{ms}) }{\tau}, v \right)  + (L (u^n - u^n_{ms}), v)  = 0
\quad \forall v \in V_{ms}.
\end{equation}

For $u^n - u^n_{ms} = (w^n - u^n_{ms}) + (u^n - w^n)$ with $w^n = \Pi u^n$ and $v = w^n - u^n_{ms}$, we have
\begin{equation}
\label{teqe2}
\begin{split}
&\left( C \frac{ (w^n -u^n_{ms}) }{\tau}, w^n - u^n_{ms} \right) 
+  (L (w^n - u^n_{ms}), w^n - u^n_{ms})  \\  
& =  
\left( C \frac{(w^n - u^n) }{\tau}, w^n - u^n_{ms} \right) 
+ \left( C \frac{(u^{n-1} - u_{ms}^{n-1}) }{\tau}, w^n - u^n_{ms} \right)  
+ (L (w^n - u^n), w^n - u^n_{ms}).
\end{split}
\end{equation}

For the left part of \eqref{teqe2}, we have
\begin{equation}
\left( C (w^n -u^n_{ms}), w^n - u^n_{ms} \right) = ||w^n - u^n_{ms}||_C^2, \quad
(L (w^n - u^n_{ms}), w^n - u^n_{ms})   = ||w^n - u^n_{ms}||^2_L.
\end{equation}

For the right part of \eqref{teqe2} by Young's inequality, we have
\begin{equation}
\begin{split}
&\left( C (w^n - u^n), w^n - u^n_{ms} \right) \leq 
\frac{1}{4 \delta_1}||w^n - u^n ||^2_{C} + \delta_1 ||w^n - u^n_{ms} ||^2_{C},
\\
&\left( C (u^{n-1} - u_{ms}^{n-1}), w^n - u^n_{ms} \right)  \leq 
\frac{1}{4 \delta_2}||u^{n-1} - u_{ms}^{n-1} ||^2_{C} + \delta_2 ||w^n - u^n_{ms} ||^2_{C},
\end{split}
\end{equation}
and
\begin{equation}
(L (w^n - u^n), w^n - u^n_{ms}) \leq 
\frac{1}{4 \delta_3} ||w^n - u^n||^2_L + \delta_3 ||w^n - u^n_{ms}||^2_L.
\end{equation}

Then with $\delta_1 = \delta_2 = 1/4$ and $\delta_3 = 1/2$, we obtain 
\begin{equation}
||w^n - u^n_{ms}||_C^2 + \tau ||w^n - u^n_{ms}||^2_L 
\preceq 
||u^{n-1} - u_{ms}^{n-1} ||^2_{C} 
+
||w^n - u^n||^2_{C} +  \tau ||w^n - u^n||^2_L.
\end{equation}

Using triangle inequality we have
\begin{equation}
||u^n - u^n_{ms}||_C^2 
+ \tau ||u^n - u^n_{ms}||^2_L 
\preceq
||u^{n-1} - u_{ms}^{n-1} ||^2_{C} 
+
||w^n - u^n||^2_{C} + \tau ||w^n - u^n||^2_L.
\end{equation}
Therefore
\begin{equation}
||u^n - u^n_{ms}||_C^2 
+ \tau \sum_{k=1}^n ||u^n - u^n_{ms}||^2_L 
\preceq
|| u^0 - u_{ms}^0 ||^2_{C} 
+ \sum_{k=1}^n  \left(
||w^k - u^k||^2_{C} + \tau ||w^k - u^k||^2_L
\right).
\end{equation}
Finally using Lemmas \ref{t:tms1} and \ref{t:tms2}, we obtain the result.
\end{proof}

\rev{The local eigenvalue are scaled by a size of the domain $H^{-2}$ ($\lambda_{M+1} \approx  \Lambda^* H^{-2}$), then similarly to the time continuous case we have \cite{abreu2019convergence, chetverushkin2021computational}
\begin{equation}
||u^n - u^n_{ms}||_C^2 + \tau \sum_{k=1}^n || u^k - u^k_{ms}||_L^2
\preceq_T  H^2 (\Lambda^*)^{-1}.
\end{equation}
The convergence study demonstrates that the error bound of the multiscale method is proportional to $H/(\Lambda^*)^{1/2}$, where $\Lambda^*$ is related to the minimum of the eigenvalues that are not included in the coarse space \cite{efendiev2011multiscale}.}

\section{Numerical results}

\rev{In this section, we consider illustrative examples for time-dependent discrete network models. We start with the illustration of the three geometric structures that a commonly occur in pore network models:  (1) structured regular network, (2) structured irregular
network, and (3) unstructured irregular. 
In Figure \ref{fig:netstr}, we show three network structures in two-dimensional and three-dimensional case, where nodes and connections are depicted in blue and red colors, respectively. 
Such structures often represent pore-structure in reservoir simulation \cite{joekar2012analysis, cui2022pore}. In this application, the complex pore structure is approximated by a pore-network model, leading to reduced computational cost \cite{bluntpnm, blunt2013pore}. Pore-network models are characterized by pore bodies connected by pore throats. The spatial location of pore bodies (structured or unstructured lattice) and pore connectivity (regular or irregular) are the main topological characteristics of the network. 
As the simplest network, we consider a structure regular network with equally spaced nodes (pore bodies) with four connections in a two-dimensional case and six connections in a three-dimensional case. 
In the second case, the structured pores distribution can be characterized by smaller connectivity, where some connections randomly dropped with a post-coming filtration of the disconnected clusters \cite{raoof2010new}. 
A real porous medium is usually characterized by an unstructured lattice with a more extensive connectivity. To mimic the disordered arrangement of pores and connections, an unstructured pore-network model is used and became very popular due to its closeness to the realistic structure \cite{cui2022pore}. The unstructured irregular networks can capture a large range of pore connectivities from one to twenty-six with the orientation in thirteen instead of the limitation in three directions.  To construct the considered networks, we use OpenPNM library \cite{gostick2016openpnm}.  }

\begin{figure}[h!]
\centering
\begin{subfigure}{0.3\textwidth}
\centering
\includegraphics[width=0.7\linewidth]{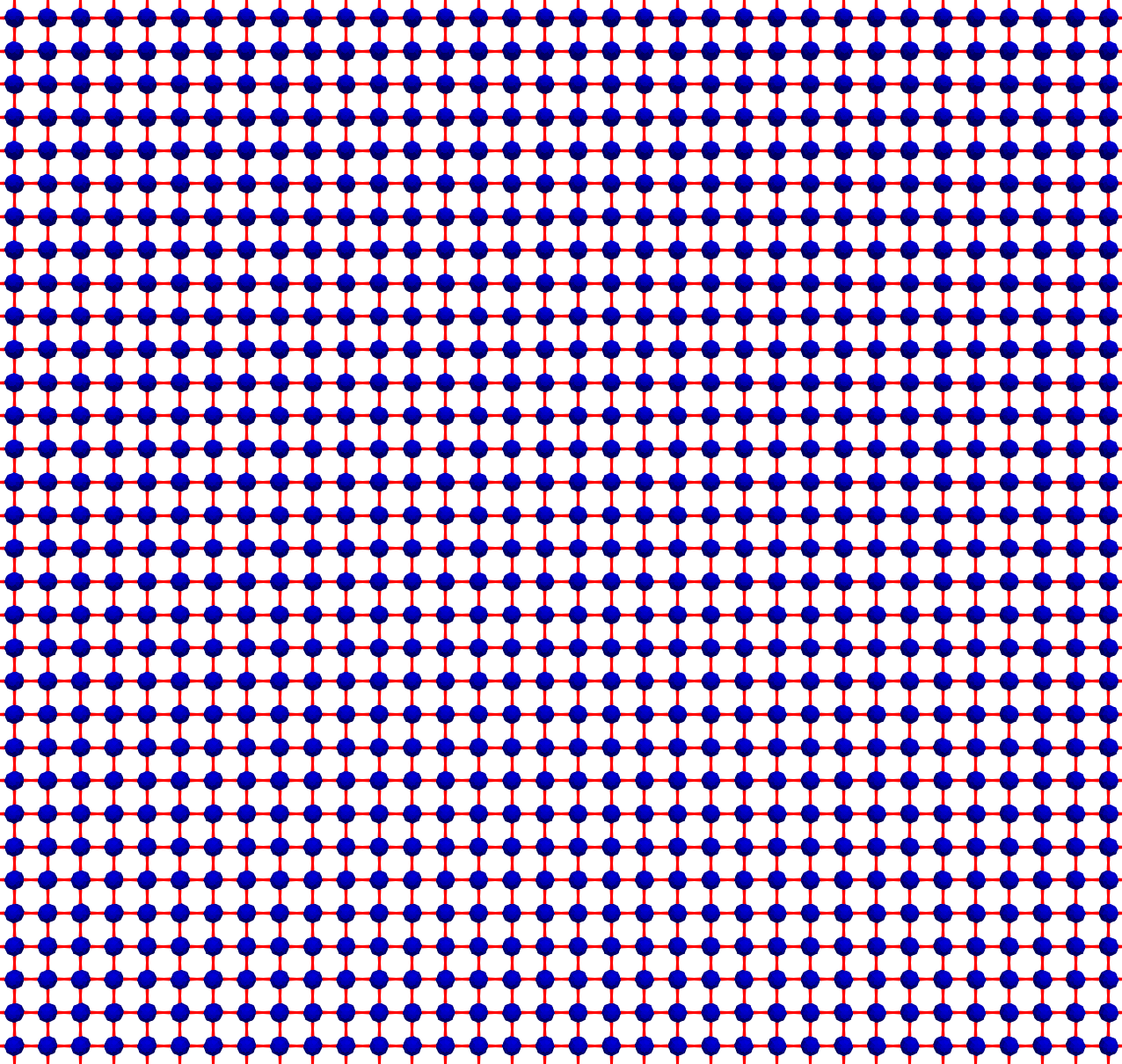}
\caption{\rev{(1) Structured regular network in 2D}}
\end{subfigure}
\ \ \ \ \ \
\begin{subfigure}{0.3\textwidth}
\centering
\includegraphics[width=0.7\linewidth]{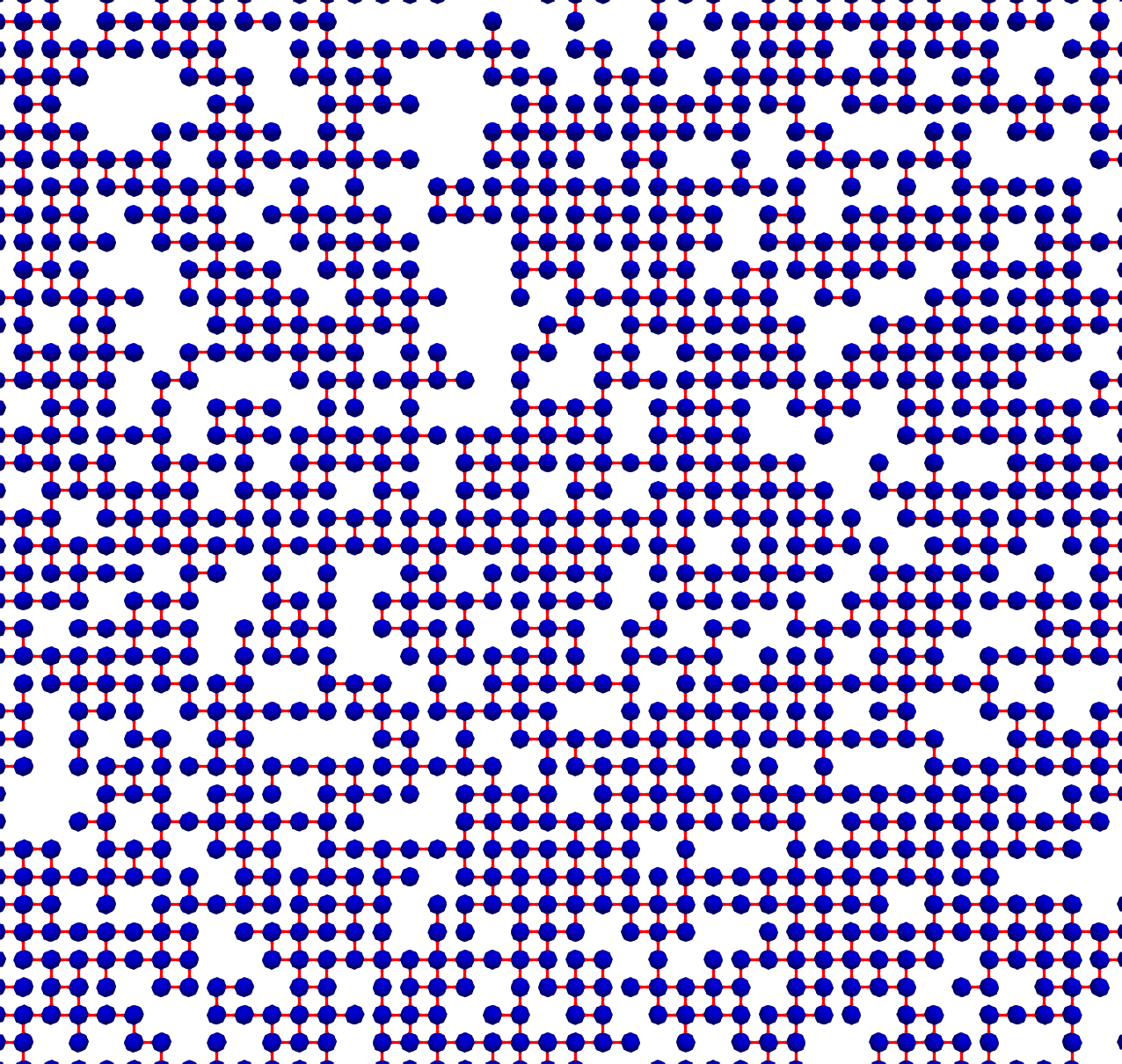}
\caption{\rev{(2) Structured irregular network in 2D}}
\end{subfigure}
\ \ \ \ \ \
\begin{subfigure}{0.3\textwidth}
\centering
\includegraphics[width=0.7\linewidth]{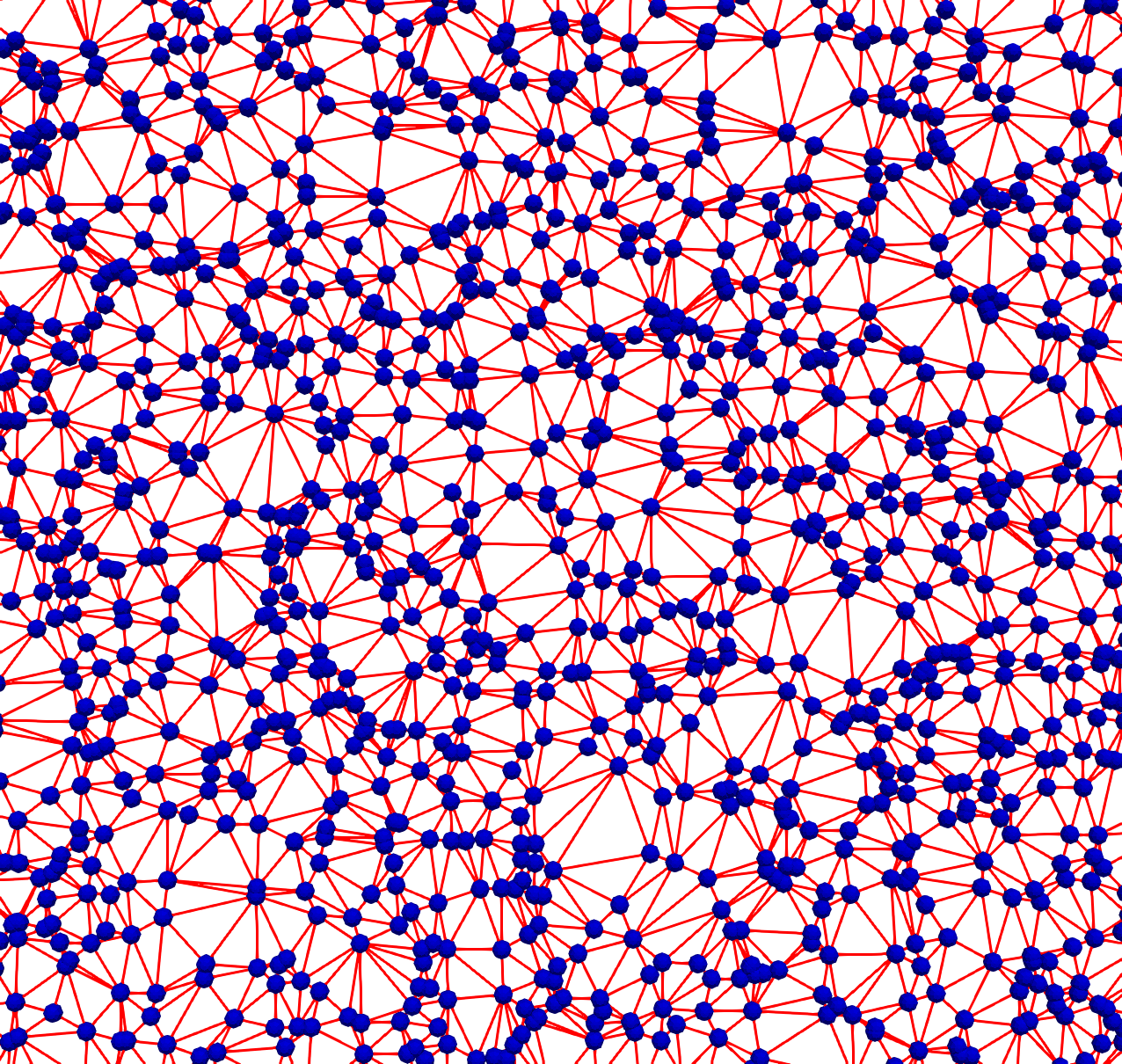}
\caption{\rev{(3) Unstructured regular network in 2D}}
\end{subfigure}
\begin{subfigure}{0.3\textwidth}
\includegraphics[width=1\linewidth]{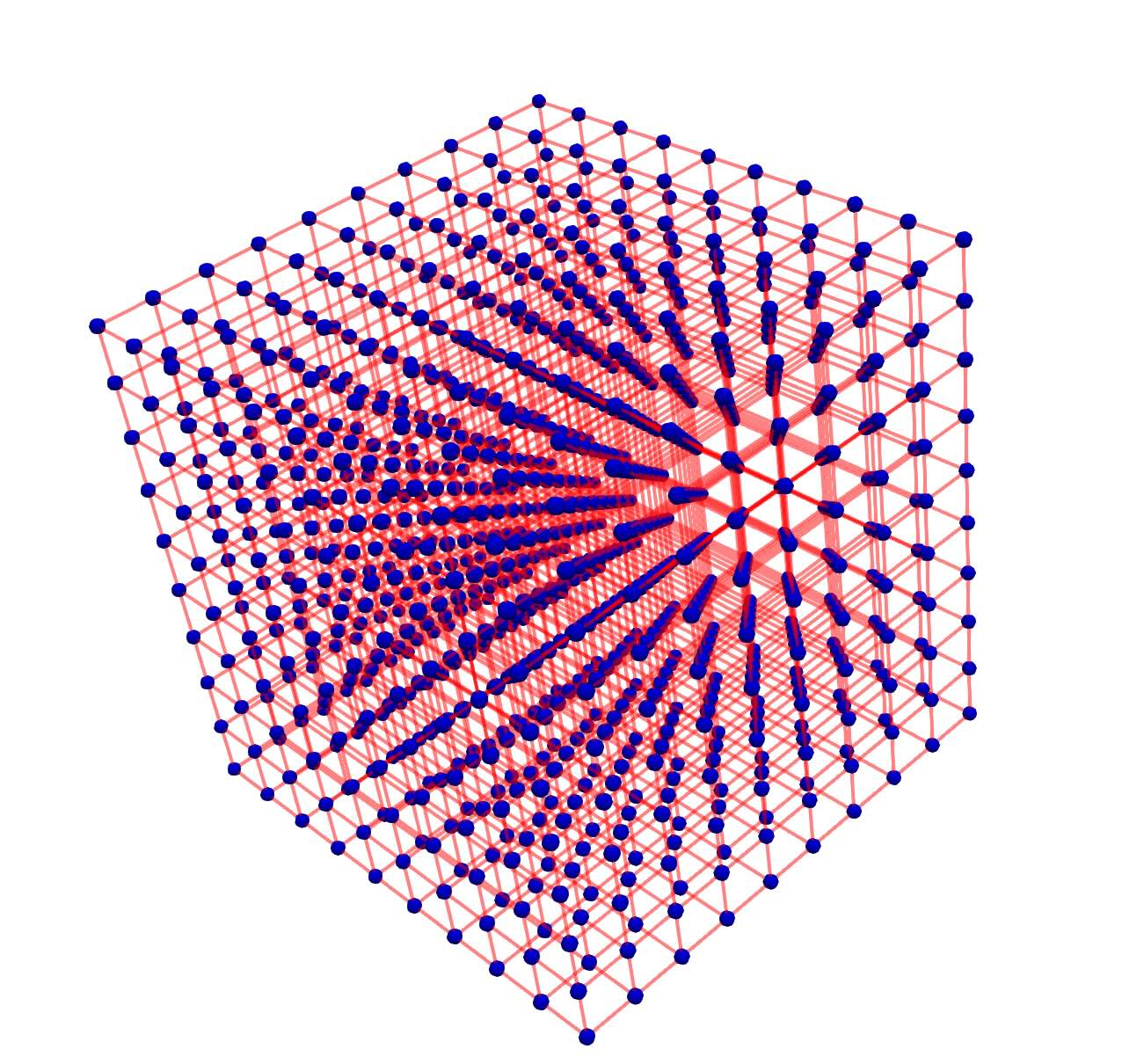}
\caption{\rev{(1) Structured regular network in 3D}}
\end{subfigure}
\ \ \ \ \ \
\begin{subfigure}{0.3\textwidth}
\includegraphics[width=1\linewidth]{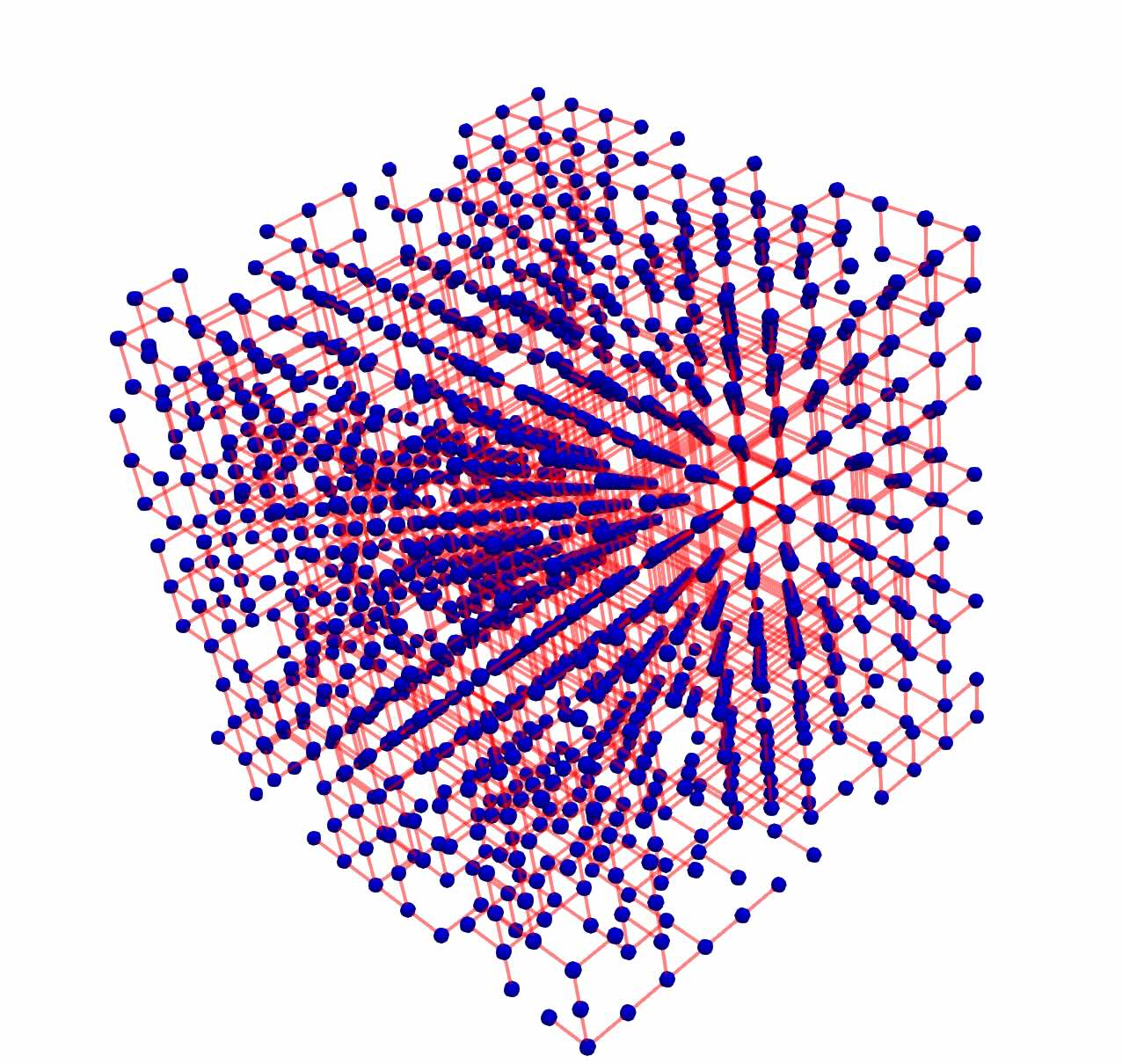}
\caption{\rev{(2) Structured irregular network in 3D}}
\end{subfigure}
\ \ \ \ \ \
\begin{subfigure}{0.3\textwidth}
\includegraphics[width=1\linewidth]{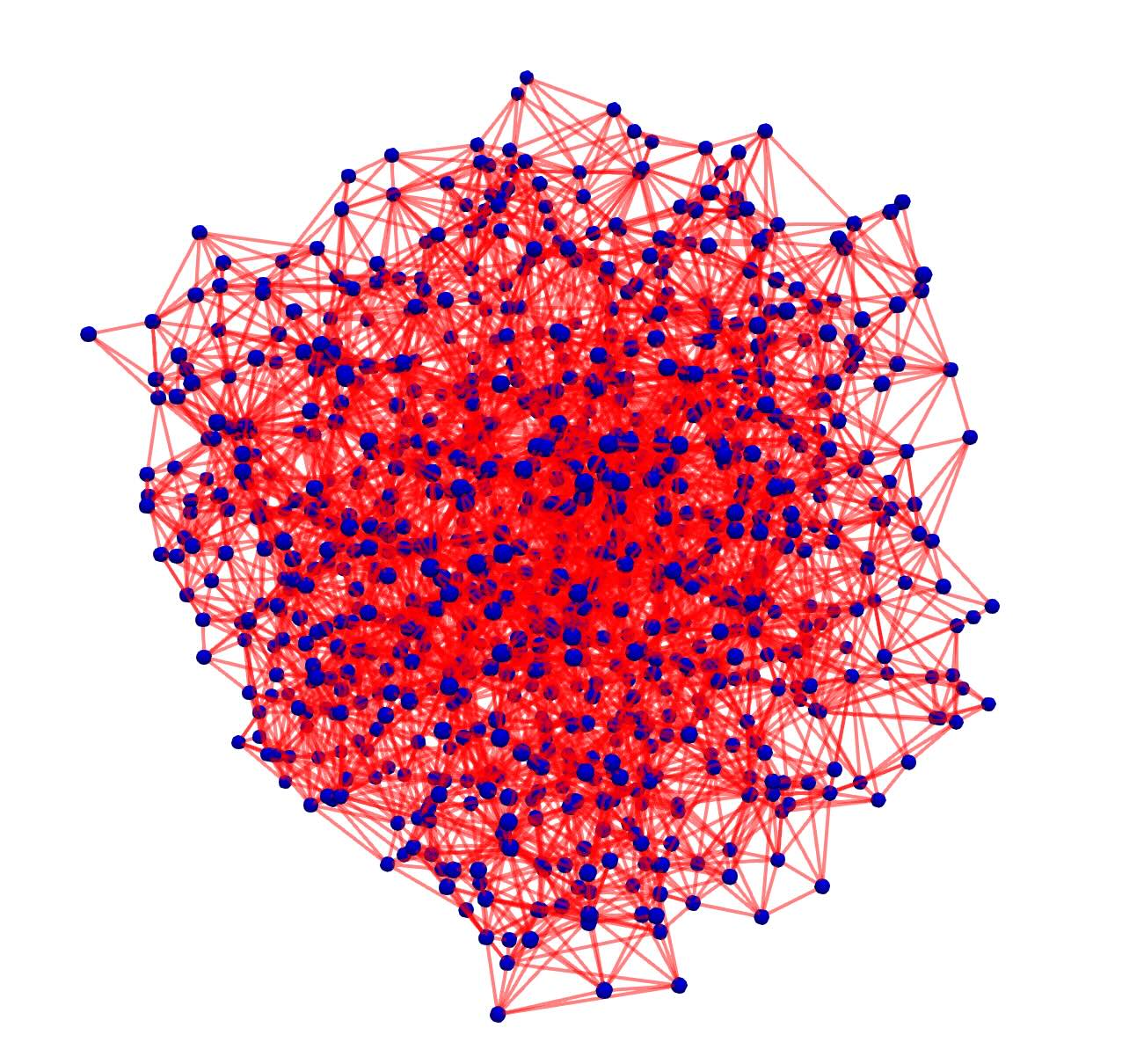}
\caption{\rev{(3) Unstructured regular network in 3D}}
\end{subfigure}
\caption{\rev{Network geometric structure in 2D and 3D}}
\label{fig:netstr}
\end{figure}

\rev{To generate a heterogeneous property, we randomly assign node and throat diameters, find the length of each throat based on the network's geometric structure, and calculate the connectivity coefficient based on the area and length of the throat. In this work, we consider a simplified approach to pore network construction, where the pores are represented by sphere geometry and throats are represented by cylinders. 
Moreover, the structured regular network can be considered a structured grid with a finite volume approximation of the parabolic equation. In this case, pore body volume is the volume of cells, and the throat is the connection with weight proportional to cell permeability, distance, and the surface of facets between two cells. 
Additionally, we consider high-contrast properties generated by assigning node diameters in different subdomains and calculating flow parameters based on them. The Hagen-Poiseuille equation for single phase flow in a cylindrical tube is used to calculate connection weights \cite{gostick2016openpnm}. }

\rev{We start by explaining the test cases in detail. Then, we illustrate the multiscale basis for network structures. Finally, we investigate the accuracy of the proposed multiscale method for the presented test cases and compare it with the upscaling method. }

\subsection{\rev{Test cases: networks and flow properties}}

\rev{We consider three network structures in two-dimensional (2D) and three-dimensional (3D) formulations: }
\begin{itemize}
\item \textit{Network-1} - \rev{structured regular network:}
\begin{itemize}
\item \rev{2D: Network-1a with 40000 nodes and 79600 connections ($200 \times 200$ structured grid)}
\item \rev{3D: Network-1b with 15625 nodes and 45000 connections ($25 \times 25 \times 25$ structured grid)}
\end{itemize}

\item \textit{Network-2} - \rev{structured irregular network:}
\begin{itemize}
\item \rev{2D: Network-2a with 46006 nodes and 62538 connections ($240 \times 240$ structured grid with random removal of nodes/connections and additional removal of disconnected clusters)}
\item \rev{3D: Network-2b with 22092 nodes and 43075 connections ($30 \times 30 \times 30$ structured grid  with random removal of nodes/connections and additional removal of disconnected clusters)}
\end{itemize}

\item \textit{Network-3} - \rev{unstructured irregular network:}
\begin{itemize}
\item \rev{2D: Network-3a with 40000 nodes and 119297 connections (related to Delaunay tessellation of arbitrary base points implemented in openpnm with highly connected 40000 nodes)}
\item \rev{3D: Network-3b with 15625 nodes and 114147 connections (Delaunay tessellation of arbitrary base points implemented in openpnm with highly connected 15625 nodes)}
\end{itemize}
\end{itemize}
\rev{In Figures \ref{fig:test2} and  \ref{fig:test3}, we plot the structured regular,  structured irregular, and unstructured irregular networks for 2D and 3D cases, respectively. Note that we set labels for nodes on the top and bottom boundaries to assign boundary conditions for discrete network models. }

\begin{figure}[h!]
\centering
\begin{subfigure}{0.3\textwidth}
\includegraphics[width=1\linewidth]{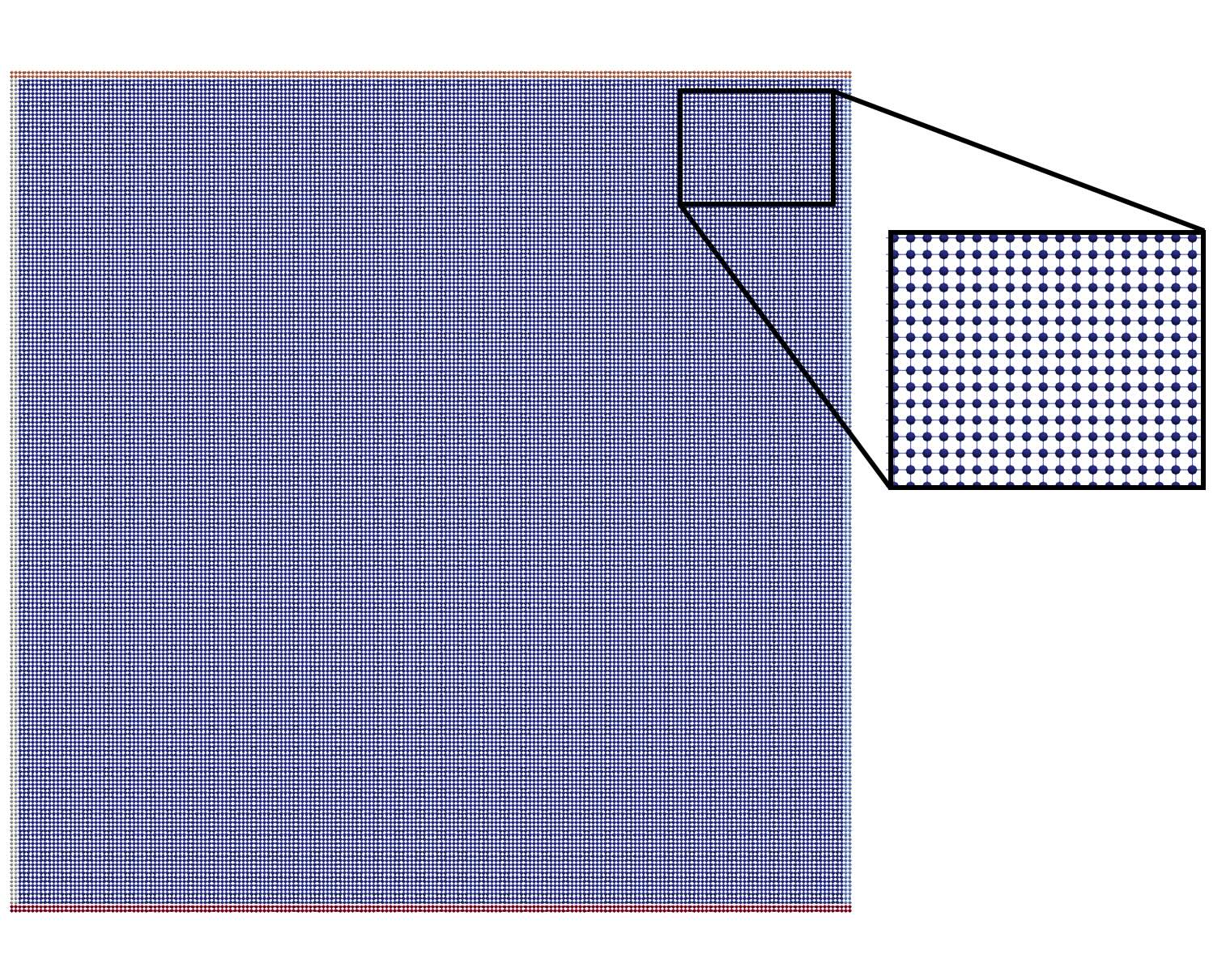}
\caption{\rev{Network-1a: Structured regular network}}
\end{subfigure}
\ \ \ \ \ \
\begin{subfigure}{0.3\textwidth}
\includegraphics[width=1\linewidth]{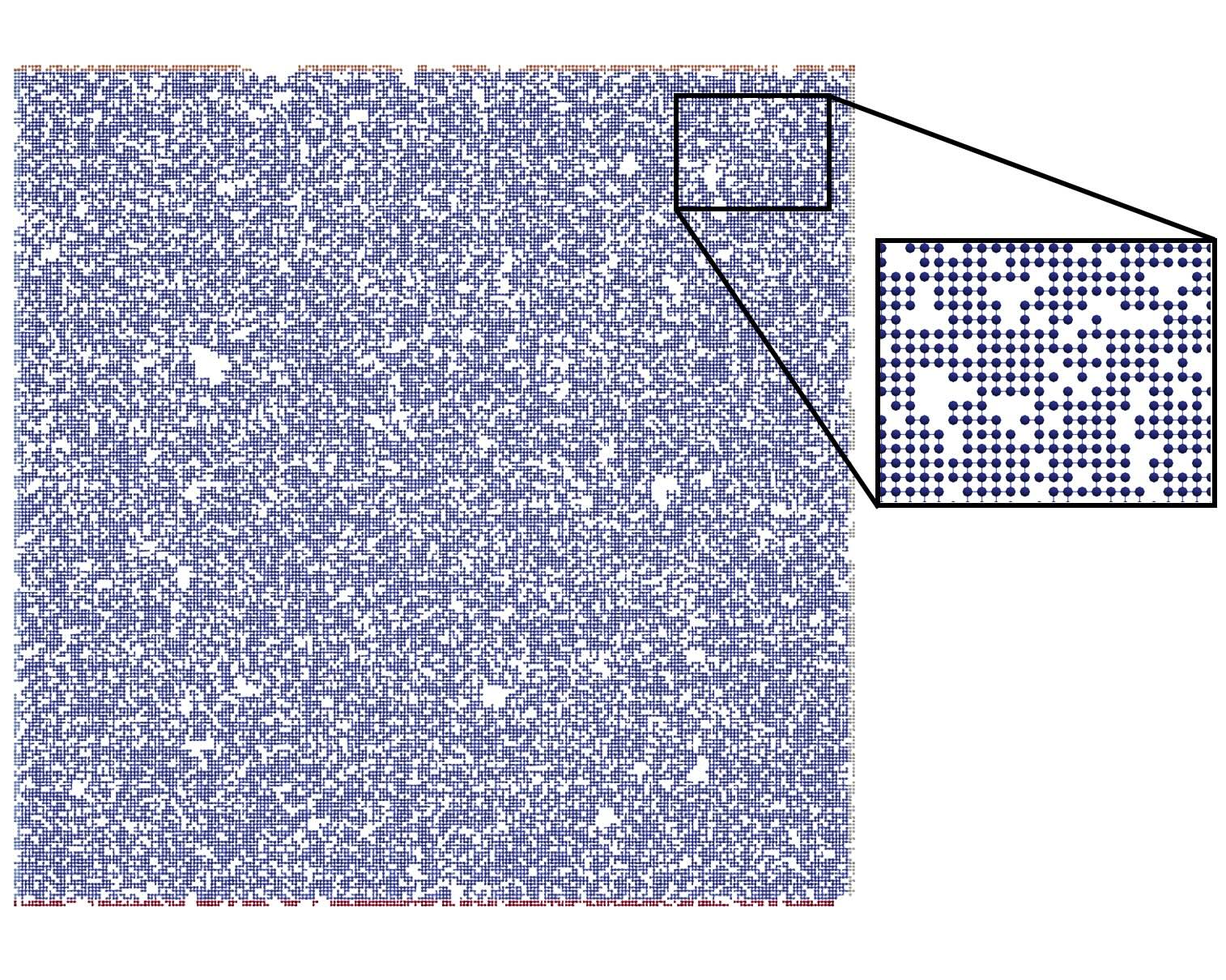}
\caption{\rev{Network-2a: Structured irregular network}}
\end{subfigure}
\ \ \ \ \ \
\begin{subfigure}{0.3\textwidth}
\includegraphics[width=1\linewidth]{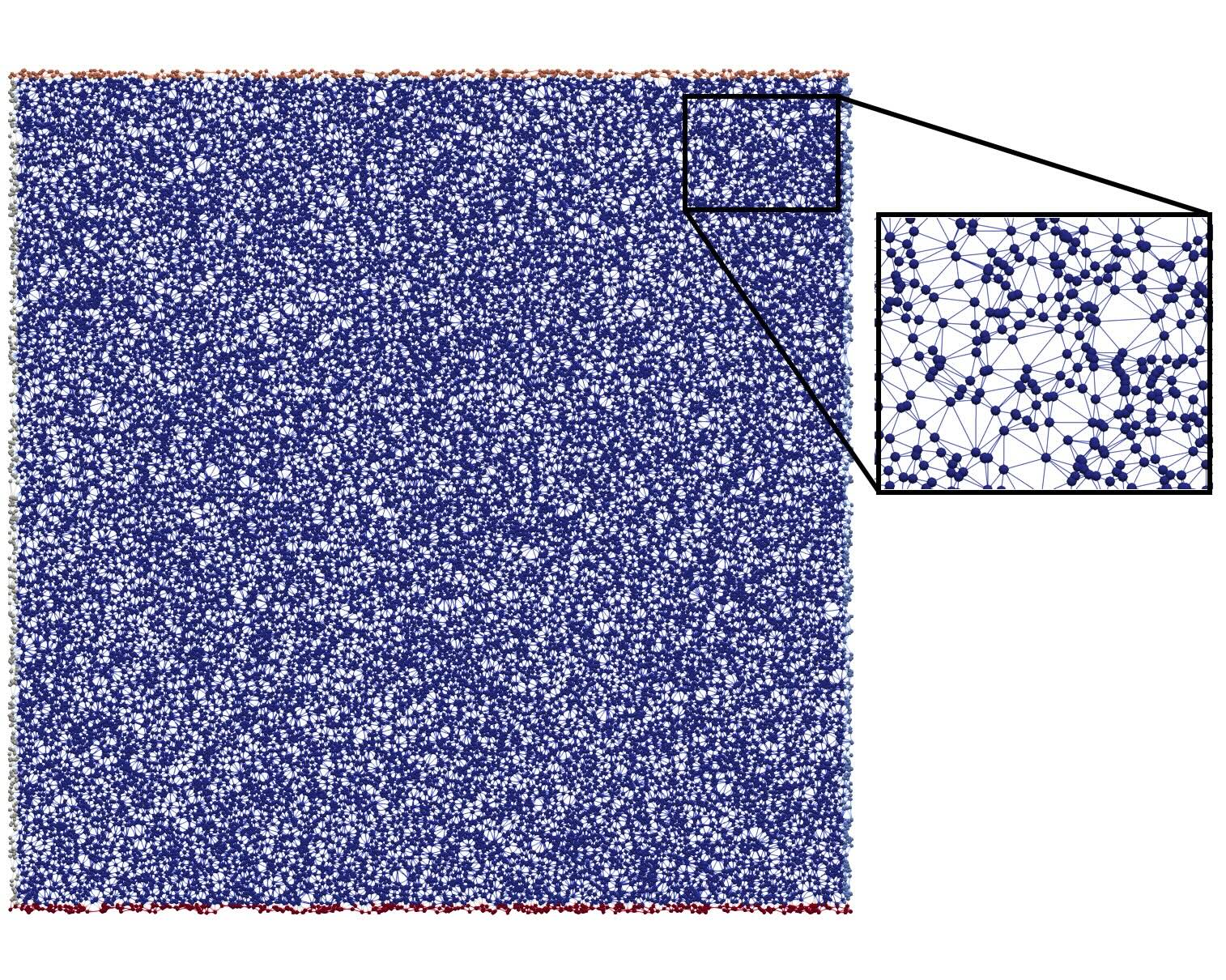}
\caption{\rev{Network-3a: Unstructured regular network}}
\end{subfigure}
\caption{\rev{Networks in 2D}}
\label{fig:test2}
\end{figure}

\begin{figure}[h!]
\centering
\begin{subfigure}{0.3\textwidth}
\includegraphics[width=1\linewidth]{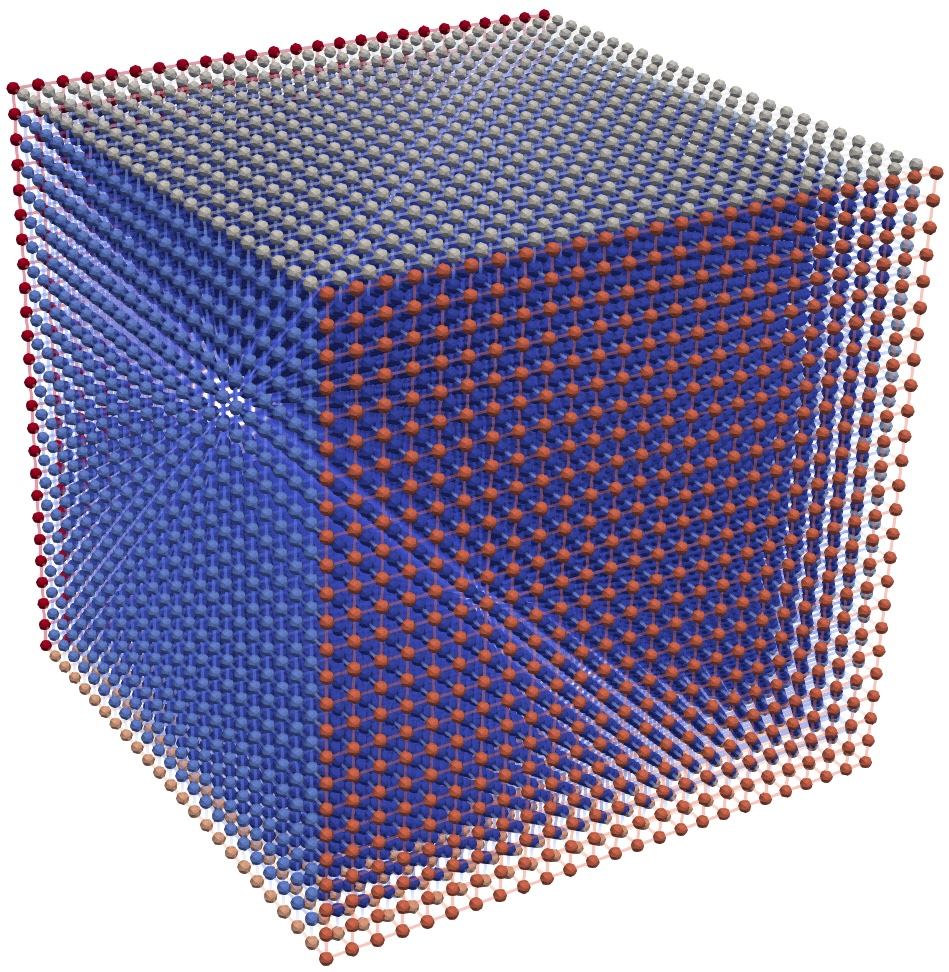}
\caption{\rev{Network-1b: Structured regular network}}
\end{subfigure}
\ \ \ \ \ \
\begin{subfigure}{0.3\textwidth}
\includegraphics[width=1\linewidth]{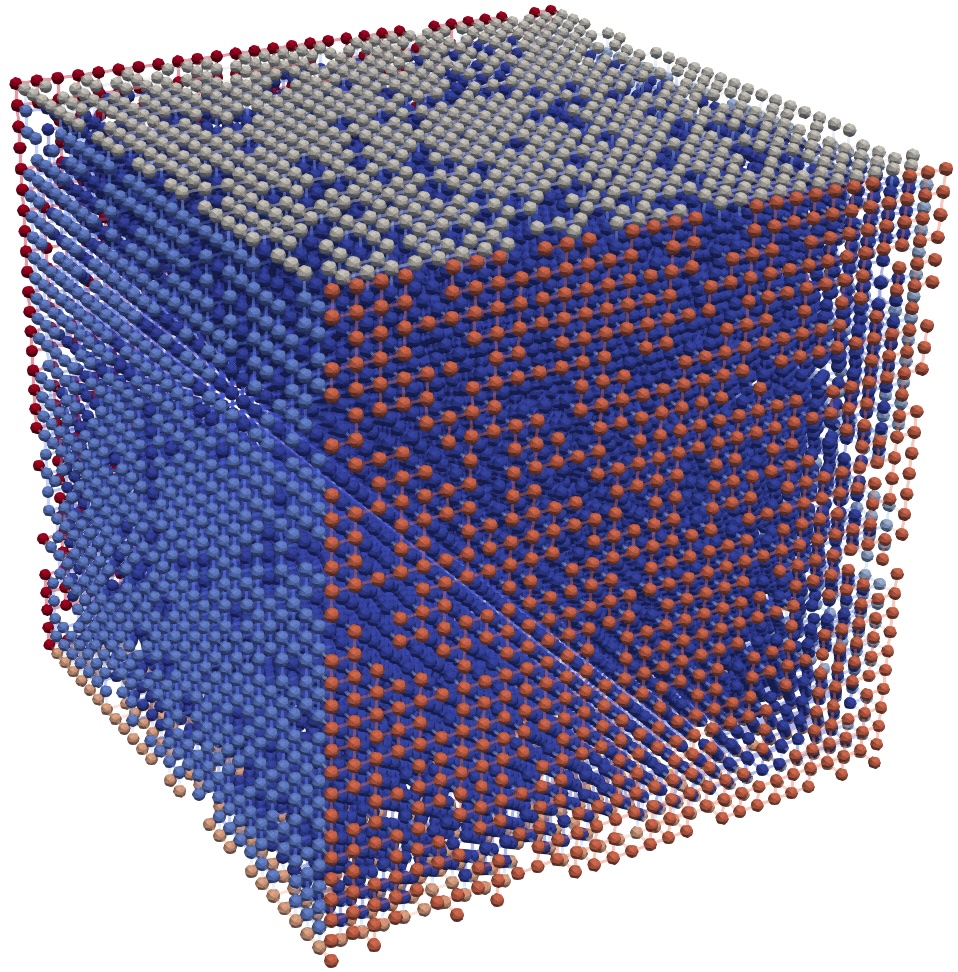}
\caption{\rev{Network-2b: Structured irregular network}}
\end{subfigure}
\ \ \ \ \ \
\begin{subfigure}{0.3\textwidth}
\includegraphics[width=1\linewidth]{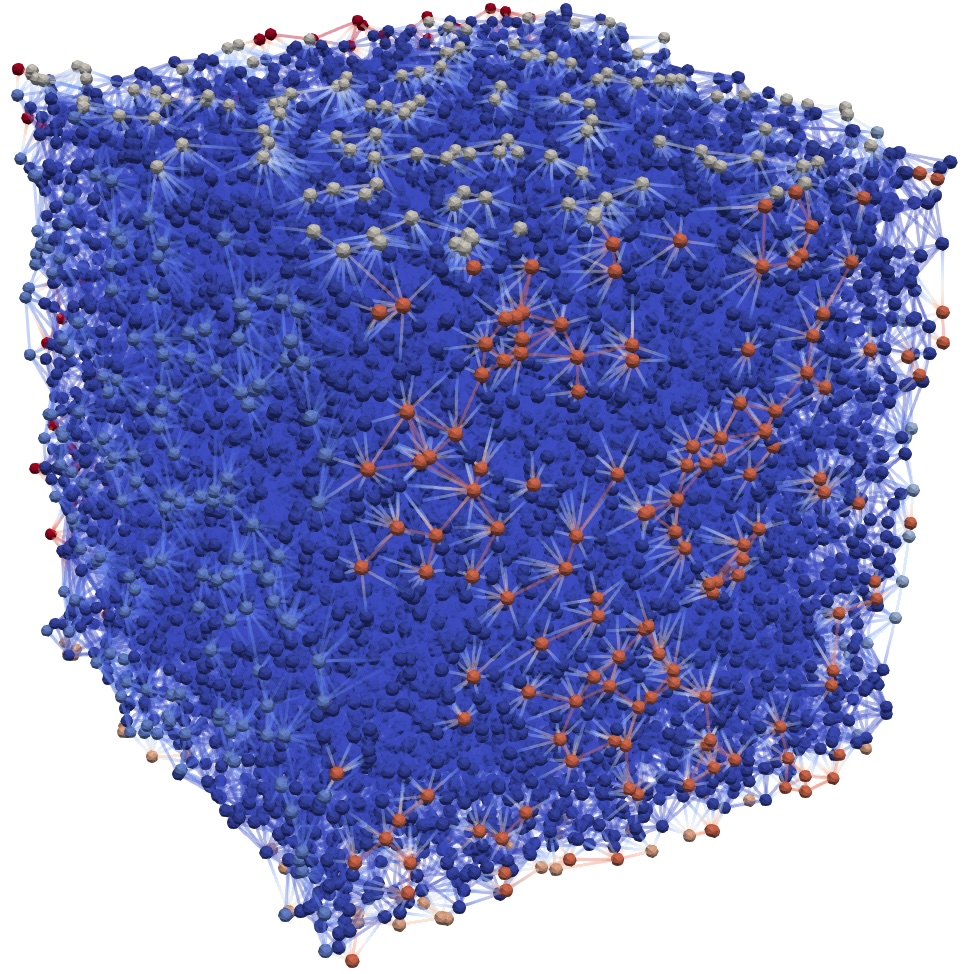}
\caption{\rev{Network-3b: Unstructured regular network}}
\end{subfigure}
\caption{\rev{Networks in 3D}}
\label{fig:test3}
\end{figure}

\begin{figure}[h!]
\centering
\begin{subfigure}{0.3\textwidth}
\includegraphics[width=1\linewidth]{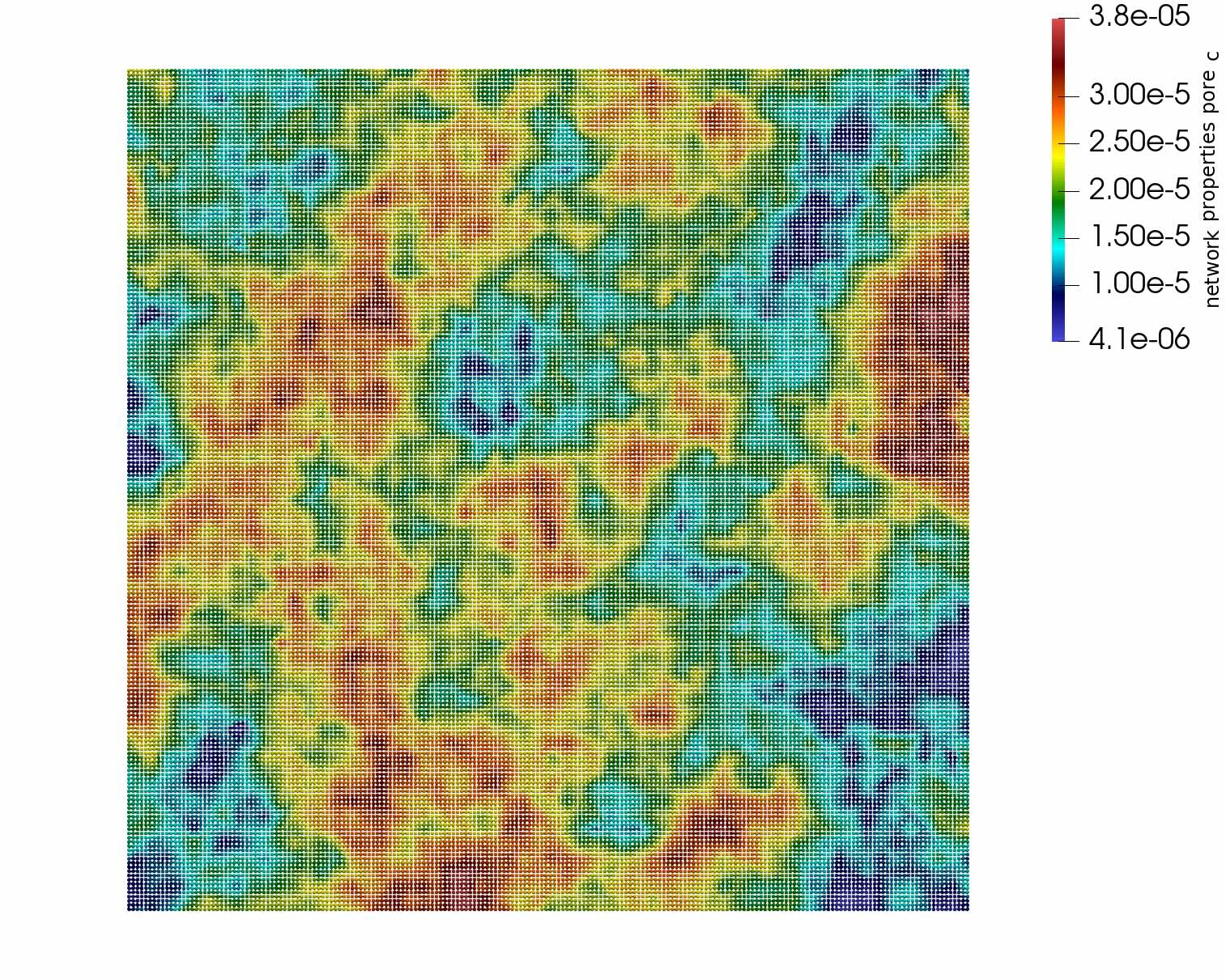}
\includegraphics[width=1\linewidth]{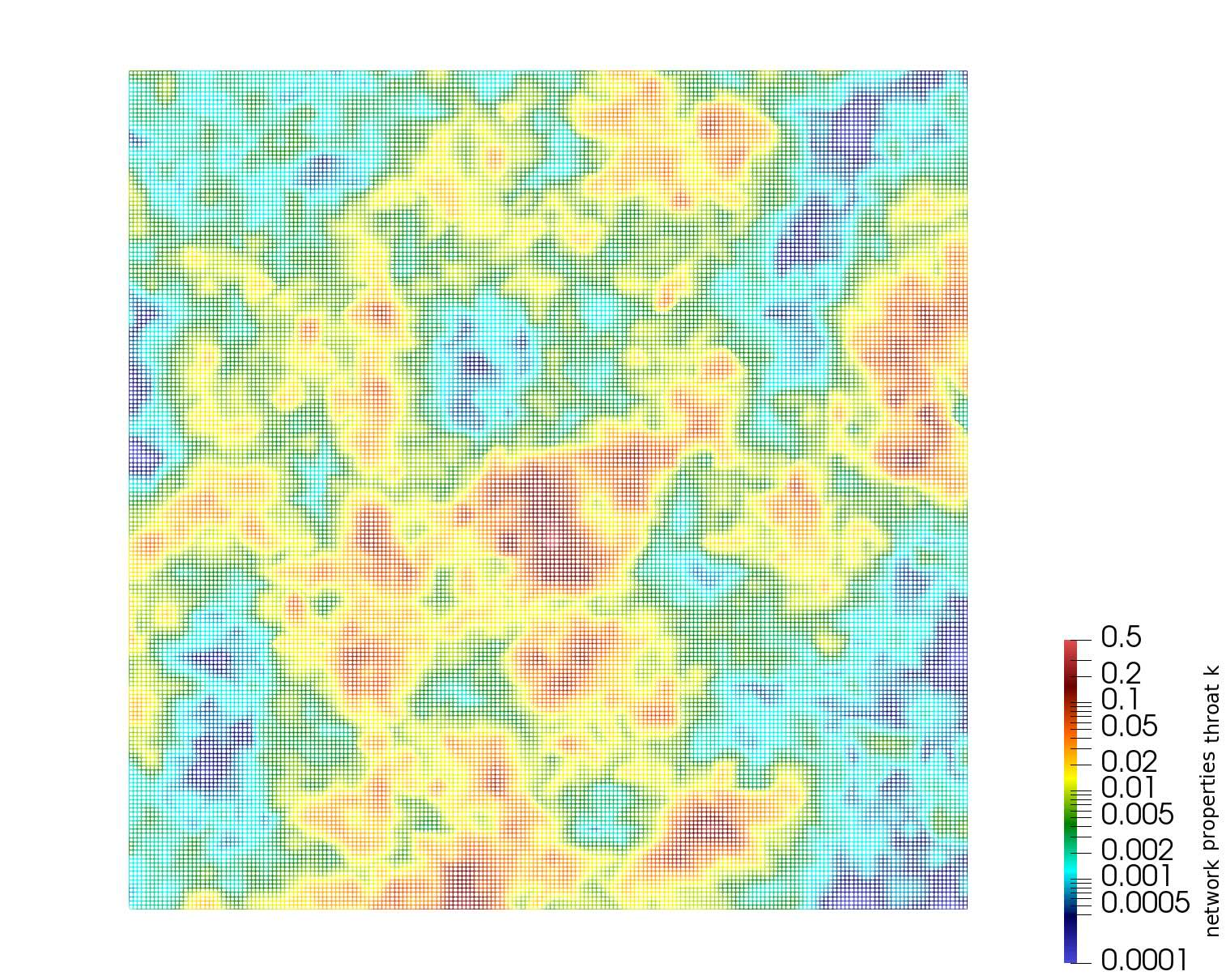}
\caption{\rev{Test-1a: SPE10 properties on Network-1a}}
\end{subfigure}
\ \ \ \ \ \
\begin{subfigure}{0.3\textwidth}
\includegraphics[width=1\linewidth]{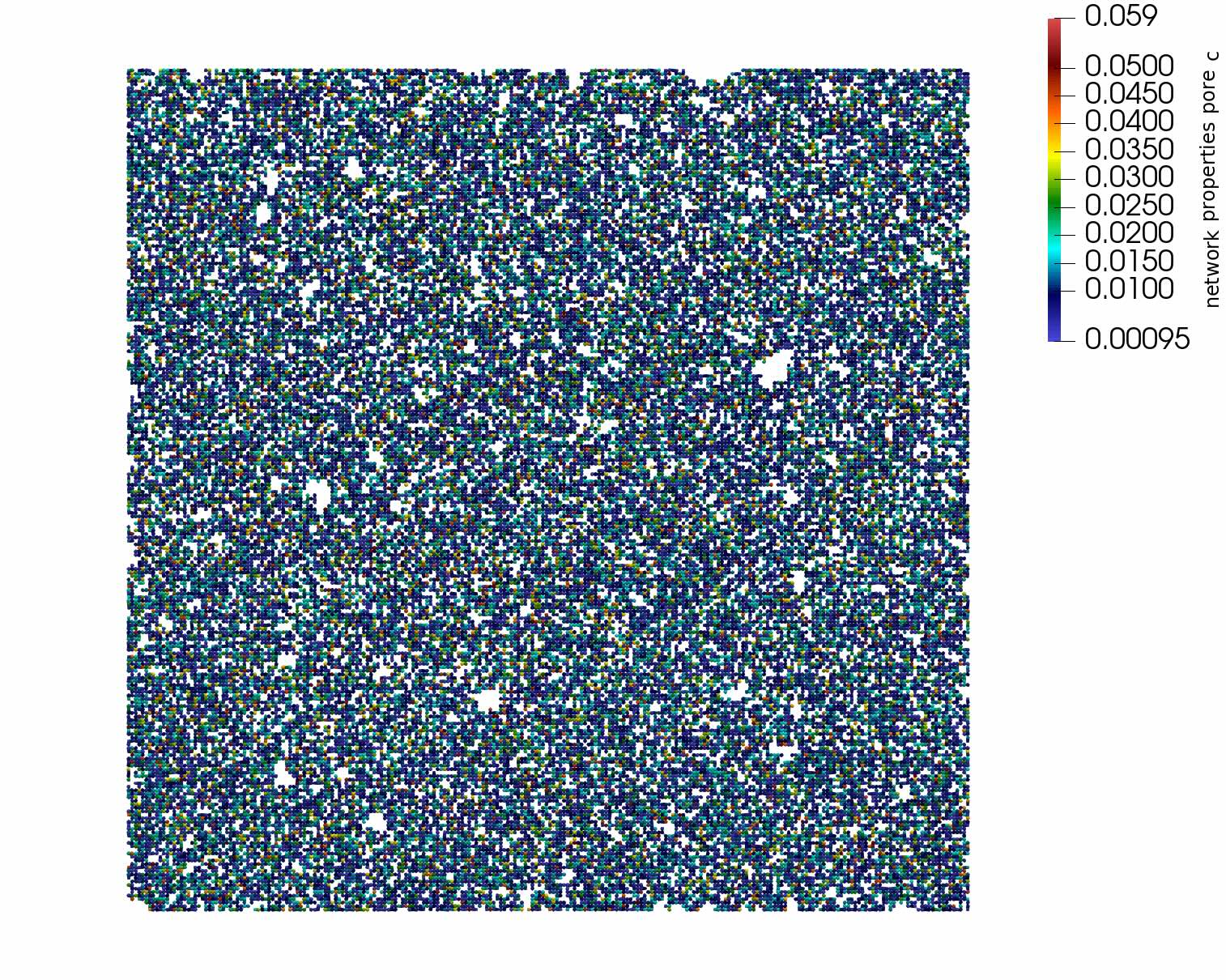}
\includegraphics[width=1\linewidth]{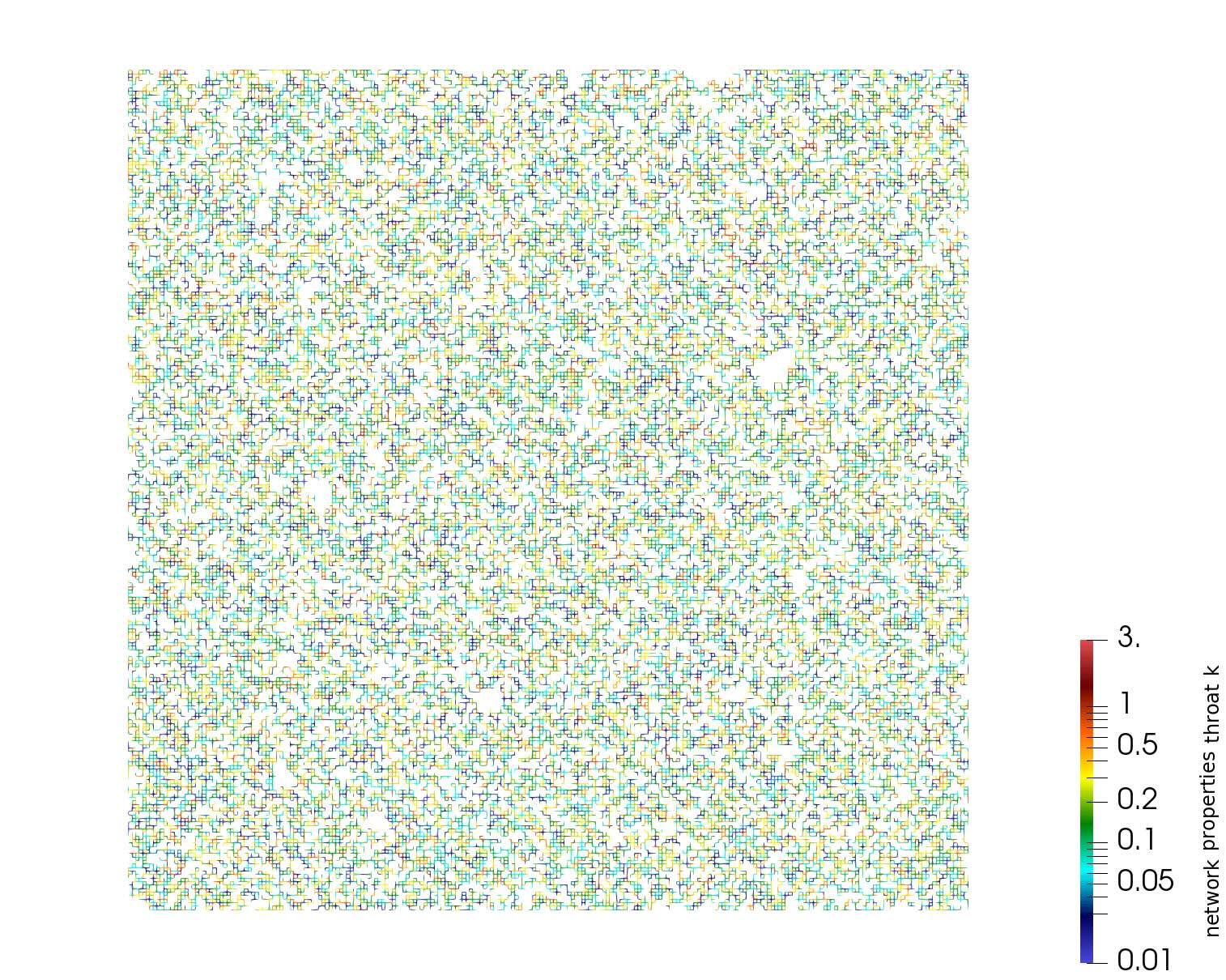}
\caption{\rev{Test-2a: Random properties on Network-2a}}
\end{subfigure}
\ \ \ \ \ \
\begin{subfigure}{0.3\textwidth}
\includegraphics[width=1\linewidth]{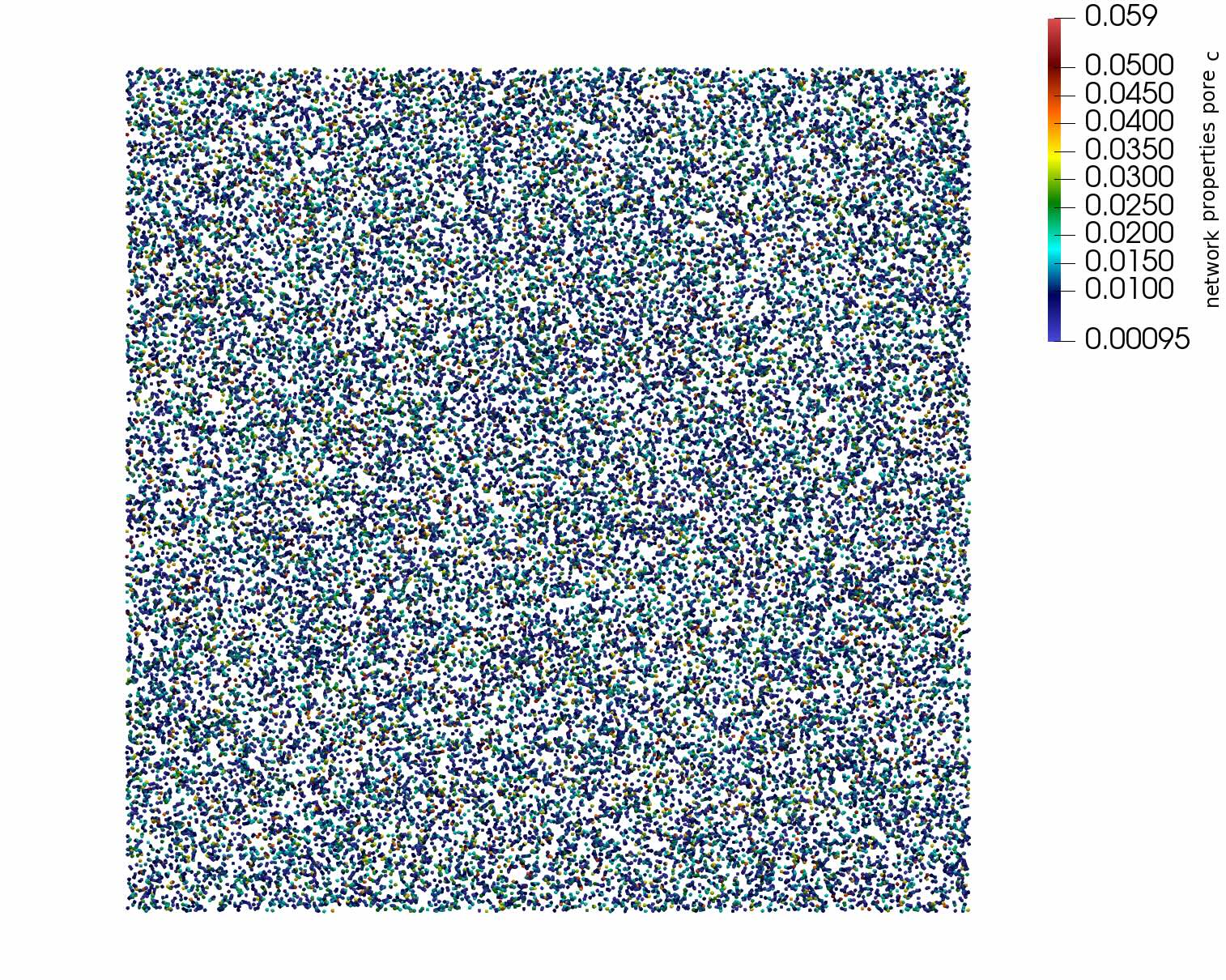}
\includegraphics[width=1\linewidth]{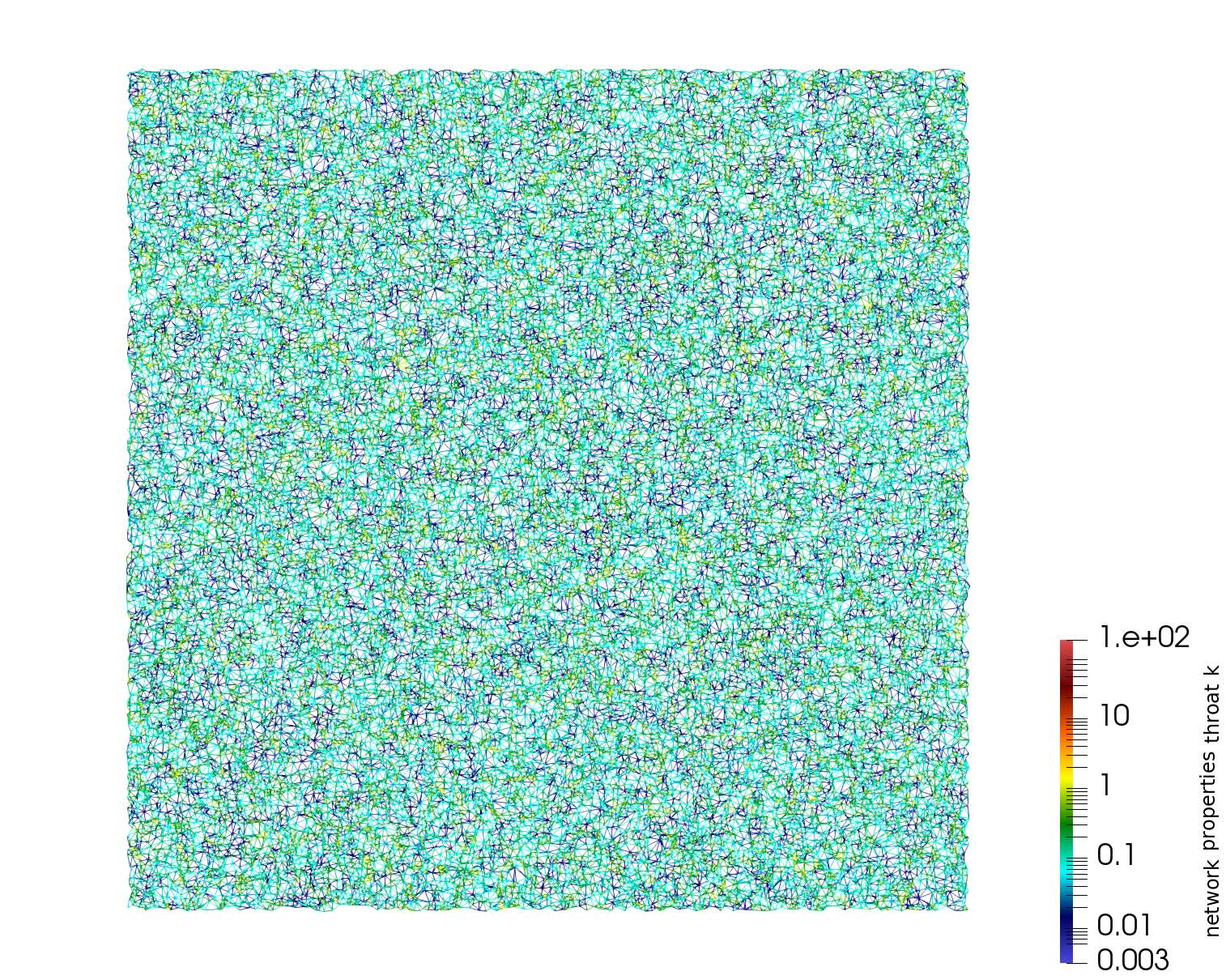}
\caption{\rev{Test-3a: Random properties on Network-3a}}
\end{subfigure}
\caption{\rev{Heterogeneous coefficients $c_i$ (first row) and $w_{ij}$ (second row) for 2D networks}}
\label{fig:test2k}
\end{figure}

\begin{figure}[h!]
\centering
\begin{subfigure}{0.3\textwidth}
\includegraphics[width=1\linewidth]{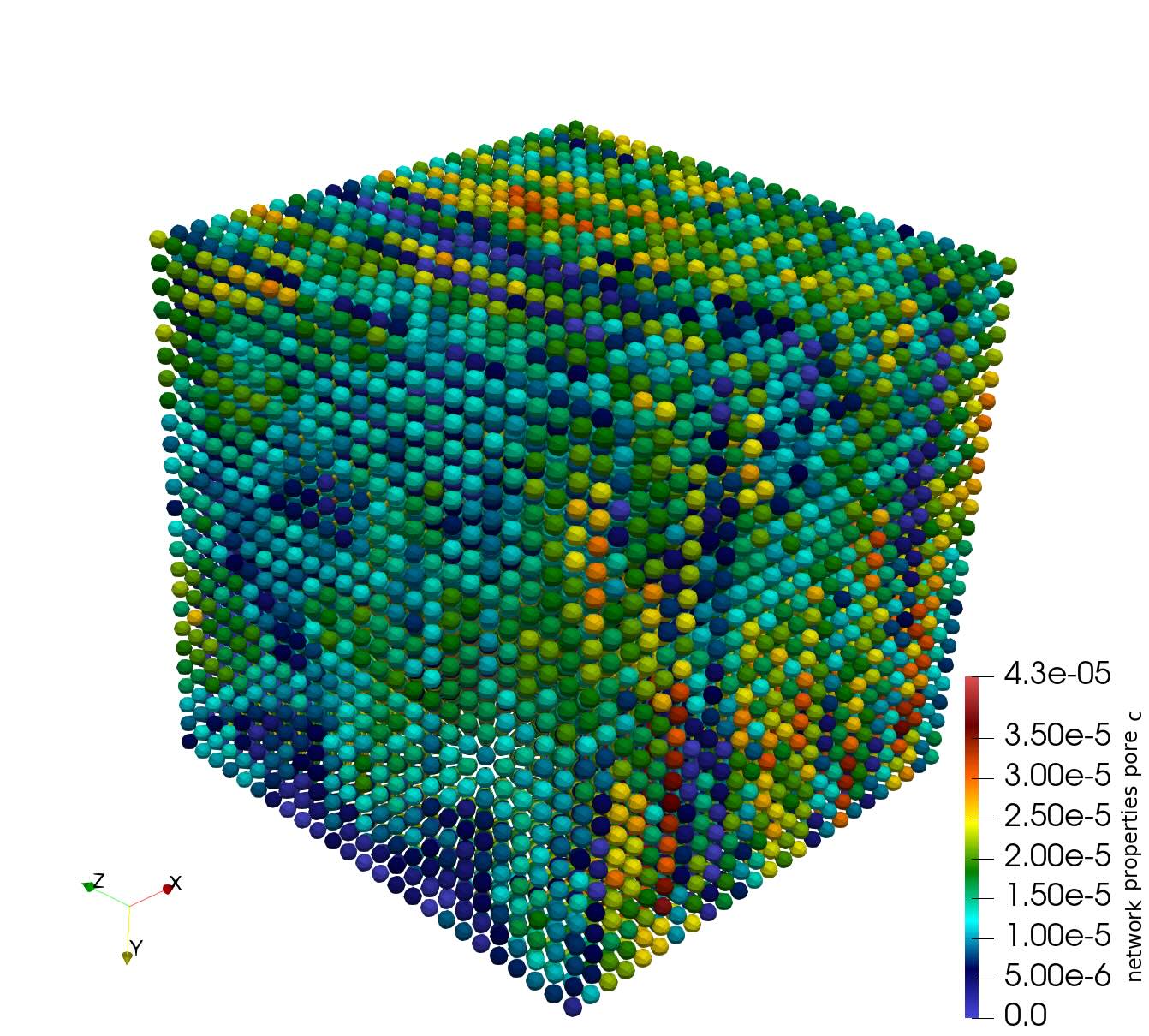}
\includegraphics[width=1\linewidth]{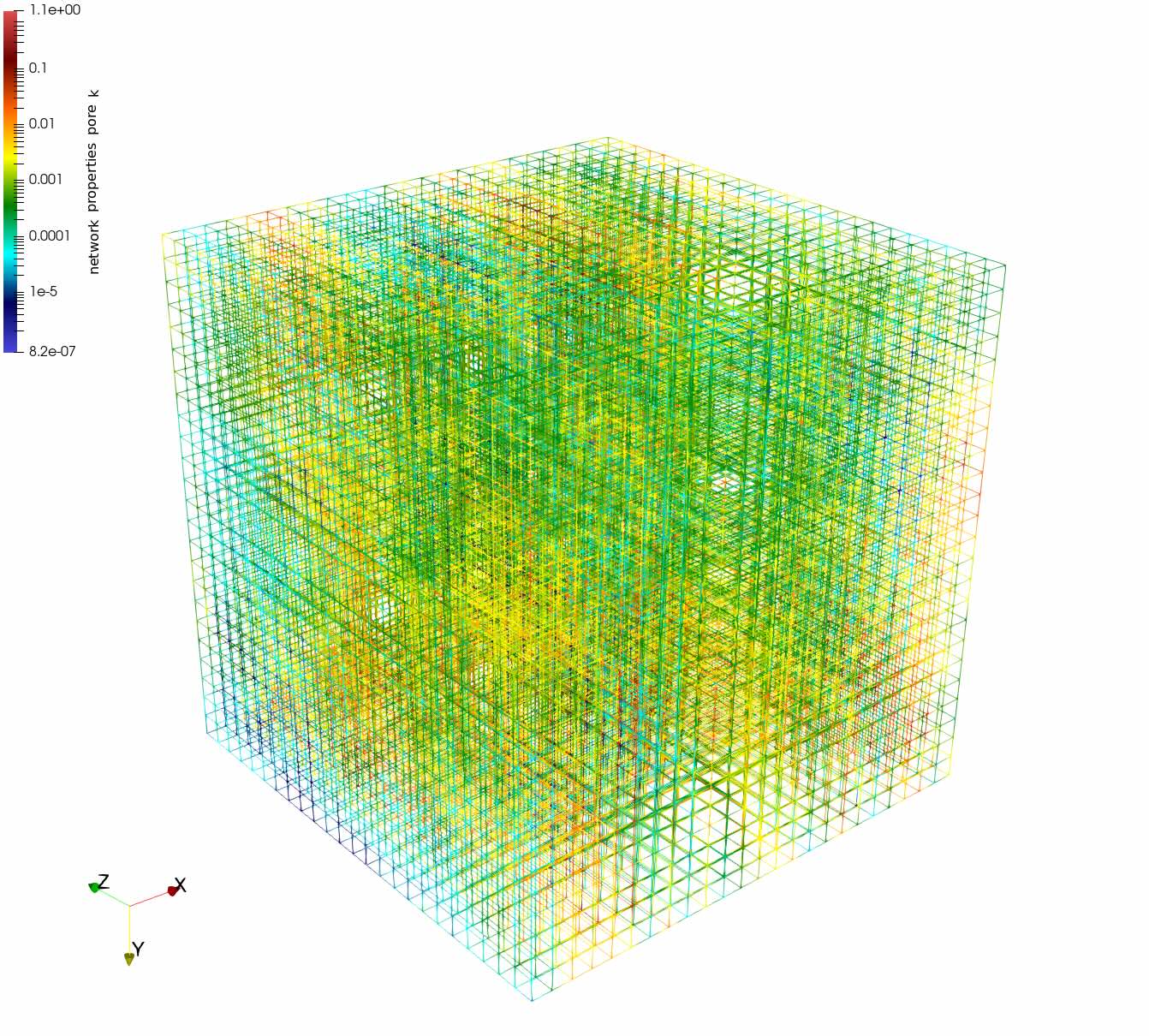}
\caption{\rev{Test-1b: SPE10 properties on Network-1b}}
\end{subfigure}
\ \ \ \ \ \
\begin{subfigure}{0.3\textwidth}
\includegraphics[width=1\linewidth]{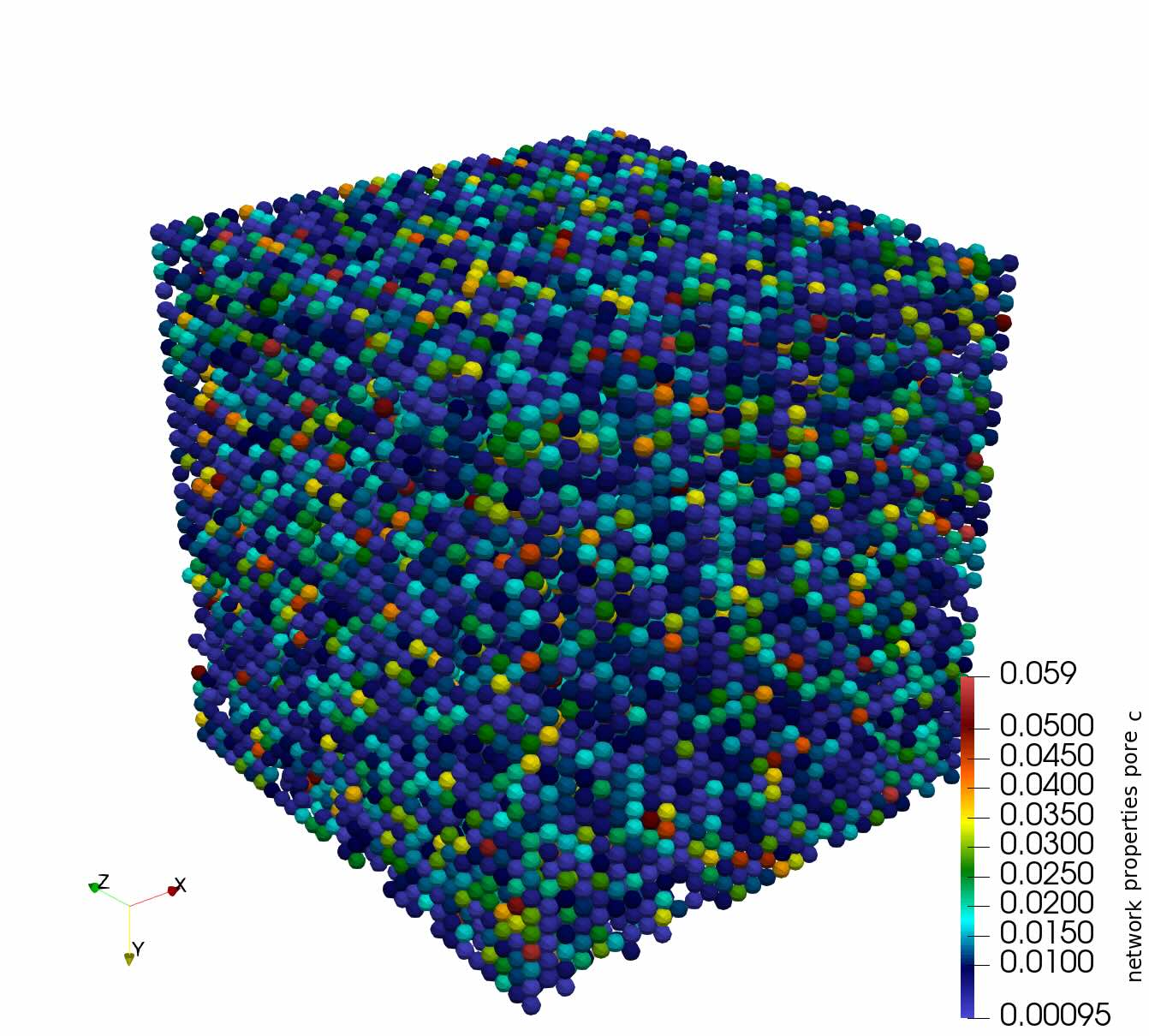}
\includegraphics[width=1\linewidth]{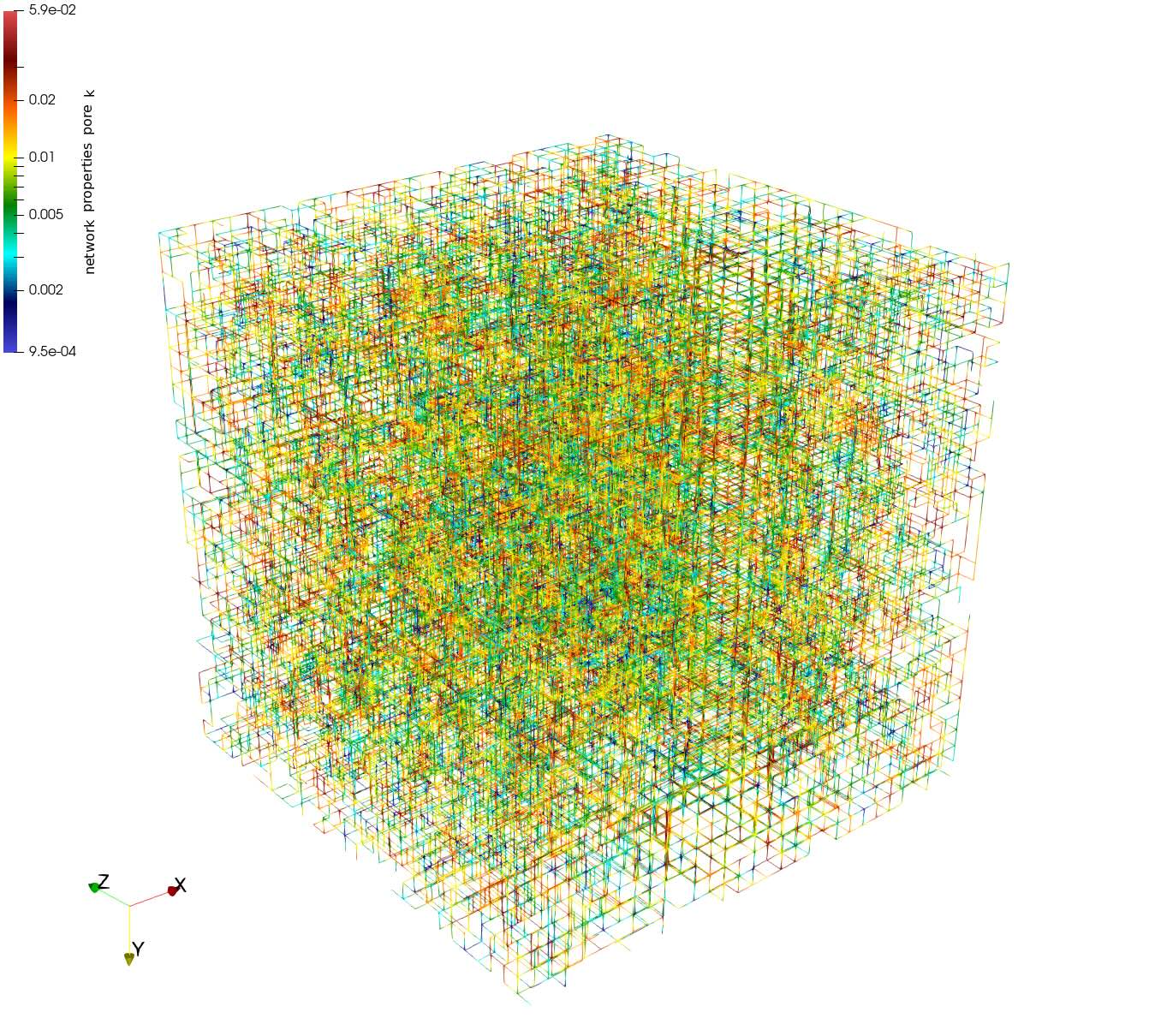}
\caption{\rev{Test-2b: Random properties on Network-2b}}
\end{subfigure}
\ \ \ \ \ \
\begin{subfigure}{0.3\textwidth}
\includegraphics[width=1\linewidth]{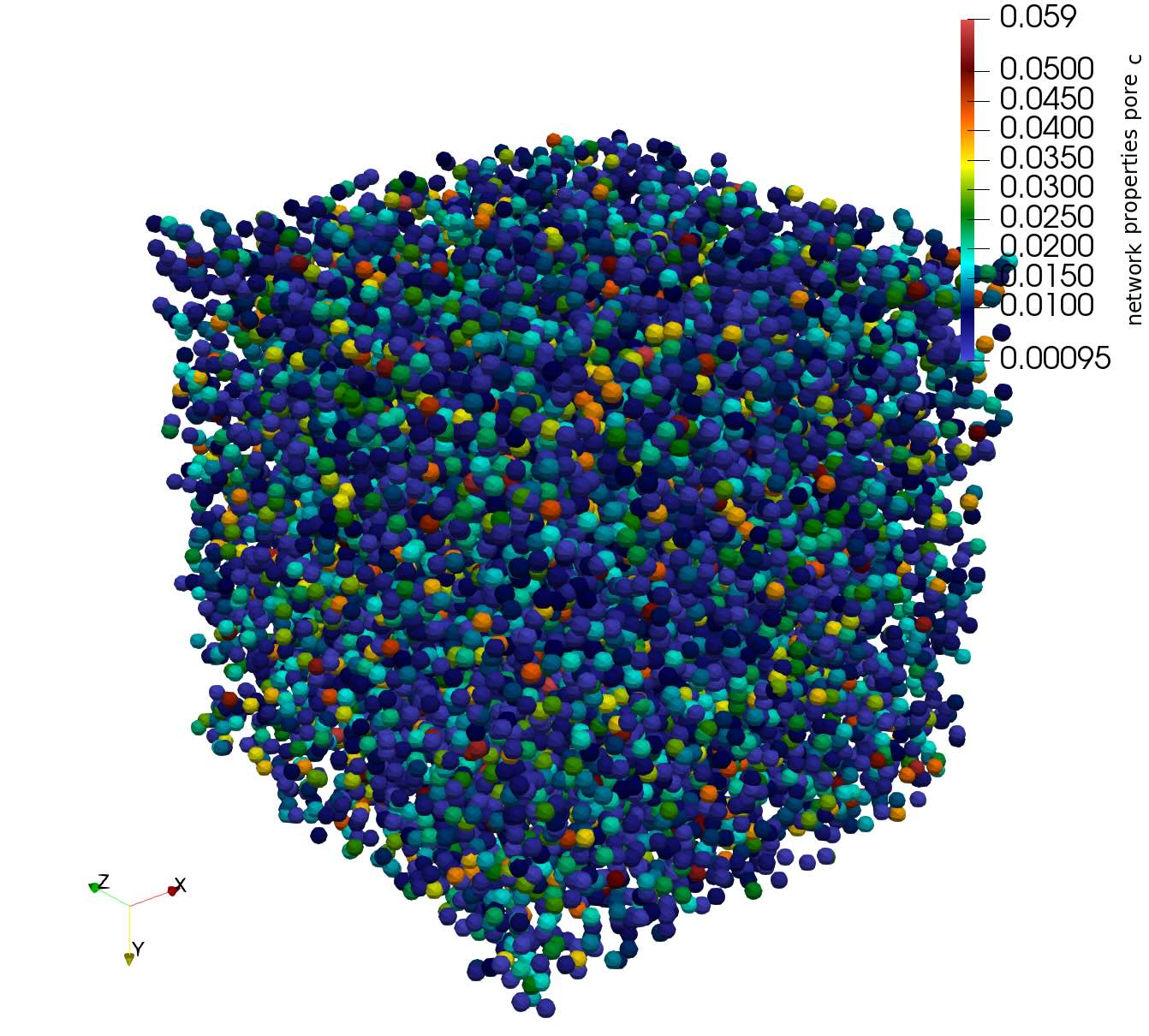}
\includegraphics[width=1\linewidth]{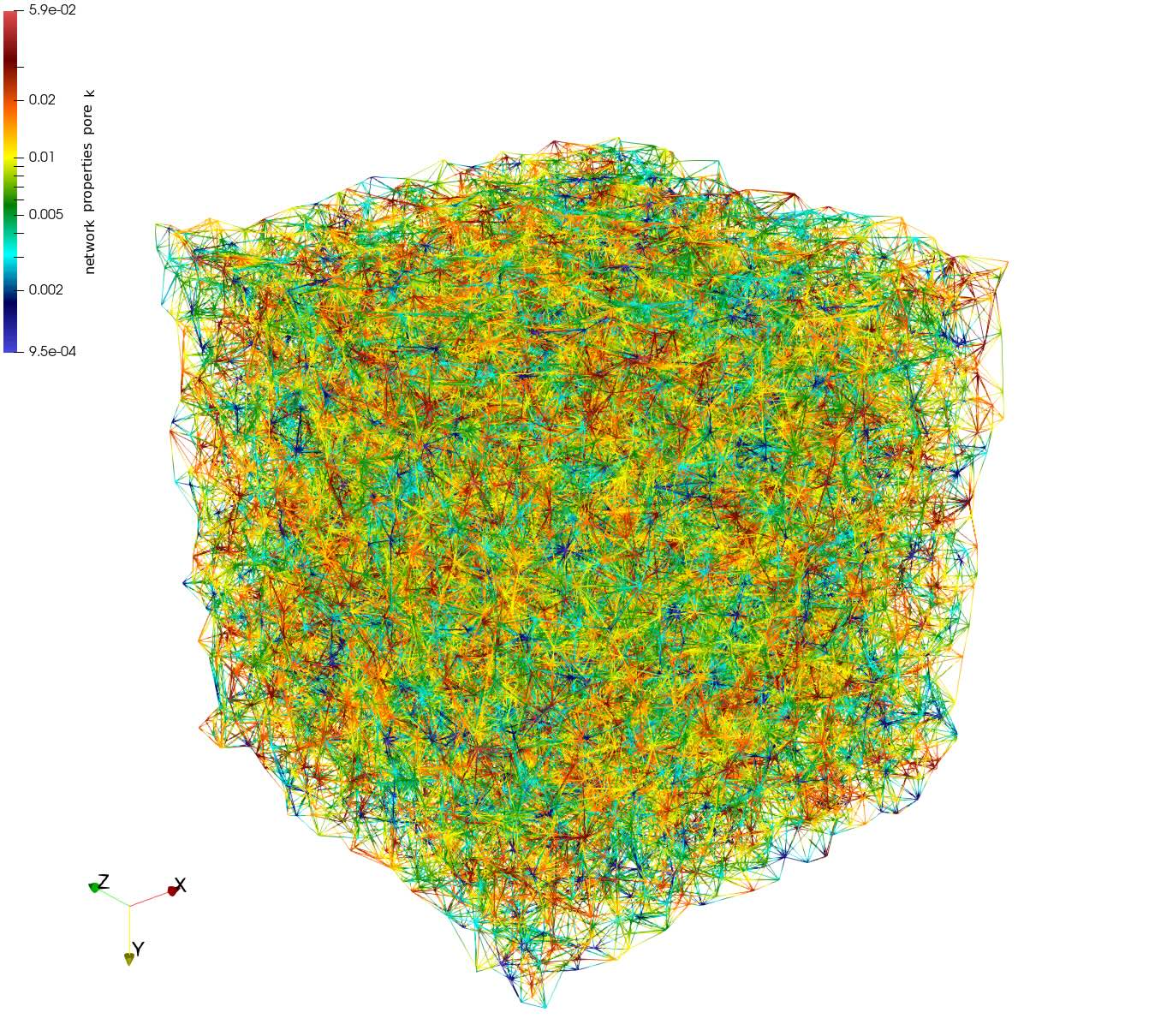}
\caption{\rev{Test-3b: Random properties on Network-3b}}
\end{subfigure}
\caption{\rev{Heterogeneous coefficients $c_i$ (first row) and $w_{ij}$  (second row) for 3D networks}}
\label{fig:test3k}
\end{figure}

\begin{figure}[h!]
\centering
\begin{subfigure}{0.3\textwidth}
\includegraphics[width=1\linewidth]{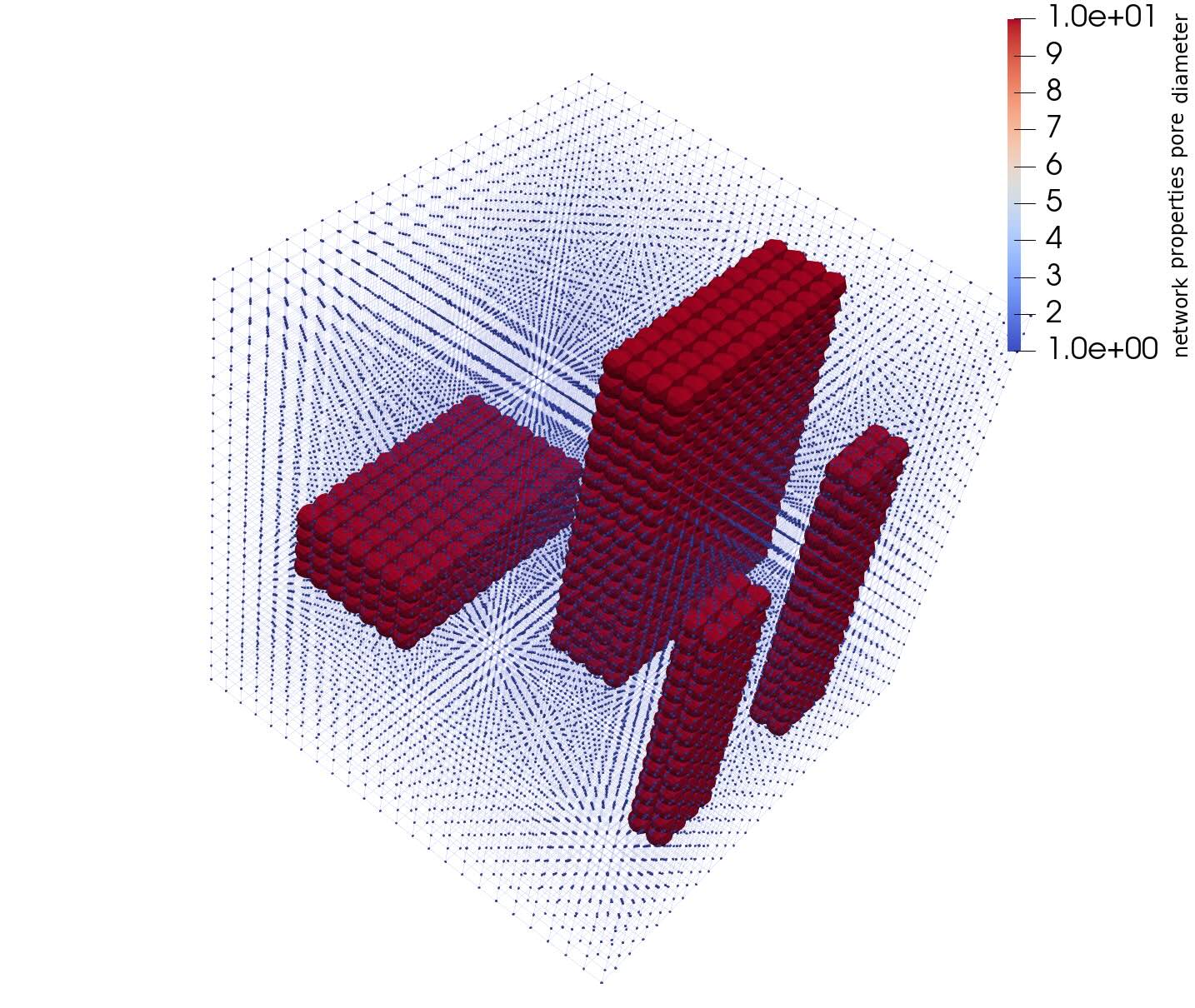}
\caption{\rev{Test-1c on Network-1b}}
\end{subfigure}
\ \ \ \ \ \
\begin{subfigure}{0.3\textwidth}
\includegraphics[width=1\linewidth]{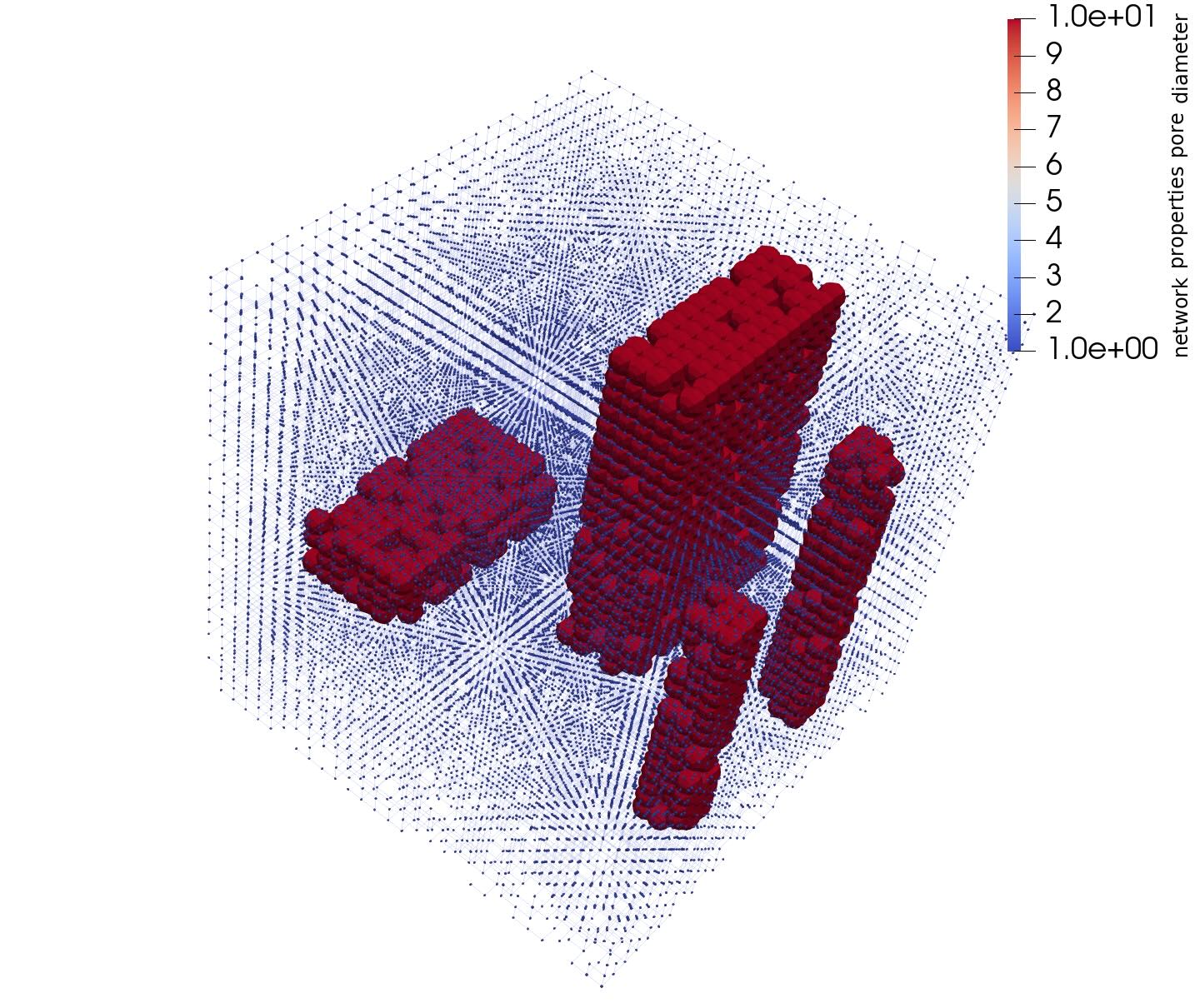}
\caption{\rev{Test-2c on Network-2b}}
\end{subfigure}
\ \ \ \ \ \
\begin{subfigure}{0.3\textwidth}
\includegraphics[width=1\linewidth]{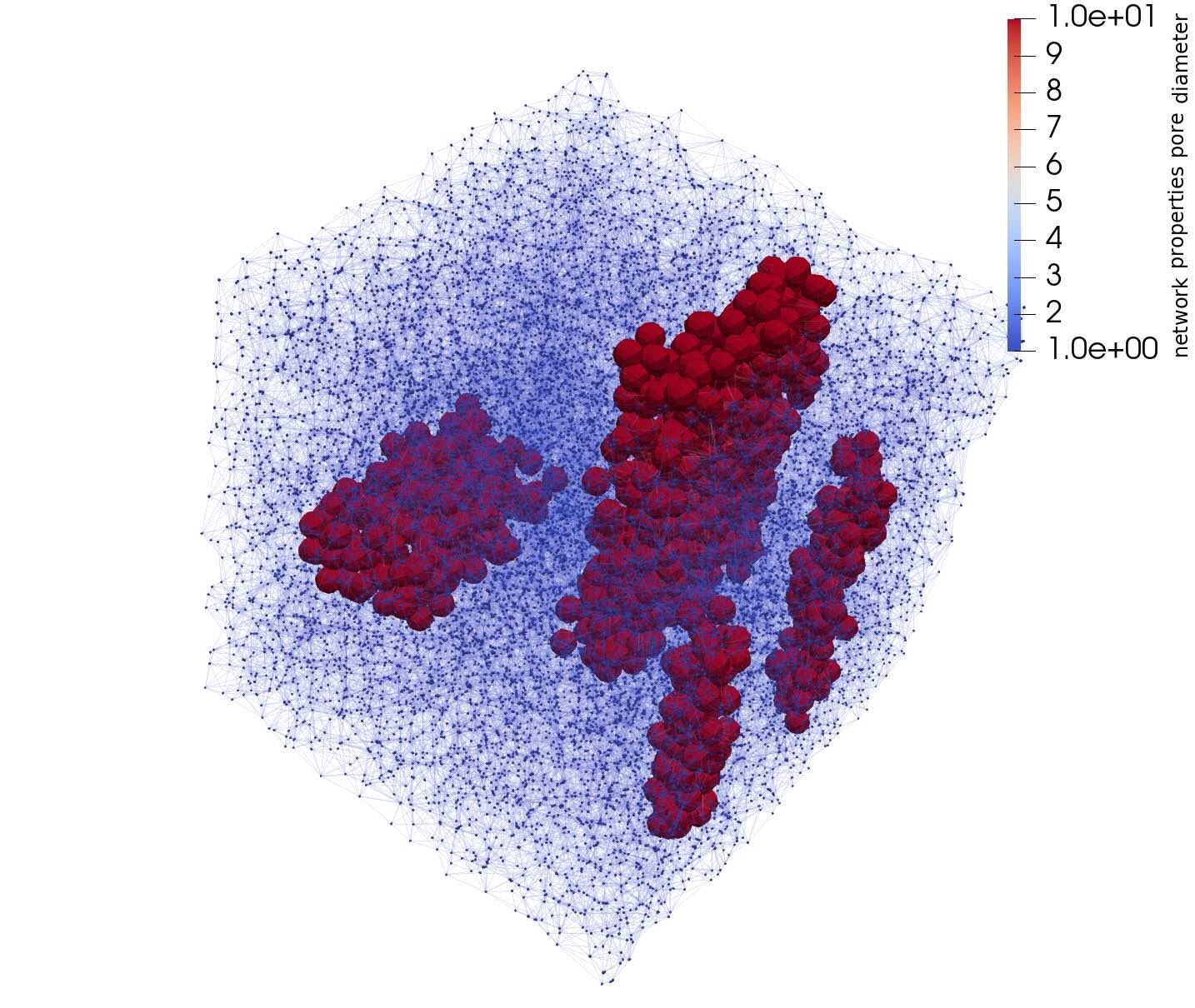}
\caption{\rev{Test-3c on Network-3b}}
\end{subfigure}
\caption{\rev{High-contrast 3D networks associated with a pore size}}
\label{fig:test3k2}
\end{figure}

\rev{The structured regular network is related to the finite volume approximation of the parabolic equation in the heterogeneous two-dimensional and three-dimensional domains. In this network, we assign a heterogeneous diffusion coefficient to each node (cell) and calculate connection weight as a harmonic average of the cell's coefficients that it connect. The heterogeneous properties are assigned from the SPE10 test \cite{christie2001tenth}.
To set heterogeneous properties of other networks, we follow the approach from the pore-network modeling that follows the geometrical representation of the pore structure. We randomly assign pore and throat diameters ($d^p_{i}$ and $d^t_{ij}$), then we set time-derivative coefficient $c_i$ as a pore volume (sphere)  and calculate connection weight $w_{ij}$ using the Hagen-Poiseuille equation for single phase flow in a cylindrical tube \cite{gostick2016openpnm}
\begin{equation}
c_i = \frac{4}{3} \pi R^3_i, \quad 
w_{ij} = \frac{\pi R^4_{ij}}{8 \mu L_{ij}},
\end{equation}
where $R_i = d^p_{i}/2$ is the pore radius, $L_{ij}$ and $R_{ij} = d^t_{ij}/2$ are the length and radius of the throat 
connecting pores $i$ and $j$.}

\rev{We consider the following test cases: }
\begin{itemize}
\item 
\rev{\textit{Network-1} - structured regular network:}
\begin{itemize}
\item \rev{Test-1a with $T = 20$ on Network-1a,   
$4.08 \cdot 10^{-6} \leq c_i \leq 3.82 \cdot 10^{-5}$ and  
$0.0001 \leq w_{ij} \leq 0.5085$}
\item \rev{Test-1b with $T = 0.6$ on Network-1b,   
$3.01 \cdot 10^{-6}  \leq c_i \leq 2.0$ and  
$2.58 \cdot 10^{-7} \leq w_{ij} \leq 2.0$}
\item \rev{Test-1c with $T = 10$ on Network-1b,   
$0.52  \leq c_i \leq 523.60$ and  
$6.14 \leq w_{ij} \leq 6.14 \cdot 10^4$}
\end{itemize}
\item \rev{\textit{Network-2} - structured irregular network:}
\begin{itemize}
\item \rev{Test-2a  with $T = 4000$ on Network-2a,   
$0.0007 \cdot 10^{-6} \leq c_i \leq 0.0309$ and  
$0.0007 \leq w_{ij} \leq  1.005$}
\item \rev{Test-2b  with $T = 200$ on Network-2b,   
$2.74 \cdot 10^{-6}  \leq c_i \leq  0.0587$ and  
$5.58 \cdot 10^{-7} \leq w_{ij} \leq 0.3073$}
\item \rev{Test-2c with $T = 10$ on Network-2b,   
$0.52  \leq c_i \leq 523.60$ and  
$7.36 \leq w_{ij} \leq 7.36 \cdot 10^4$}
\end{itemize}
\item \rev{\textit{Network-3} - unstructured irregular network:}
\begin{itemize}
\item \rev{Test-3a  with $T = 300$ on Network-3a,   
$2.74 \cdot 10^{-6} \leq c_i \leq 0.0588$ and  
$1.41 \cdot 10^{-7} \leq w_{ij} \leq 3.4619$}
\item \rev{Test-3b  with $T = 20$ on Network-3b,   
$2.77 \cdot 10^{-6}  \leq  c_i \leq 0.0588$ and  
$1.56 \cdot 10^{-8} \leq w_{ij} \leq 0.2330$}
\item \rev{Test-3c with $T = 10$ on Network-3b,   
$0.52  \leq c_i \leq 523.60$ and  
$1.81 \leq w_{ij} \leq 1.19 \cdot 10^6$}
\end{itemize}
\end{itemize}
\rev{The network properties are depicted in Figures \ref{fig:test2k}, \ref{fig:test3k} and  \ref{fig:test3k2}for 2D and 3D tests. 
From first column in Figures \ref{fig:test2k} and \ref{fig:test3k}, we observe the heterogeneous coefficients based on SPE10 test used in continuum-scale reservoir simulations \cite{christie2001tenth}. We took subdomain from the original heterogeneity, where we use porosity to generate node coefficient $c_i$ and use permeability to assign connection weight $w_{ij}$.  }

\rev{As the third test case for 3D networks, we consider high-contrast properties. To generate a high-contrast network, we set node diameter equal to 10 in several subdomains; in the rest of the domain, we set node diameter equal to 1. In the Figure \ref{fig:test3k2}, we depicted three network structures with nodes and connections, where we scale the node's volume based on the node diameter. Such a case can be considered when we have a subdomain with different microstructures in pore network models. Then, we set throat diameter as a harmonic average between pores diameters that the throat is connecting. The coefficients $c_i$ and $w_{ij}$ are generated based on the Hagen-Poiseuille equation discussed above. In the second and third rows of  Figure \ref{fig:test3k2}, we depicted corresponded values of high-contrast coefficients. }

\begin{figure}[h!]
\centering
\begin{subfigure}{1\textwidth}
\centering
\includegraphics[width=0.17\linewidth]{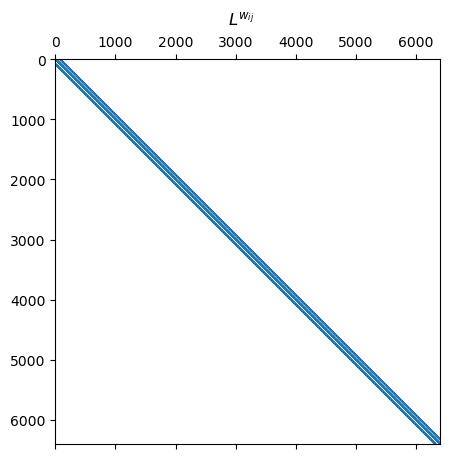}
\includegraphics[width=0.19\linewidth]{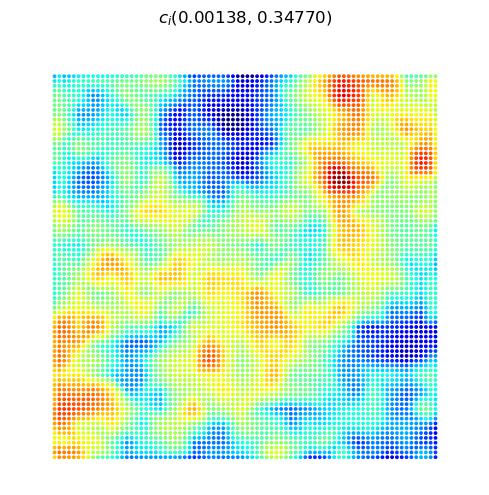}
\includegraphics[width=0.19\linewidth]{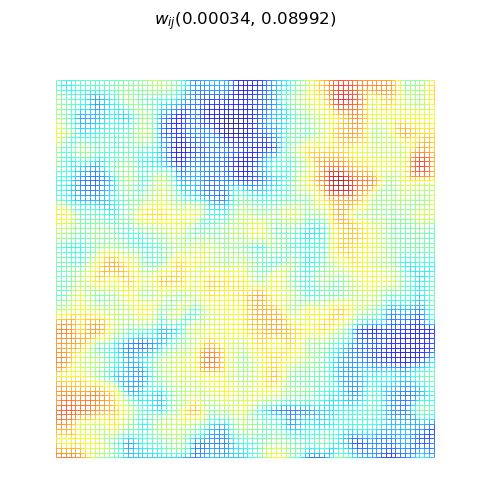}
\includegraphics[width=0.31\linewidth]{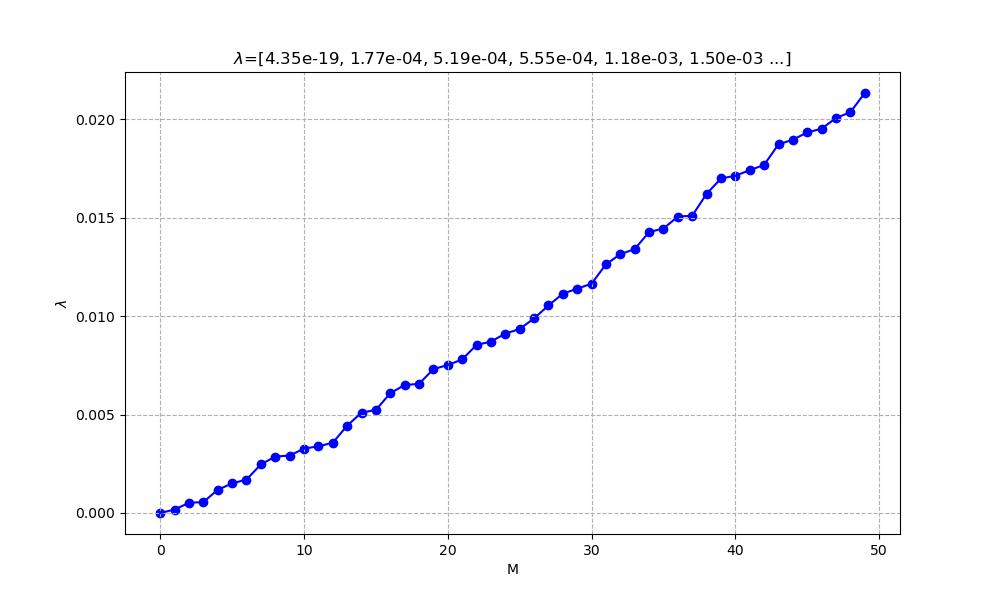}
\caption{\rev{Test-1a: Local matrix $L^{\omega_{ij}} = D^{\omega_{ij}} - W^{\omega_{ij}}$; illustration of  $D^{\omega_{ij}}$ and $W^{\omega_{ij}}$ and plot of the first smallest eigenvalues.}}
\end{subfigure}
\begin{subfigure}{1\textwidth}
\includegraphics[width=0.12\linewidth]{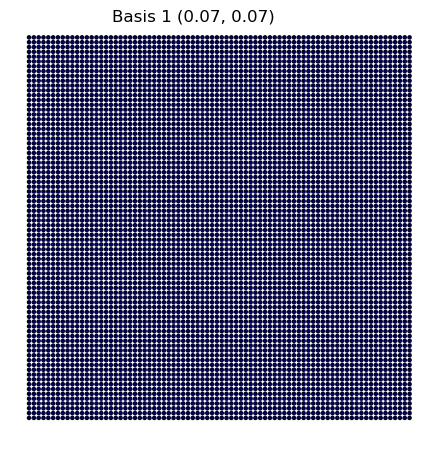}
\includegraphics[width=0.12\linewidth]{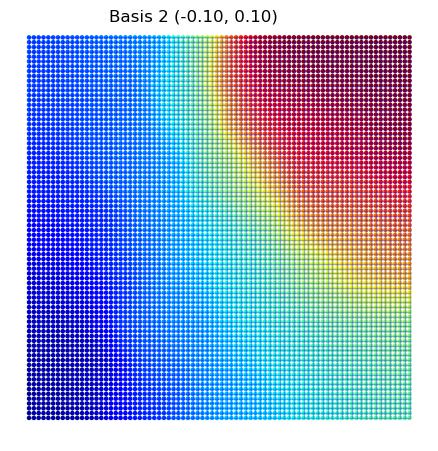}
\includegraphics[width=0.12\linewidth]{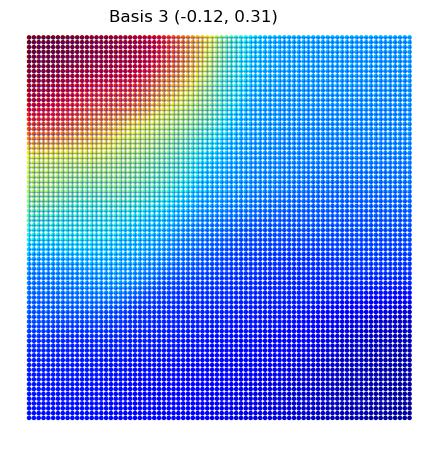}
\includegraphics[width=0.12\linewidth]{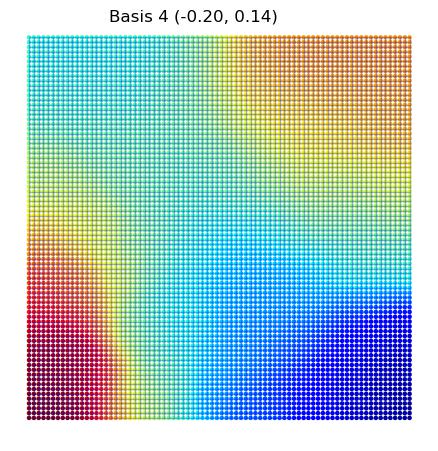}
\includegraphics[width=0.12\linewidth]{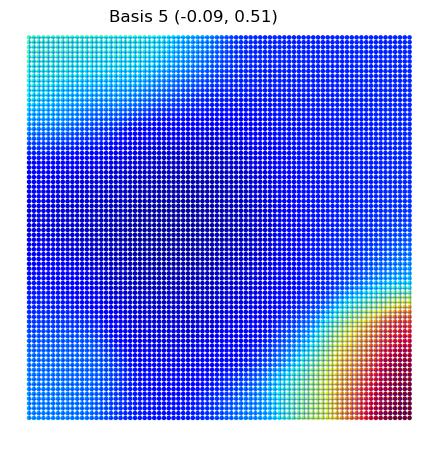}
\includegraphics[width=0.12\linewidth]{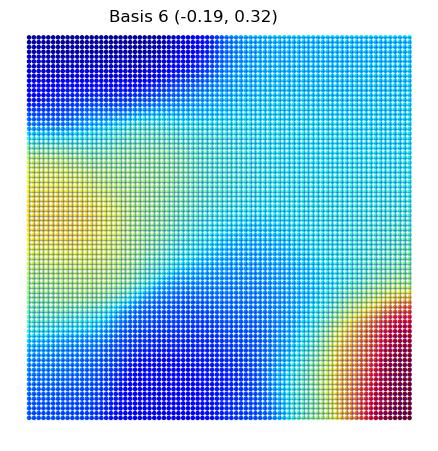}
\includegraphics[width=0.12\linewidth]{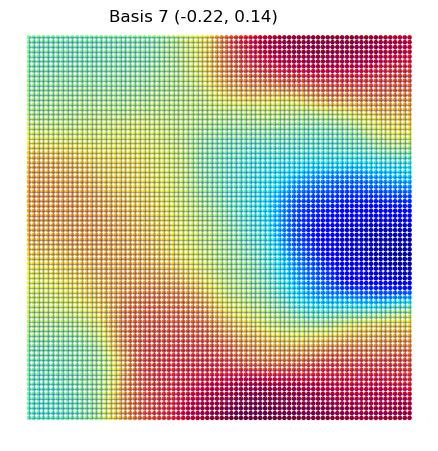}
\includegraphics[width=0.12\linewidth]{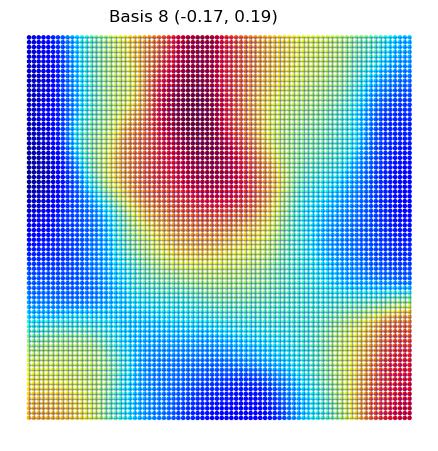}
\caption{\rev{Test-1a: The first eight eigenvectors corresponded to the smallest eigenvalues of $L^{\omega_{ij}} \phi^{\omega_{ij}}_r = \lambda^{\omega_{ij}}_r D^{\omega_{ij}} \phi^{\omega_{ij}}_r$.}}
\end{subfigure}
\begin{subfigure}{1\textwidth}
\centering
\includegraphics[width=0.17\linewidth]{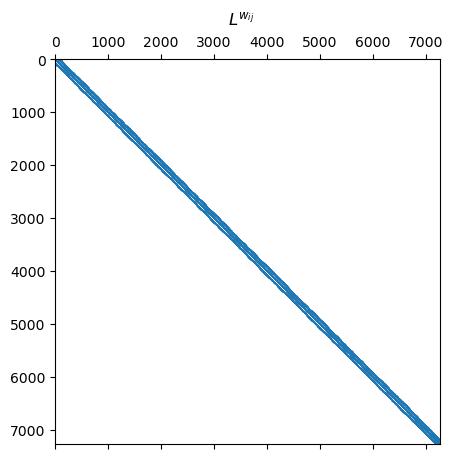}
\includegraphics[width=0.19\linewidth]{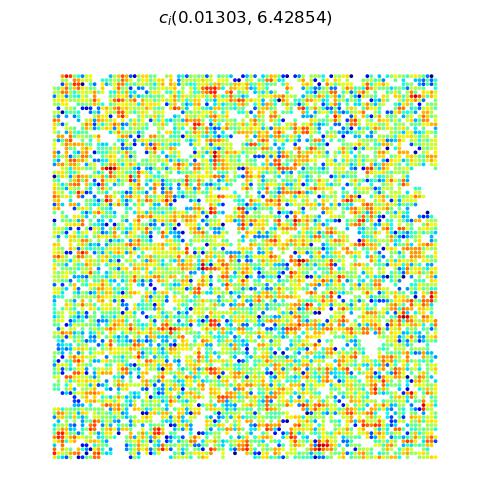}
\includegraphics[width=0.19\linewidth]{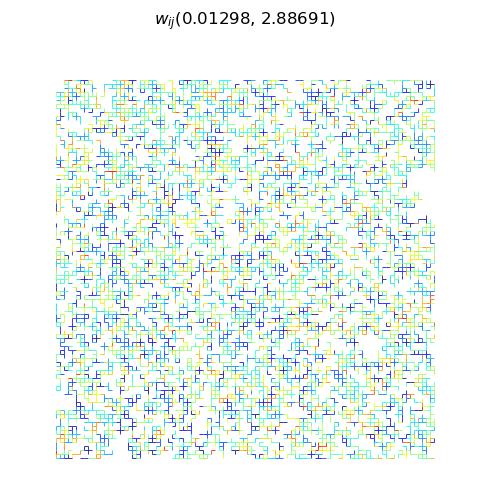}
\includegraphics[width=0.31\linewidth]{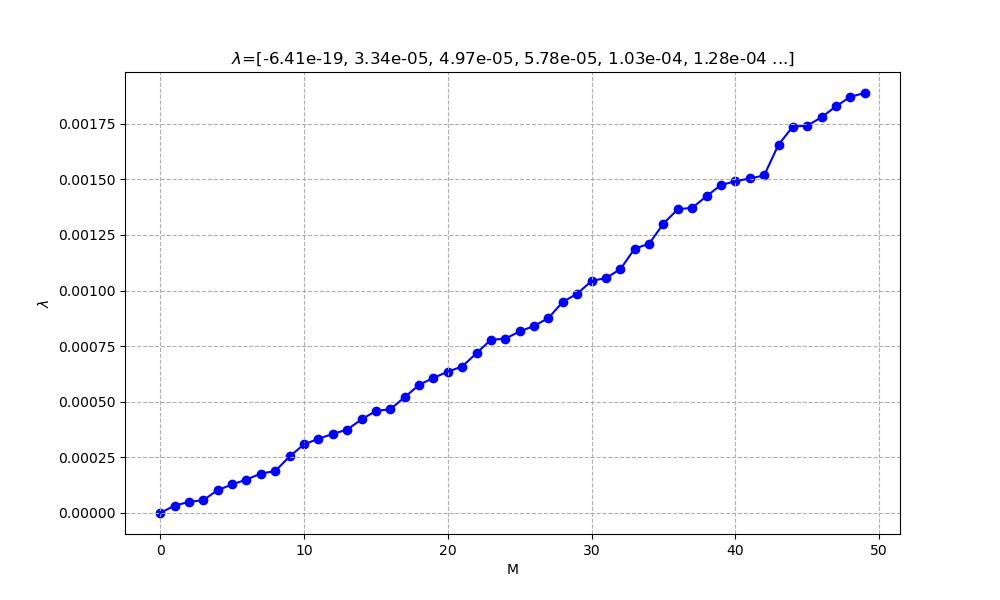}
\caption{\rev{Test-2a: Local matrix $L^{\omega_{ij}} = D^{\omega_{ij}} - W^{\omega_{ij}}$; illustration of  $D^{\omega_{ij}}$ and $W^{\omega_{ij}}$ and plot of the first smallest eigenvalues.}}
\end{subfigure}
\begin{subfigure}{1\textwidth}
\includegraphics[width=0.12\linewidth]{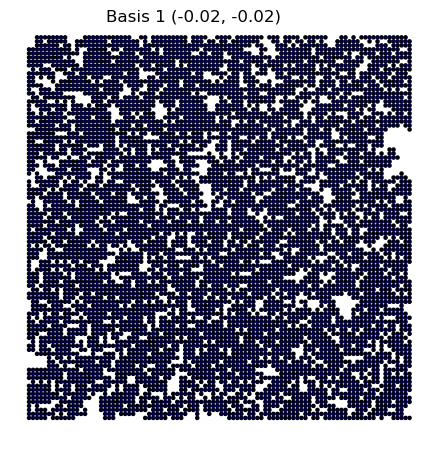}
\includegraphics[width=0.12\linewidth]{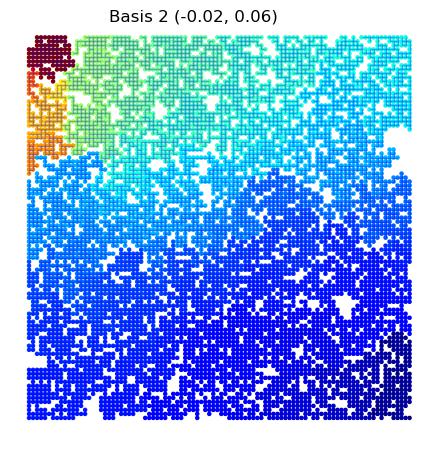}
\includegraphics[width=0.12\linewidth]{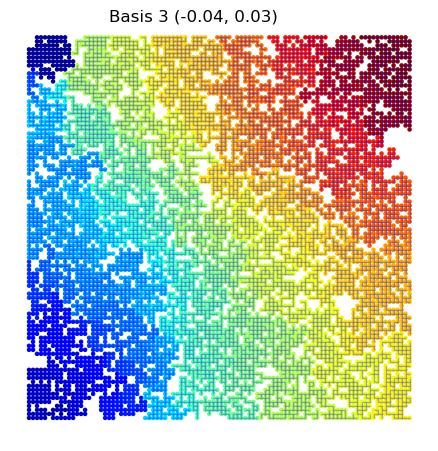}
\includegraphics[width=0.12\linewidth]{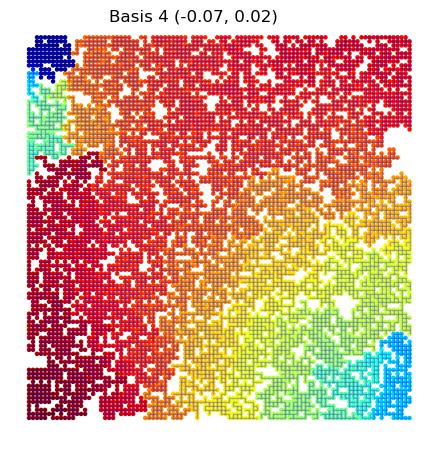}
\includegraphics[width=0.12\linewidth]{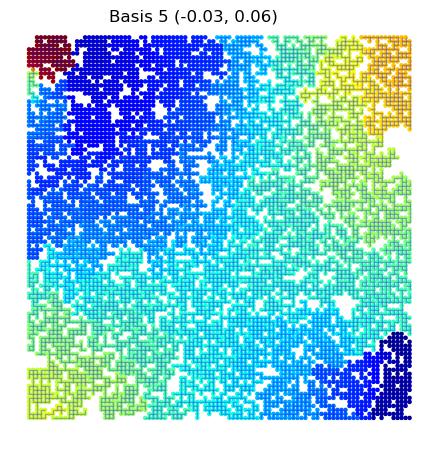}
\includegraphics[width=0.12\linewidth]{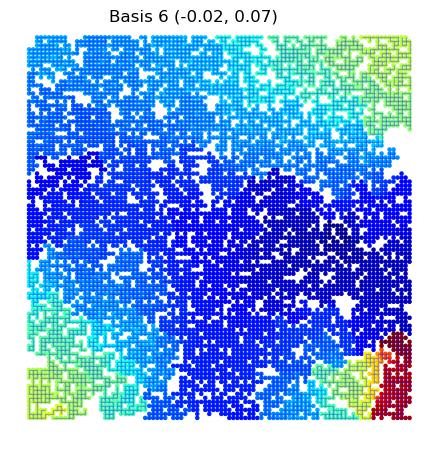}
\includegraphics[width=0.12\linewidth]{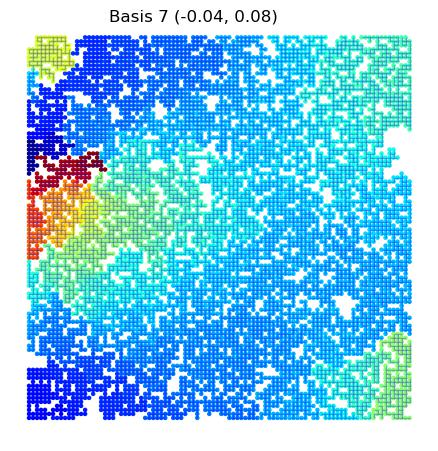}
\includegraphics[width=0.12\linewidth]{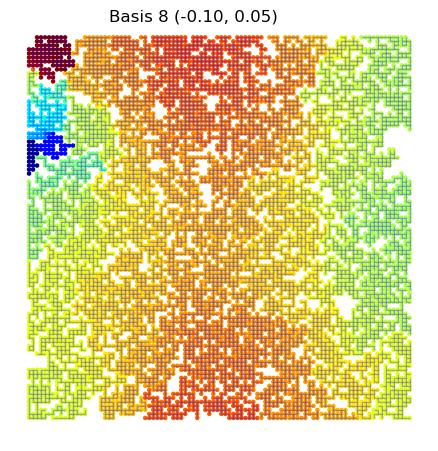}
\caption{\rev{Test-2a: The first eight eigenvectors corresponded to the smallest eigenvalues of $L^{\omega_{ij}} \phi^{\omega_{ij}}_r = \lambda^{\omega_{ij}}_r D^{\omega_{ij}} \phi^{\omega_{ij}}_r$.}}
\end{subfigure}
\begin{subfigure}{1\textwidth}
\centering
\includegraphics[width=0.17\linewidth]{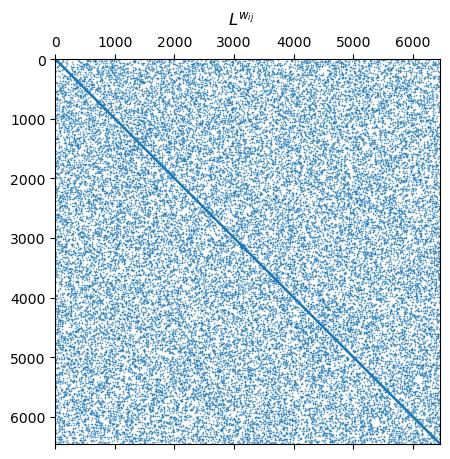}
\includegraphics[width=0.19\linewidth]{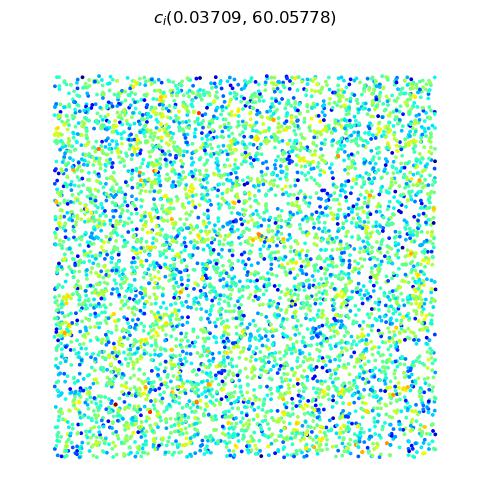}
\includegraphics[width=0.19\linewidth]{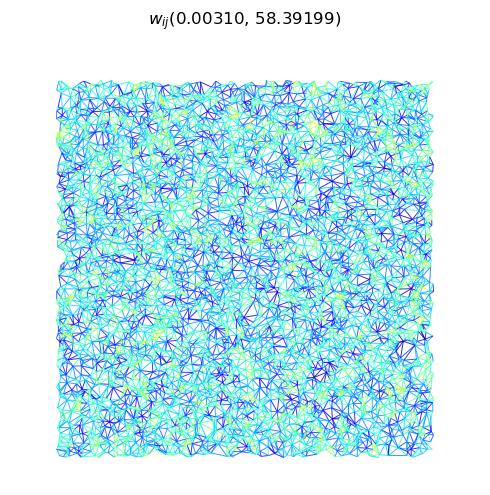}
\includegraphics[width=0.31\linewidth]{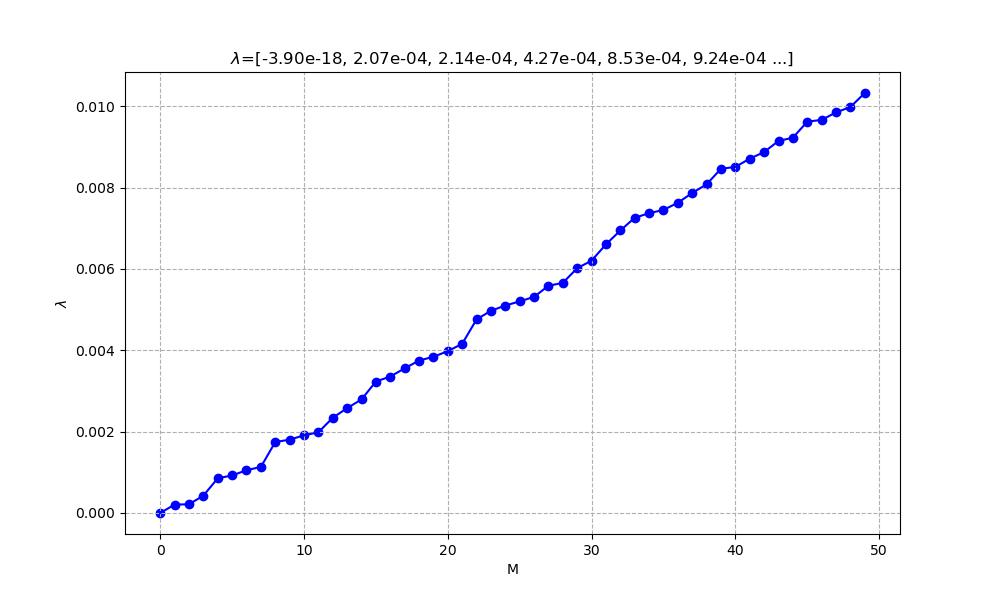}
\caption{\rev{Test-3a: Local matrix $L^{\omega_{ij}} = D^{\omega_{ij}} - W^{\omega_{ij}}$; illustration of  $D^{\omega_{ij}}$ and $W^{\omega_{ij}}$ and plot of the first smallest eigenvalues.}}
\end{subfigure}
\begin{subfigure}{1\textwidth}
\includegraphics[width=0.12\linewidth]{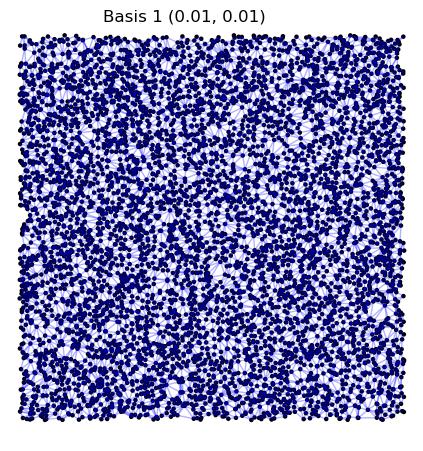}
\includegraphics[width=0.12\linewidth]{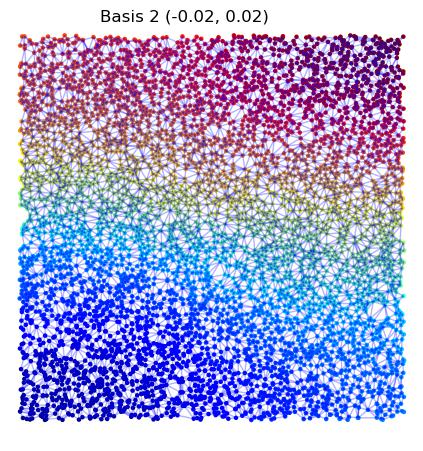}
\includegraphics[width=0.12\linewidth]{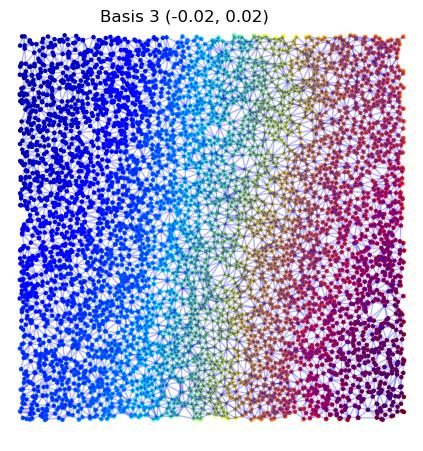}
\includegraphics[width=0.12\linewidth]{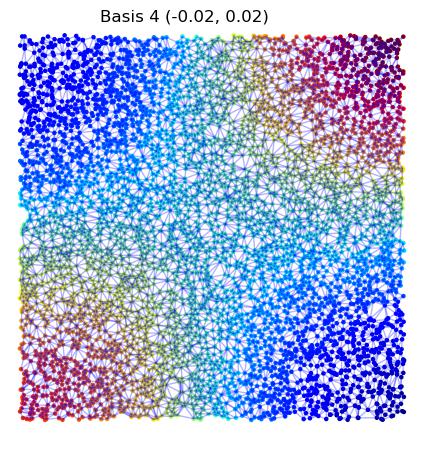}
\includegraphics[width=0.12\linewidth]{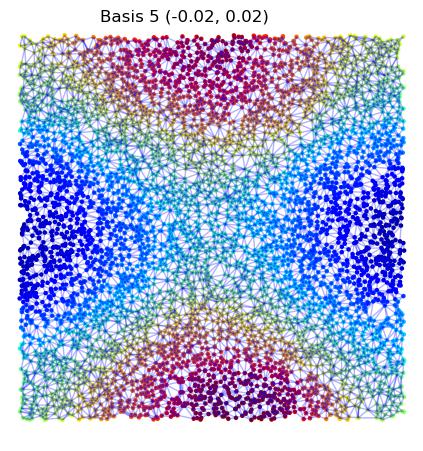}
\includegraphics[width=0.12\linewidth]{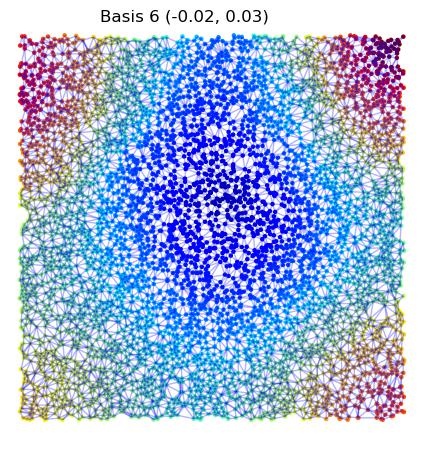}
\includegraphics[width=0.12\linewidth]{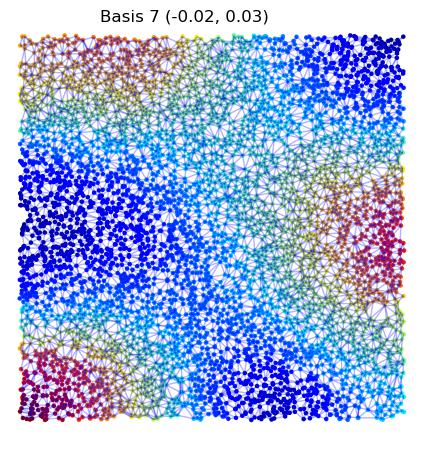}
\includegraphics[width=0.12\linewidth]{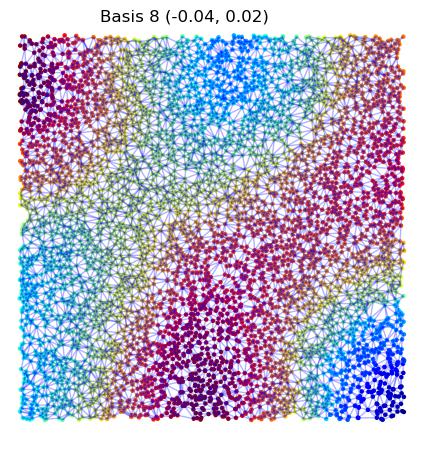}
\caption{\rev{Test-3a: The first eight eigenvectors corresponded to the smallest eigenvalues of $L^{\omega_{ij}} \phi^{\omega_{ij}}_r = \lambda^{\omega_{ij}}_r D^{\omega_{ij}} \phi^{\omega_{ij}}_r$.}}
\end{subfigure}
\caption{\rev{Illustration of the local network spectral properties for Test-1a, Test-2a and  Test-3a.}}
\label{fig:eig-t2d}
\end{figure}

\begin{figure}[h!]
\centering
\begin{subfigure}{1\textwidth}
\centering
\includegraphics[width=0.17\linewidth]{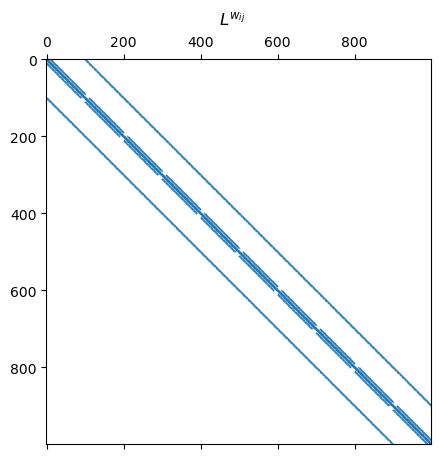}
\includegraphics[width=0.19\linewidth]{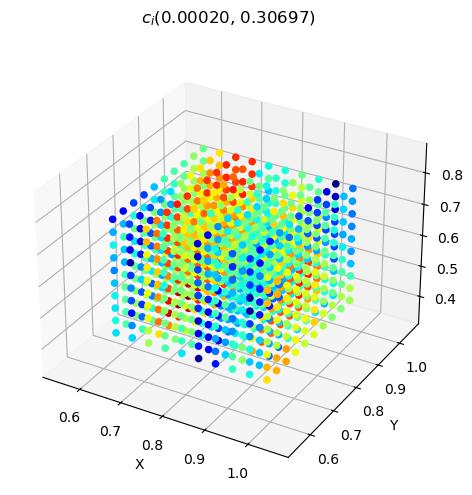}
\includegraphics[width=0.19\linewidth]{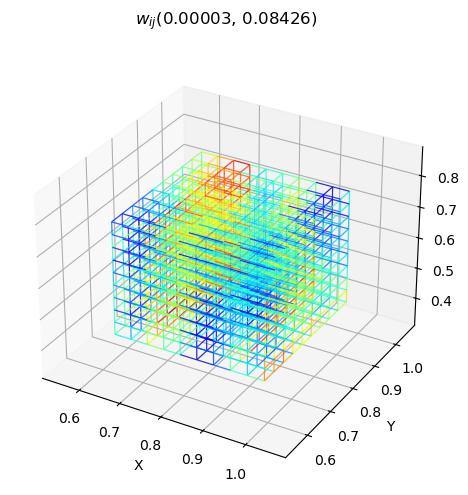}
\includegraphics[width=0.31\linewidth]{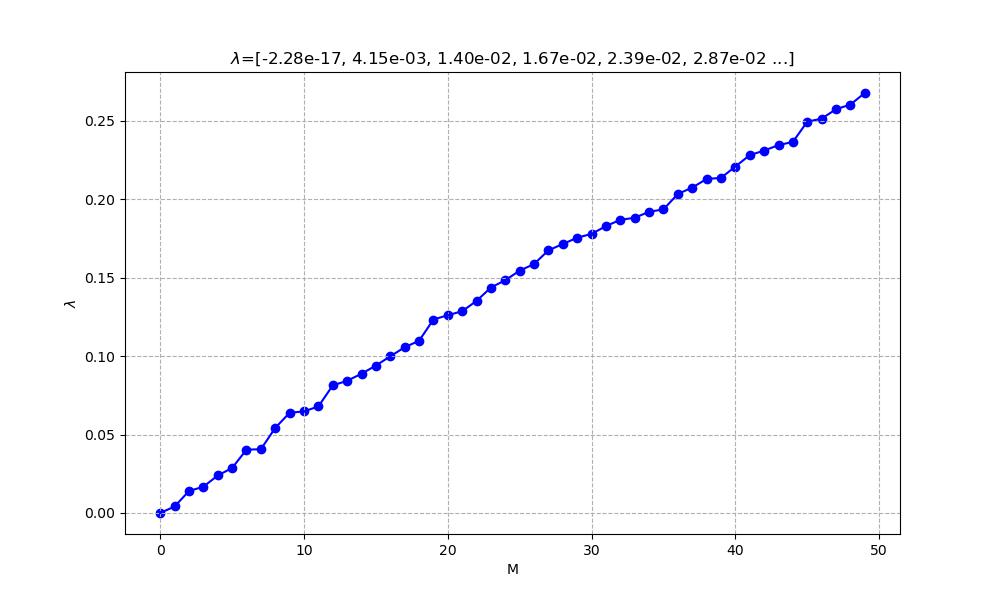}
\caption{\rev{Test-1b: Local matrix $L^{\omega_{ij}} = D^{\omega_{ij}} - W^{\omega_{ij}}$; illustration of  $D^{\omega_{ij}}$ and $W^{\omega_{ij}}$ and plot of the first smallest eigenvalues.}}
\end{subfigure}
\begin{subfigure}{1\textwidth}
\includegraphics[width=0.12\linewidth]{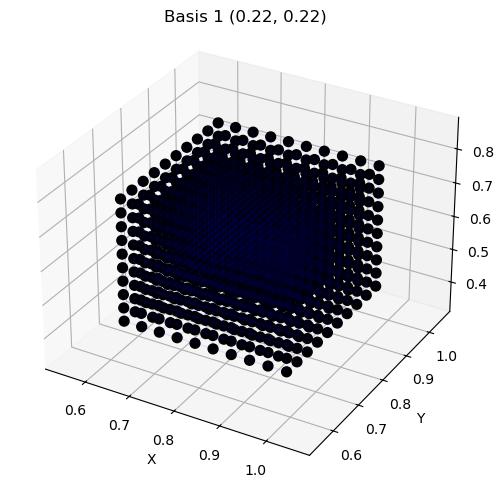}
\includegraphics[width=0.12\linewidth]{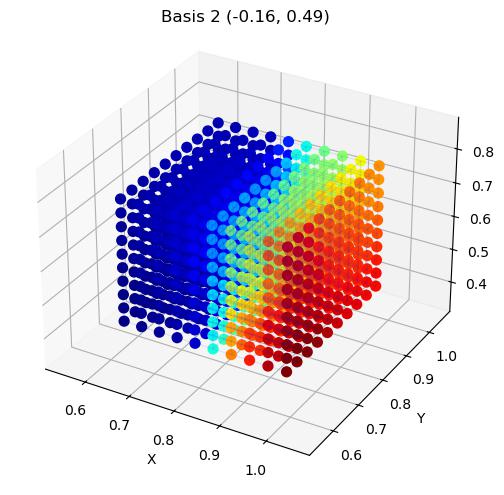}
\includegraphics[width=0.12\linewidth]{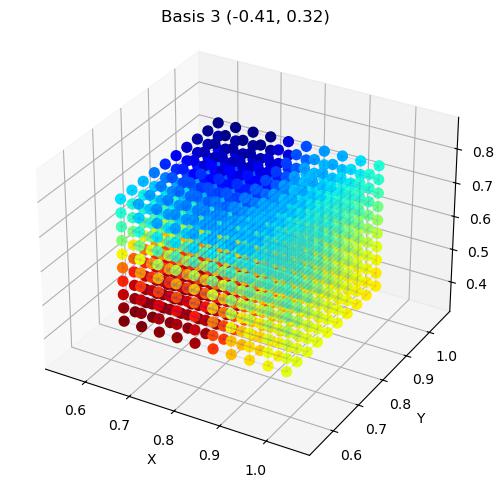}
\includegraphics[width=0.12\linewidth]{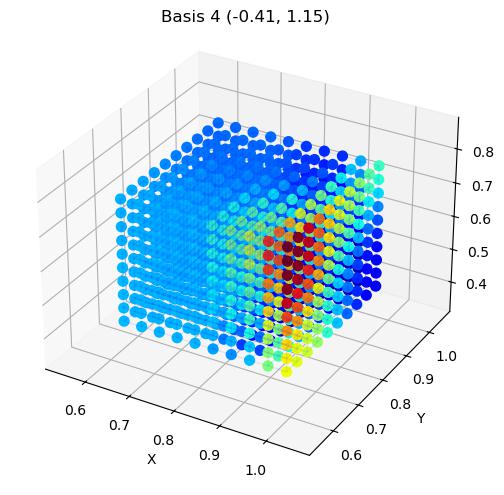}
\includegraphics[width=0.12\linewidth]{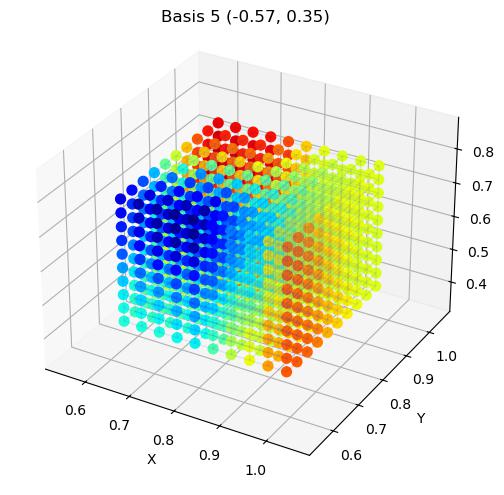}
\includegraphics[width=0.12\linewidth]{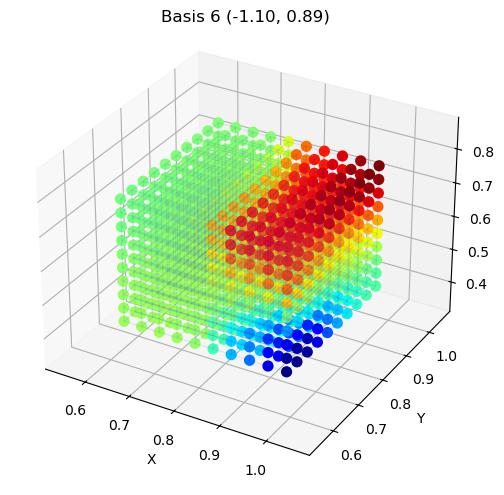}
\includegraphics[width=0.12\linewidth]{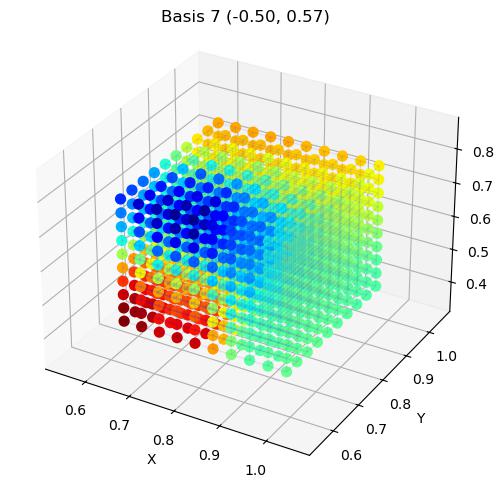}
\includegraphics[width=0.12\linewidth]{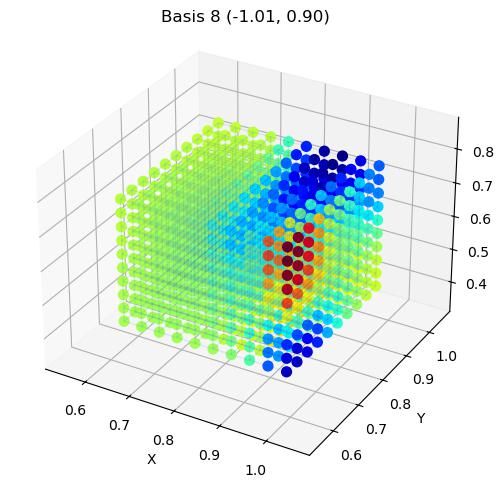}
\caption{\rev{Test-1b: The first eight eigenvectors corresponded to the smallest eigenvalues of $L^{\omega_{ij}} \phi^{\omega_{ij}}_r = \lambda^{\omega_{ij}}_r D^{\omega_{ij}} \phi^{\omega_{ij}}_r$.}}
\end{subfigure}
\begin{subfigure}{1\textwidth}
\centering
\includegraphics[width=0.17\linewidth]{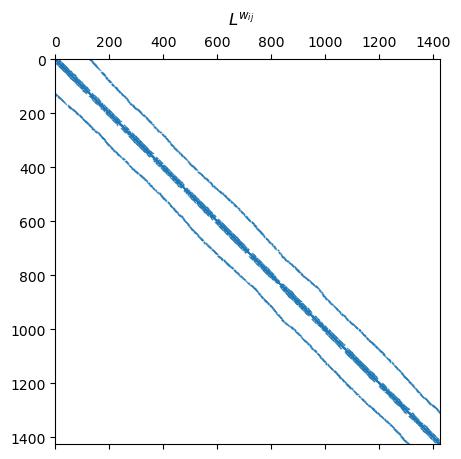}
\includegraphics[width=0.19\linewidth]{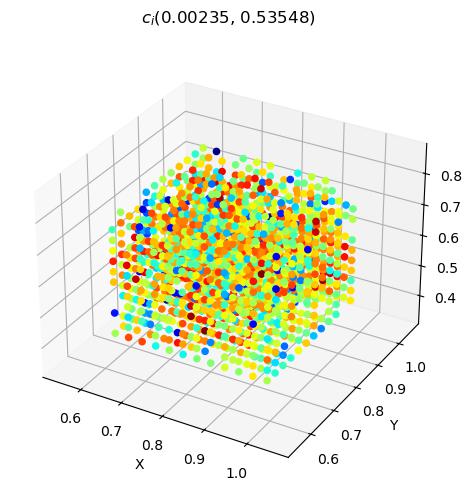}
\includegraphics[width=0.19\linewidth]{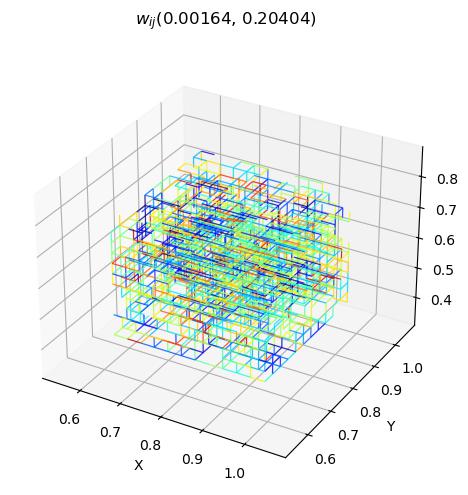}
\includegraphics[width=0.31\linewidth]{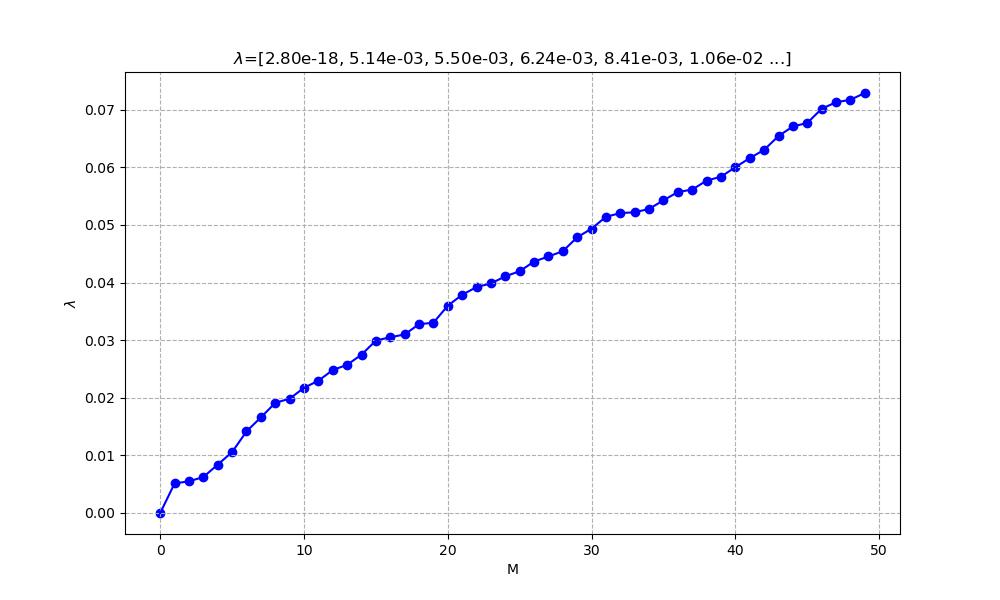}
\caption{\rev{Test-2b: Local matrix $L^{\omega_{ij}} = D^{\omega_{ij}} - W^{\omega_{ij}}$; illustration of  $D^{\omega_{ij}}$ and $W^{\omega_{ij}}$ and plot of the first smallest eigenvalues.}}
\end{subfigure}
\begin{subfigure}{1\textwidth}
\includegraphics[width=0.12\linewidth]{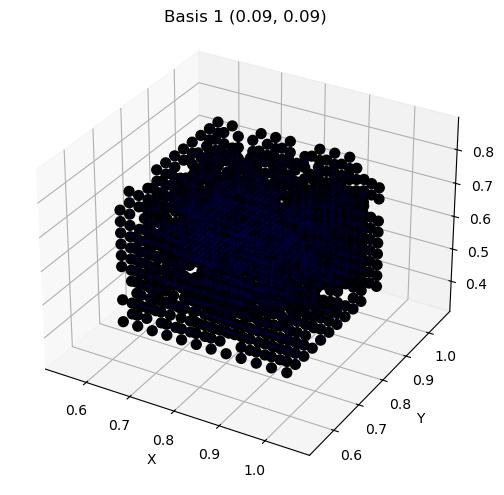}
\includegraphics[width=0.12\linewidth]{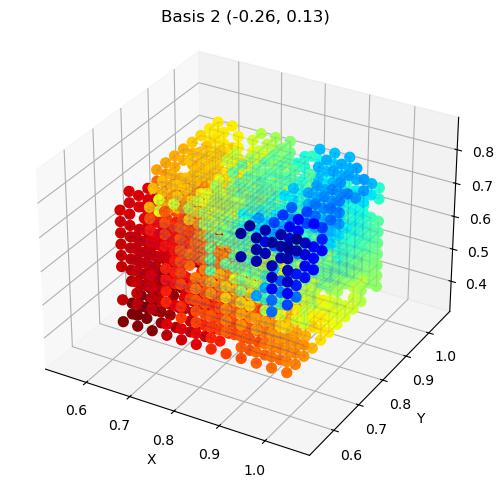}
\includegraphics[width=0.12\linewidth]{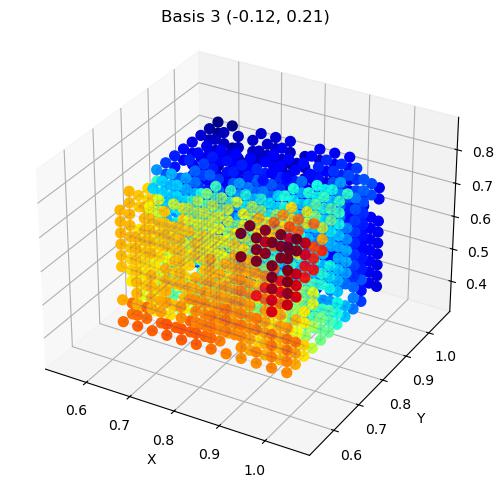}
\includegraphics[width=0.12\linewidth]{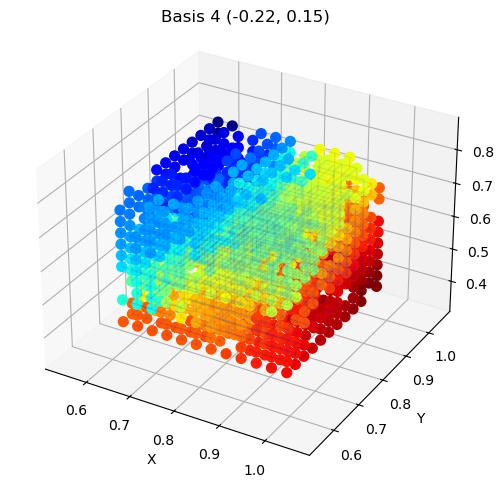}
\includegraphics[width=0.12\linewidth]{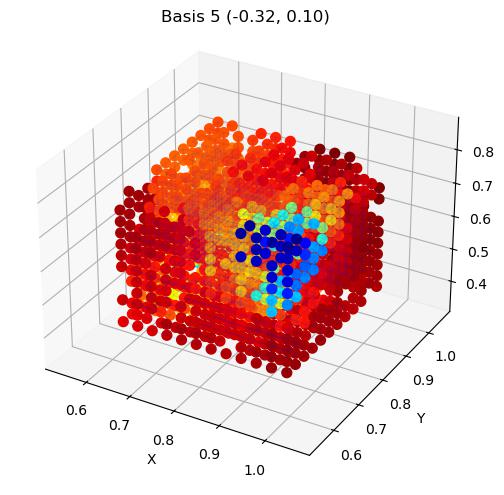}
\includegraphics[width=0.12\linewidth]{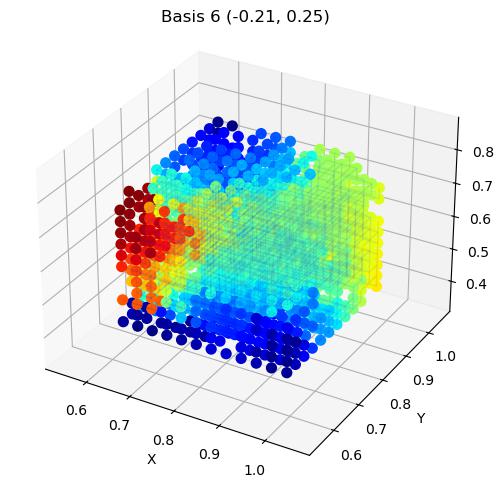}
\includegraphics[width=0.12\linewidth]{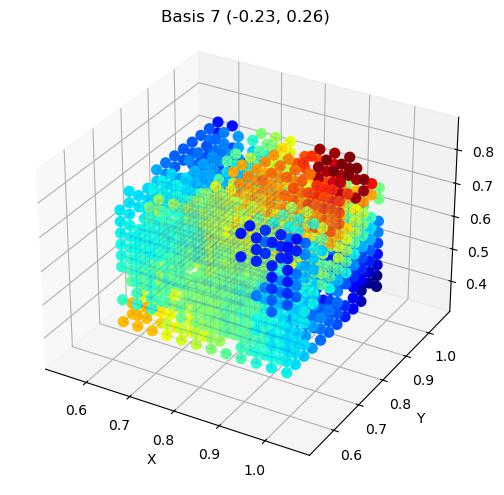}
\includegraphics[width=0.12\linewidth]{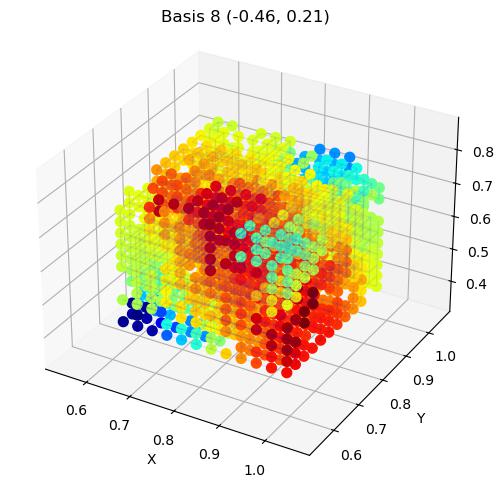}
\caption{\rev{Test-2b: The first eight eigenvectors corresponded to the smallest eigenvalues of $L^{\omega_{ij}} \phi^{\omega_{ij}}_r = \lambda^{\omega_{ij}}_r D^{\omega_{ij}} \phi^{\omega_{ij}}_r$.}}
\end{subfigure}
\begin{subfigure}{1\textwidth}
\centering
\includegraphics[width=0.17\linewidth]{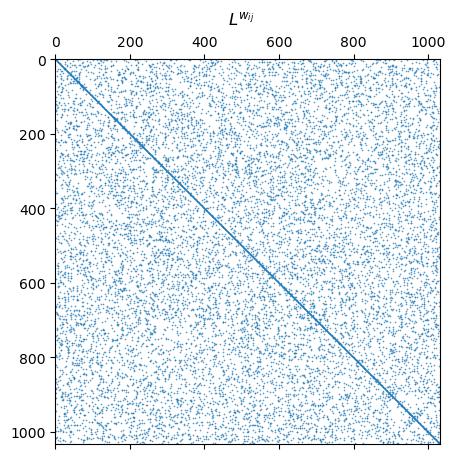}
\includegraphics[width=0.19\linewidth]{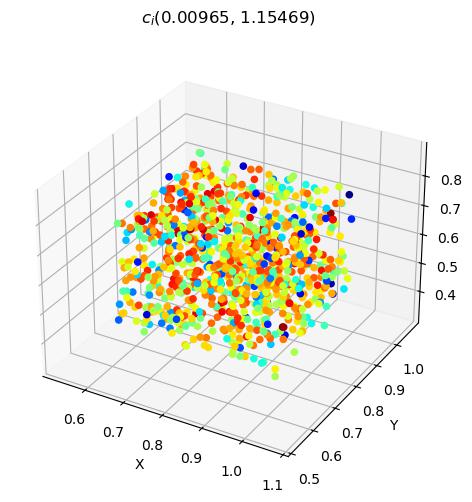}
\includegraphics[width=0.19\linewidth]{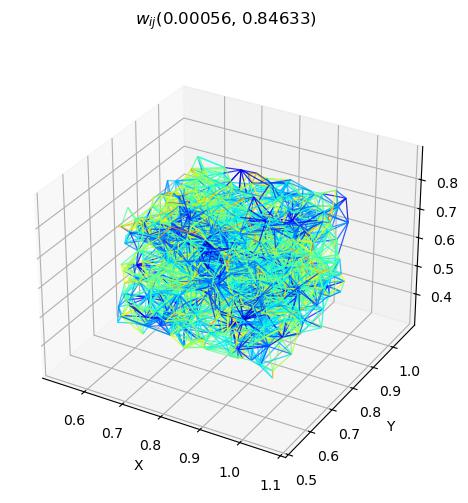}
\includegraphics[width=0.31\linewidth]{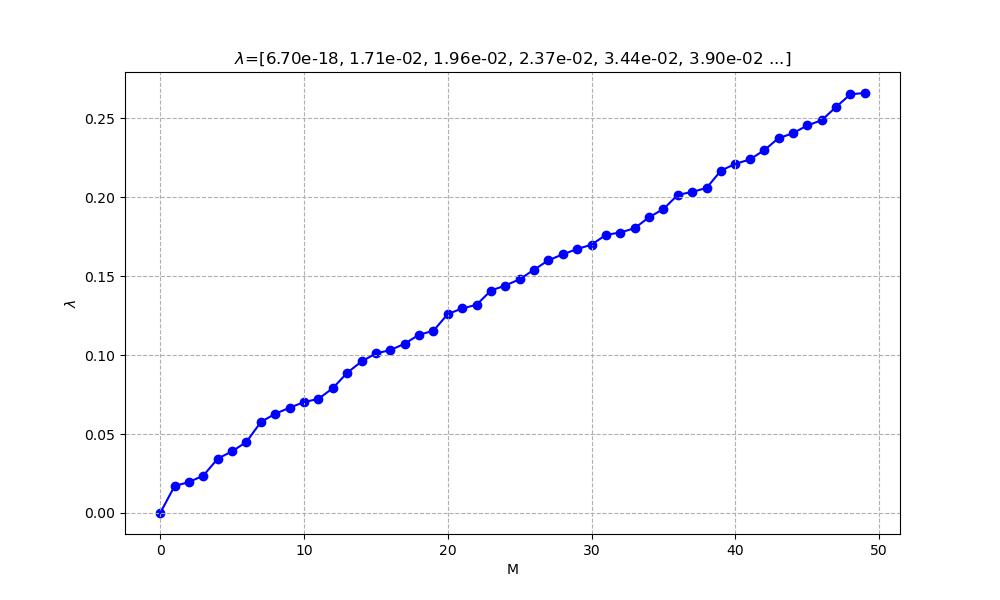}
\caption{\rev{Test-3b: Local matrix $L^{\omega_{ij}} = D^{\omega_{ij}} - W^{\omega_{ij}}$; illustration of  $D^{\omega_{ij}}$ and $W^{\omega_{ij}}$ and plot of the first smallest eigenvalues.}}
\end{subfigure}
\begin{subfigure}{1\textwidth}
\includegraphics[width=0.12\linewidth]{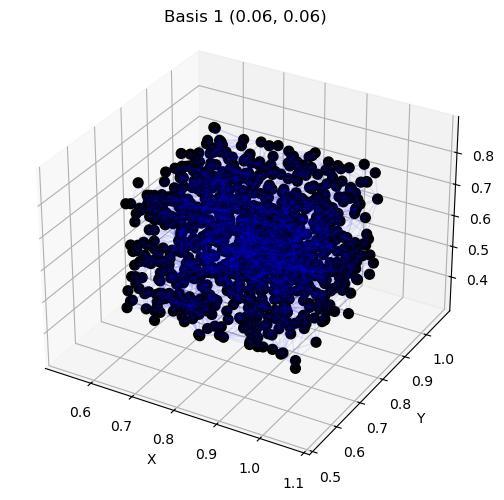}
\includegraphics[width=0.12\linewidth]{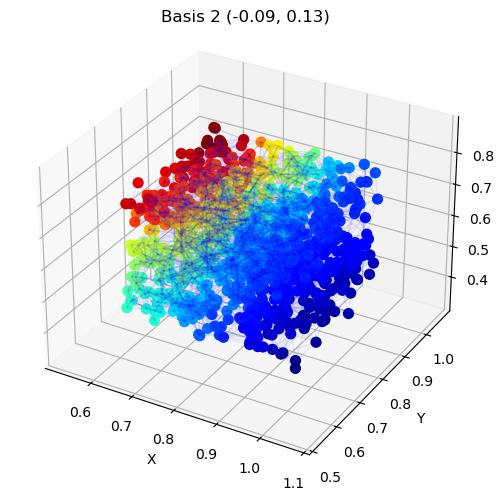}
\includegraphics[width=0.12\linewidth]{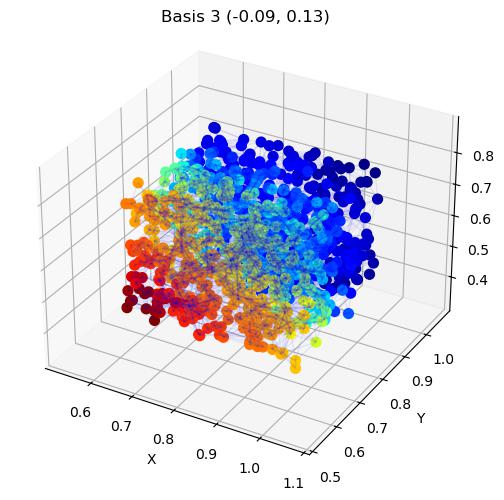}
\includegraphics[width=0.12\linewidth]{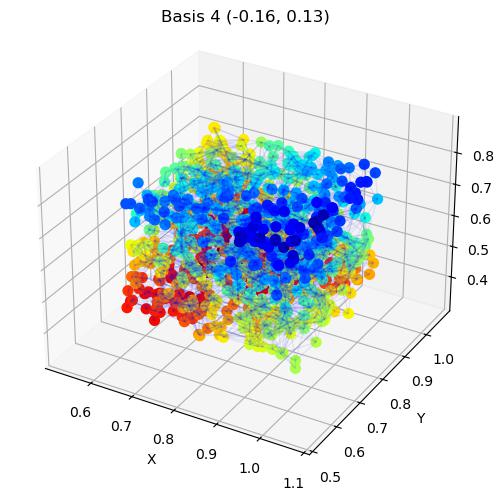}
\includegraphics[width=0.12\linewidth]{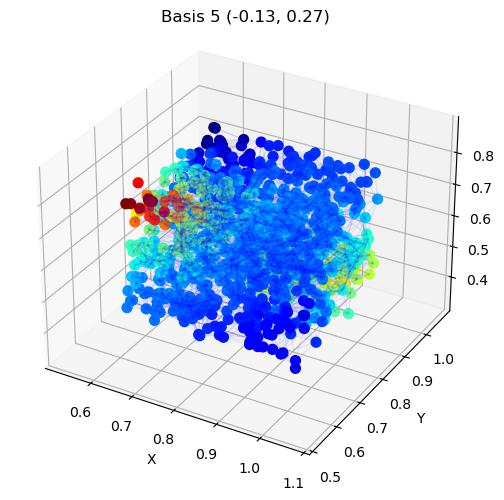}
\includegraphics[width=0.12\linewidth]{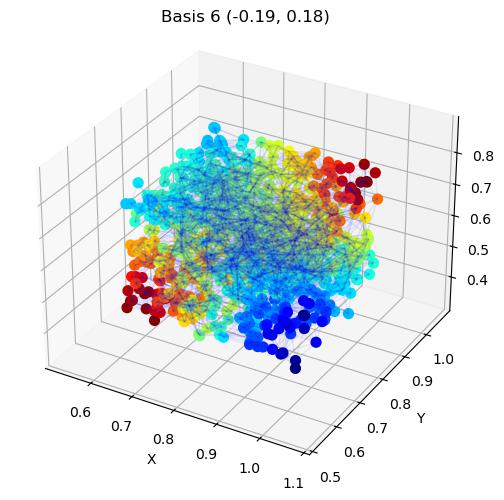}
\includegraphics[width=0.12\linewidth]{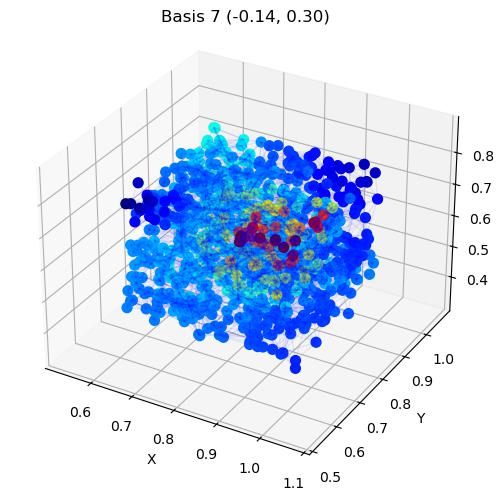}
\includegraphics[width=0.12\linewidth]{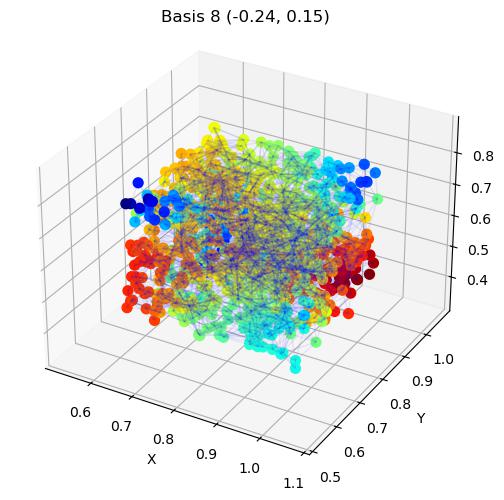}
\caption{\rev{Test-3b: The first eight eigenvectors corresponded to the smallest eigenvalues of $L^{\omega_{ij}} \phi^{\omega_{ij}}_r = \lambda^{\omega_{ij}}_r D^{\omega_{ij}} \phi^{\omega_{ij}}_r$.}}
\end{subfigure}
\caption{\rev{Illustration of the local network spectral properties for Test-1b, Test-2b and  Test-3b.}}
\label{fig:eig-t3d}
\end{figure}

\begin{figure}[h!]
\centering
\begin{subfigure}{1\textwidth}
\centering
\includegraphics[width=0.17\linewidth]{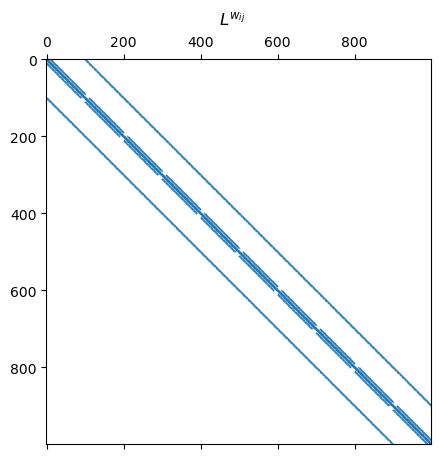}
\includegraphics[width=0.19\linewidth]{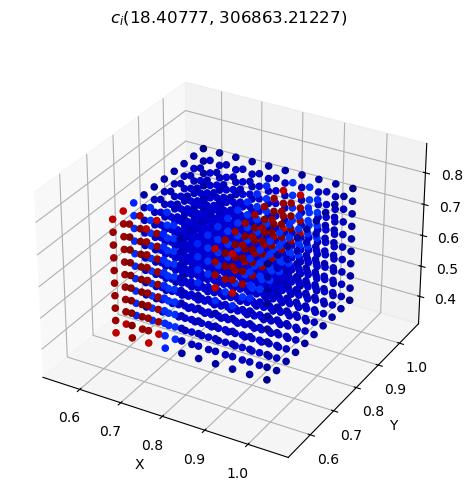}
\includegraphics[width=0.19\linewidth]{pic/t12-d3kk}
\includegraphics[width=0.31\linewidth]{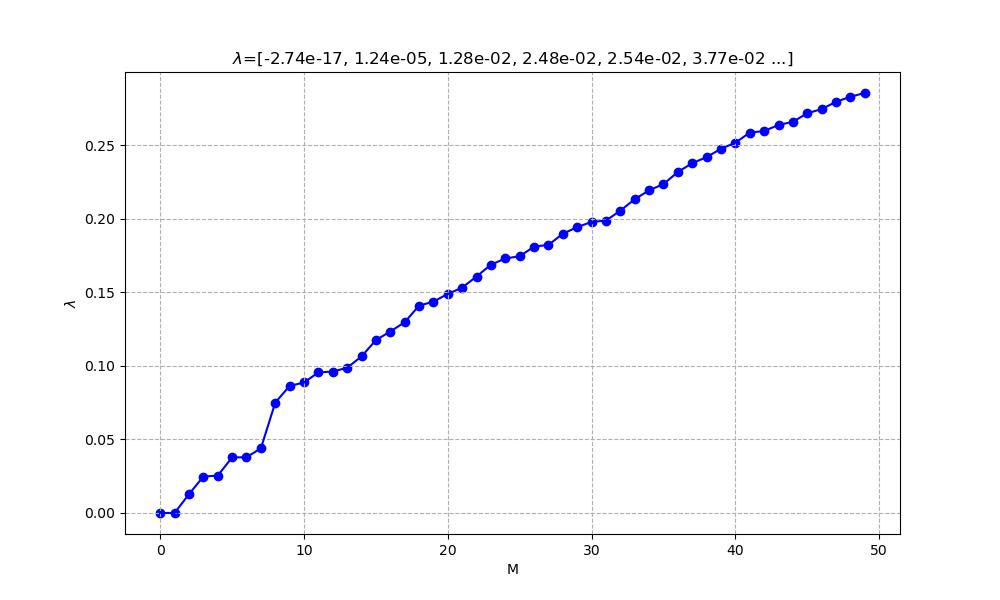}
\caption{\rev{Test-1c: Local matrix $L^{\omega_{ij}} = D^{\omega_{ij}} - W^{\omega_{ij}}$; illustration of  $D^{\omega_{ij}}$ and $W^{\omega_{ij}}$ and plot of the first smallest eigenvalues.}}
\end{subfigure}
\begin{subfigure}{1\textwidth}
\includegraphics[width=0.12\linewidth]{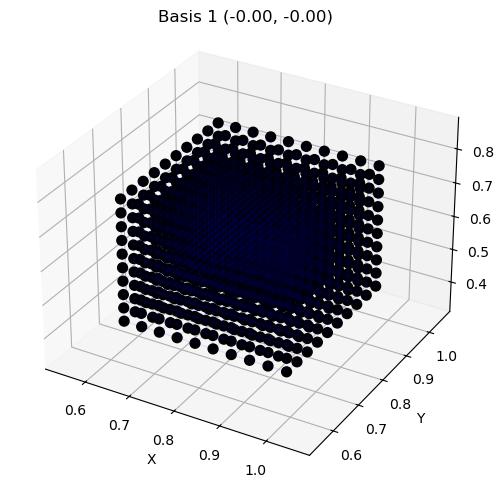}
\includegraphics[width=0.12\linewidth]{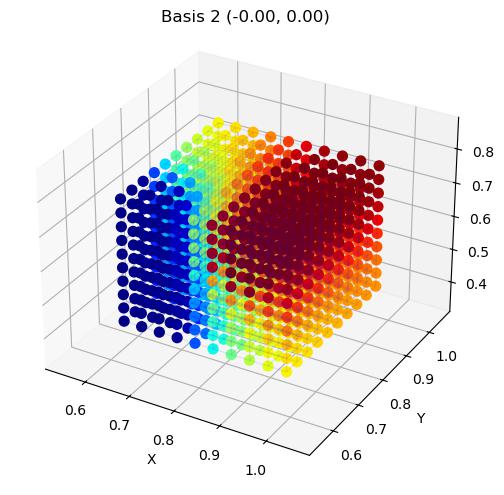}
\includegraphics[width=0.12\linewidth]{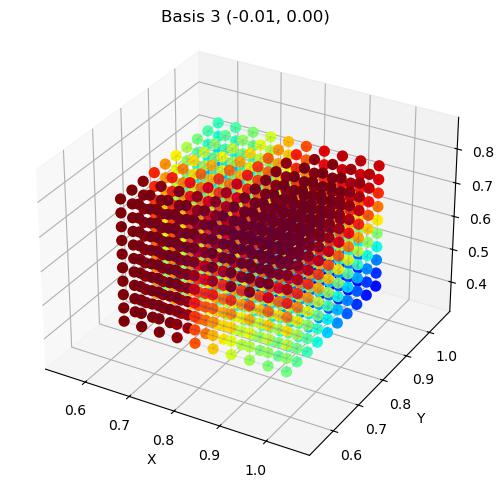}
\includegraphics[width=0.12\linewidth]{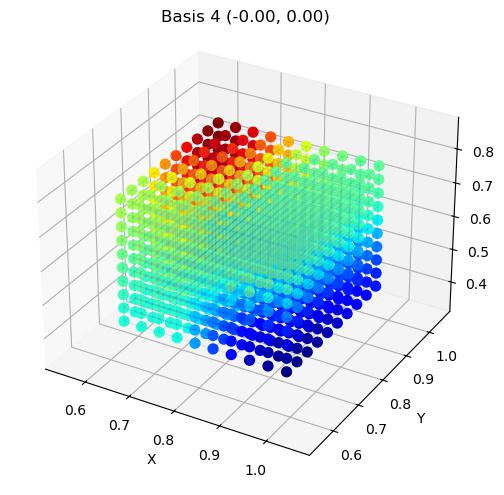}
\includegraphics[width=0.12\linewidth]{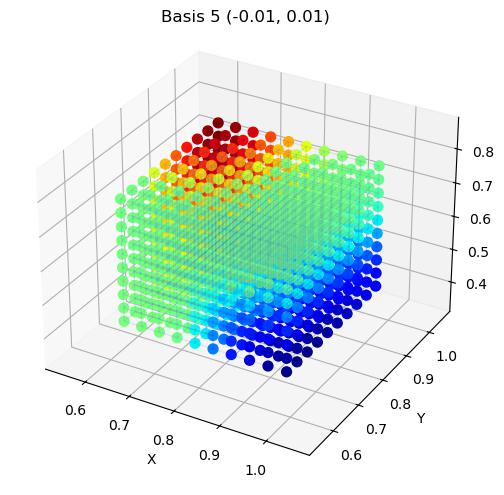}
\includegraphics[width=0.12\linewidth]{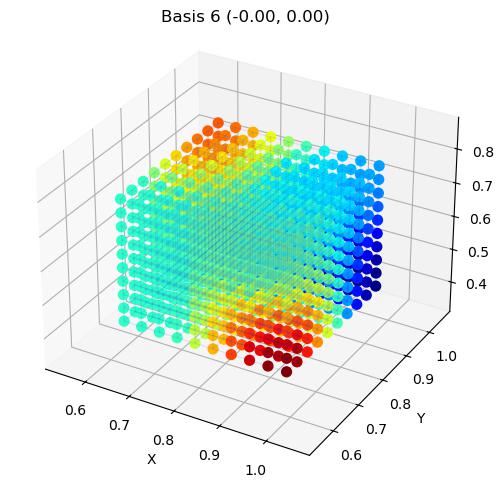}
\includegraphics[width=0.12\linewidth]{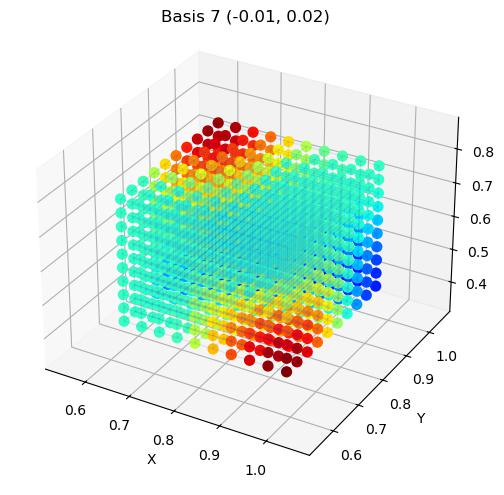}
\includegraphics[width=0.12\linewidth]{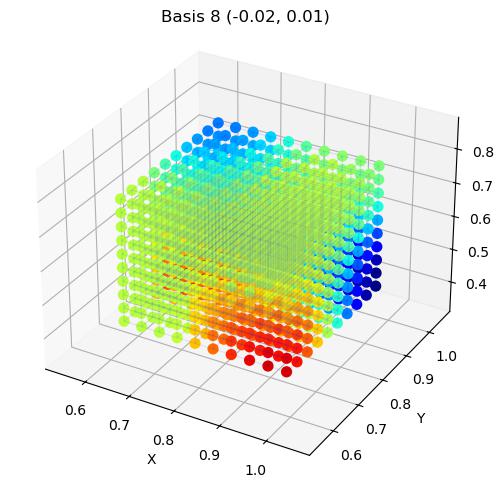}
\caption{\rev{Test-1c: The first eight eigenvectors corresponded to the smallest eigenvalues of $L^{\omega_{ij}} \phi^{\omega_{ij}}_r = \lambda^{\omega_{ij}}_r D^{\omega_{ij}} \phi^{\omega_{ij}}_r$.}}
\end{subfigure}
\begin{subfigure}{1\textwidth}
\centering
\includegraphics[width=0.17\linewidth]{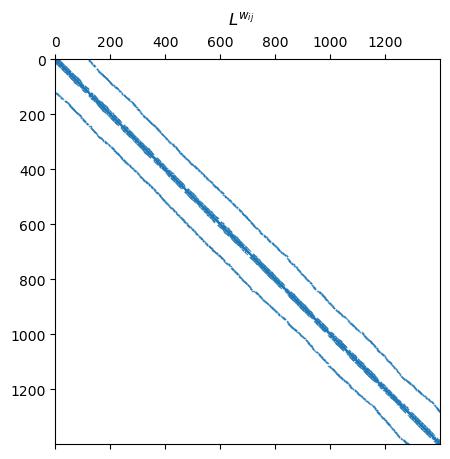}
\includegraphics[width=0.19\linewidth]{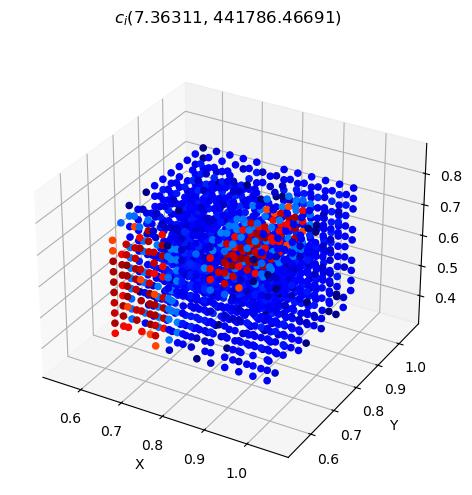}
\includegraphics[width=0.19\linewidth]{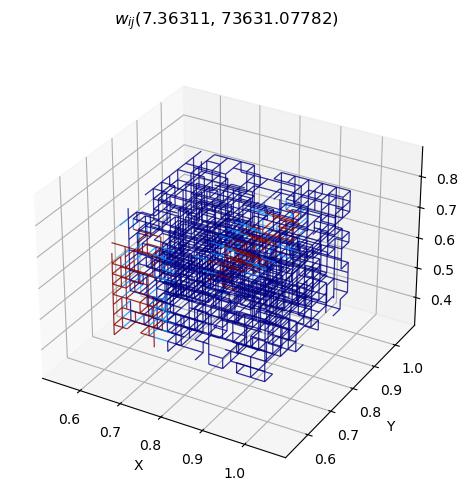}
\includegraphics[width=0.31\linewidth]{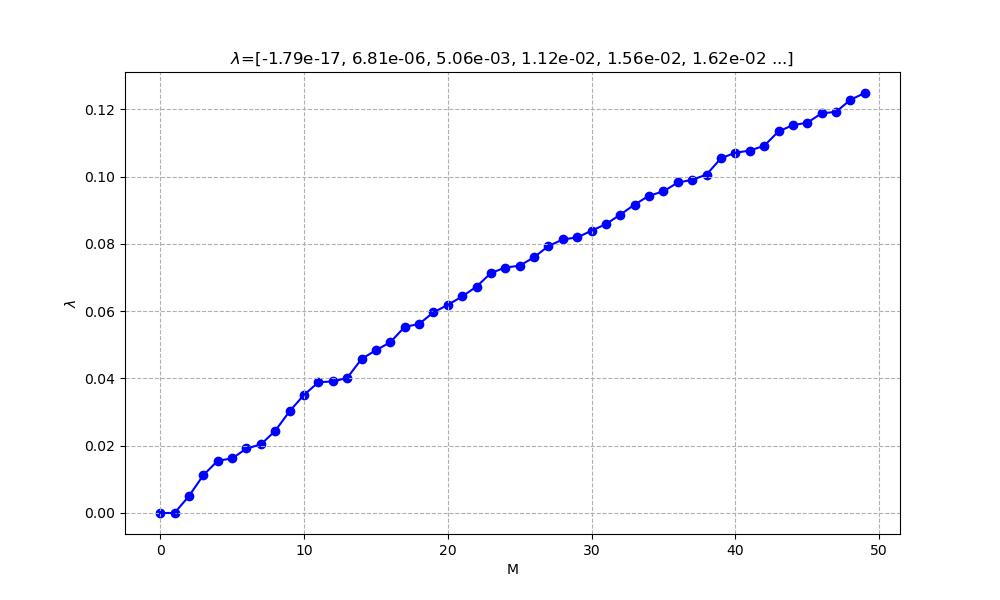}
\caption{\rev{Test-2c: Local matrix $L^{\omega_{ij}} = D^{\omega_{ij}} - W^{\omega_{ij}}$; illustration of  $D^{\omega_{ij}}$ and $W^{\omega_{ij}}$ and plot of the first smallest eigenvalues.}}
\end{subfigure}
\begin{subfigure}{1\textwidth}
\includegraphics[width=0.12\linewidth]{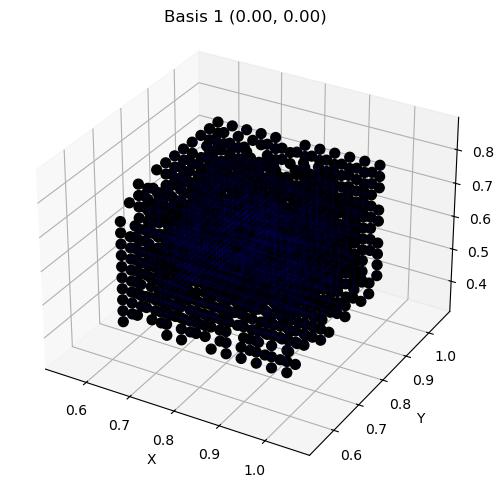}
\includegraphics[width=0.12\linewidth]{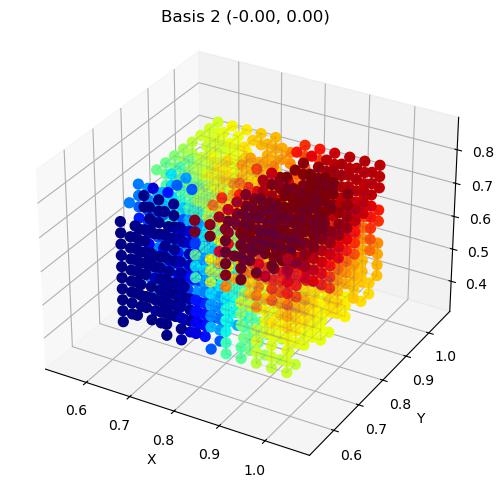}
\includegraphics[width=0.12\linewidth]{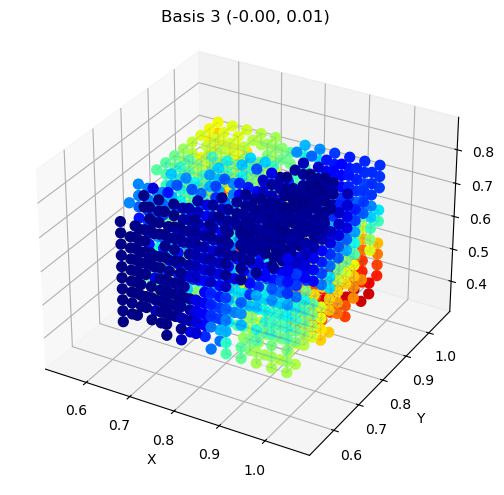}
\includegraphics[width=0.12\linewidth]{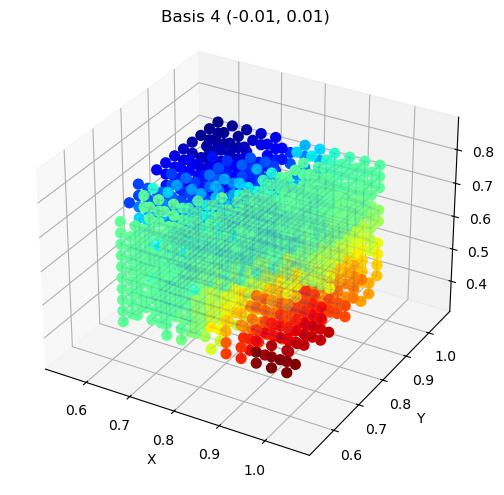}
\includegraphics[width=0.12\linewidth]{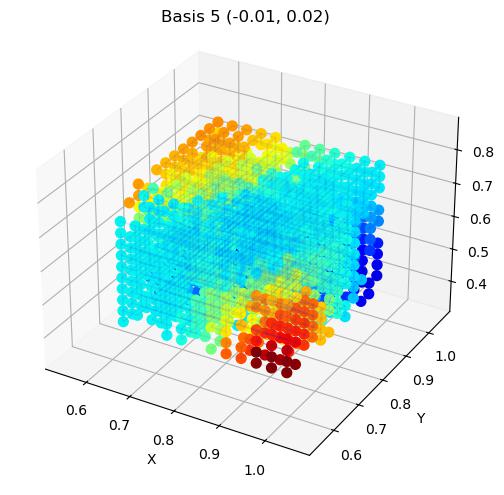}
\includegraphics[width=0.12\linewidth]{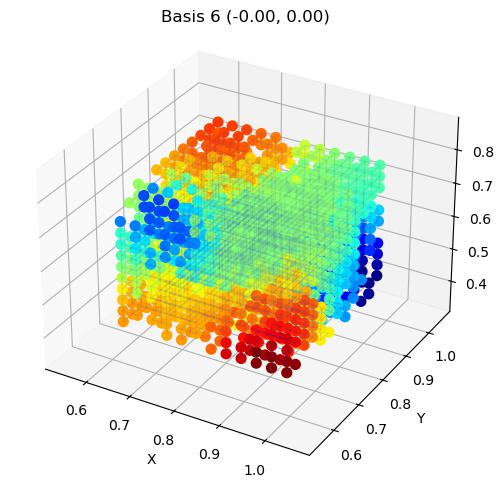}
\includegraphics[width=0.12\linewidth]{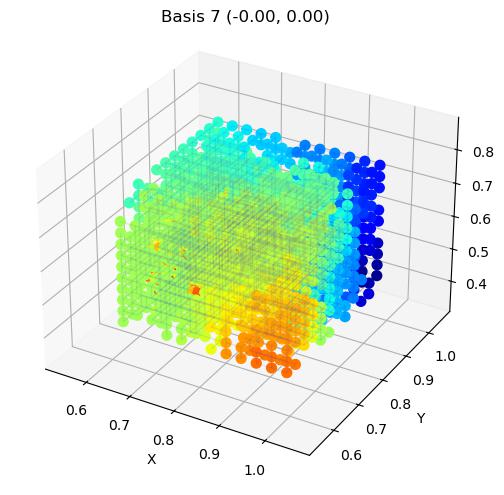}
\includegraphics[width=0.12\linewidth]{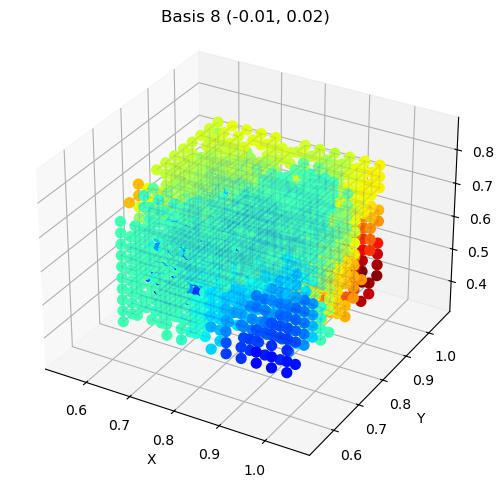}
\caption{\rev{Test-2c: The first eight eigenvectors corresponded to the smallest eigenvalues of  $L^{\omega_{ij}} \phi^{\omega_{ij}}_r = \lambda^{\omega_{ij}}_r D^{\omega_{ij}} \phi^{\omega_{ij}}_r$.}}
\end{subfigure}
\begin{subfigure}{1\textwidth}
\centering
\includegraphics[width=0.17\linewidth]{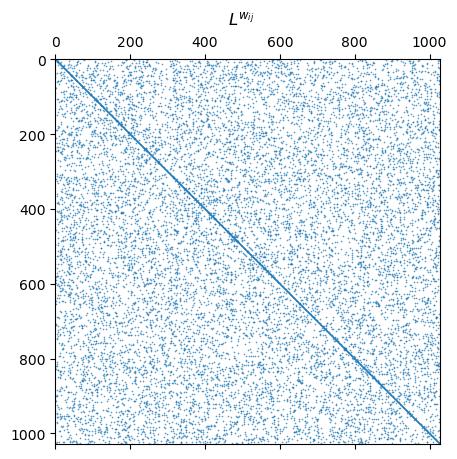}
\includegraphics[width=0.19\linewidth]{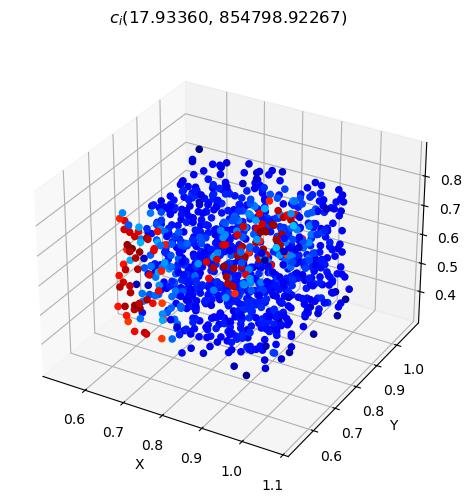}
\includegraphics[width=0.19\linewidth]{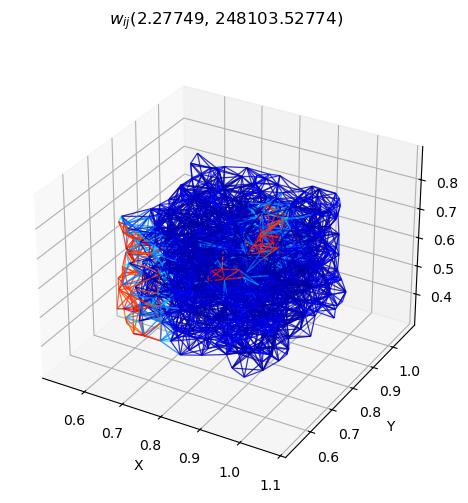}
\includegraphics[width=0.31\linewidth]{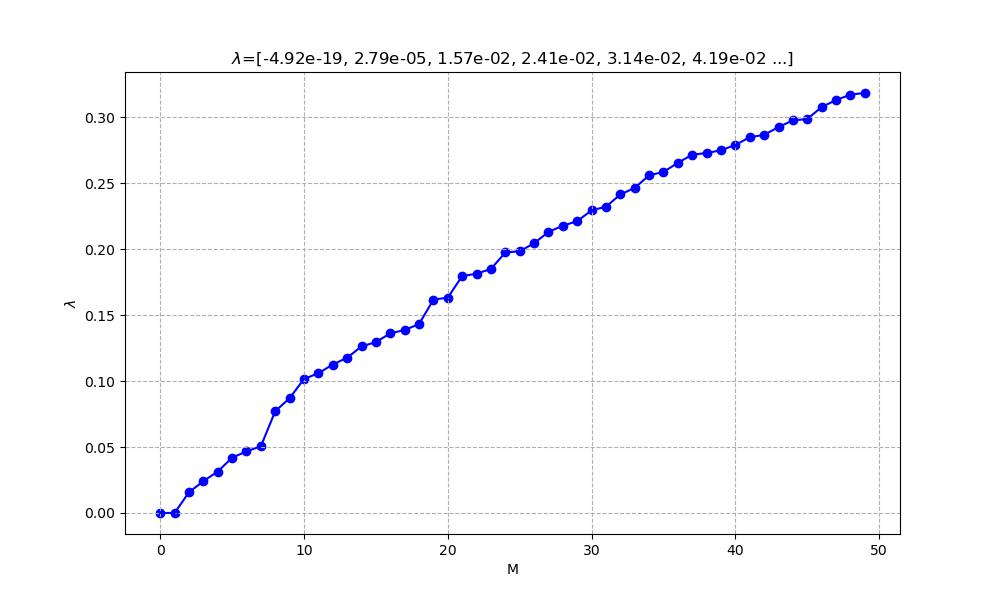}
\caption{\rev{Test-3c: Local matrix $L^{\omega_{ij}} = D^{\omega_{ij}} - W^{\omega_{ij}}$; illustration of  $D^{\omega_{ij}}$ and $W^{\omega_{ij}}$ and plot of the first smallest eigenvalues.}}
\end{subfigure}
\begin{subfigure}{1\textwidth}
\includegraphics[width=0.12\linewidth]{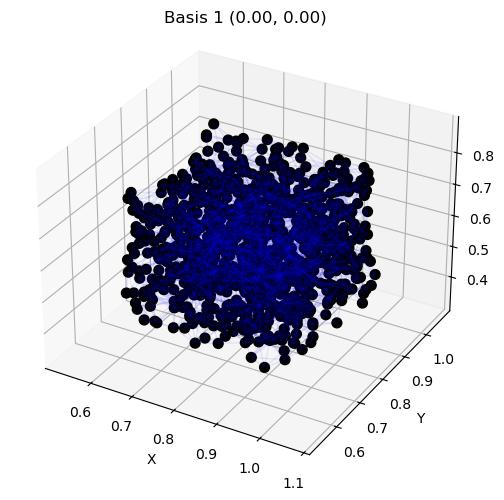}
\includegraphics[width=0.12\linewidth]{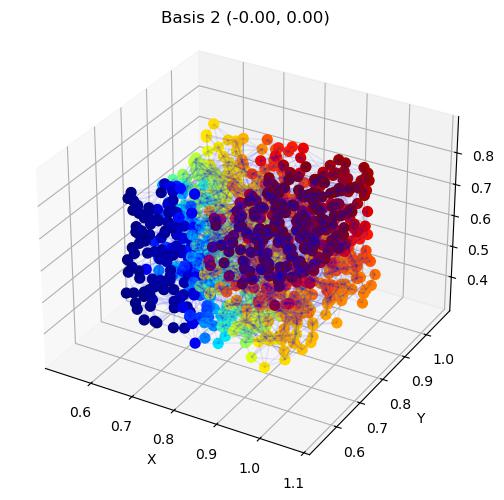}
\includegraphics[width=0.12\linewidth]{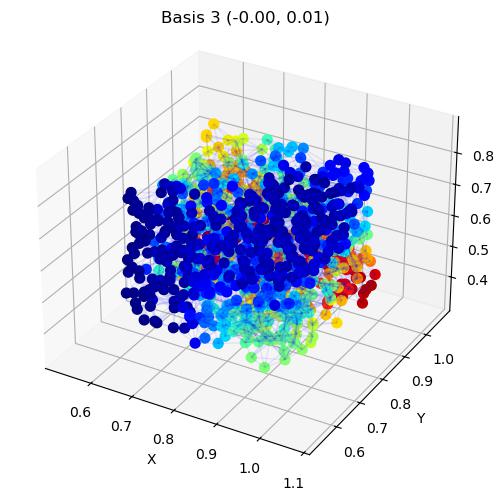}
\includegraphics[width=0.12\linewidth]{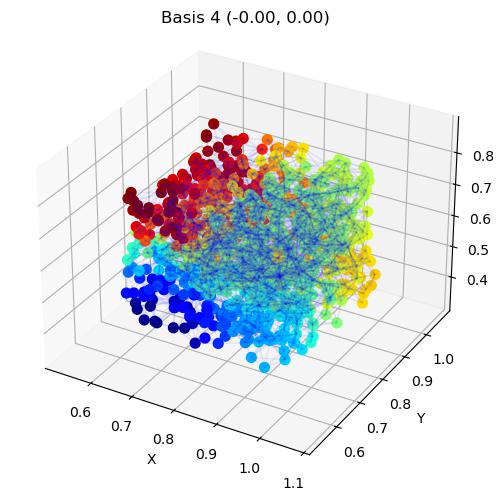}
\includegraphics[width=0.12\linewidth]{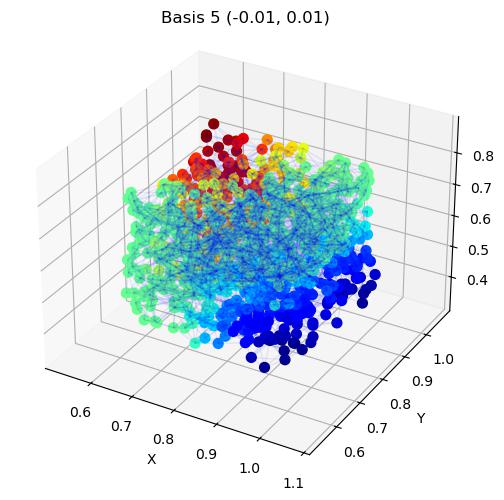}
\includegraphics[width=0.12\linewidth]{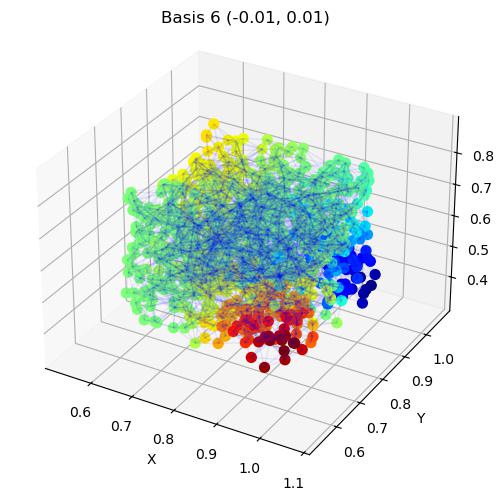}
\includegraphics[width=0.12\linewidth]{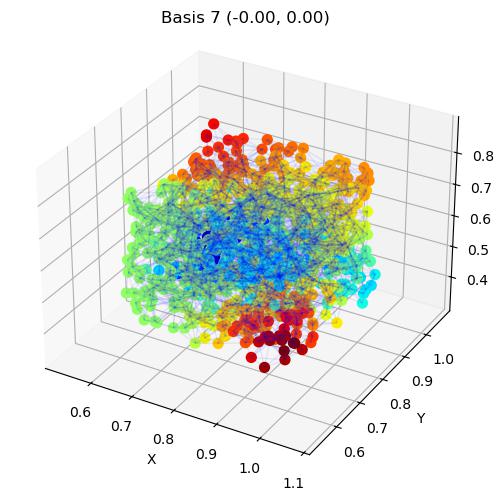}
\includegraphics[width=0.12\linewidth]{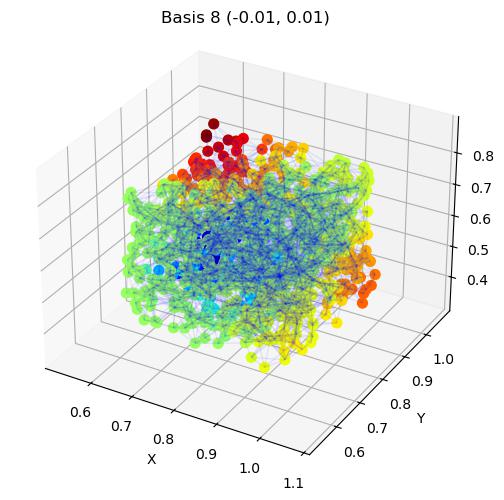}
\caption{\rev{Test-3c: The first eight eigenvectors corresponded to the smallest eigenvalues of $L^{\omega_{ij}} \phi^{\omega_{ij}}_r = \lambda^{\omega_{ij}}_r D^{\omega_{ij}} \phi^{\omega_{ij}}_r$.}}
\end{subfigure}
\caption{\rev{Illustration of the local network spectral properties for Test-1c, Test-2c and  Test-3c.}}
\label{fig:eig-t3d2}
\end{figure}

\subsection{\rev{Illustration of multiscale basis functions for networks}}

\rev{Next, we consider the spectral properties of the local problems. The multiscale basis functions are based on the eigenvectors $\phi^{\omega_{ij}}_r$ of generalized eigenvalue problem \eqref{sp1} and constructed as a result of multiplication of eigenvectors to the linear partition of unity functions $\chi_i$ in local subnetwork $G^{\omega_{ij}}$, $\psi^{\omega_{ij}}_r = \chi_i \phi^{\omega_{ij}}_r$. }

\rev{
In Figures \ref{fig:eig-t2d}, \ref{fig:eig-t3d} and \ref{fig:eig-t3d2}, we represent local problem properties and eigenvectors corresponding to the first eight smallest eigenvalues for each network. 
To illustrate the approximation properties of eigenvectors, we start with a plot of the local matrix. From the first plots of Figures \ref{fig:eig-t2d}, \ref{fig:eig-t3d} and \ref{fig:eig-t3d2}, we see that the structure of the local matrix  $L^{\omega_{ij}} = D^{\omega_{ij}} - W^{\omega_{ij}}$ highly depends on the network structure and especially connectivity. For the unstructured networks (Network-3a and Network-3b), we observe a denser matrix than for other structured networks (Network-1a and Network-1b, Network-2a and Network-2b). We note that the sparsity of the matrix is associated with the graph connectivity patterns and doesn't depend on assigned values of coefficients on networks. The network coefficients give heterogeneous weight (coefficients). To show the weights of connections and values of the corresponding degree matrix $D^{\omega_{ij}}$, we plot them on the local network. The degree matrix $D^{\omega_{ij}}$ is a diagonal matrix and can be represented as the coefficients of the nodes, and the weight matrix $W^{\omega_{ij}}$ can be represented as the coefficients of the throats. From the presented figures, we observe highly heterogeneous values in $D^{\omega_{ij}}$ and  $W^{\omega_{ij}}$ that correspond to local heterogeneity. The plot of the degree and weight matrix in the figures are given with a corresponding range of values. 

Next, we look to the distribution of the eigenvalues of generalized eigenvalue problem  $L^{\omega_{ij}} \phi^{\omega_{ij}}_r = \lambda^{\omega_{ij}}_r D^{\omega_{ij}} \phi^{\omega_{ij}}_r$. We have sorted eigenvalues in ascending order, plotted the first 50 values, and shown the first six values in the title. The first eight eigenvectors are plotted to illustrate their ability to understand a given local heterogeneity and connectivity. For the tests with heterogeneous properties (Test-1a/1b/2a/2b/3a/3b), we have a smooth decreasing behavior of the eigenvalues due to the small-scale oscillations of the coefficients. For the tests with high-contrast properties (Test-1c/2c/3c), we see a jump in the eigenvalues after the second eigenvector, which shows that we should include corresponding eigenvectors in the multiscale space. This behavior was observed in \cite{efendiev2011multiscale, efendiev2013generalized} and can be connected to the multicontinuum theory and adaptive choice of necessary number of basis functions in local domains. 
}

\subsection{\rev{Multiscale solver for networks}}

The presented multiscale method is applied for each network to investigate accuracy for different number numbers of basis functions.  
\rev{Additionally, we present results for a classic upscaling method. In the upscaling method, we construct a continuums scale model on a coarse grid by using flux averaging. The detaled explanation of the upscaling method presented in Appendix \ref{appendx}. }

To compare the accuracy of the presented \rev{upscaling and  multiscale methods}, we use the relative $L_2$ error in percentage.  
The errors are calculated using the following formula on the fine grid: 
\rev{
\begin{equation}
e_1^h = \frac{||u - u_{ms}||}{||u||} \times 100 \%,
\quad
e_2^h = \frac{||u - u_{ms}||_L}{||u||_L} \times 100 \%,
\quad
e_1^H = \frac{||\bar{u} - \bar{u}_{ms}||}{||\bar{u}||} \times 100 \%,
\end{equation}
where $||u||^2 = u^T u$ and $||u||^2_L = u^T L u$ are the L$_2$ and energy norms, $u$ and $u_{ms}$ are the reference solution and multiscale solution at the final time,  $\bar{u}$ and $\bar{u}_{ms}$ are the coarse grid reference solution and multiscale solution at the final time with $\bar{u} = \int_{K_i} u dx/|K_i|$ ($K_i$ is the coarse grid cell). }
We use the corresponding fine-grid solution for each test problem as a reference solution.

\begin{figure}[h!]
\centering
\begin{subfigure}{1\textwidth}
\centering
\includegraphics[width=0.24\linewidth]{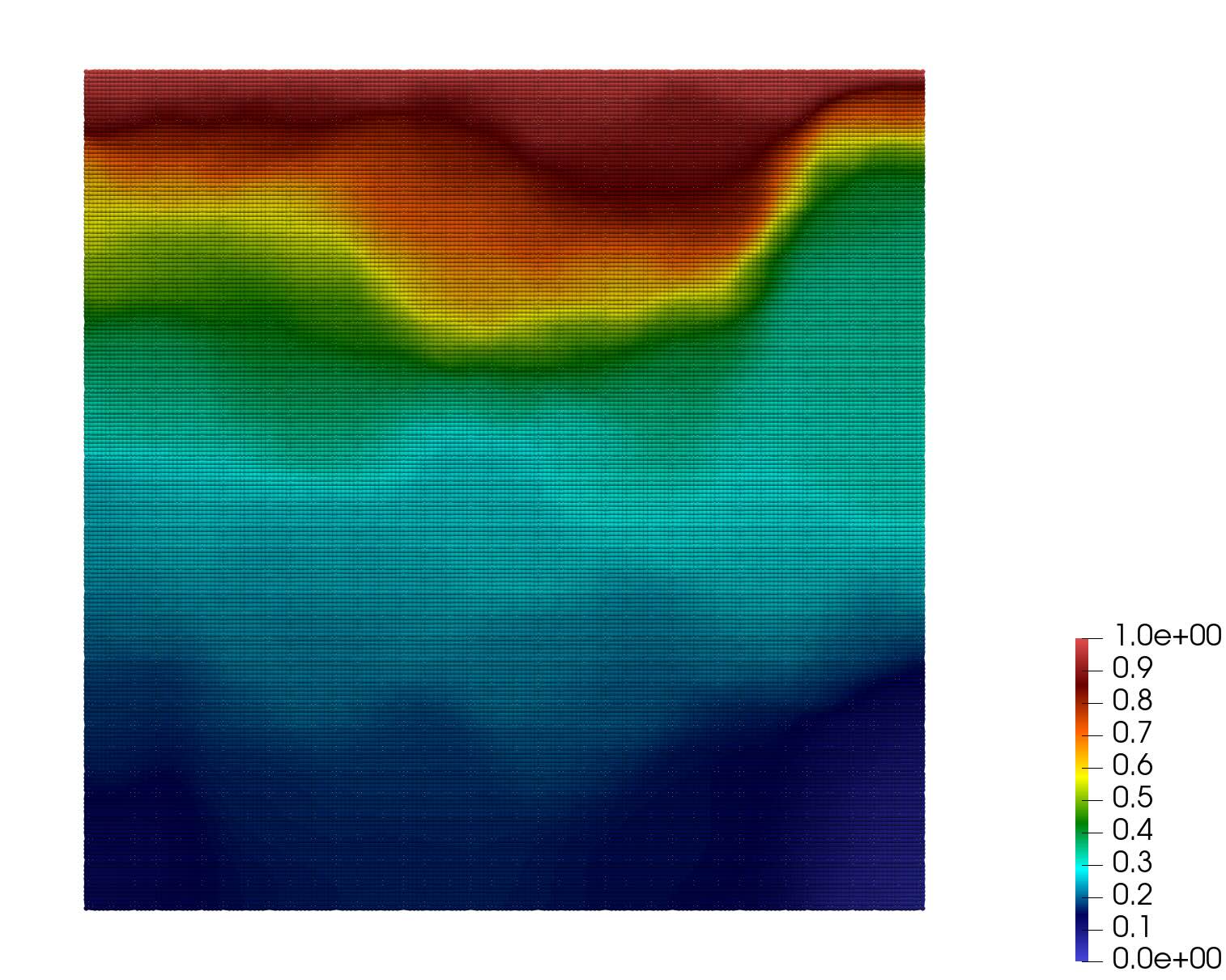}
\includegraphics[width=0.24\linewidth]{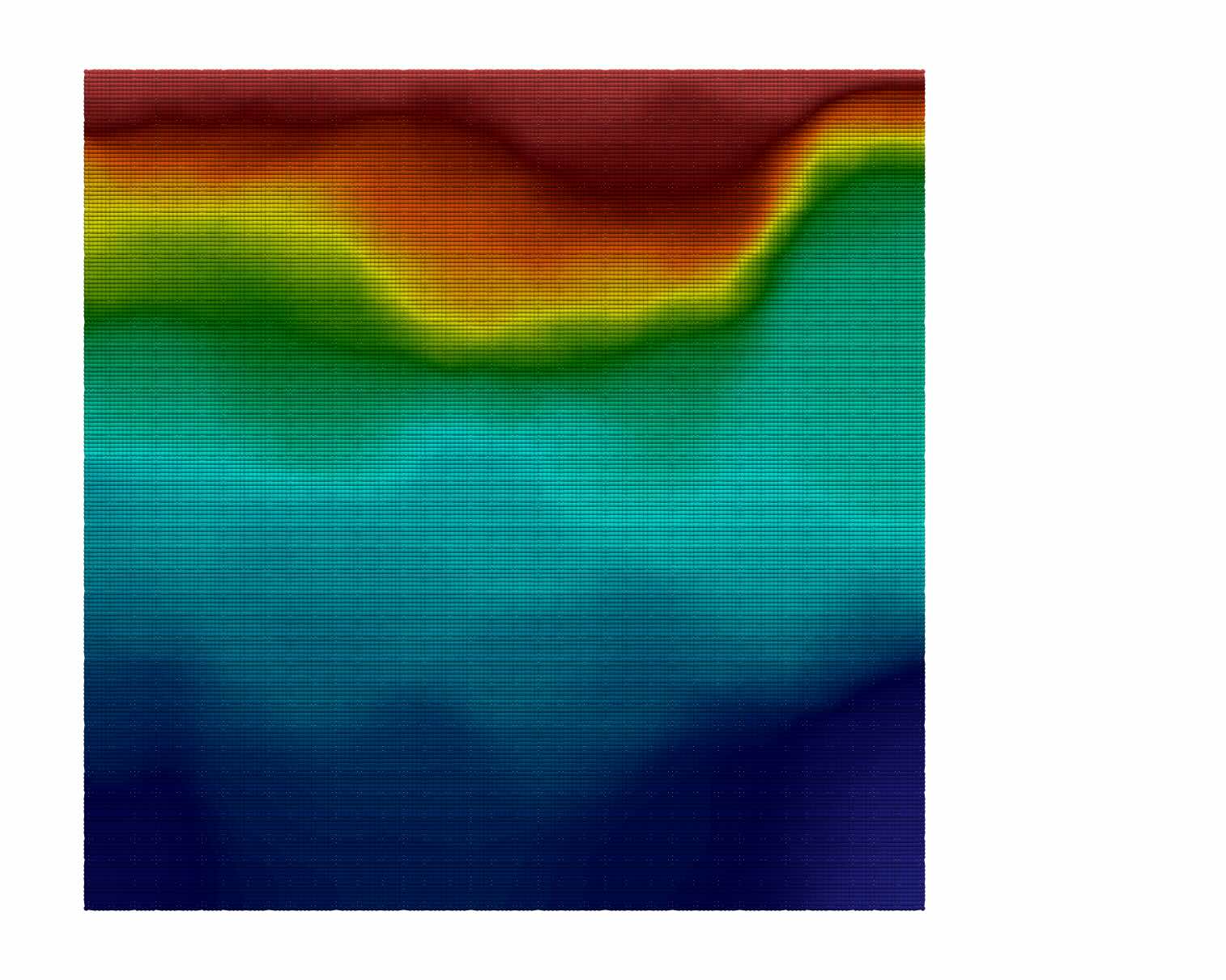}
\includegraphics[width=0.24\linewidth]{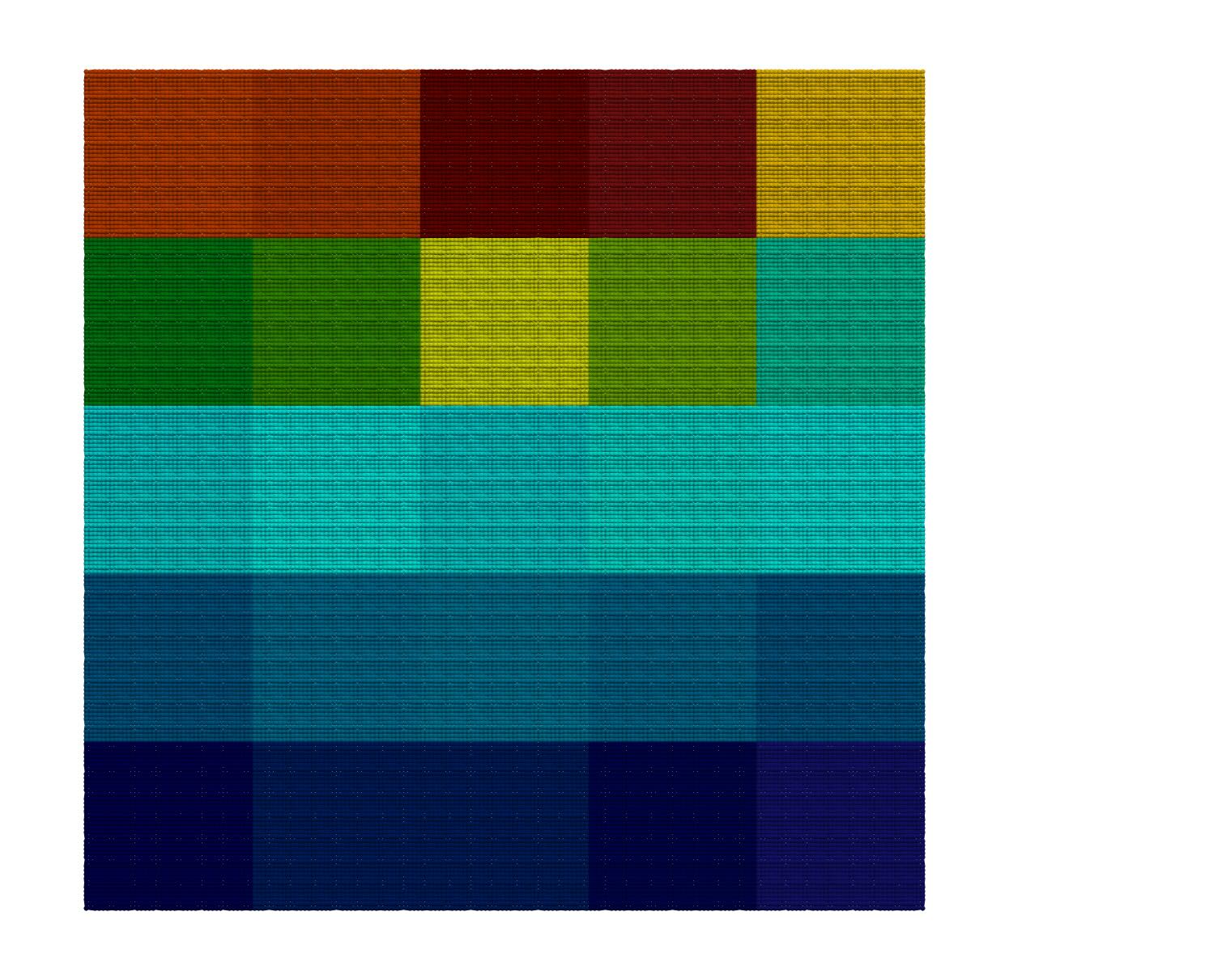}
\includegraphics[width=0.24\linewidth]{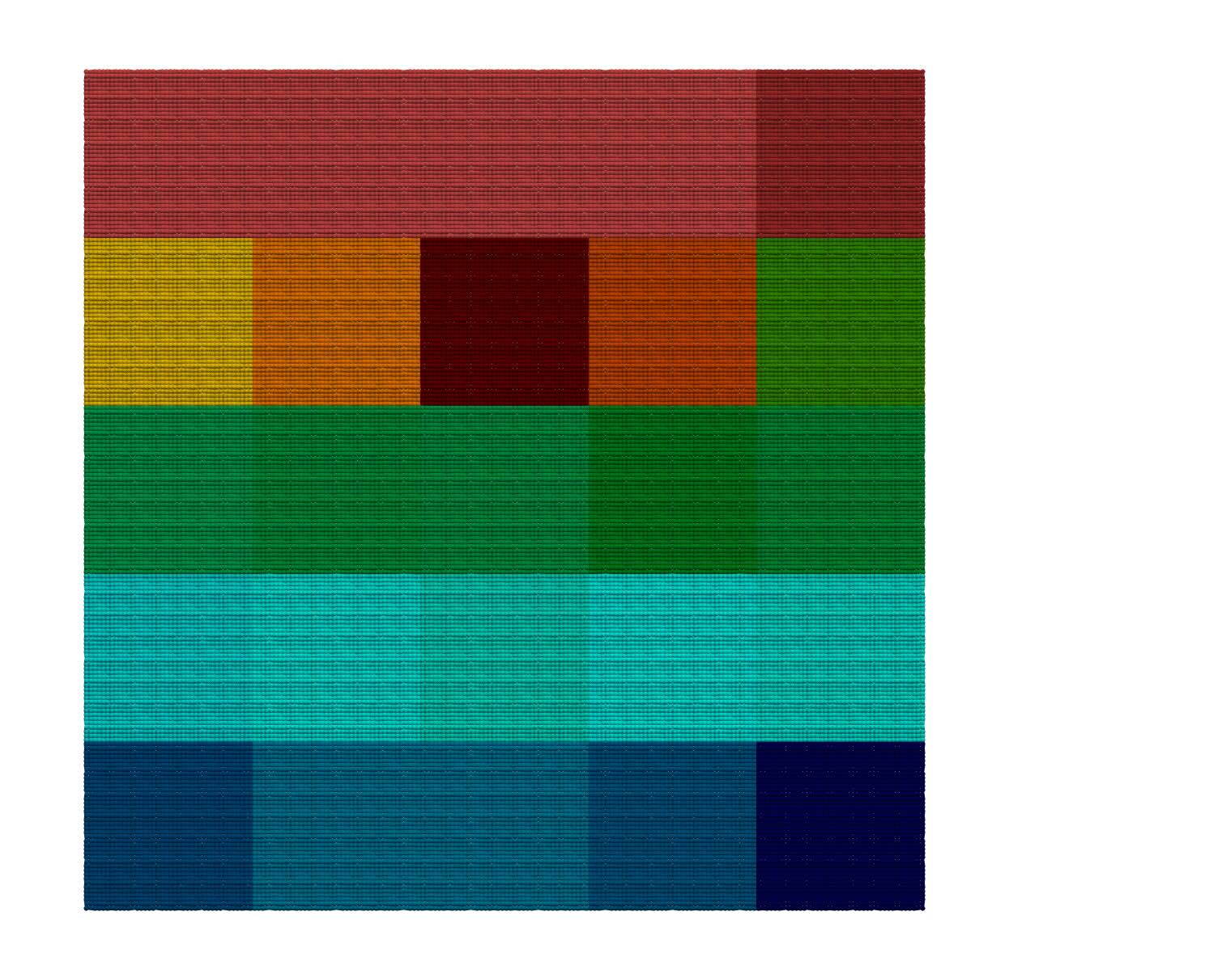}
\caption{\rev{Test-1a: SPE10 properties on Network 1a}}
\end{subfigure}
\begin{subfigure}{1\textwidth}
\centering
\includegraphics[width=0.24\linewidth]{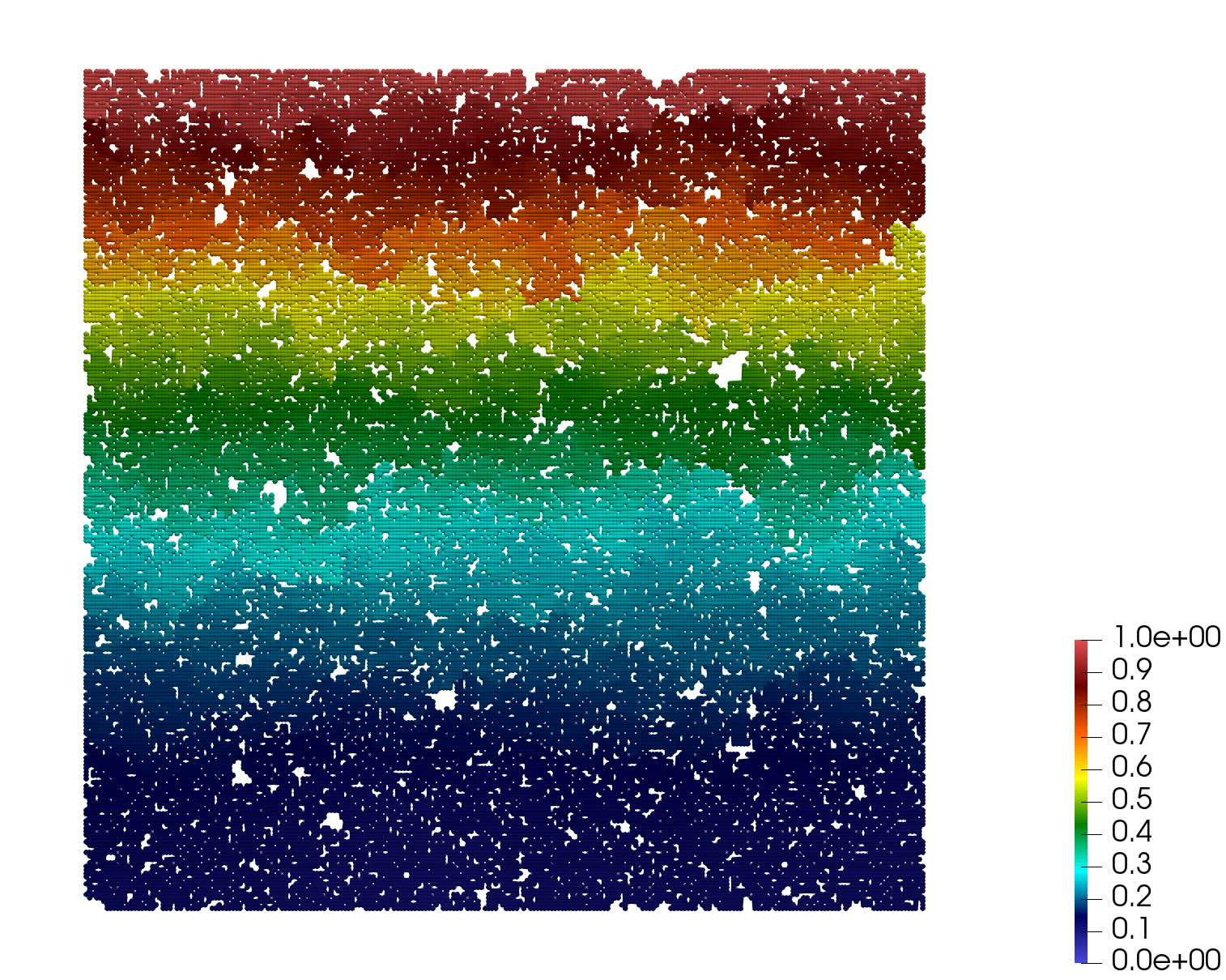}
\includegraphics[width=0.24\linewidth]{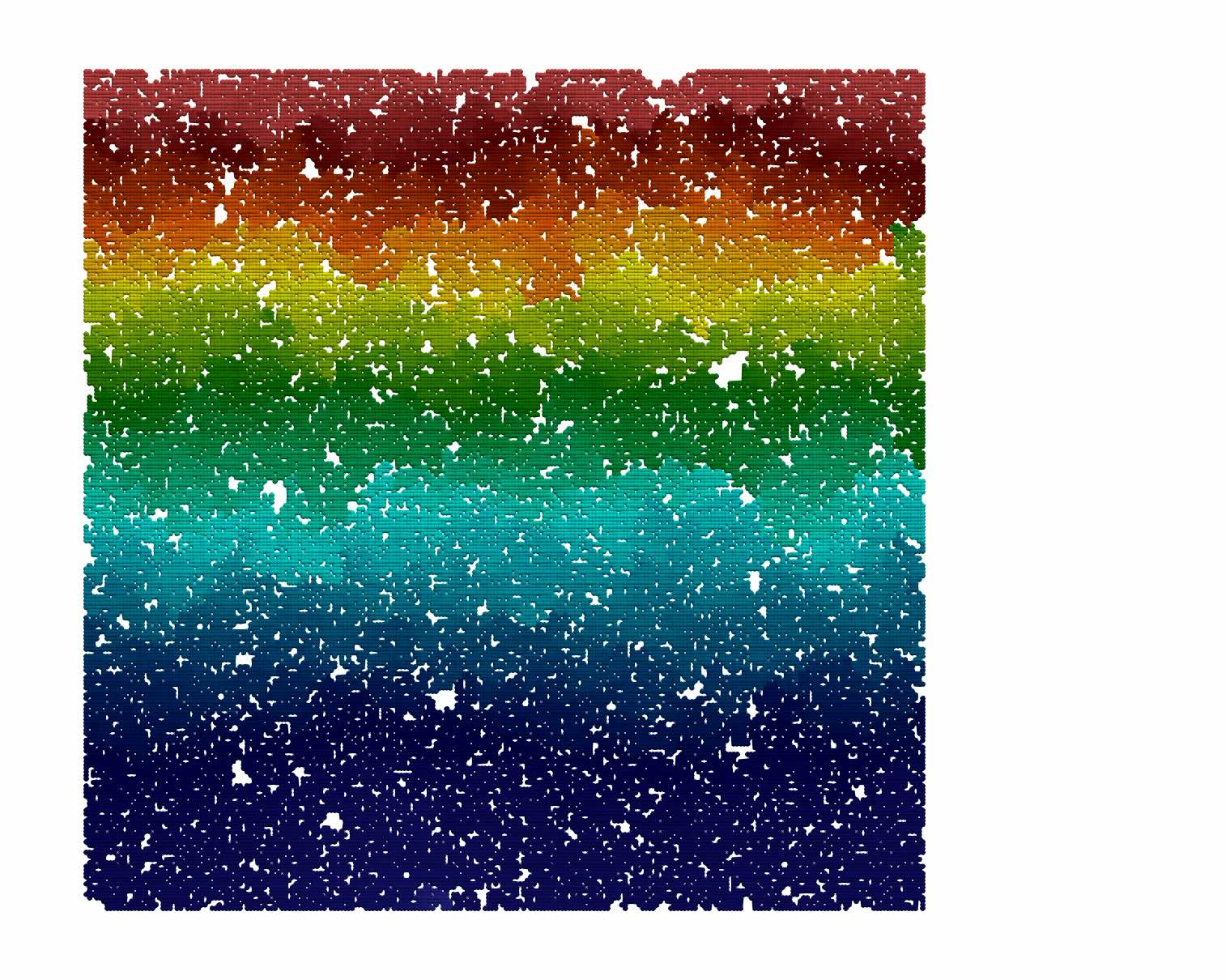}
\includegraphics[width=0.24\linewidth]{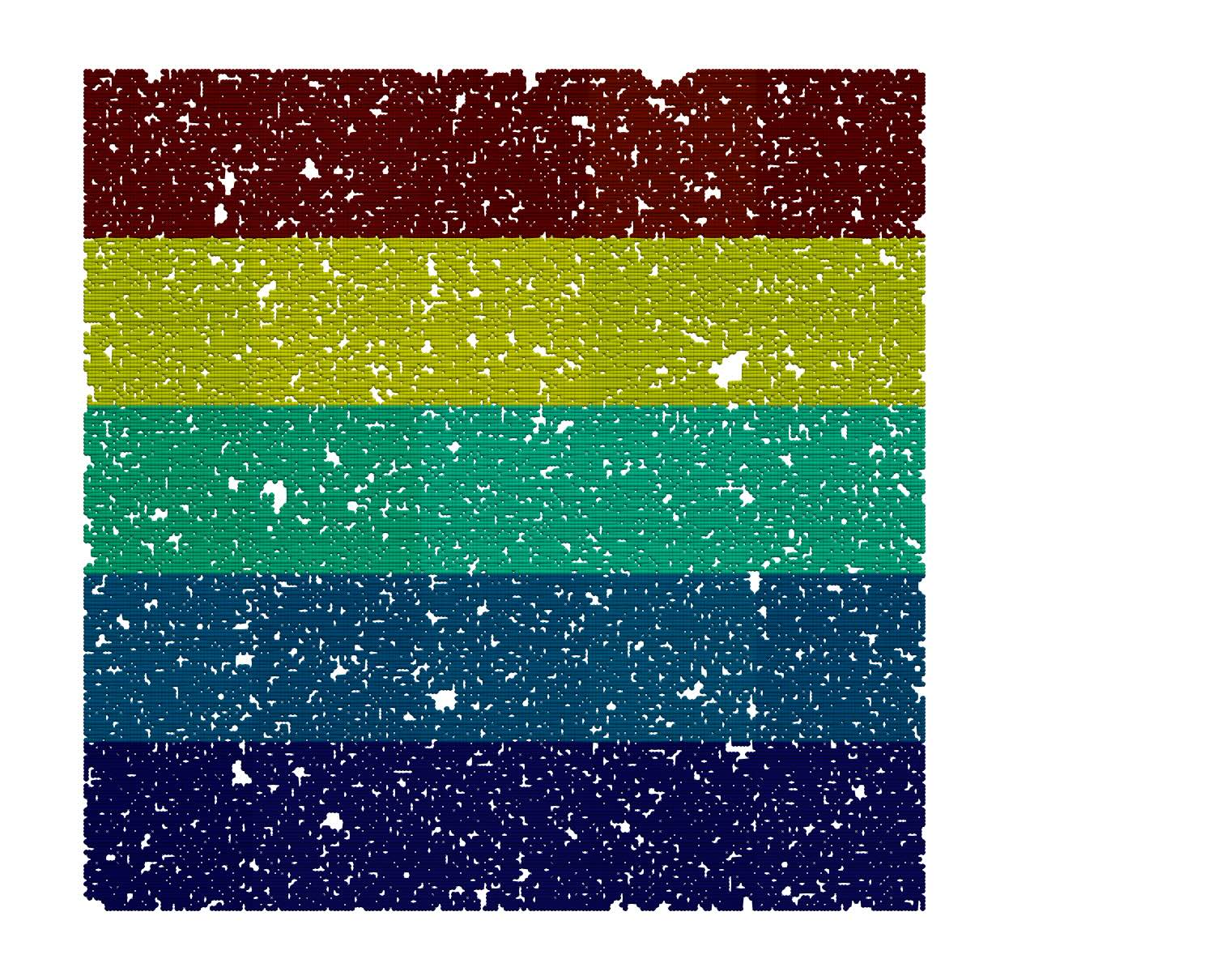}
\includegraphics[width=0.24\linewidth]{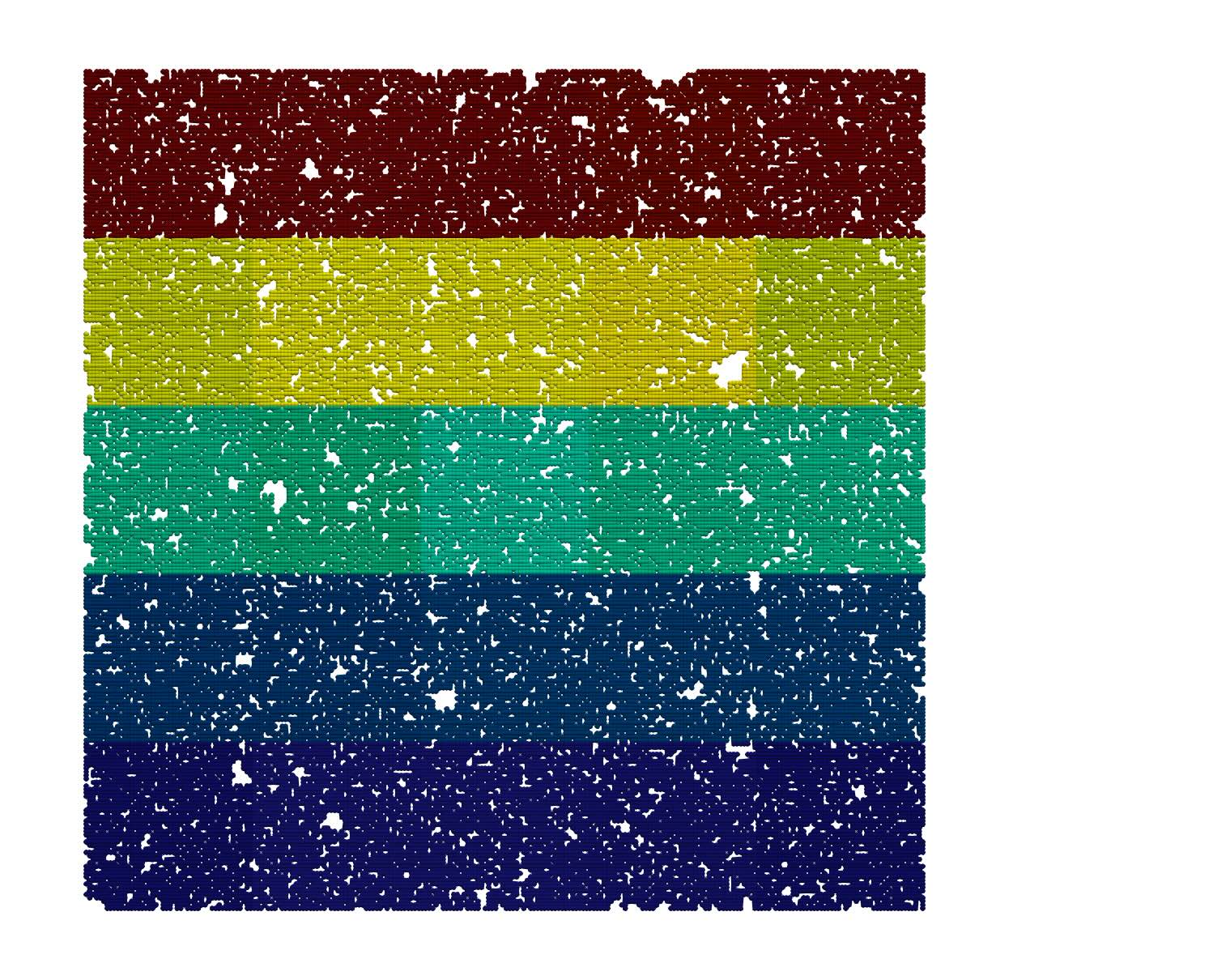}
\caption{\rev{Test-2a: Random properties on Network 2a}}
\end{subfigure}
\begin{subfigure}{1\textwidth}
\centering
\includegraphics[width=0.24\linewidth]{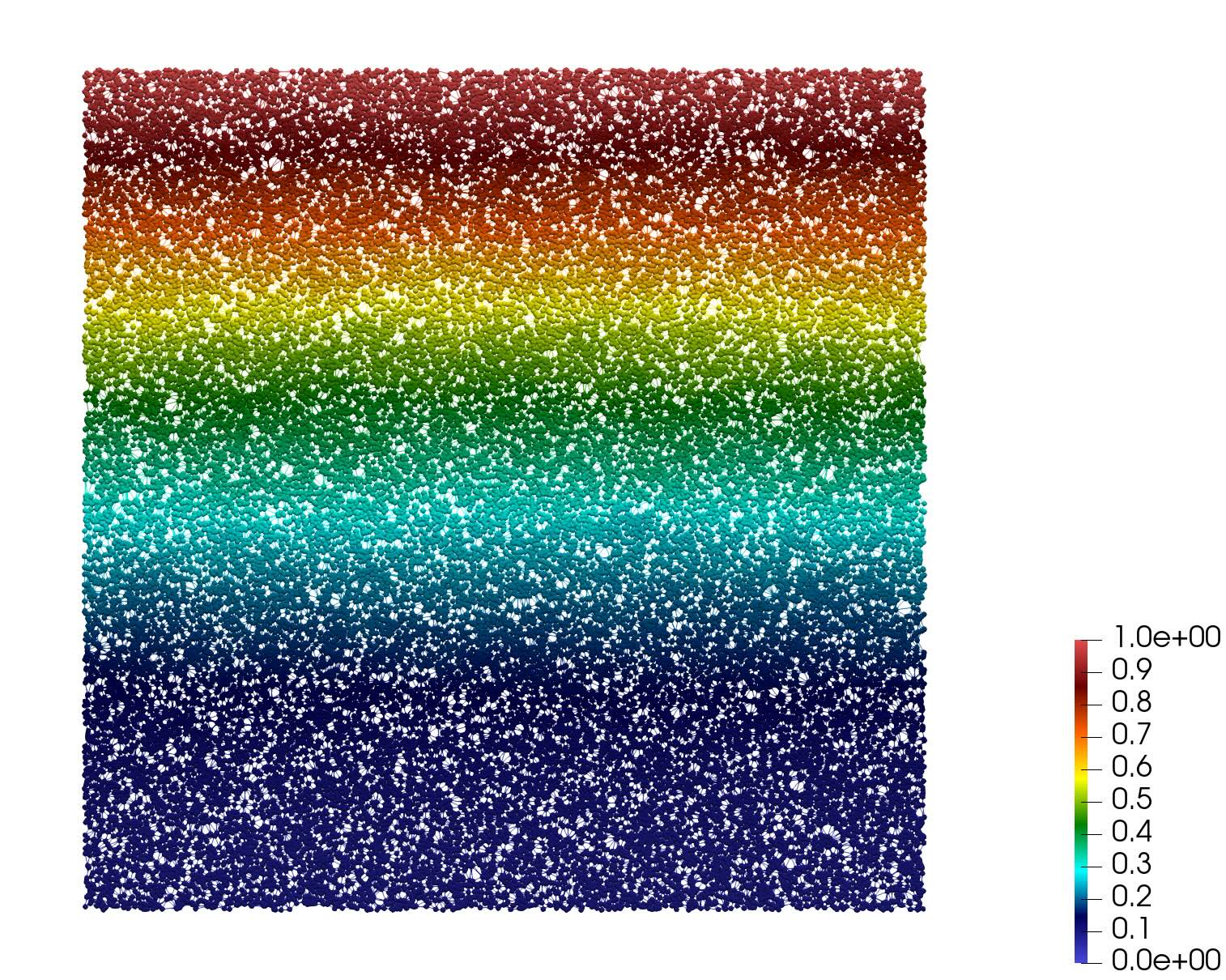}
\includegraphics[width=0.24\linewidth]{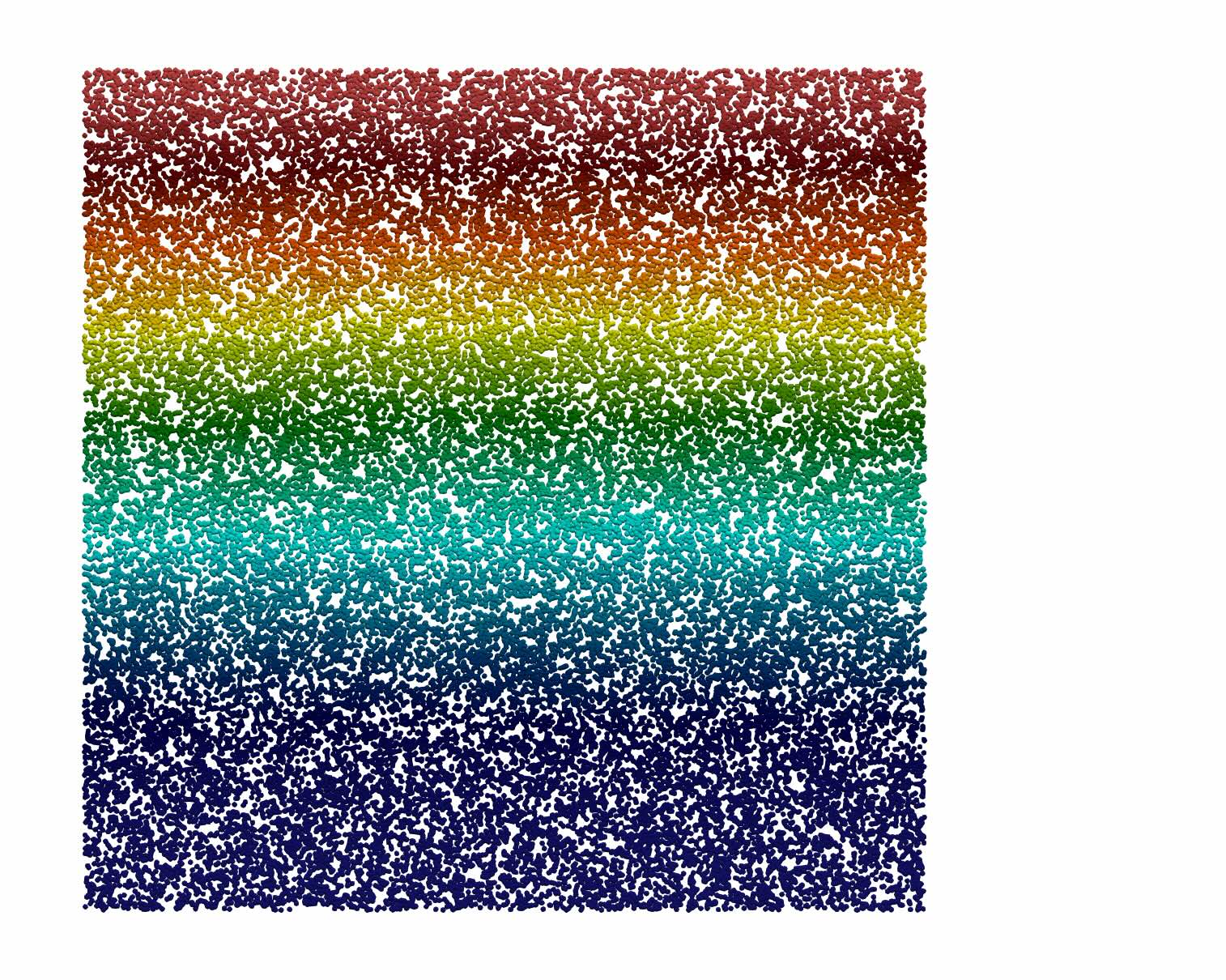}
\includegraphics[width=0.24\linewidth]{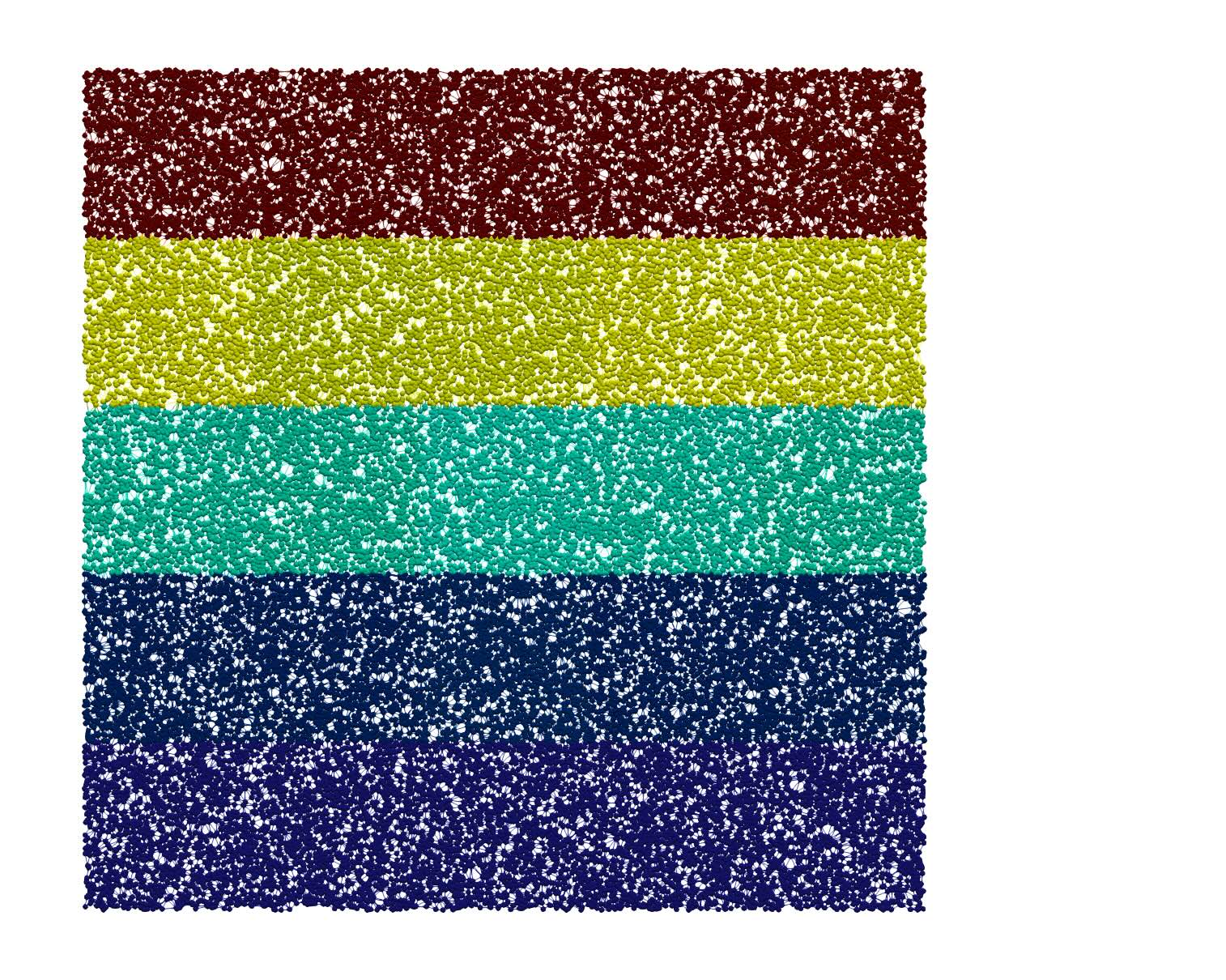}
\includegraphics[width=0.24\linewidth]{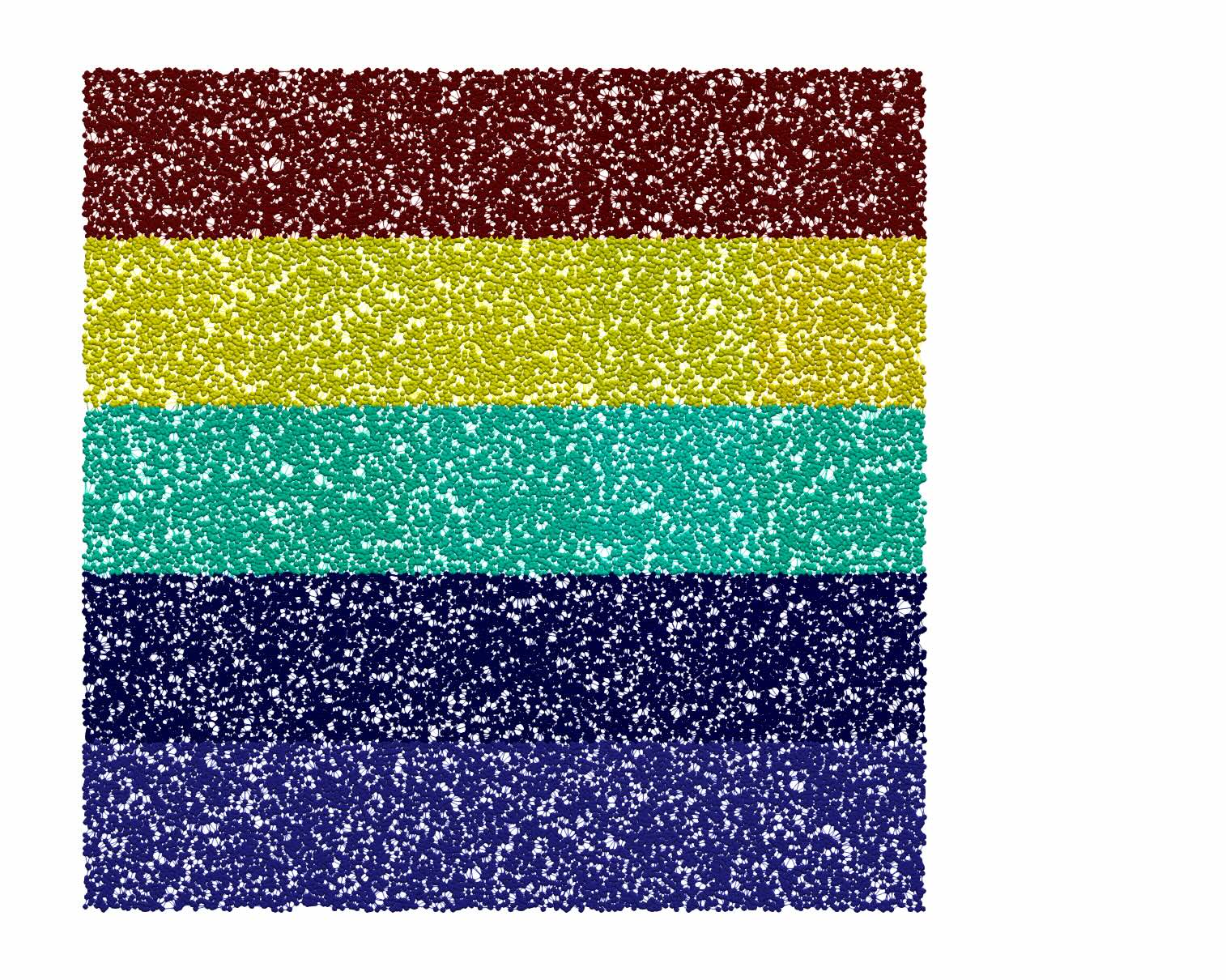}
\caption{\rev{Test-3a: Random properties on Network 3a}}
\end{subfigure}
\caption{\rev{Reference solution for the 2D case on the fine-scale network (first column), multiscale solution (second column), average reference solution on the coarse grid (third column) and upscaled solution on the coarse grid (fourth column).}}
\label{fig:solu2}
\end{figure}

\begin{figure}[h!]
\centering
\begin{subfigure}{1\textwidth}
\centering
\includegraphics[width=0.24\linewidth]{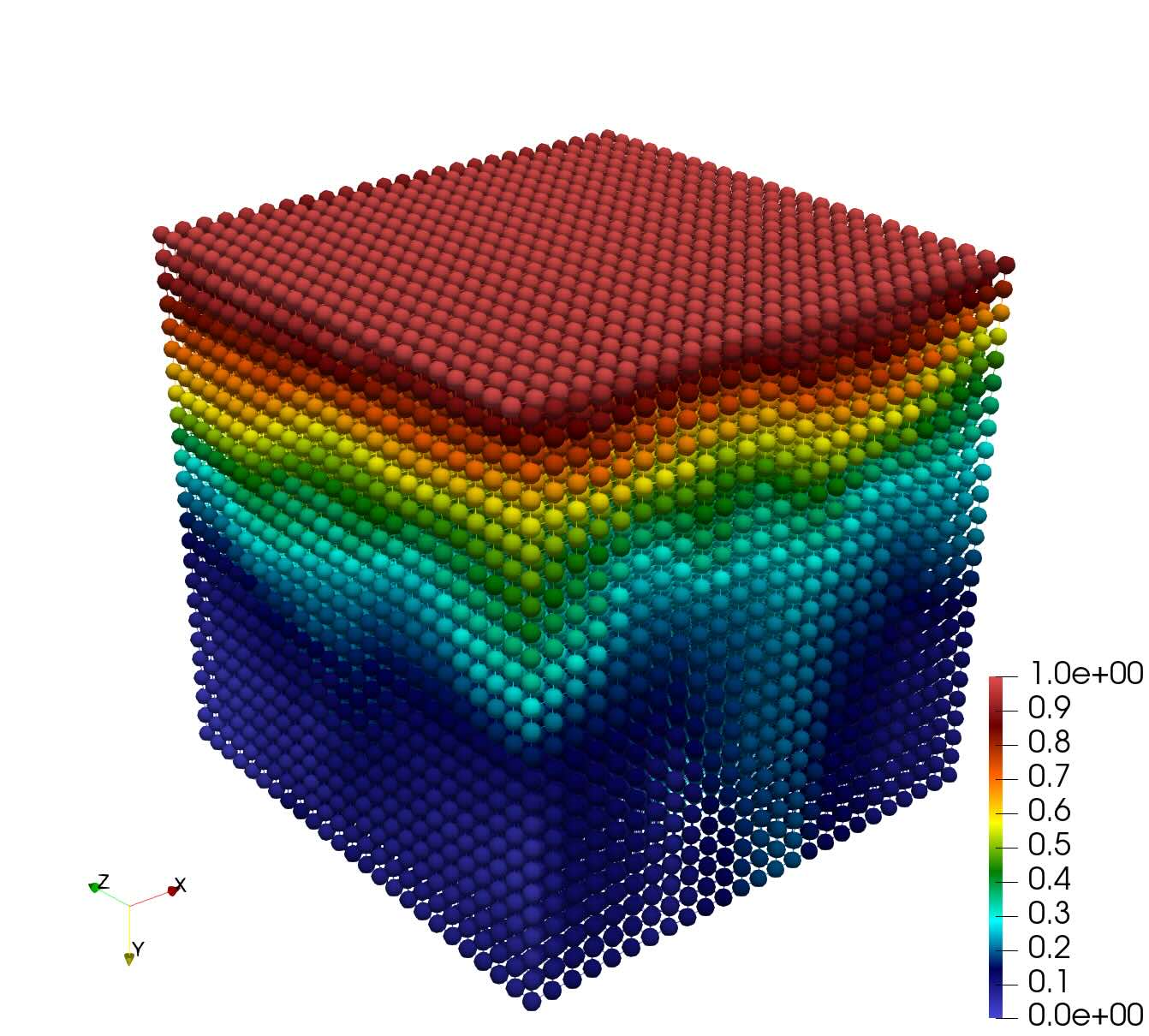}
\includegraphics[width=0.24\linewidth]{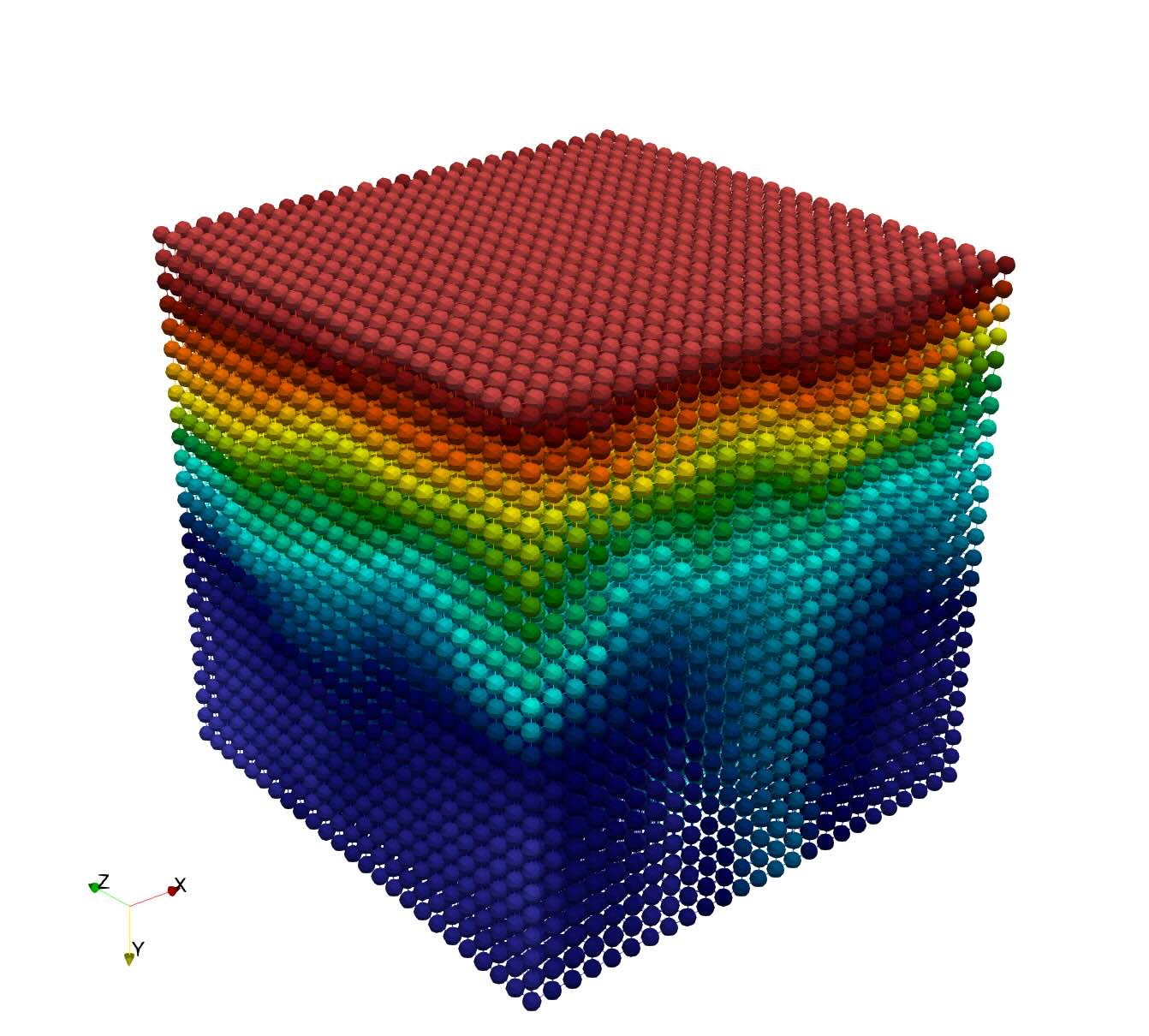}
\includegraphics[width=0.24\linewidth]{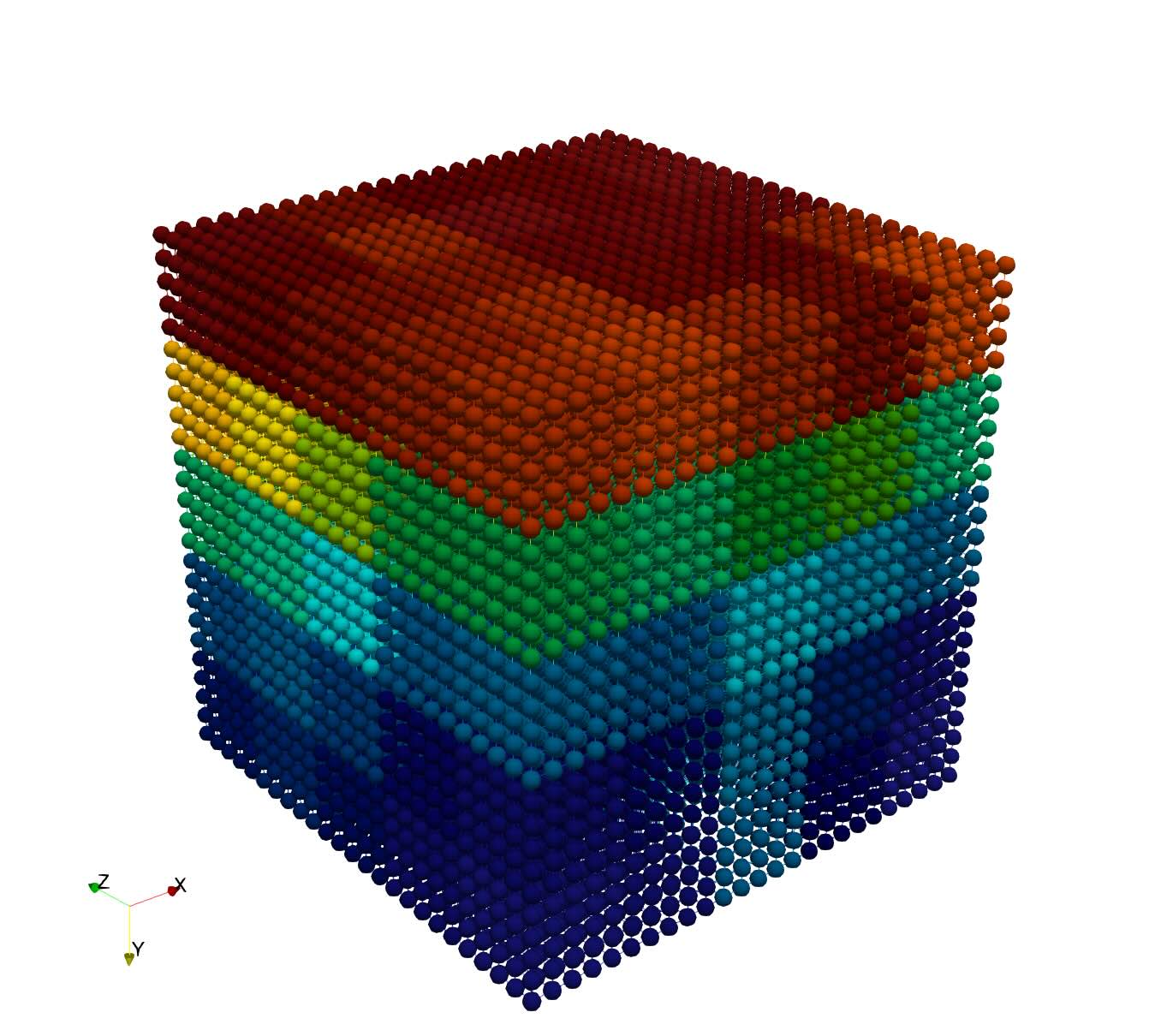}
\includegraphics[width=0.24\linewidth]{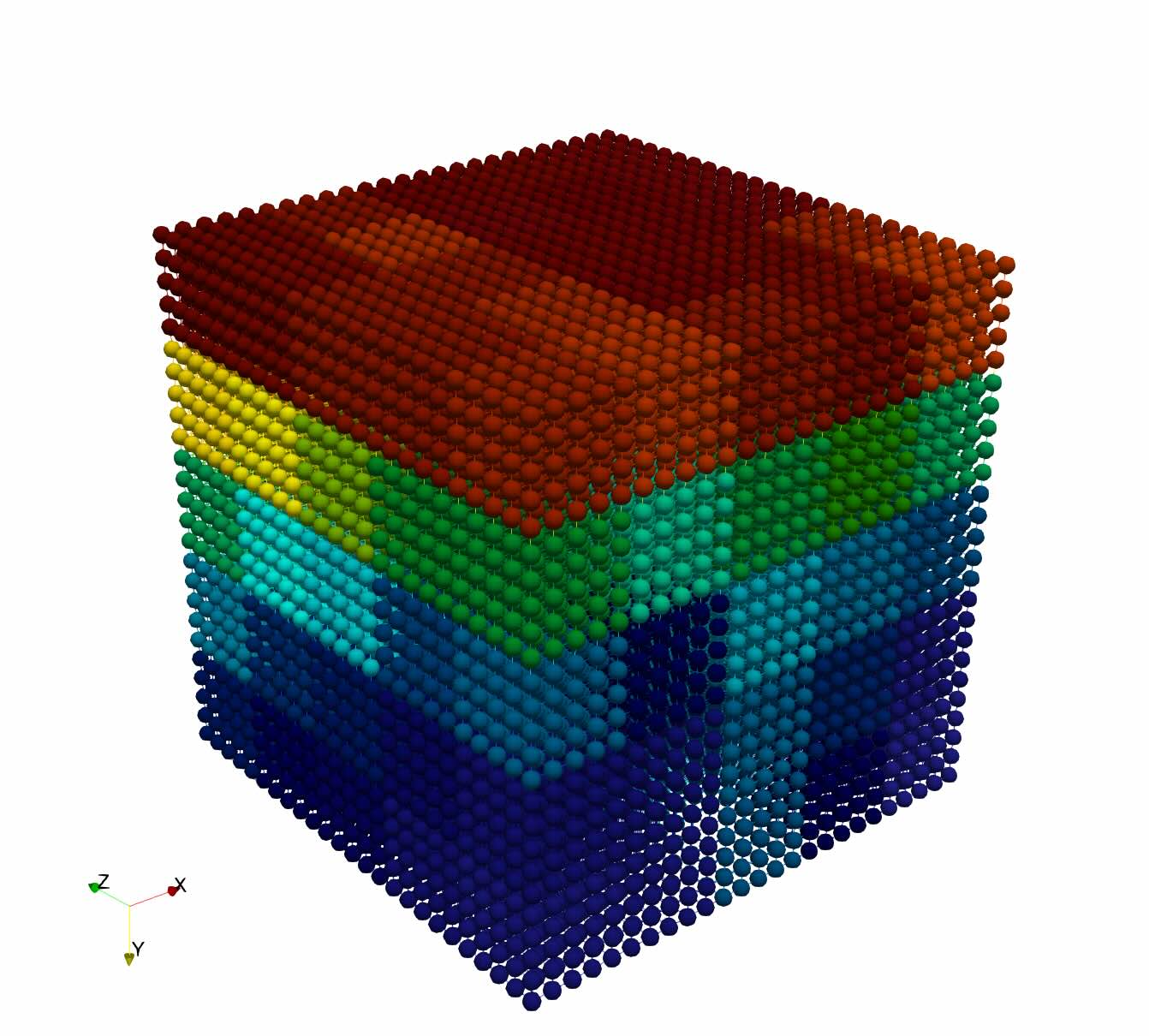}
\caption{\rev{Test-1b: SPE10 properties on Network 1b}}
\end{subfigure}
\begin{subfigure}{1\textwidth}
\centering
\includegraphics[width=0.24\linewidth]{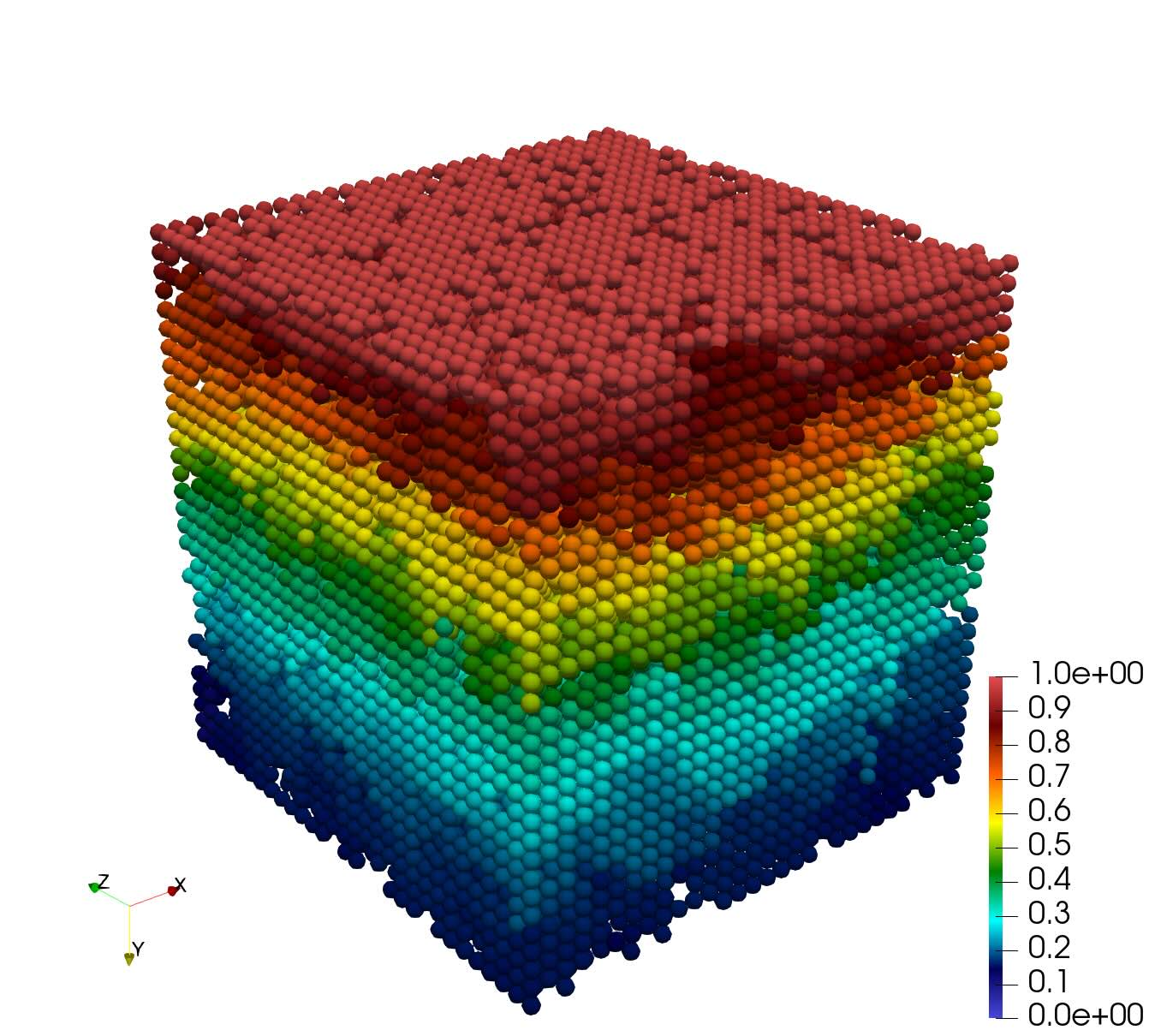}
\includegraphics[width=0.24\linewidth]{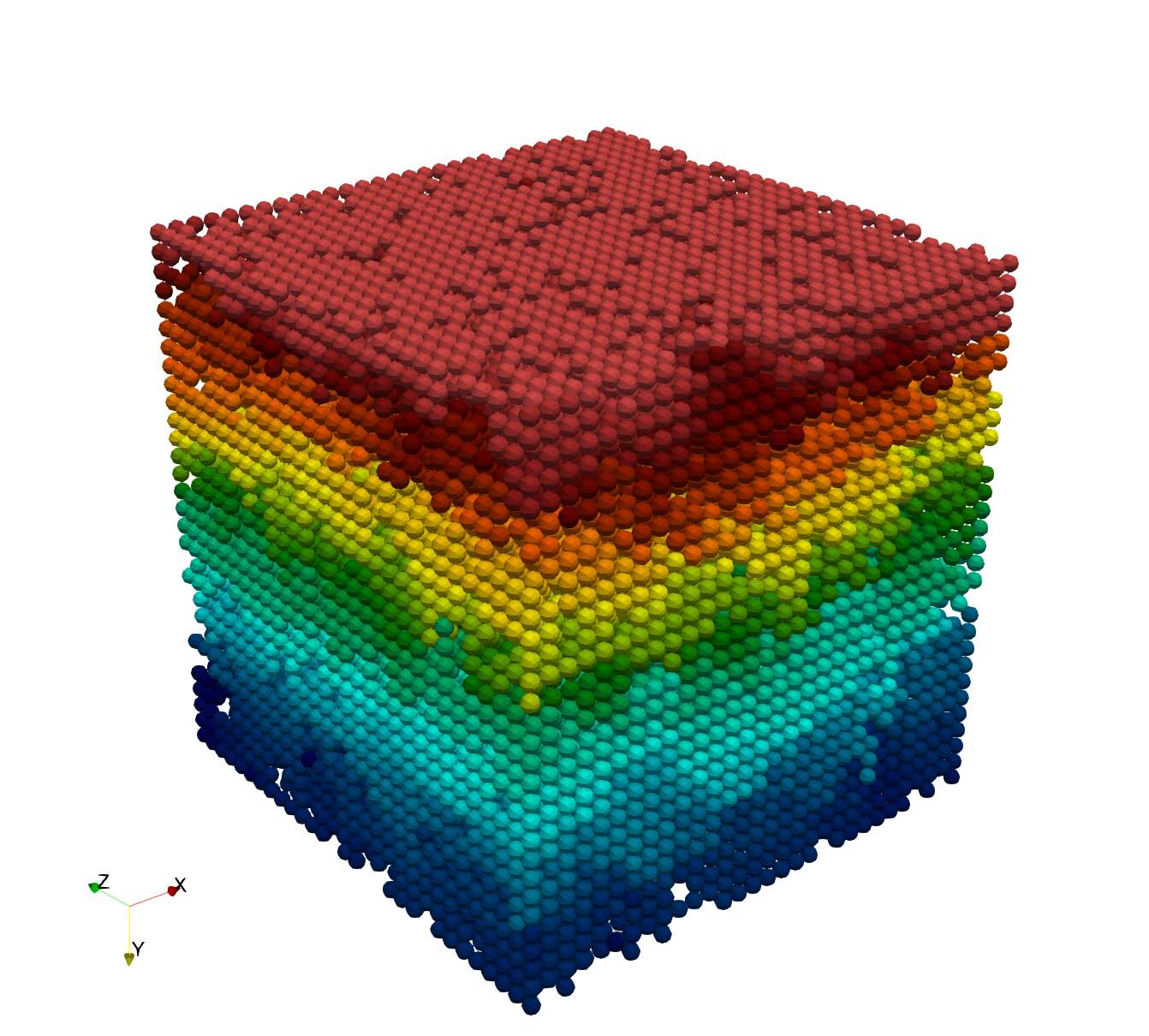}
\includegraphics[width=0.24\linewidth]{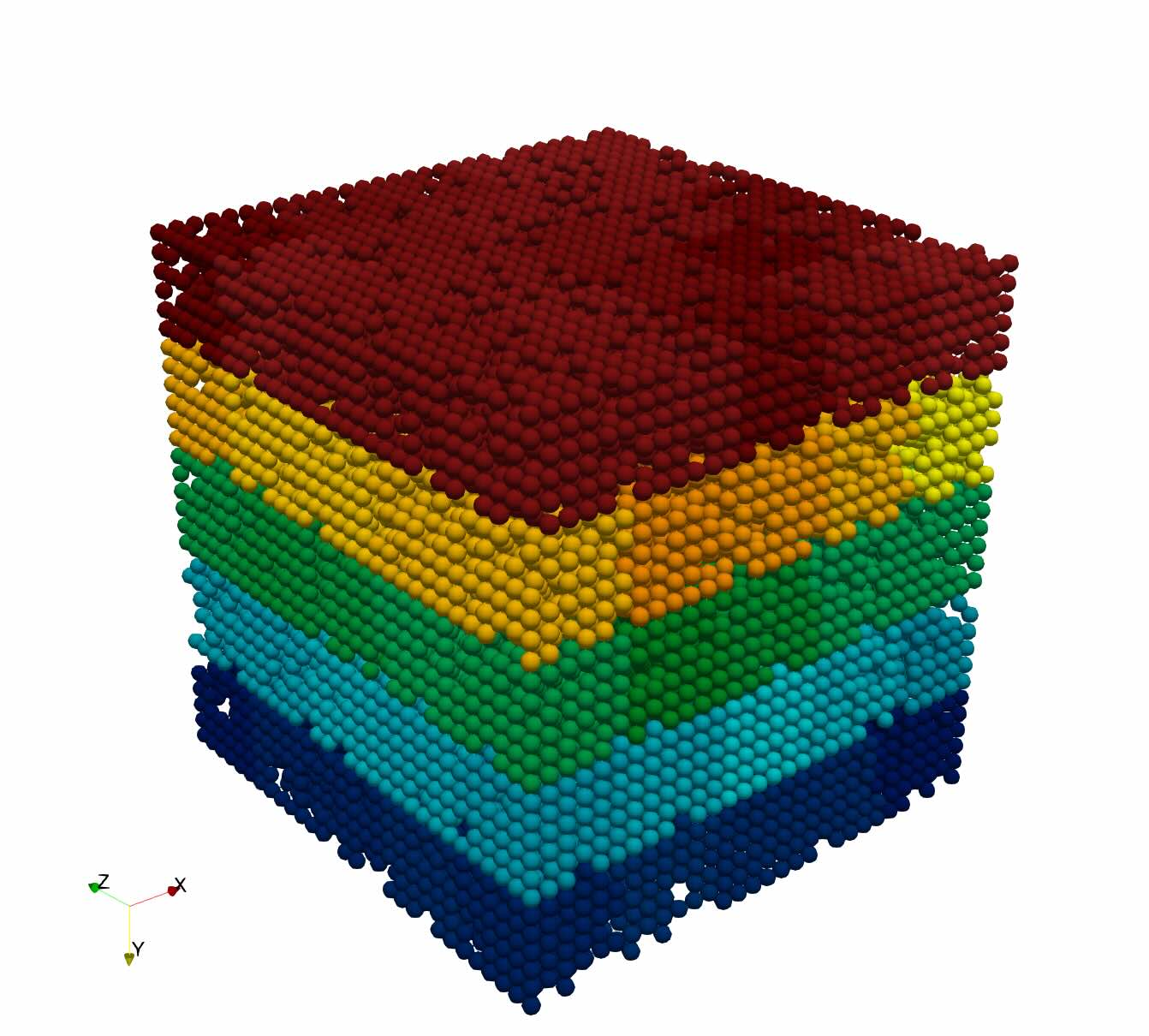}
\includegraphics[width=0.24\linewidth]{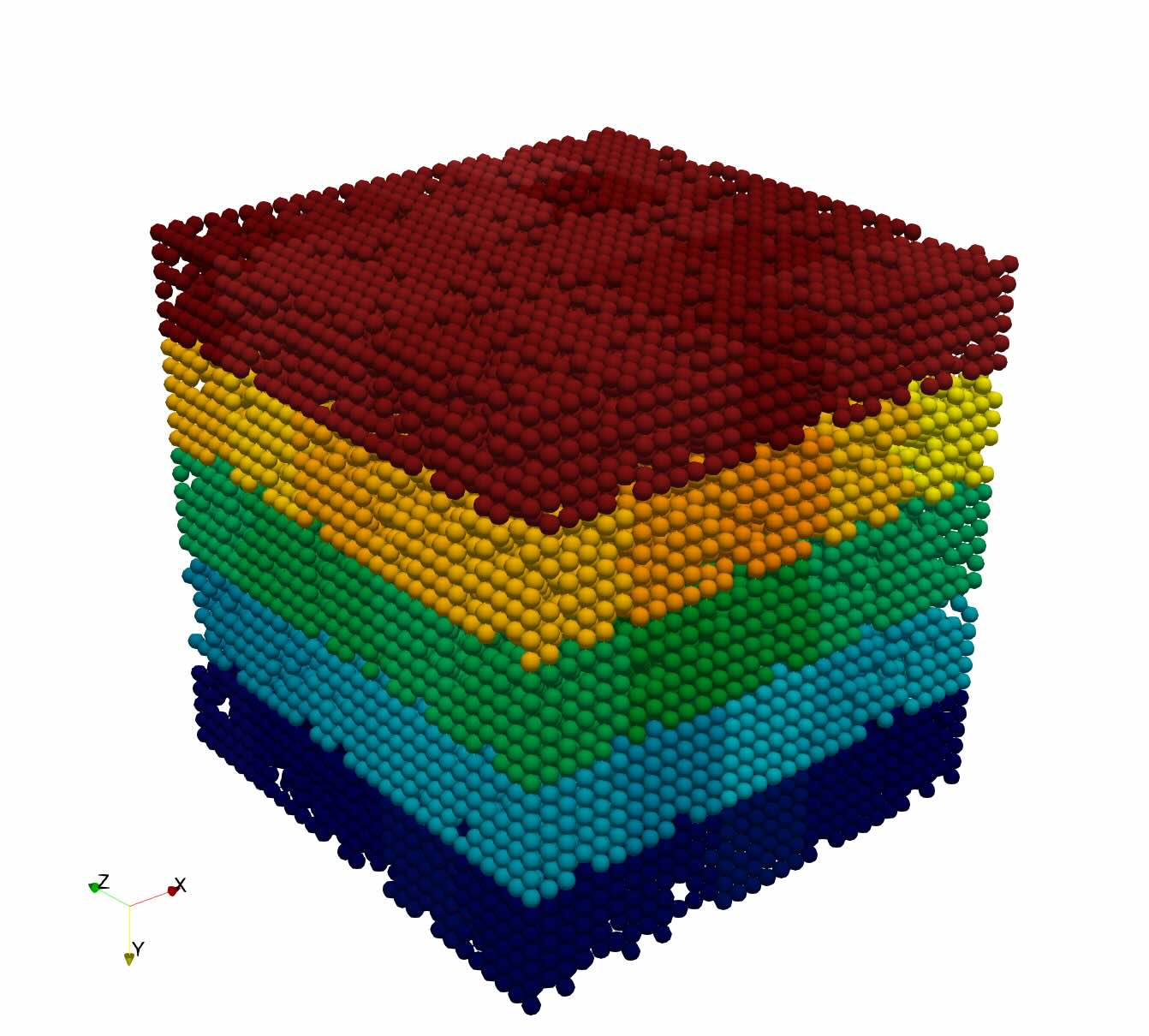}
\caption{\rev{Test-2b: Random properties on Network 2b}}
\end{subfigure}
\begin{subfigure}{1\textwidth}
\centering
\includegraphics[width=0.24\linewidth]{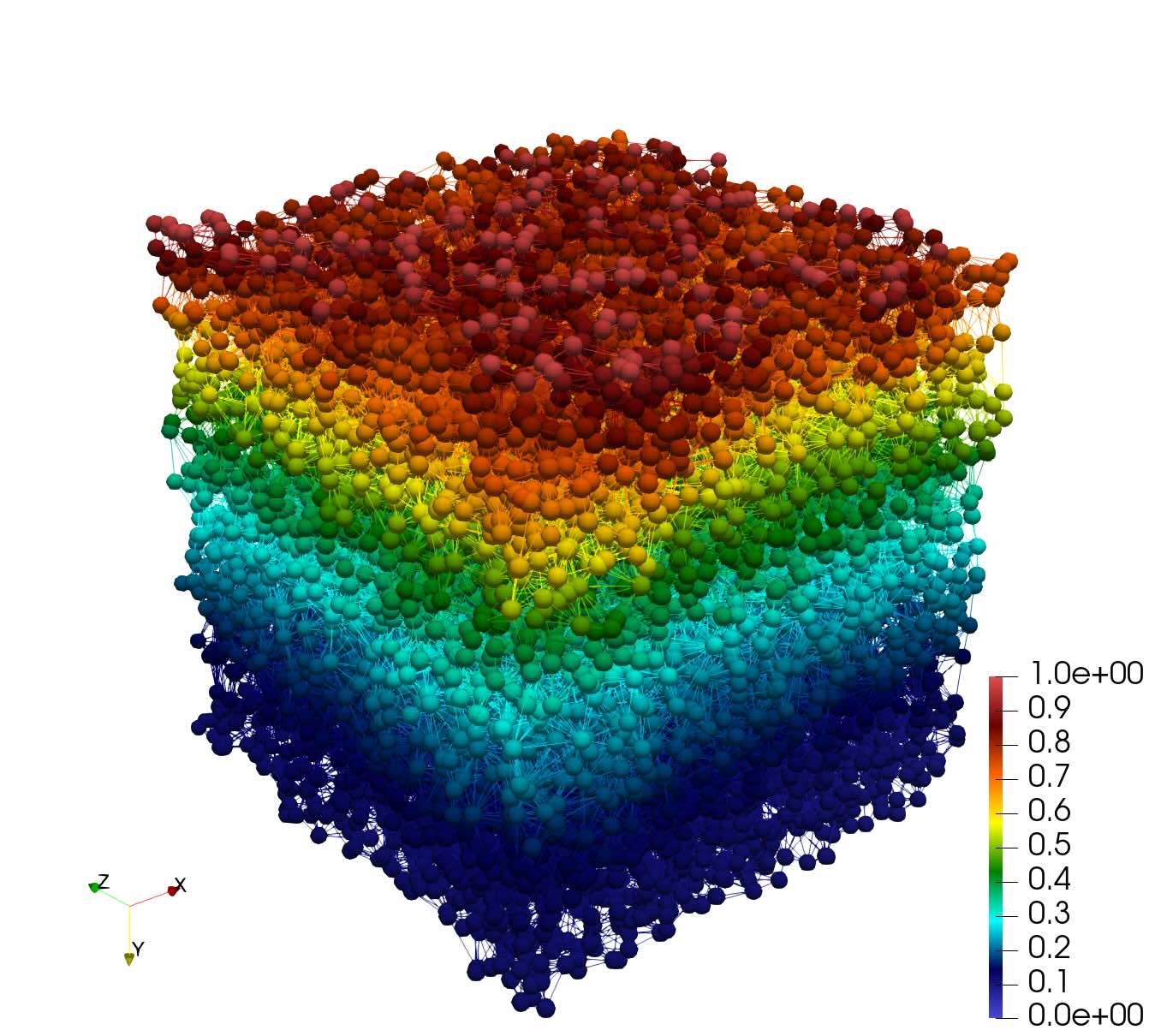}
\includegraphics[width=0.24\linewidth]{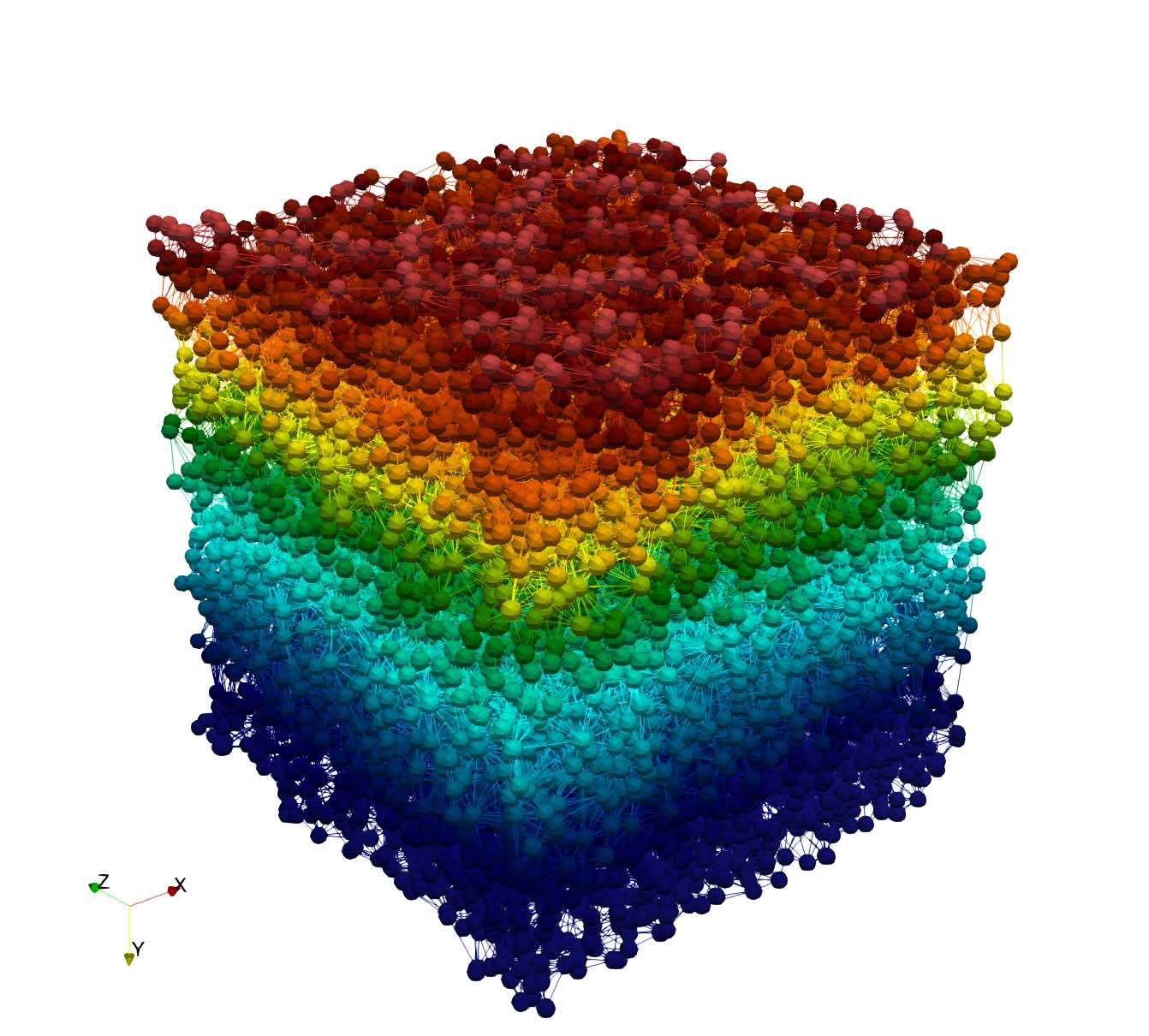}
\includegraphics[width=0.24\linewidth]{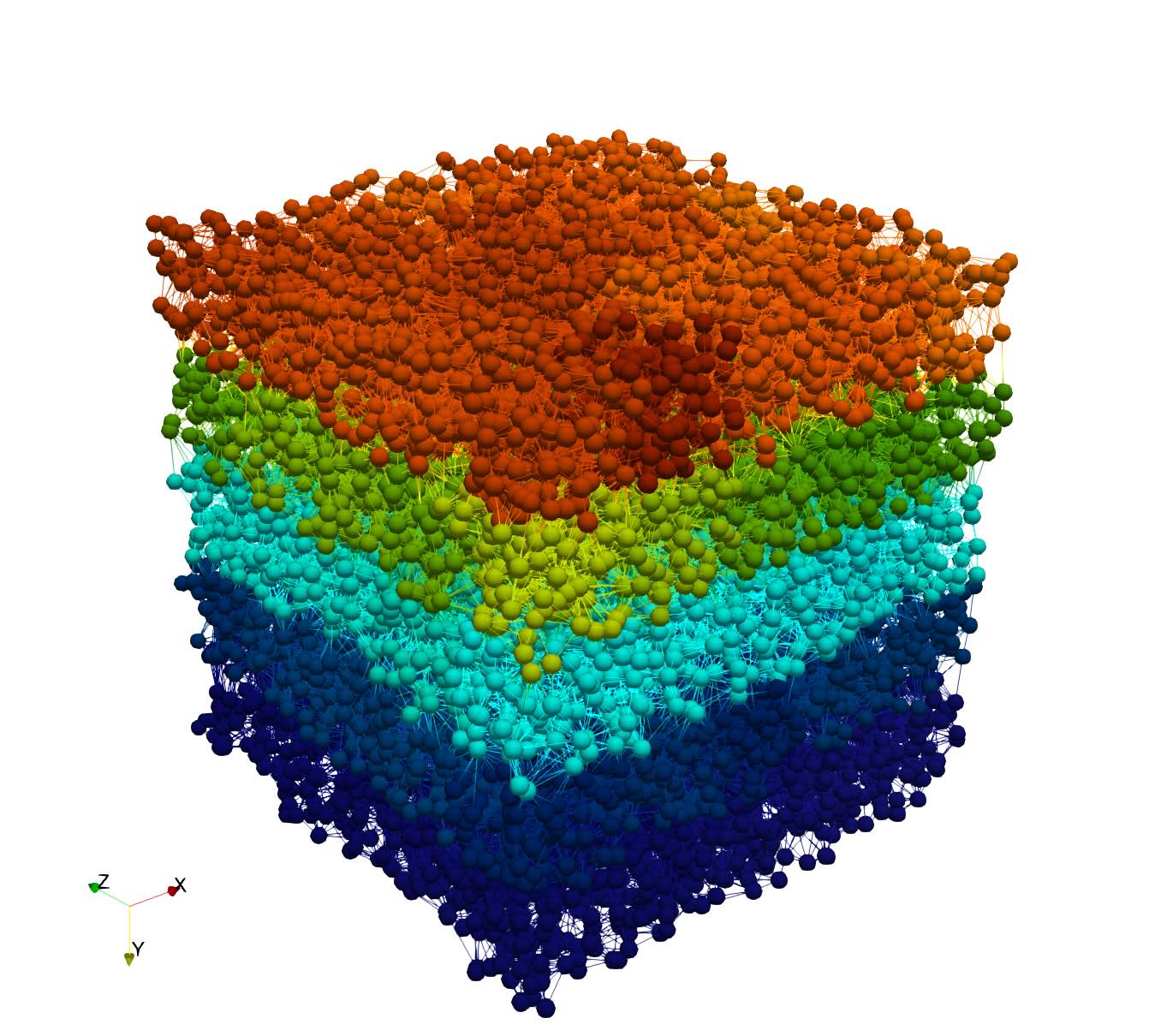}
\includegraphics[width=0.24\linewidth]{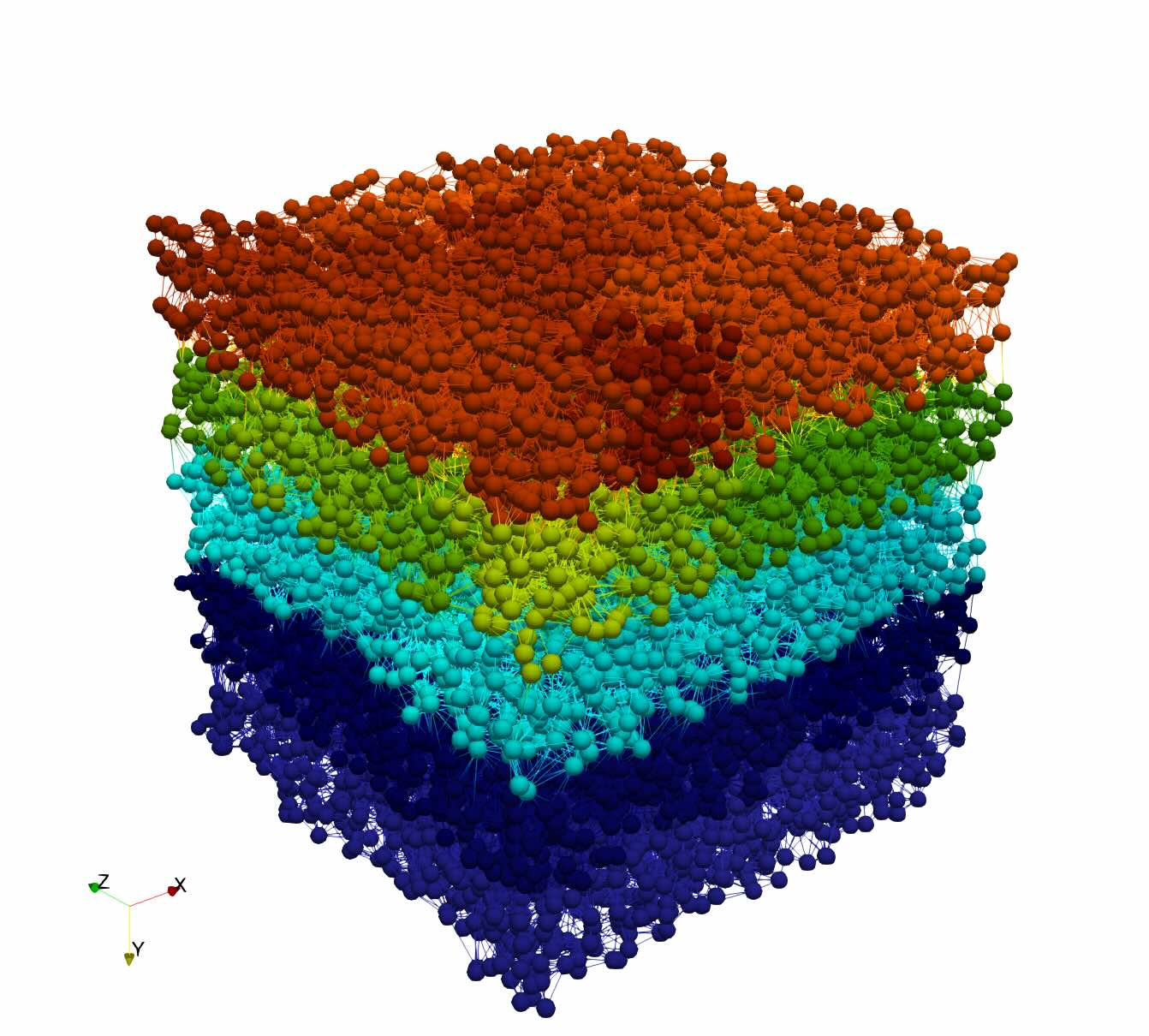}
\caption{\rev{Test-3b: Random properties on Network 3b}}
\end{subfigure}
\caption{\rev{Reference solution for the 3D case on the fine-scale network (first column), multiscale solution (second column), average reference solution on the coarse grid (third column) and upscaled solution on the coarse grid (fourth column).}}
\label{fig:solu3}
\end{figure}

\begin{figure}[h!]
\centering
\begin{subfigure}{1\textwidth}
\centering
\includegraphics[width=0.24\linewidth]{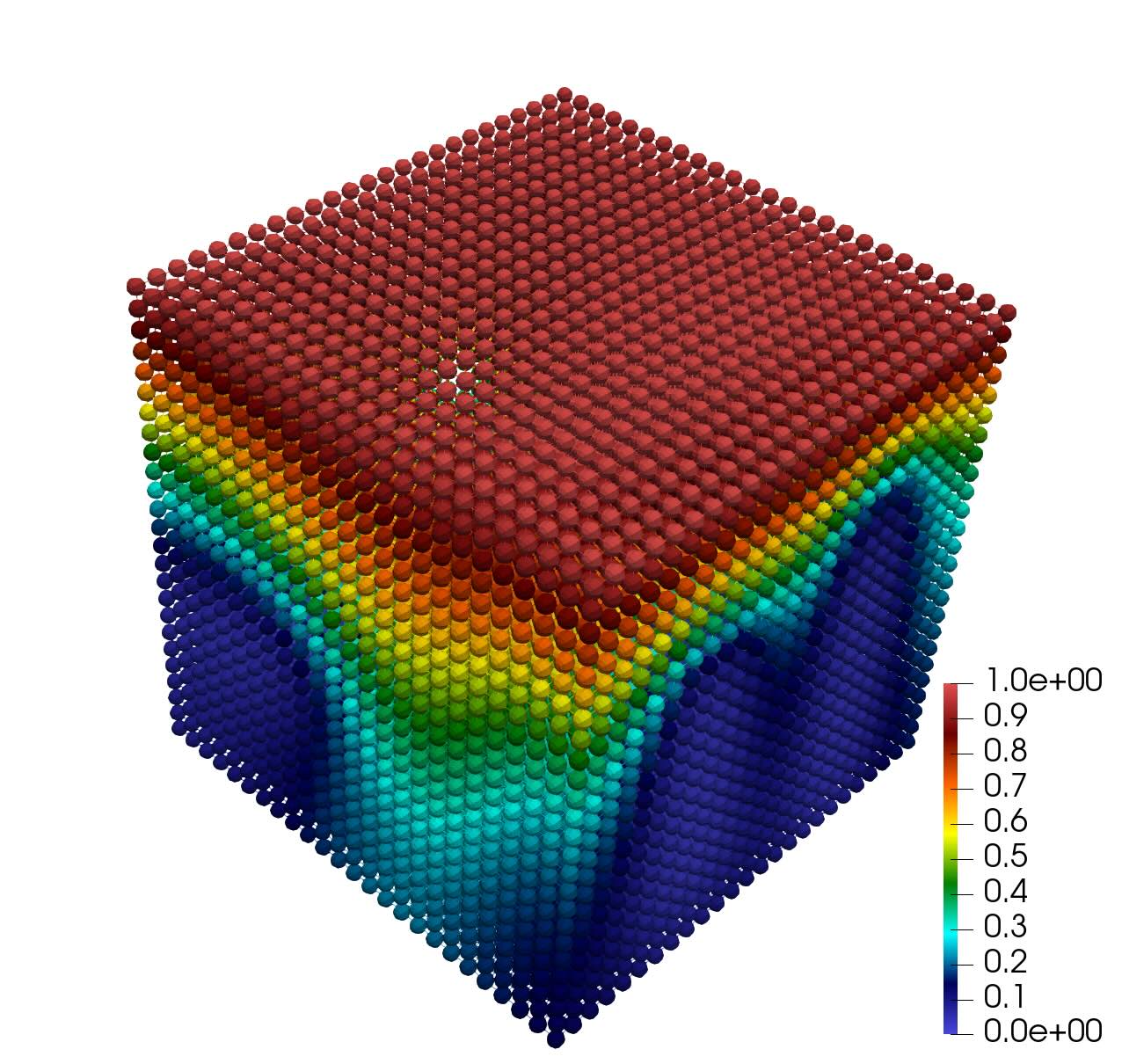}
\includegraphics[width=0.24\linewidth]{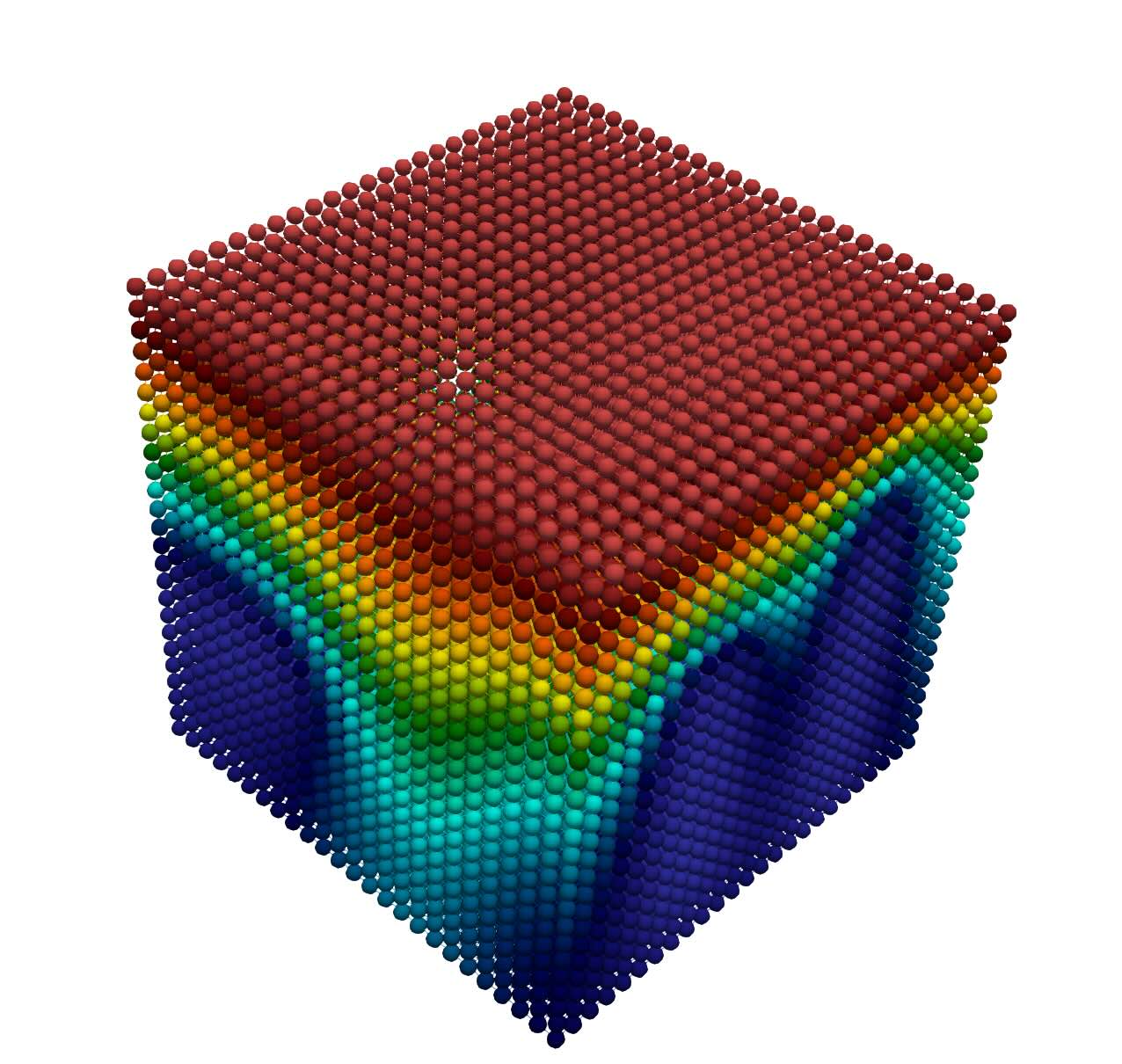}
\includegraphics[width=0.24\linewidth]{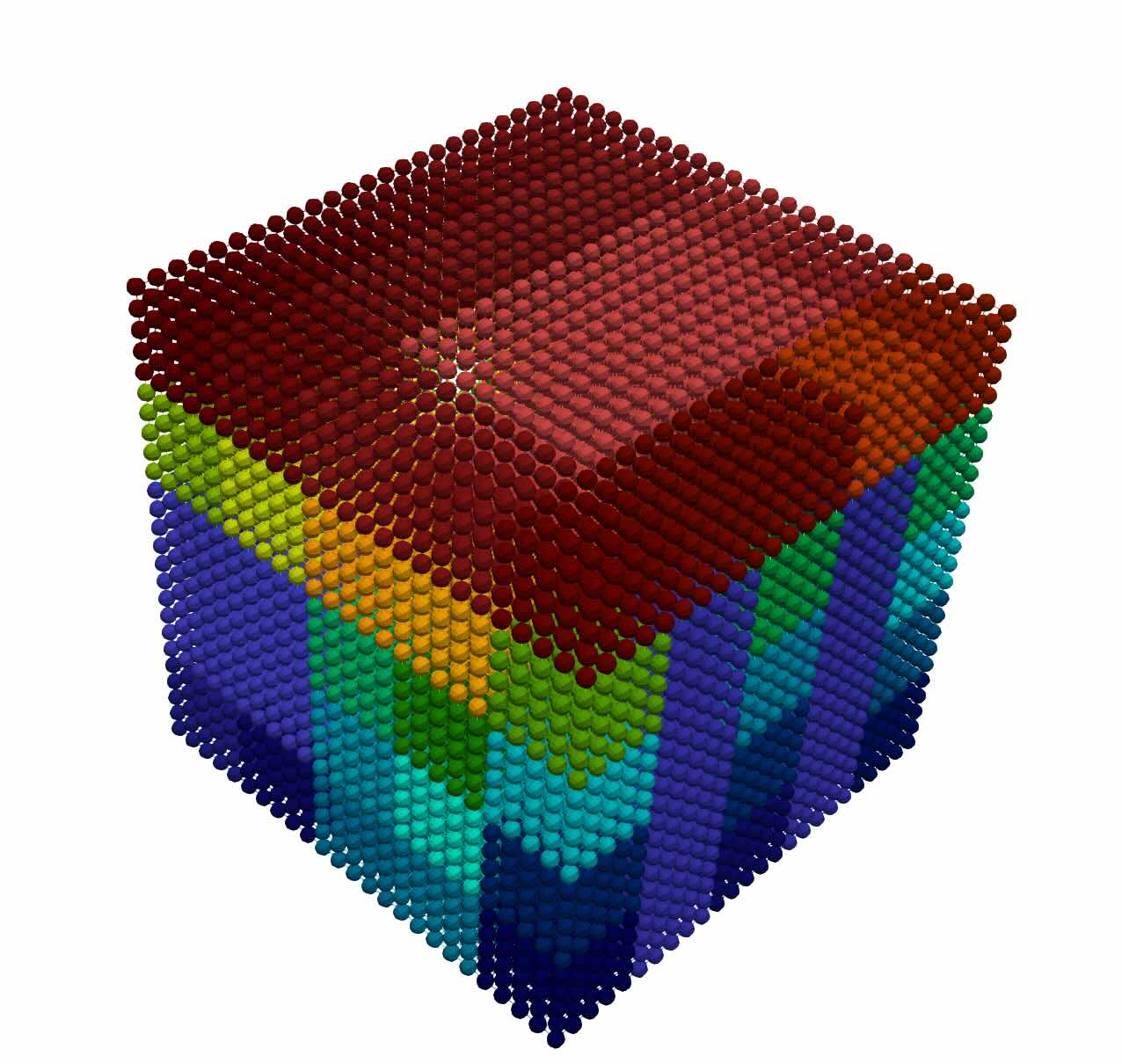}
\includegraphics[width=0.24\linewidth]{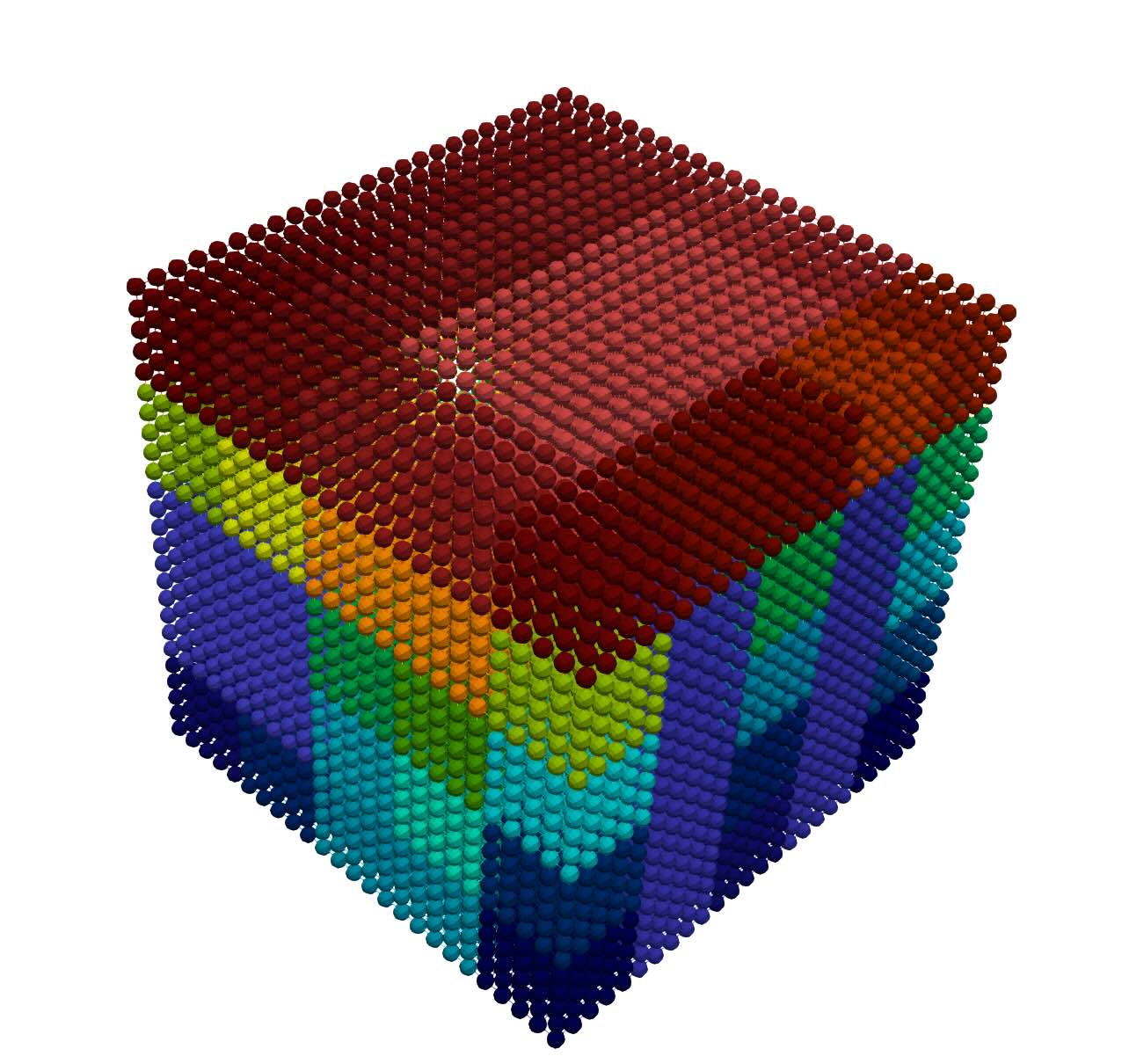}
\caption{\rev{Test-1c: High-contrast properties on Network 1b}}
\end{subfigure}
\begin{subfigure}{1\textwidth}
\centering
\includegraphics[width=0.24\linewidth]{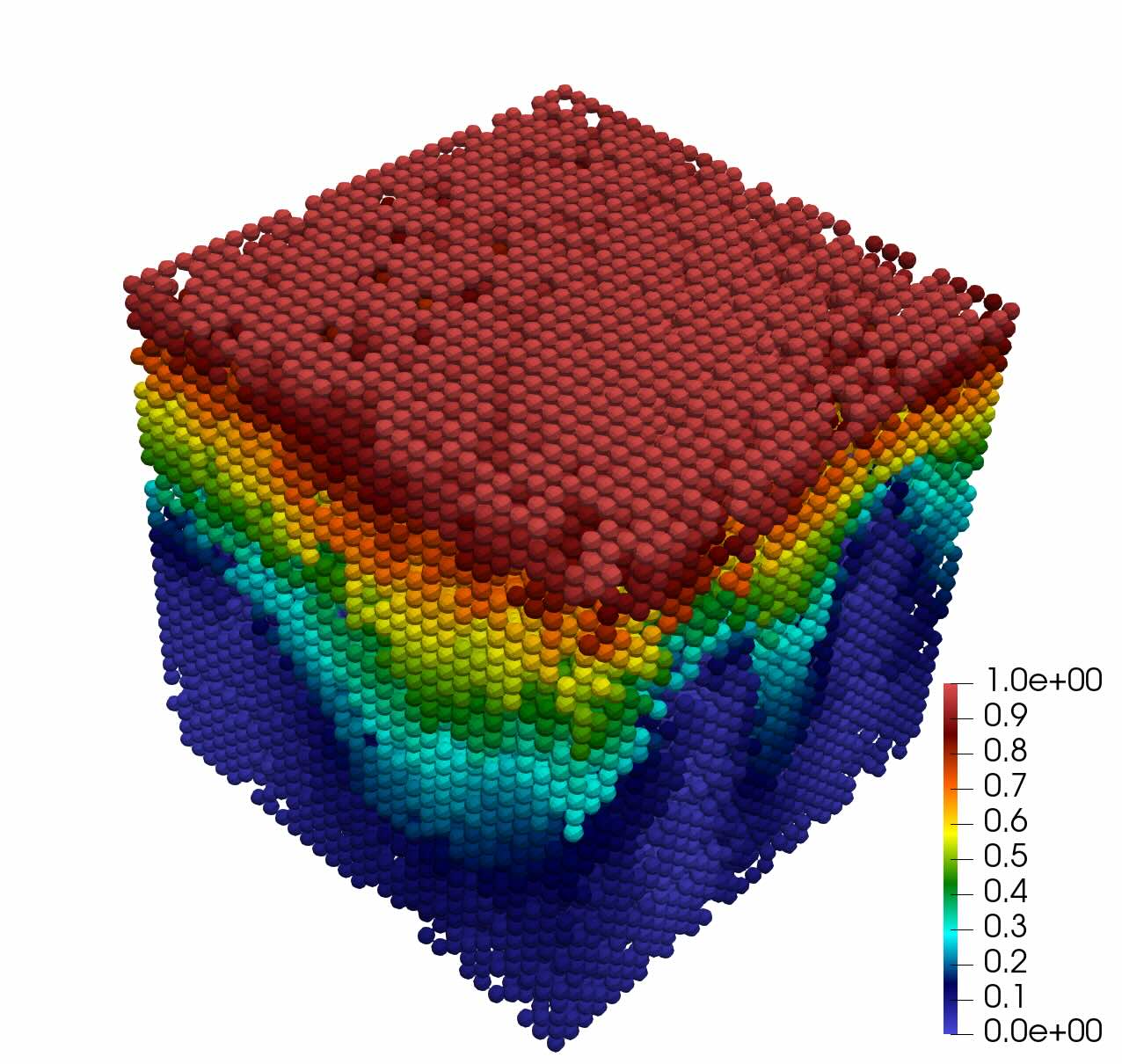}
\includegraphics[width=0.24\linewidth]{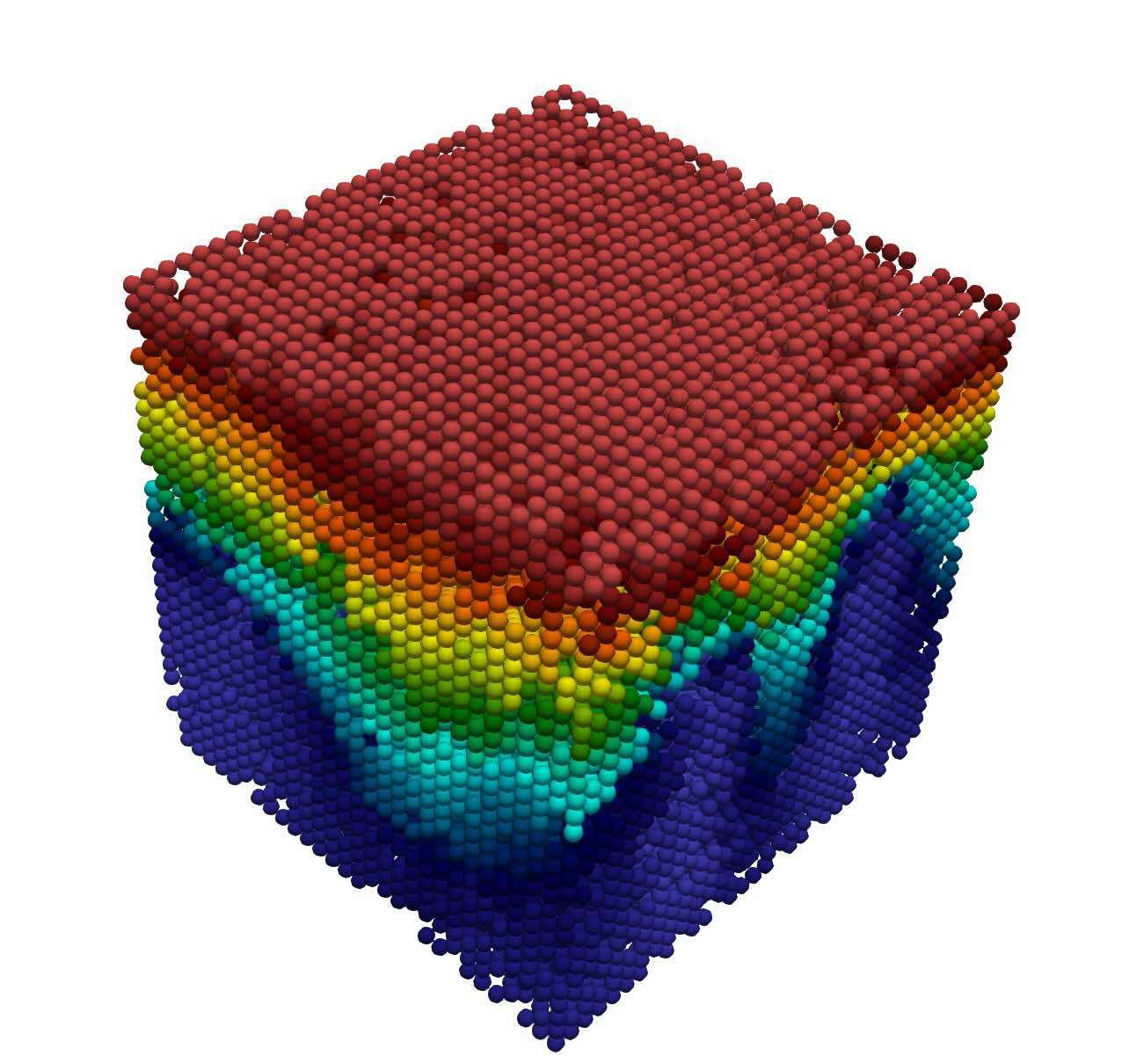}
\includegraphics[width=0.24\linewidth]{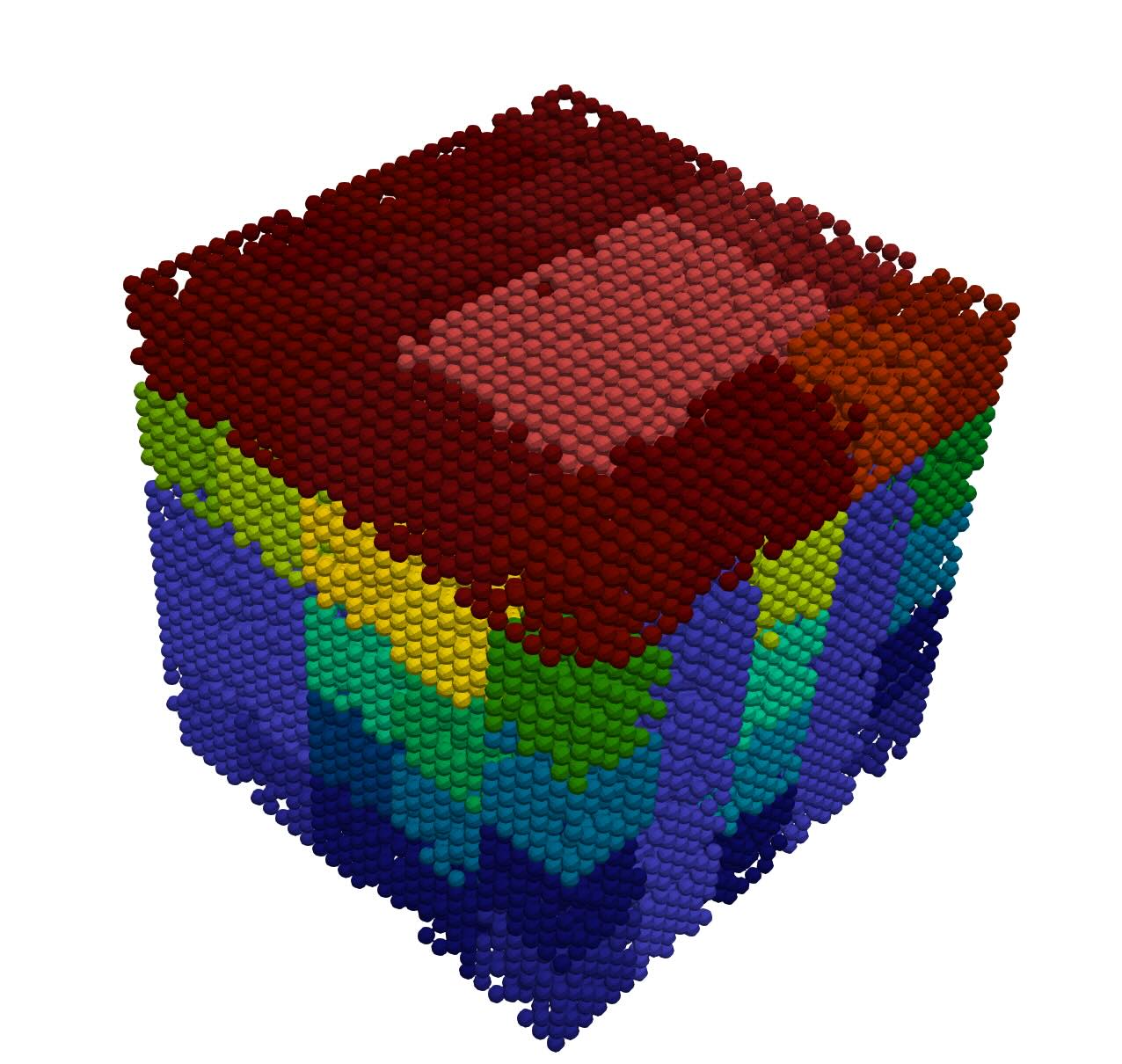}
\includegraphics[width=0.24\linewidth]{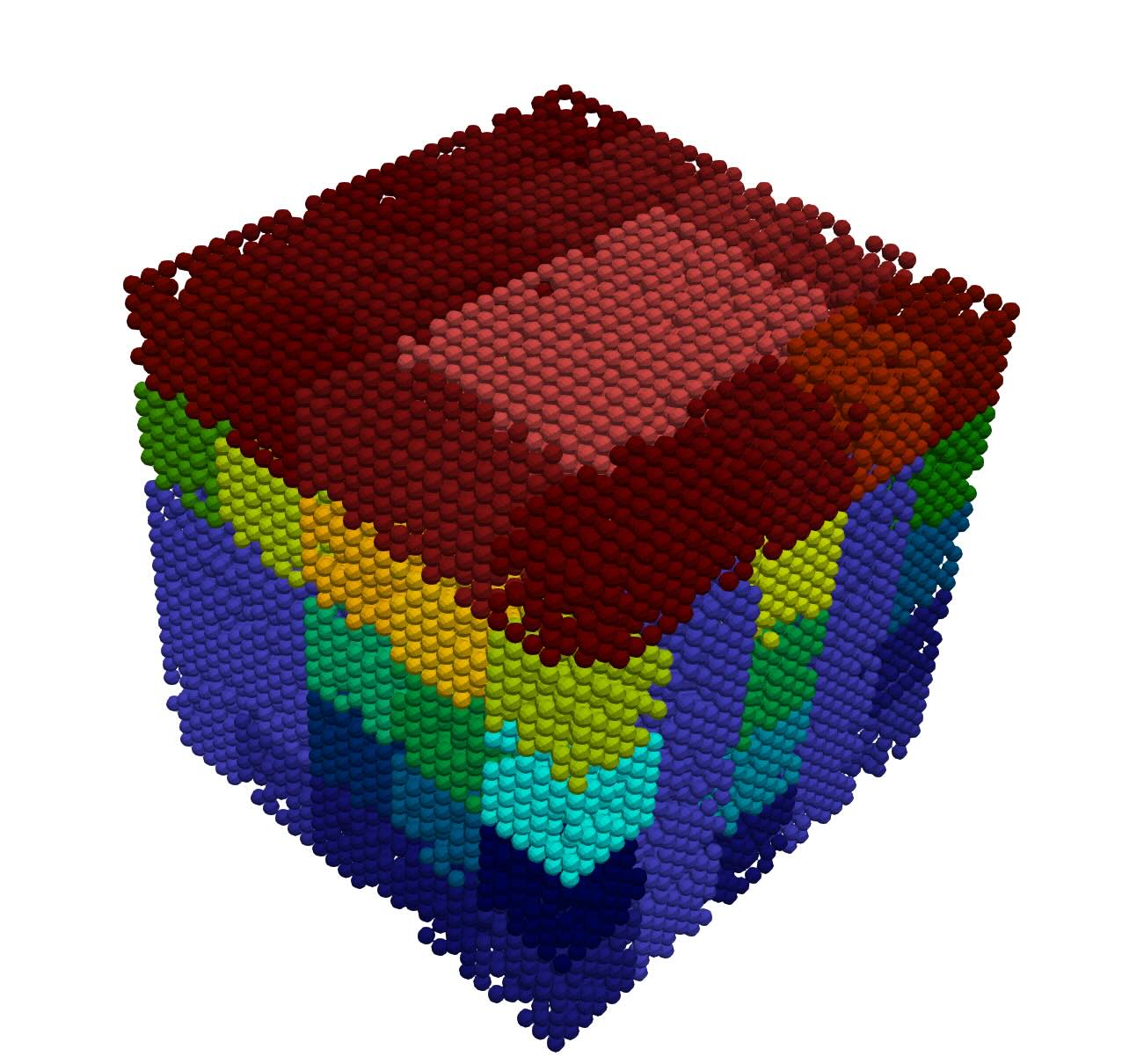}
\caption{\rev{Test-2c: High-contrast properties on Network 2b}}
\end{subfigure}
\begin{subfigure}{1\textwidth}
\centering
\includegraphics[width=0.24\linewidth]{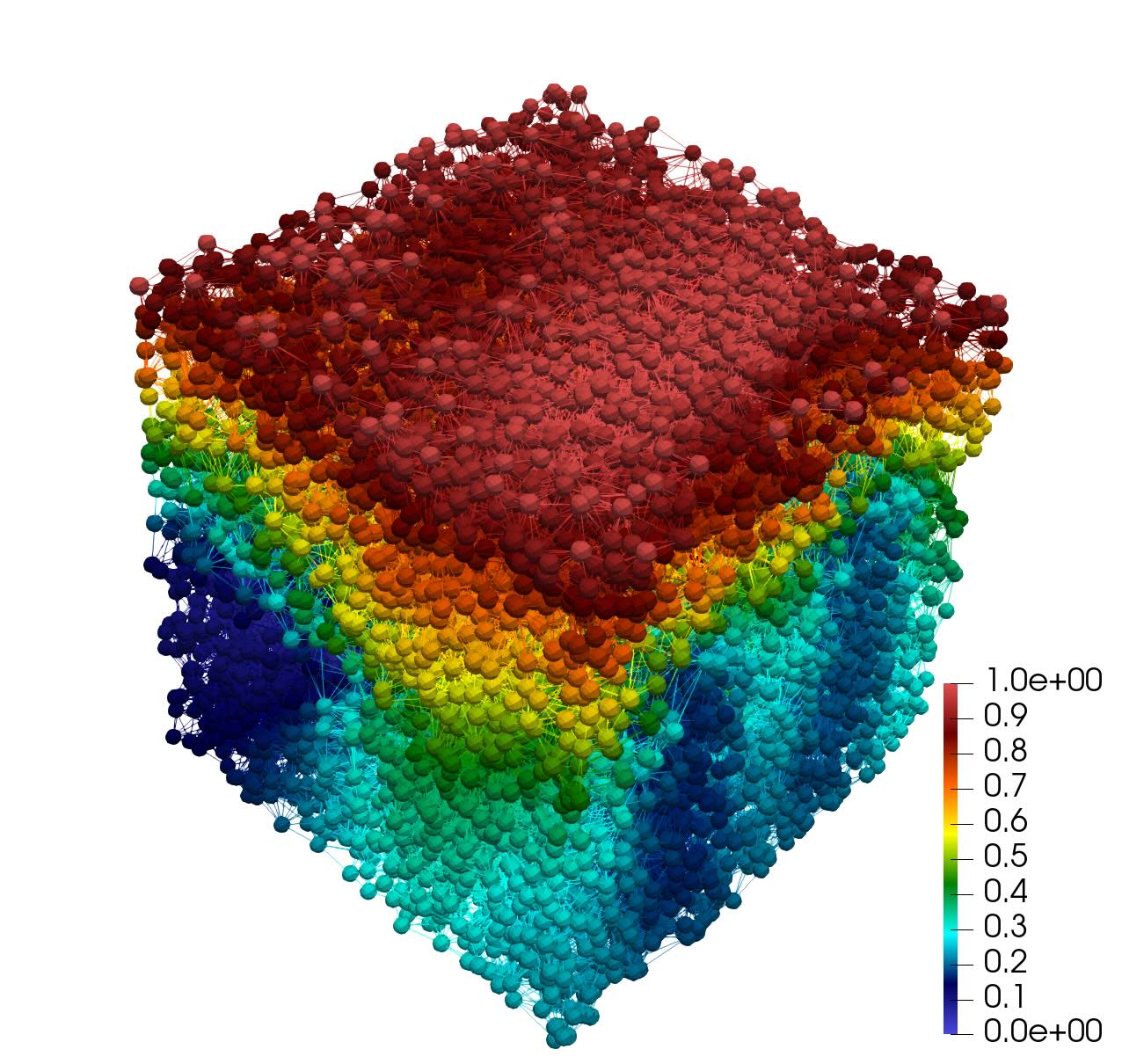}
\includegraphics[width=0.24\linewidth]{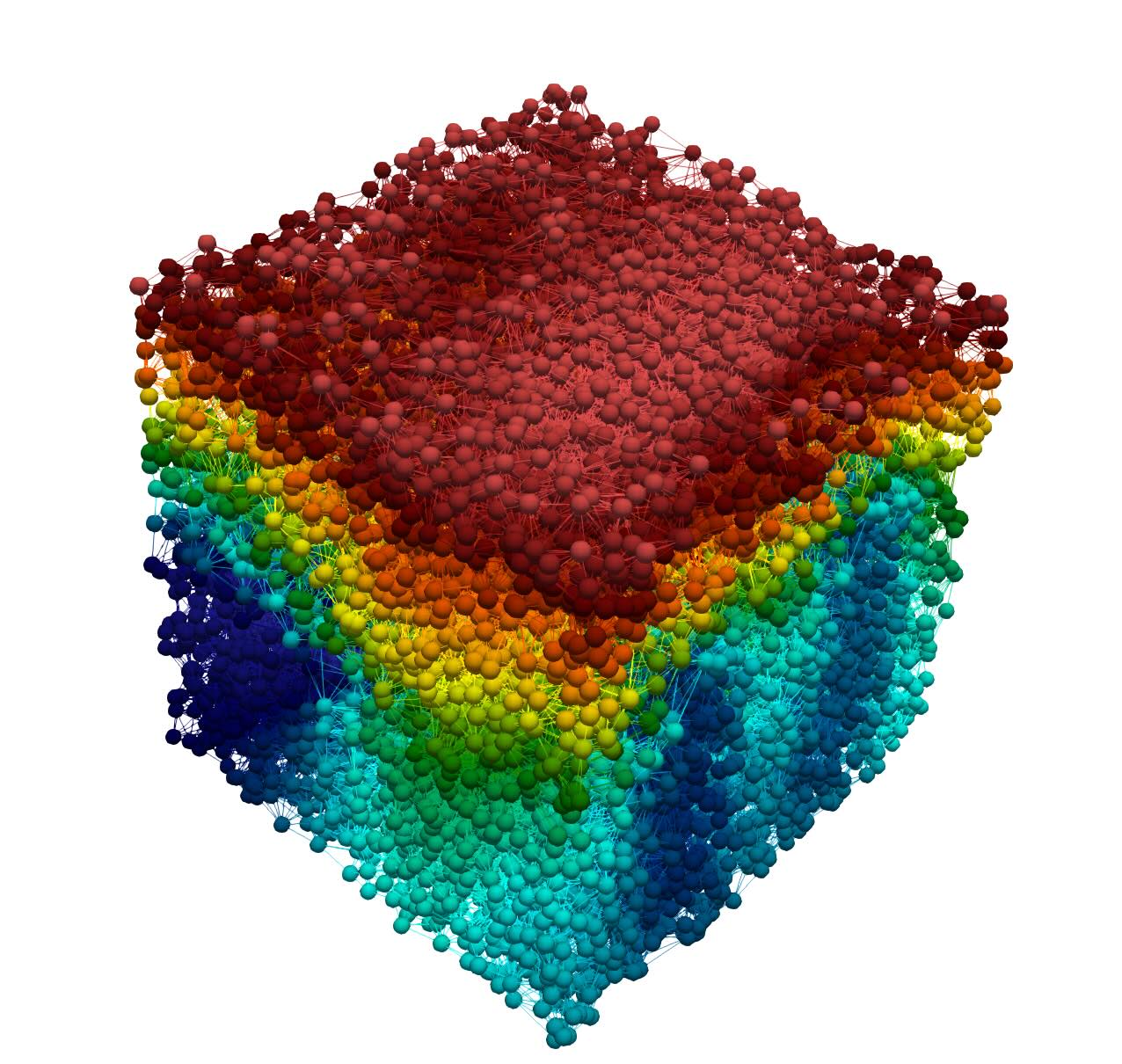}
\includegraphics[width=0.24\linewidth]{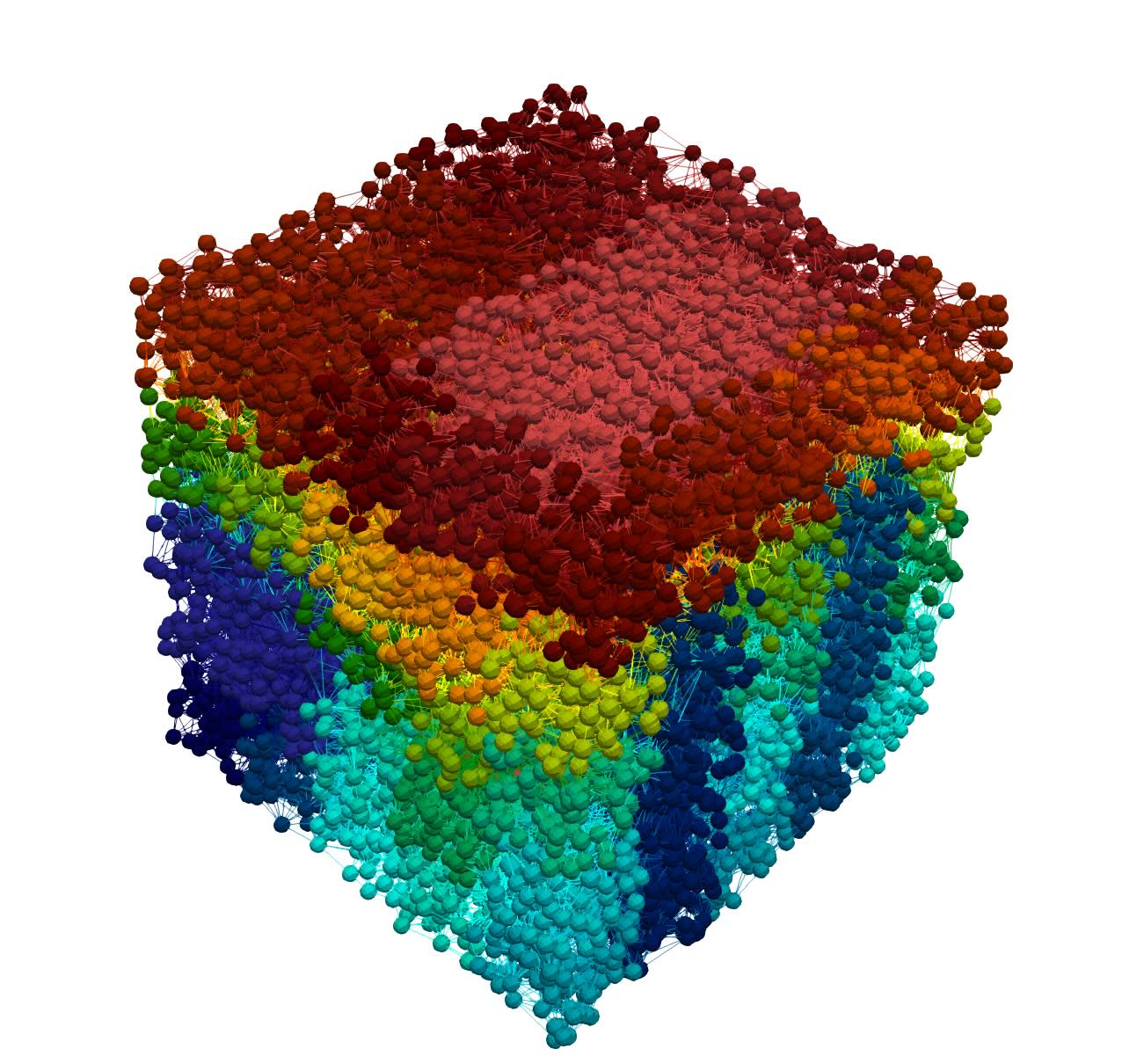}
\includegraphics[width=0.24\linewidth]{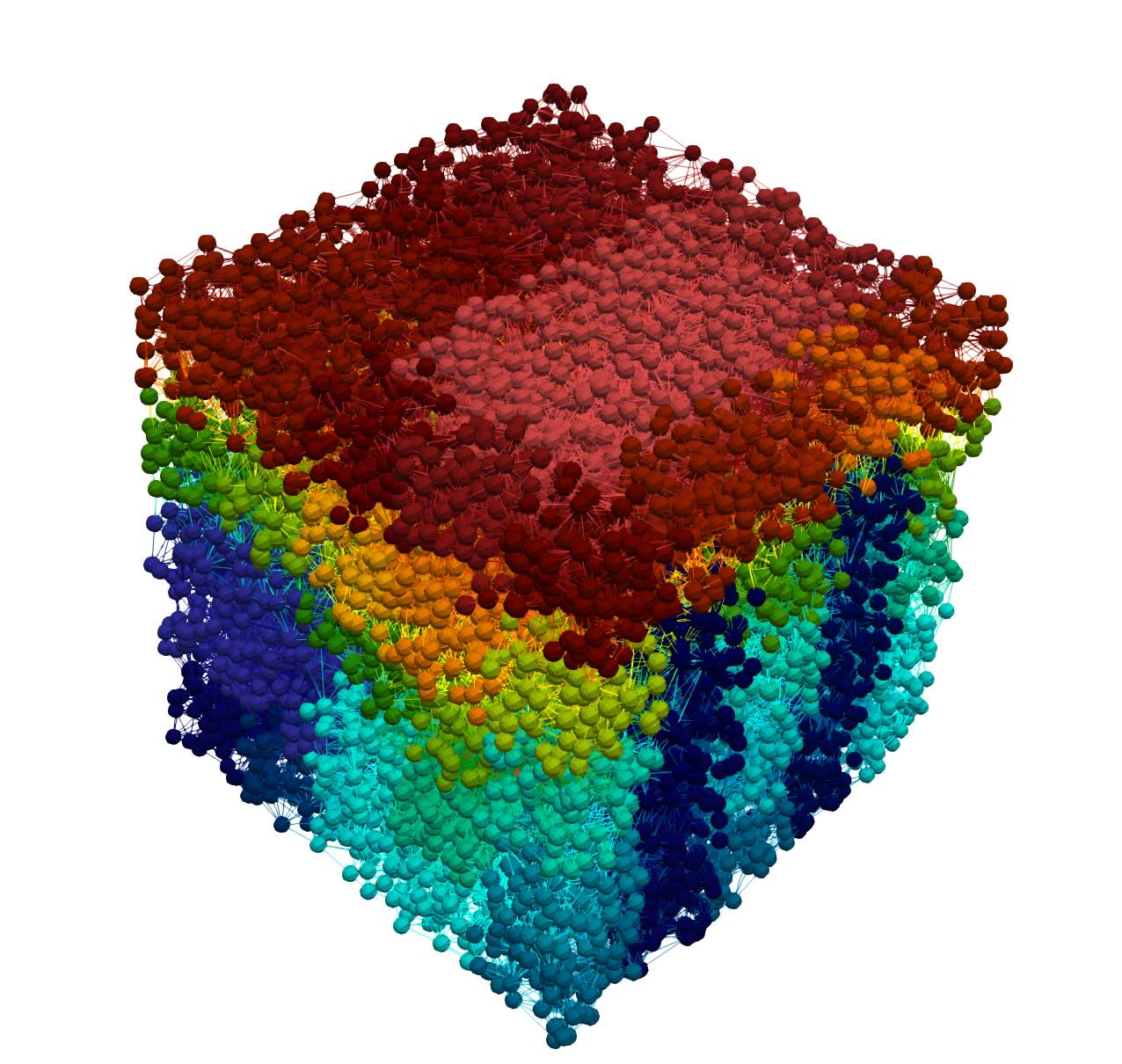}
\caption{\rev{Test-3c: High-contrast properties on Network 3b}}
\end{subfigure}
\caption{\rev{Reference solution for the 3D case on the fine-scale network (first column), multiscale solution (second column), average reference solution on the coarse grid (third column) and upscaled solution on the coarse grid (fourth column).}}
\label{fig:solu3c}
\end{figure}

\begin{table}[h!]
\centering
\begin{tabular}{|c || c|cc|c || c|cc|c || c|cc|c ||}
\hline
\multirow{2}{*}{$DOF_H$} 
& $e_1^H$ & $e_1^h$ & $e_2^h$ & $t_{sol}$
& $e_1^H$ & $e_1^h$ & $e_2^h$ & $t_{sol}$
& $e_1^H$ & $e_1^h$ & $e_2^h$ & $t_{sol}$ \\ \cline{2-13} 
& \multicolumn{4}{|c||}{Test-1a} 
& \multicolumn{4}{|c||}{Test-2a} 
& \multicolumn{4}{|c||}{Test-3a} \\ \hline
25 & 3.33 & 17.36 & 22.95 & 0.049 & 3.16 & 13.01 & 15.44 & 0.052 & 3.72 & 13.08 & 13.58 & 0.035
 \\ \hline
\hline
 $DOF_H$
& \multicolumn{4}{|c||}{Test-1b} 
& \multicolumn{4}{|c||}{Test-2b} 
& \multicolumn{4}{|c||}{Test-3b} \\ \hline
125 & 3.27 & 15.00 & 17.08 & 0.075 & 2.66 & 11.30 & 13.84 & 0.081 & 7.70 & 14.65 & 33.41 & 0.524 \\ \hline
\hline
 $DOF_H$
& \multicolumn{4}{|c||}{Test-1c} 
& \multicolumn{4}{|c||}{Test-2c} 
& \multicolumn{4}{|c||}{Test-3c} \\ \hline
125 & 2.63 & 26.09 & 10.93 & 0.09 & 3.47 & 29.34 & 11.52 & 0.05 & 2.88 & 23.02 & 13.58 & 0.25 \\ \hline
\end{tabular}
\caption{\rev{Solution time and accuracy between fine-scale network model and upscaled coarse network, $DOF_H$ is the number of degrees of freedom on the coarse grid, $t_{sol}$ is solution time in seconds, $e_1^h$ is the relative L$_2$ error on the coarse-grid in percentage, $e_1^h$ and $e_2^h$ are relative L$_2$ and energy errors on the fine-scale network in percentage.}}
\label{table-up}
\end{table}

\begin{table}[h!]
\centering
\begin{tabular}{|| ccc || ccc || ccc ||}
\hline
$DOF_h$ & $N_{conn}$ & $t_{sol}$ &
$DOF_h$ & $N_{conn}$ & $t_{sol}$ &
$DOF_h$ & $N_{conn}$ & $t_{sol}$ \\ \hline
\multicolumn{3}{||c||}{Test-1a} & 
\multicolumn{3}{|c||}{Test-2a} & 
\multicolumn{3}{|c||}{Test-3a} \\ \hline
40000 & 79600 & 9.491 & 46006 & 62538 & 1.854 & 40000 & 119297 & 18.315 \\ \hline
\hline
\multicolumn{3}{||c||}{Test-1b} & 
\multicolumn{3}{|c||}{Test-2b} & 
\multicolumn{3}{|c||}{Test-3b} \\ \hline
15625 & 45000 & 135.487 & 22092 & 43075 & 110.655 & 15625 & 114147 & 246.295 \\ \hline
\hline
\multicolumn{3}{||c||}{Test-1c} & 
\multicolumn{3}{|c||}{Test-2c} & 
\multicolumn{3}{|c||}{Test-3c} \\ \hline
15625 & 45000 & 131.429 & 22072 & 43075 & 113.782 & 15625 & 114147 & 271.428 \\ \hline
\end{tabular}
\caption{\rev{Solution time for fine-scale network model with $DOF_h=N_{pore}$, where $t_{sol}$ is solution time in seconds, $N_{pore}$ and $N_{conn}$ are the number of nodes and connections.}}
\label{table-f}
\end{table}

\begin{table}[h!]
\centering
\begin{tabular}{|c | c || cc|c || cc|c || cc|c ||}
\hline
\multirow{2}{*}{$DOF_H$} & \multirow{2}{*}{M} 
& $e_1^h$ & $e_2^h$ & $t_{sol}$
& $e_1^h$ & $e_2^h$ & $t_{sol}$
& $e_1^h$ & $e_2^h$ & $t_{sol}$ \\ \cline{3-11} 
&
& \multicolumn{3}{|c||}{Test-1a} 
& \multicolumn{3}{|c||}{Test-2a} 
& \multicolumn{3}{|c||}{Test-3a} \\
\hline
36 & 1 & 50.01 & 0.43 & 0.385 & 100.24 & 0.85 & 0.341 & 51.16 & 0.18 & 0.337 \\ \hline
72 & 2 & 29.87 & 0.34 & 0.475 & 49.32 & 0.77 & 0.354 & 26.00 & 0.13 & 0.290 \\ \hline
144 & 4 & 15.58 & 0.24 & 0.425 & 17.53 & 0.66 & 0.400 & 15.98 & 0.10 & 0.381 \\ \hline
288 & 8 & 9.11 & 0.18 & 0.257 & 10.02 & 0.61 & 0.480 & 10.47 & 0.08 & 0.486 \\ \hline
432 & 12 & 7.14 & 0.16 & 0.654 & 3.51 & 0.55 & 0.441 & 6.92 & 0.06 & 0.463 \\ \hline
576 & 16 & 5.26 & 0.13 & 0.604 & 2.34 & 0.52 & 0.542 & 6.16 & 0.06 & 0.622 \\ \hline
864 & 24 & 3.65 & 0.11 & 0.696 & 3.75 & 0.44 & 0.797 & 4.81 & 0.05 & 0.745 \\ \hline
1152 & 32 & 2.75 & 0.10 & 0.792 & 3.70 & 0.41 & 0.804 & 4.06 & 0.05 & 0.749 \\ \hline
\hline
$DOF_H$ & M 
& \multicolumn{3}{|c||}{Test-1b} 
& \multicolumn{3}{|c||}{Test-2b} 
& \multicolumn{3}{|c||}{Test-3b} \\ 
\hline
216 & 1 & 11.70 & 0.47 & 0.177 & 35.99 & 2.26 & 0.120 & 21.55 & 6.30 & 0.162 \\ \hline
432 & 2 & 9.66 & 0.41 & 0.272 & 23.42 & 1.96 & 0.271 & 20.45 & 5.94 & 0.194 \\ \hline
864 & 4 & 7.57 & 0.33 & 0.243 & 11.19 & 1.41 & 0.274 & 17.74 & 5.30 & 0.376 \\ \hline
1728 & 8 & 5.06 & 0.27 & 0.409 & 6.50 & 1.18 & 0.487 & 11.72 & 4.27 & 0.413 \\ \hline
2592 & 12 & 3.15 & 0.20 & 0.539 & 4.34 & 1.05 & 0.605 & 9.31 & 3.77 & 3.409 \\ \hline
3456 & 16 & 2.20 & 0.16 & 2.261 & 3.27 & 0.97 & 2.115 & 7.63 & 3.39 & 2.457 \\ \hline
5184 & 24 & 1.09 & 0.10 & 2.562 & 2.05 & 0.85 & 2.160 & 4.93 & 2.70 & 2.139 \\ \hline
6912 & 32 & 0.64 & 0.07 & 4.328 & 1.56 & 0.78 & 4.086 & 3.11 & 2.11 & 5.594 \\ \hline
\hline
$DOF_H$ & M 
& \multicolumn{3}{|c||}{Test-1c} 
& \multicolumn{3}{|c||}{Test-2c} 
& \multicolumn{3}{|c||}{Test-3c} \\ 
\hline
216 & 1 & 64.66 & 1.06 & 0.15 & 70.97 & 0.97 & 0.15 & 78.18 & 0.43 & 0.18 \\ \hline
432 & 2 & 4.51 & 0.15 & 0.24 & 17.91 & 0.60 & 0.30 & 11.68 & 0.15 & 0.27 \\ \hline
864 & 4 & 2.82 & 0.10 & 0.37 & 6.55 & 0.56 & 0.36 & 5.56 & 0.09 & 0.38 \\ \hline
1728 & 8 & 1.35 & 0.06 & 0.45 & 3.98 & 0.51 & 0.45 & 2.81 & 0.06 & 0.51 \\ \hline
2592 & 12 & 0.85 & 0.04 & 0.72 & 3.47 & 0.48 & 0.64 & 2.09 & 0.05 & 0.64 \\ \hline
3456 & 16 & 0.67 & 0.04 & 3.25 & 2.81 & 0.46 & 4.09 & 1.63 & 0.04 & 2.57 \\ \hline
5184 & 24 & 0.42 & 0.03 & 2.83 & 1.70 & 0.45 & 4.46 & 1.03 & 0.04 & 3.24 \\ \hline
6912 & 32 & 0.26 & 0.02 & 4.88 & 1.15 & 0.44 & 5.33 & 0.69 & 0.03 & 5.44 \\ \hline
\end{tabular}
\caption{\rev{Solution time and accuracy between fine-scale network model and multiscale network, $DOF_H$ is the number of degrees of freedom on the coarse grid, $M$ is number of basis functions, $t_{sol}$ is solution time in seconds, $e_1^h$ and $e_2^h$ are relative L$_2$ and energy errors  on the fine-scale network in percentage.}}
\label{table-ms}
\end{table}

\rev{We set zero source terms and Dirichlet boundary conditions on the top boundary with $g = 1$}.  We then simulate with a fixed number of time steps, $n = 50$. 
The networks are embedded into domain \rev{$\Omega=[0,1]^d$ with $d=2$ and $3$ for 2D and 3D cases. The coarse grid is $5^d$ for both upscaling and multiscale methods. In upscaling, the resulting coarse-scale model has $DOF_H = 5^d$. In the multiscale method, we have $DOF_H = M \cdot 6^d$, where $M$ is the number of multiscale basis functions. In the fine-scale network, we have $DOF_h = N_{pore}$, where $N_{pore}$ is the number of nodes in the network. } 
Numerical implementation of the network model is performed based on the PETSc DMNetwork framework \cite{balay2019petsc, maldonado2017scalable, betrie2018scalable}. The basis functions are constructed using a generalized eigenvalue solver from SLEPc  \cite{hernandez2005slepc}.

\rev{In Figures \ref{fig:solu2},  \ref{fig:solu3} and  \ref{fig:solu3c}, we depict reference solution on the fine-scale network (first column), multiscale solution (second column), average reference solution on the coarse grid (third column) and upscaled solution on the coarse grid (fourth column). 
From the fine-scale solution, we observe a significant influence of the heterogeneity properties from the SPE10 test and for high-contrast properties on the solution in Test-1a/1b and Test1c/2c/3c. However, we observe smoother solutions for Test-2a/2b and 3a/3b with random coefficients. From Test-1c/2c/3c, we observe a significant influence of the subdomains filled with larger pores that give a subdomain with much higher connection weights. Moreover, we observe an influence of the network structure on the solution, where results on Network-2b give a slower flow due to the smaller connectivity of the network, and faster flow is observed in Network-3b with the biggest connectivity. From the results of the multiscale solver, we observe visually a very good solution for all test cases. 

In Table \ref{table-f}, we present the time of solution for the fine-scale network model. 
In the 2D case, we have $DOF_h = 46,006$ for Network-2a, $DOF_h = 40,000$ for Network-1a and 3a.  
In the 3D case, we have $DOF_h = 22,092$ for Network-2b, $DOF_h = 15,625$ for Network-1b and 3b.  
The solution time is presented in Table \ref{table-f}. On the fine-scale, we use a direct solver. The solution time on the fine-scale network depends on the network structures (number of nodes $N_{pore}$ and connectivity $N_{conn}$). The network with stronger connections takes more time to solve. For example, we have solution time $t_{sol} = 9.5$ sec for Network-1a with $DOF_h = 40,000$; for Network-3a with larger connectivity, we have $t_{sol} = 18.3$ sec. In the 3D case, we obtain the same behavior depending on network connectivity with approximately two-times speed-up for a network with 2-times less number of connections, where we have $t_{sol} = 135.5$ sec for Network-1b and $t_{sol} = 246.3$ sec. for Network-3b in Test-3c. We observe the same time of solution between Test-1b/2b/3b and Test-1c/2c/3c  because the same networks are used, and we only vary heterogeneous coefficients. 

In Tables \ref{table-up} and \ref{table-ms}, we present results for the upscaling method and multiscale method.  For the structured coarse grid with quadrilateral cells, in the upscaling method, the coarse grid model in two-dimensional cases contains four connections associated with flow in x and y directions and six connections in the 3D case for flow in x,y, and z directions. The GMsFEM model is related to the finite element method, where we have $C=8$ connections in 2D and $C=26$ for 3D. Therefore, in total, we have $C + (M-1)*(C+1)= C + M(C+1)-C-1 = (C+1)*M-1$ connections in GMsFEM ($C$ connections in current continua plus connection with all other continua in local support). For example, with $M=1$ (one basis function per local domain), we have $(C+1)*M-1=9*2-1=17$ connections in 2D and $(C+1)*M-1=27*2-1=53$ connections in 3D for each interior node. For $M=2$, we have $(C+1)*M-1=9*4-1=35$ connections in 2D and $(C+1)*M-1=27*4-1=107$ connections in 3D. We see that the GMsFEM approximation with a sufficient number of basis functions gives a rich connectivity pattern and provides a good approximation to capture networks with many connections.
By using the upscaling method, we construct a continuum scale model with upscaled parameters and reduce the size to $DOF_H = 25$, leading to very fast calculations. However, because the upscaling method only provides solutions on the coarse grid, we obtain a big error between fine-scale and coarse-scale models when we project the coarse-scale solutions back to fine resolution. We should mention that, in the upscaling method, we can construct a prolongation operator accurately, for example, based on the fine-scale solution of local problems, and obtain a better representation. We will investigate the upscaling method in our future works. 
By constructing a multiscale method in a 3D case, we reduce the size of the system to $DOF_H = 1728$ and $DOF_H = 3456$ for 8  and 16 multiscale basis functions, respectively. In the 2D case, we have $DOF_H = 288$ and $DOF_H = 576$ for 8  and 16 multiscale basis functions, respectively. 
We observe that the error reduces when we take more basis functions, as shown in the convergence analysis. Furthermore, we can reduce error by considering a finer coarse grid. 
Compared with the upscaling method, we can reconstruct a fine-scale representation and use the proposed method as a two-grid preconditioner in future works. 
The size of the resulting coarse-scale system directly affects the solution time. Moreover, we also see that in the multiscale method, we obtain the same connectivity induced by a Galerkin projection for different underlying fine-scale networks. Therefore, the solution time for different networks on the coarse grid is similar. For Network-1b, we obtain very good results with less than one \% of error using 32 multiscale basis functions with solution time $t_{sol} = 4.3$ sec. Compared with the fine-scale solution time $t_{sol} = 135.5$ sec, we obtain a 31 times faster solution. For the example with an unstructured irregular network with very high connectivity (Network-3b), we obtain 44 times faster solution with $t_{sol} = 5.9$ sec. for multiscale method and $t_{sol} = 246.3$ sec. for a fine-scale solution. We also see that the connectivity of the underlying network highly affects the number of basis functions. We will investigate it in future works. Moreover, we will develop online techniques to reduce errors further and interplay with network connectivity. 
Also, we didn't include an investigation of very large network models in this work. We note that a bigger problem will require an interplay between local domain (subnetwork) size and coarse grid size. The size of a coarse grid will affect the compression properties of the multiscale method related to the ability to capture a complex fine-scale behavior in a few degrees of freedom and affect the coarse-system size. Very coarse grids lead to a larger number of spectral bases per local domain and affect the connectivity of the coarse-scale network representation (sparsity of the matrix). On the other side, a finer coarse grid will reduce the number of bases (connectivity) but increase the number of unknowns. Moreover, when we reach a limit of the direct solver for a coarse grid system, it leads to the development of the efficient preconditioner for multiscale approximation or combining it with a multilevel approach with a coarse grid hierarchy. On the other side, we should also have an 'optimal' size of the local domain (subnetwork) for efficient solving of a generalized eigenvalue problem, which can be done in parallel. It opens many interesting questions and research directions; we will investigate this aspect in future works.
}

\section{Conclusion}

We considered a time-dependent model on structured and random networks.  The time approximation was performed using an implicit scheme.  The stability of the semi-discrete and discrete networks was presented. 
The multiscale method for the network model was developed and analyzed for network models.  The proposed approach is based on the generalized multiscale finite element method.  To find a multiscale basis function, we solve local eigenvalue problems in sub-networks.  
Convergence analysis of the proposed method was presented for semi-discrete and discrete network models with stability estimates.   
Numerical results were presented for structured and random heterogeneous networks to confirm the theory.

\subsection*{\rev{Appendix. Upscaling of the network model}}\label{appendx}

\rev{For problems in non-periodic media, numerical homogenization (or upscaling) is used to calculate effective properties in each local domain \cite{vasilyeva2020learning, vasilyeva2021machine, vasilyeva2019convolutional}. 
We consider a coarse grid $\mathcal{T}_H$ with coarse cell $K_i$. We let $\mathcal{E}_H$ be the set of all faces of the coarse grid and $N_e$ be the total number of coarse faces. 
We let $E_{ij}$ be the coarse grid face, and we define the local domain by
\[
\omega_{E_{ij}} = K_i \cup K_j, \quad K_i, K_j \in \mathcal{T}_H,
\]
where $\omega_{E_{ij}}$ is a union of two coarse cells, when $E_{ij}$ lies in the interior of the domain $\Omega$. 
}

\begin{figure}[h!]
\centering
\includegraphics[width=0.5\linewidth]{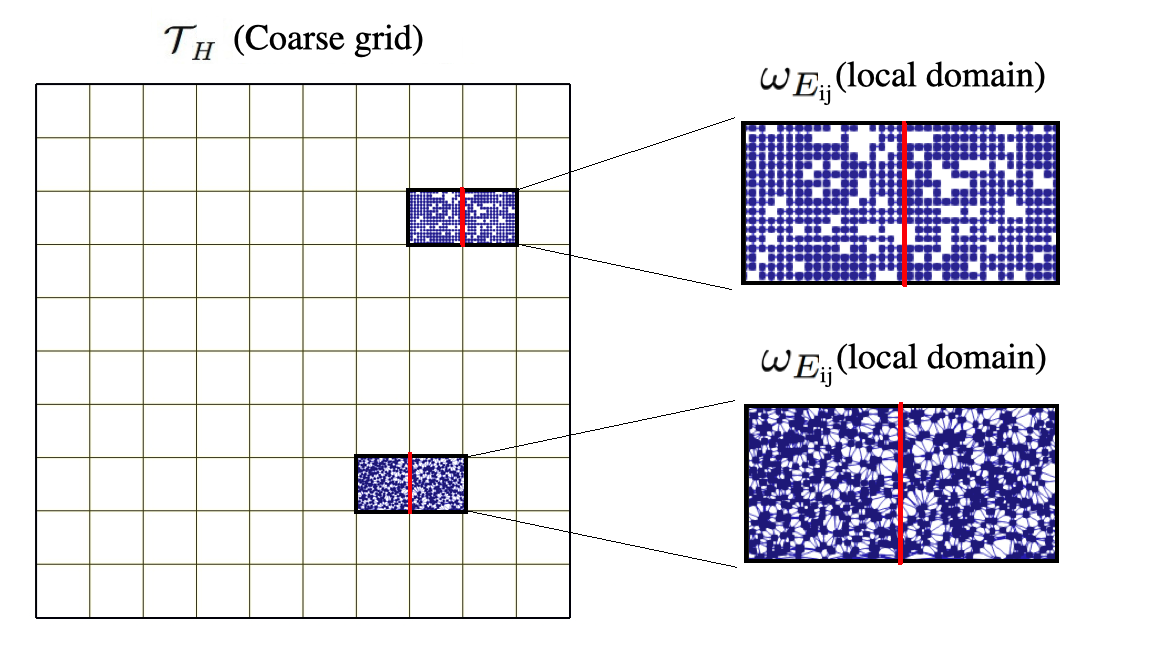}
\caption{\rev{Illustration of local domain in upscaling algorithm with corresponded subnetwork $G^{\omega_{E_{ij}}}$.}}
\label{fig:up}
\end{figure}

\rev{
We formulate a coarse-scale problem as follows
\begin{equation}
\bar{c}_i \frac{ \bar{u}^{n}_i - {\bar{u}}^{n-1}_i }{\tau} 
+ \sum_{j}  \bar{w}_{ij}  (\bar{u}^{n}_i -\bar{u}^{n}_j) = \bar{f}_i, 
\end{equation}
where $\bar{c}_i$ and $\bar{w}_{ij}$ are the upscaled coefficients. 
Note that this formulation is related to the finite volume method on the coarse grid, which leads to the continuum scale model on a coarse level with effective properties. 

For calculation of the upscaled connection weight $\bar{w}_{ij}$ for coarse face $E_{ij}$, we solve local problems in corresponding subnetwork with linear boundary conditions in $G^\omega_{E_{ij}}$. We solve  the following local problem
\begin{equation}
L^{\omega_{E_{ij}}} u^{\omega_{E_{ij}}} = 0,
\end{equation}
with Dirichlet boundary conditions on the inflow and outflow boundaries ($u = g$ with $g = 1$ for the inflow boundary and $g=0$ for the outflow boundary) and set zero fluxes for all other boundaries. 
We have two types of face orientation in the 2D case ($x$, $y$) and three orientations in the 3D case that correspond to the flow in $x$, $y$, and $z$ directions. For the flow in $x$-direction, we define the left boundary as an inflow boundary and the right boundary as an outflow; the rest of the boundaries are marked as a boundary with zero flow boundary conditions.

We calculate upscaled coefficient $\bar{w}_{ij}$ using average flux on the interface between two cells 
\begin{equation}
\bar{w}_{ij} = \frac{ \bar{q}_{ij} }{\bar{u}^{\omega_{E_{ij}}}_i - \bar{u}^{\omega_{E_{ij}}}_j}, 
\quad 
\bar{q}_{ij} = \sum_{l \sim n} w^{\omega_{E_{ij}}}_{ln}(u^{\omega_{E_{ij}}}_l - u^{\omega_{E_{ij}}}_n), 
\quad 
\bar{u}_i^{\omega_{E_{ij}}} = \frac{\sum_{\xi_l \in G^{K_i}} u_l^{\omega_{E_{ij}}} |\xi_i| }{\sum_{\xi_l \in G^{K_i}} |\xi_i|}, 
\end{equation}
and upscaled coefficient $\bar{c}_i$ is calculated  as follows $\bar{c}_i = \sum_{\xi_l \in G^{K_i}} c_l$ with $c_l = \xi_l$. 
Here $\xi_l$ is the fine-scale mode in $G^{\omega_{E_{ij}}}$, $|\xi_i|$ is the node volume, 
$j \sim  i$ if nodes $\xi_i$ and $\xi_j$ are connected  and $\xi_l \in G^{K_i}$, $\xi_n \in G^{K_j}$. 
For the interfaces on the boundary with zero flux boundary conditions, we will use no flux boundary conditions and, therefore, not need to calculate the upscaled connection weight. For the faces on the boundary where we set Dirichlet boundary conditions, we calculate the upscaled coefficient in the coarse cell that has a given global boundary with fixed boundary conditions. 
}

\bibliographystyle{unsrt}
\bibliography{lit}

\end{document}